\documentclass[11pt]{amsart}
\usepackage{amsmath, amsthm, amsfonts, amssymb, graphicx, bbm} 
\usepackage{bm,verbatim,mathrsfs}
\usepackage{enumerate,tabularx}
\usepackage{pdfpages}

\DeclareFontFamily{U}{wncy}{}
    \DeclareFontShape{U}{wncy}{m}{n}{<->wncyr10}{}
    \DeclareSymbolFont{mcy}{U}{wncy}{m}{n}
    \DeclareMathSymbol{\Sha}{\mathord}{mcy}{"58} 
    \DeclareMathSymbol{\Chu}{\mathord}{mcy}{"51}
    \DeclareMathSymbol{\Tsu}{\mathord}{mcy}{"43} 
    \DeclareMathSymbol{\Yu}{\mathord}{mcy}{"10}
    \DeclareMathSymbol{\yu}{\mathord}{mcy}{"18}  
    \DeclareMathSymbol{\Zhu}{\mathord}{mcy}{"11} 
    \DeclareMathSymbol{\zhu}{\mathord}{mcy}{"19}
    \DeclareMathSymbol{\Ya}{\mathord}{mcy}{"17}
    \DeclareMathSymbol{\cyrl}{\mathord}{mcy}{"6C} 
    \DeclareMathSymbol{\cyrL}{\mathord}{mcy}{"4C}

\theoremstyle{plain}
\newtheorem{thm}{Theorem}[section]
\newtheorem{lem}[thm]{Lemma}
\newtheorem{prop}[thm]{Proposition}
\newtheorem{cor}[thm]{Corollary}

\newtheorem{defn}[thm]{Definition}
\newtheorem*{rmk}{Remark}
\numberwithin{equation}{section}

\usepackage{color}

\usepackage{tikz}

\DeclareMathAlphabet{\curly}{U}{rsfs}{m}{n}  

\newcommand{\ZZ}{{\mathbb Z}}

\newcommand{\RR}{{\mathbb R}}
\newcommand{\CC}{{\mathbb C}}
\newcommand{\NN}{{\mathbb N}}

\makeatletter
\renewcommand{\pmod}[1]{\allowbreak\mkern7mu({\operator@font mod}\,\,#1)}
\makeatother
\newcommand{\be}{\begin{equation}}
\newcommand{\ee}{\end{equation}}

\renewcommand{\a}{\ensuremath{\alpha}}
\renewcommand{\b}{\ensuremath{\beta}}

\newcommand{\eps}{\ensuremath{\varepsilon}}

\renewcommand{\le}{\leqslant}

\renewcommand{\ge}{\geqslant}

\newcommand{\ssum}[1]{\sum_{\substack{#1}}}  
\newcommand{\sprod}[1]{\prod_{\substack{#1}}}  
\newcommand{\mint}[1]{\idotsint\limits_{\substack{#1}}} 

\newcommand{\fl}[1]{{\ensuremath{\left\lfloor {#1} \right\rfloor}}}
\newcommand{\cl}[1]{{\ensuremath{\left\lceil #1 \right\rceil}}}
\newcommand{\order}{\asymp}      
\renewcommand{\(}{\left(}
\renewcommand{\)}{\right)}

\newcommand{\pfrac}[2]{\left(\frac{#1}{#2}\right)}  

\newcommand{\emptyvec}{\varnothing}  

\newcommand{\balpha}{\ensuremath{\boldsymbol{\alpha}}}
\newcommand{\bbeta}{\ensuremath{\boldsymbol{\beta}}}

\newcommand{\bphi}{\ensuremath{\boldsymbol{\phi}}}
\newcommand{\bpsi}{\ensuremath{\boldsymbol{\psi}}}
\newcommand{\bxi}{\ensuremath{\boldsymbol{\xi}}}

\newcommand{\one}{\ensuremath{\mathbbm{1}}}  
\newcommand{\bone}{\ensuremath{\mathbf{1}}}  

\newcommand{\ssc}[2]{\ensuremath{{#1}_{#2}^{\phantom{2}}}}
\newcommand{\epszero}{\ssc{\eps}{0}}

\newcommand{\cA}{\ensuremath{\mathcal{A}}}
\newcommand{\cB}{\ensuremath{\mathcal{B}}}
\newcommand{\cC}{\ensuremath{\mathcal{C}}}
\newcommand{\cD}{\ensuremath{\mathcal{D}}}
\newcommand{\cE}{\ensuremath{\mathcal{E}}}

\newcommand{\cG}{\ensuremath{\mathcal{G}}}
\newcommand{\cH}{\ensuremath{\mathcal{H}}}
\newcommand{\cI}{\ensuremath{\mathcal{I}}}
\newcommand{\cJ}{\ensuremath{\mathcal{J}}}
\newcommand{\cK}{\ensuremath{\mathcal{K}}}
\newcommand{\cL}{\ensuremath{\mathcal{L}}}
\newcommand{\cM}{\ensuremath{\mathcal{M}}}
\newcommand{\cN}{\ensuremath{\mathcal{N}}}

\newcommand{\cP}{\ensuremath{\mathcal{P}}}
\newcommand{\cQ}{\ensuremath{\mathcal{Q}}}
\newcommand{\cR}{\ensuremath{\mathcal{R}}}
\newcommand{\cS}{\ensuremath{\mathcal{S}}}
\newcommand{\cU}{\ensuremath{\mathcal{U}}}
\newcommand{\cT}{\ensuremath{\mathcal{T}}}
\newcommand{\cV}{\ensuremath{\mathcal{V}}}
\newcommand{\cW}{\ensuremath{\mathcal{W}}}

\newcommand{\cZ}{\ensuremath{\mathcal{Z}}}

\newcommand{\sF}{\ensuremath{\mathscr{F}}}
\newcommand{\sG}{\ensuremath{\mathscr{G}}}

\newcommand{\sL}{\ensuremath{\mathscr{L}}}
\newcommand{\sM}{\ensuremath{\mathscr{M}}}

\newcommand{\bd}{\ensuremath{\mathbf{d}}}

\newcommand{\bmm}{\ensuremath{\mathbf{m}}}

\newcommand{\bn}{\ensuremath{\mathbf{n}}}

\newcommand{\bt}{\ensuremath{\mathbf{t}}}
\newcommand{\bu}{\ensuremath{\mathbf{u}}}
\newcommand{\bv}{\ensuremath{\mathbf{v}}}

\newcommand{\bx}{\ensuremath{\mathbf{x}}}
\newcommand{\by}{\ensuremath{\mathbf{y}}}
\newcommand{\bz}{\ensuremath{\mathbf{z}}}

\newcommand{\tf}{\widetilde{f}}
\newcommand{\tl}{\ensuremath{\widetilde{\lambda}}}

\newcommand{\ft}{\tf}  

\newcommand{\two}{\ensuremath{[\theta,\theta+\nu]}}

\newcommand{\CB}{C_{\text{bd}}}

\begin{document}

\setlength{\extrarowheight}{3pt}  

\title{On the theory of prime-producing sieves}
\author{Kevin Ford}
\address{KF: Department of Mathematics, University of Illinois at Urbana-Champaign}
\author{James Maynard}
\address{JM: Mathematical Institute, University of Oxford}
\email{james.alexander.maynard@gmail.com}
\date{\today}

\begin{abstract}
We develop the foundations of a general framework for producing optimal upper and
lower bounds on the sum $\sum_p a_p$ over primes $p$, where $(a_n)_{x/2<n\le x}$ is an arbitrary  non-negative sequence satisfying Type I and Type II estimates. 

Our lower bounds on $\sum_p a_p$ depend on a new
sieve method, which is non-iterative and uses all of the Type I and Type II information
at once.  We also give a complementary general procedure for constructing sequences $(a_n)$
satisfying the Type I and Type II estimates, which in many cases proves that our lower 
bounds on $\sum_p a_p$ are best possible.  A key role in both the sieve method and
the construction method is played by the geometry of special subsets 
of $\RR^k$.

This allows us to determine precisely the ranges of Type I and Type II estimates for which an asymptotic for $\sum_p a_p$ is guaranteed, that a substantial Type II range is always necessary to guarantee a non-trivial lower bound for $\sum_p a_p$, and to determine the optimal bounds in some naturally occurring families of parameters from the literature. We also demonstrate that the optimal upper and lower bounds for $\sum_p a_p$ exhibit many discontinuities with respect to the Type I and Type II ranges, ruling out the possibility of a particularly simple characterization.
\end{abstract}

\keywords{Sieves, primes, optimality, }
\subjclass[2010]{Primary 11N05, 11N35}

\maketitle

{\Large
\section{Introduction}\label{sec:Introduction}
}
The main technique for estimating the number of primes in a finite set which doesn't possess particular multiplicative structure is the method of Type I/Type II sums. This allows one to obtain an asymptotic estimate or a non-trivial lower bound for the number of primes in the set provided one has a suitably good understanding of the behaviour of the set in arithmetic progressions (a `Type I' estimate) and provided one has suitably good control of certain bilinear sums associated to the set (a `Type II' estimate). Unfortunately, the sieve process which one uses to translate these Type I and Type II estimates to arithmetic information about primes in the set is still poorly understood; there is limited understanding of how strong 
these estimates need to be in order to detect primes, or what are the best possible bounds on the number of primes in the set for given Type I and Type II estimates.

More generally, we consider a sequence $(a_n)_{x/2<n\le x}$ of non-negative weights, normalized to have average value about 1, and we wish to estimate the sum $\sum_p a_p$ over primes $p$. 
A typical example would be when $a_n$ is the normalized indicator function of a set 
of positive integers. 
It is often convenient to show that the size of $\sum_p a_p$ is similar to $\sum_p b_p$ for a simpler comparison sequence $(b_n)$ rather than attempt to directly estimate $\sum_p a_p$. In particular, $(b_n)$ should have similar distributional properties to $(a_n)$ but $\sum_p b_p$ should be relatively easy to bound via some version of the prime number theorem, and our goal is to bound $\sum_p a_p$ in terms
of $\sum_p b_p$.

The `Type I' and `Type II' estimates state that the sequence $w_n=a_n-b_n$ has average zero in the following strong sense. Assume that for some constants $\gamma,\theta,\nu$ with
 \begin{equation}\label{Q0}
 0<\gamma<1, \quad 0\le \theta < \tfrac12,  \quad 0 < \nu \le 1-\theta,
\end{equation} 
 and some large constant $B>0$, we have the following estimates:
\begin{itemize}
\item (Type I range $[0,1-\gamma]$): We have
\begin{equation}\label{eq:TypeI}
\sum_{m\le x^{\gamma}} \tau^B(m) \max_{\text{interval \cI}} \Big| \ssum{x/2<mn\le x \\ n\in \cI} w_{mn} \Big| \le \frac{x}{\log^{B}{x}}. \tag{I}
\end{equation}
\item (Type II range $[\theta,\theta+\nu]$): For any complex numbers $\xi_m, \kappa_n$ with $|\xi_m| \le \tau^B(m)$ and $|\kappa_n| \le \tau^B(n)$ for $m,n\in \NN$, we have
\begin{equation}\label{eq:TypeII}
\bigg|\sum_{\substack{ (x/2)^{\theta} < m\le x^{\theta+\nu} \\ x/2<mn \le x}} \xi_m \kappa_n w_{mn} \bigg| \le \frac{x}{\log^{B}{x}}.\tag{II}
\end{equation}
\end{itemize}

Under these hypotheses, an estimate such as Vaughan's identity can be used to show that if $\gamma+\nu>1$ and $B$ is sufficiently large then
\[
\sum_{x/2<n\le x}\Lambda(n)a_n=\sum_{x/2<n\le x}\Lambda(n)b_n+O\Bigl(\frac{x}{(\log{x})^{B/3}}\Bigr),
\]
so the condition $\gamma+\nu>1$ gives a sufficient condition to get a good estimate for $\sum_p a_p$ provided prime powers are negligible. Unfortunately, in many applications it is difficult (or impossible) to produce Type I and Type II ranges with $\gamma+\nu>1$ and so it is important to obtain non-trivial results with rather weaker assumptions. 
In some cases the Heath-Brown identity gives an asymptotic formula from \eqref{eq:TypeI} and \eqref{eq:TypeII} even when Vaughan's identity fails to do so. 
On the other hand, Selberg \cite{Selberg} showed that whenever $\nu=0$, there are examples of $a_n$ with $b_n=1$ for all $n$ which satisfy \eqref{eq:TypeI} for arbitrary $\gamma<1$ but with $\sum_p a_p=0$, so non-trivial Type II information is necessary to detect primes
(specifically, take $a_n=1+\lambda(n)$ and $b_n=1$ for all $n$, where $\lambda$ is the Liouville function). 
  Harman's sieve \cite{Harman} is a technique developed to get non-trivial lower bounds for $\sum_p a_p$ with weaker assumptions on the Type I and Type II ranges, since lower bounds are often sufficient for many applications. With this in mind, we let $C^+(\gamma,\theta,\nu)$ and $C^-(\gamma,\theta,\nu)$ be the smallest and largest constants such that
\[
\Bigl(C^-(\gamma,\theta,\nu)+o(1)\Big)\sum_{p}b_p\le\sum_{p}a_p\le \Bigl(C^+(\gamma,\theta,\nu)+o(1)\Bigr)\sum_{p}b_p
\]
for any non-negative sequence $a_n$ with $w_n=a_n-b_n$ satisfying \eqref{eq:TypeI} and \eqref{eq:TypeII}. 
(To make this precise we of course need some assumptions on $b_n$; 
see Subsections \ref{sec:bn-assumptions} and \ref{sec:C-def} for our precise setup.) 
Taking $a_n=b_n=1$ for all $n$ shows that
$C^-(\gamma,\theta,\nu)\le 1\le C^+(\gamma,\theta,\nu)$.
We have an asymptotic formula for $\sum_p a_p$ whenever $C^-(\gamma,\theta,\nu)=C^+(\gamma,\theta,\nu)=1$ and a non-trivial lower bound for primes whenever $C^-(\gamma,\theta,\nu)>0$. 

The main aim of this paper is to introduce a new method to analyze the constants $C^\pm(\gamma,\theta,\nu)$ for general parameters $\gamma,\theta,\nu$.  In contrast with
the largely ad hoc methods of many previous works, especially those relying on the iterative techniques of the Harman 
sieve, we argue directly, deploying \emph{all} of the Type I and 
Type II information at once.
As a consequence, we are able to determine the precise value of the constants in various regimes, thereby both improving previous estimates and demonstrating general limitations of the Type I/Type II setup. 

\medskip

\medskip

\subsection{Examples from the literature.}\label{sec:examples-from-literature}
Table 1 illustrates a few examples of parameters where the author(s) have been successful in counting primes in particular sets.  
In all but the first two examples, we have $\gamma+\nu<1$.
In the first, second and seventh example an asymptotic formula was proven, while in the other examples the author(s) found lower bounds on $\sum_p a_p$ of
the expected order of magnitude.\footnote{Strictly speaking, several of these results only established \eqref{eq:TypeII} for special coefficients $\xi_m,\kappa_n$, such as those that do not correlate with any character of conductor $(\log{x})^{O(1)}$, or coming out of an explicit sieve. A mild generalization of the underlying methods should allow one to establish \eqref{eq:TypeI} and \eqref{eq:TypeII} in full. 
In the interests of simplicity we will ignore this technicality.}

\begin{table}[ht]\label{table:examples}
\setlength{\extrarowheight}{4pt}  
\begin{tabular}{|c|c|c|l|l|}
\hline
$\gamma$ & $\theta$ & $\nu$ & Reference & Property of the primes $p$  \\
\hline
$\frac34$ & $\frac14$ & $\frac12$ & Friedlander-Iwaniec \cite{FI98} & $p=x^2+y^4$\\
$\frac23$ & $\frac13$ & $\frac13$ & Heath-Brown \cite{HB01} & $p=x^3+2y^3$ \\
$\frac{19}{28}$ & $\frac{9}{28}$ & $\frac{1}{28}$ & Jia \cite{Jia} & $\{ \alpha p\} < p^{-9/28}$ \\
$\frac{16}{25}$ & $\frac{9}{25}$ & $\frac{1}{16}$ & Maynard \cite{MDP} & $p$ missing a digit in base 10\\
$\frac34$ & $\frac14$ & $\frac1{12}$ & Merikoski \cite{Merikoski}, Thm. 1 & $p=x^2+(y^2+1)^2$ \\
$\frac56$ & $\frac16$ & $\frac1{18}$ &  Merikoski \cite{Merikoski}, Thm. 2 & $p=x^2+(y^3+z^3)^2$ \\
$\frac12$ & $0$ & $\frac13$ & Duke-Friedlander-Iwaniec \cite{DFI} & $x^2\equiv a\pmod{p}$ for $x/p$ in a short interval  \\
$\frac12$ & $0$ & $\frac15$ & Sarnak-Ubis \cite{SU} & dynamical systems at prime times \\
\hline
\end{tabular}
\medskip
\caption{Examples from the literature (epsilons omitted)}
\end{table}

\bigskip

%
%
{\Large \section{Main Results}\label{sec:main-results}}
%
%

As mentioned in the introduction, the main contribution of this paper is to introduce a new framework to study the method of Type I/II sums, developed in Sections \ref{sec:constructions} and \ref{sec:sieving}. To demonstrate the benefits of this framework, we first give various consequences on the optimality and limitations of the method of Type I/II sums which follow from this approach.

Throughout the paper, we assume that the sequences $(a_n)$,
$(b_n)$ and $(w_n)$ are supported on integers in $(x/2,x]$.
In addition to the Type I bound \eqref{eq:TypeI}
and Type II bound \eqref{eq:TypeII}, we postulate mild growth
conditions on the sequence $(w_n)$ which depends on a parameter $\varpi$ and the quantity $\nu$ from \eqref{eq:TypeII}:
\begin{equation}\label{w}
\sum_n |w_n|\tau(n) \le x(\log x)^{\varpi}, \qquad w_n\ge -x^{\nu/10}\;\; (x/2<n\le x).
\tag{$w$}
\end{equation}

\subsection{Minimal Type II range}

Our first result shows that for any $\gamma<1$ there is a minimum amount of Type II information, measured by $\nu$,
required to detect primes.

\begin{thm}[Minimal Type II range]\label{thm:TypeII-minimum}
For all $\gamma<1$, there is a constant $\ssc{\nu}{0}(\gamma)>0$ such that the following holds. If $B>0$ and $(\gamma,\theta,\nu)$ satisfies \eqref{Q0} with $\nu\le \ssc{\nu}{0}(\gamma)$, then for all $x$ suffiicently large (in terms of $B,\gamma,\theta,\nu$), there is a bounded non-negative sequence $(a_n)$ such that
$w_n=a_n-1$ satisfies \eqref{eq:TypeI} and \eqref{eq:TypeII}, but $a_p=0$ for all primes $p$.
In particular,
\[
C^-(\gamma,\theta,\nu)=0.
\]
\end{thm}

Theorem \ref{thm:TypeII-minimum} gives the first examples (at least with $\gamma \ge 1/2$ and $\nu>0$) of sequences satisfying non-trivial Type I and Type II estimates but not containing primes, thereby putting limitations on when a non-trivial lower bound can be obtained using the method of Type I/II sums. In Theorem \ref{thm:TypeII-minimum} the comparison sequence is simply $b_n=1$.

The principal tool in the proof is the construction of a family of functions
which resemble the Liouville function and may be of independent interest.  These are fed into our general method of constructing example sequences.  Theorem \ref{thm:TypeII-minimum} will
be proven in Section \ref{sec:TypeII-minimum}.

\bigskip

\subsection{Asymptotic for primes}
Our second result gives a simple criterion (both necessary and sufficient) for when  an asymptotic formula is guaranteed from Type I/II estimates. It turns out that after certain relatively straightforward reductions, it suffices to consider $(\gamma,\theta,\nu)\in\cQ$, where
\begin{equation}\label{Q2}
\cQ := \Bigg\{ (\gamma,\theta,\nu) : \;\;
\begin{matrix}
\tfrac12 \le \gamma \le 1-\theta-\nu \;\; \text{ or } \;\; 1-\theta\le \gamma<1, \\
0 \le \theta < \theta+\nu \le \tfrac12 \;\; \text{ or } \;\; 0\le \theta<\tfrac12, \theta+\nu=1-\theta
\end{matrix}
\Bigg \}.
\end{equation}
This reduction is explained in Section \ref{sec:Setup} alongside the formal definition of $C^\pm (\gamma,\theta,\nu)$. 
The proof of Theorem \ref{thm:asymptotic} will be given in Section \ref{sec:asymptotic}.

\begin{thm}[Asymptotic]\label{thm:asymptotic}
Suppose that $(\gamma,\theta,\nu) \in \cQ$ and  define $M=\fl{1/(1-\gamma)}$. Let \eqref{eq:A1} and \eqref{eq:A2} be the claims
\begin{align}
\text{For all integers } n&\ge M+1, \exists a\in \NN  \text{ so that } \tfrac{a}{n}\in [\theta,\theta+\nu].\tag{A1}\label{eq:A1} \\
\text{For some positive integer } h &, h(1-\gamma) \in [\theta,\theta+\nu] \cup [1-\theta-\nu,1-\theta]. \tag{A2}\label{eq:A2}
\end{align}
Then:
\begin{itemize}
\item[(a)]
If both \eqref{eq:A1} and \eqref{eq:A2} hold, then for any $A>1$ and $\varpi\ge 1$, whenever $B$ is large enough in terms of 
$A,\varpi,\gamma,\theta,\nu$ and $(w_n)$ satisfies  \eqref{eq:TypeI}, \eqref{eq:TypeII} and \eqref{w}, we have
 \[
 \sum_{p} w_p \; \ssc{\ll}{A}\; \frac{x}{(\log x)^A}.
 \]
 In particular, $C^-(\gamma,\theta,\nu)=C^+(\gamma,\theta,\nu)=1$.
\item[(b)] If either \eqref{eq:A1} or \eqref{eq:A2} fails, then
 there is a constant $\delta>0$ such that for any $B>0$ and all $x$ sufficiently large in terms of $B,\gamma,\theta,\nu$, there 
 are bounded, non-negative sequences $(a_n^\pm)_{x/2<n\le x}$ with $w_n^\pm = a_n^\pm-1$ satisfying
 \eqref{eq:TypeI} and \eqref{eq:TypeII} and with
 \[
 \sum_p a_p^- \le (1-\delta) \sum_{x/2<p\le x} 1, \qquad \sum_p a_p^+ \ge (1+\delta)
 \sum_{x/2<p\le x}1.
 \] 
 In particular, $C^-(\gamma,\theta,\nu)<1<C^+(\gamma,\theta,\nu)$.
\end{itemize}
\end{thm}

For a given triple $(\gamma,\theta,\nu)\in \cQ$, conditions
\eqref{eq:A1} and \eqref{eq:A2} are simple to check, since the conclusion of \eqref{eq:A1} is always true
when $n \ge 1/\nu$, and we need only check $h\le M+1$ in \eqref{eq:A2}
as $(M+1)(1-\gamma)>1$. For part (b), we note that a bounded sequence $(w_n)$ always satisfies \eqref{w}
for $\varpi=2$ when $x$ is large enough.
From (a), if \eqref{eq:A1} and \eqref{eq:A2} both hold then
 we obtain the asymptotic $\sum_p a_p\sim \sum_p b_p$, provided that
$\sum_p b_p \gg x(\log x)^{-D}$ for a constant $D$, and $B$ is large enough. 
Informally, Theorem \ref{thm:asymptotic} can be thought of as showing that one can obtain an asymptotic formula for primes precisely when the Heath-Brown identity would give one, and a simple combinatorial classification of when this is the case.

In Table \ref{table:examples}, an asymptotic formula was obtained in the first, second and seventh examples. If we ignore technicalities related to whether the results are obtained with $\epsilon$ losses in the parameters (which is addressed below), a short computation shows that in these three cases both \eqref{eq:A1} and \eqref{eq:A2} hold, whereas in each of the remaining examples \eqref{eq:A1} fails but \eqref{eq:A2} holds. See Subsection \ref{subsec:A1A2} for further examples and discussion.

\bigskip

\subsection{Continuity and discontinuity of $C^\pm(\gamma,\theta,\nu)$ near the asymptotic region.}

Typically, the Type I  bound (I) and Type II bound (II) are proven
for the set of triples $P_\eps := (\gamma-\eps,\theta+\eps,\nu-2\eps)$ for some
fixed $P=(\gamma,\theta,\nu)$, where $\eps>0$ is arbitrary
and $x$ is large enough as a function of $\eps$.
In many cases we have $\lim_{\eps\to 0^+} C^\pm (P_\eps) = C^{\pm} (P)$.
However, we have discovered that there are many points $P$
where the functions $C^{\pm}$ are \emph{discontinuous}.

The nature of the situation is slightly different when $\gamma=1/2$, so first we restrict to $\gamma>1/2$ and define
\begin{align*}
\cA&:=\{P=(\gamma,\theta,\nu)\in \cQ:\,\gamma>1/2,\, C^{\pm}(\gamma,\theta,\nu)=1\},\\
\cA^*&:=\{P\in \cA:\,P_\eps\notin \cA\;\forall \eps \in (0,\nu/2)\}.
\end{align*}
Thus $\cA$ is the set of triples (with $\gamma>1/2$) for which the asymptotic holds and $\cA^*$ is the set of boundary points of $\cA$. Clearly the functions $C^{\pm}$ are continuous on the interior of $\cA$ since they are identically 1 there, so any discontinuities on $\cA$ must occur on $\cA^*$.
\begin{thm}[Continuity-discontinuity for $\gamma>\frac12$]\label{thm:continuity}
Suppose that $P=(\gamma,\theta,\nu) \in \cA^*$. Let \eqref{eq:B} be the claim
\[
\exists \, h\in \NN\, : \,  h(1-\gamma) \in [\theta,\theta+\nu) \cup [1-\theta-\nu,1-\theta).
\tag{B}\label{eq:B}
\]
Then we have the following.
\begin{itemize}
\item[(a)] If either $\theta=0$ or if \eqref{eq:B} holds, then
\begin{equation}\label{C-continuous}
C^{\pm}(P_\eps) = 1 + \ssc{O}{P}(\eps).
\end{equation}
\item[(b)] If $\theta>0$ and \eqref{eq:B} fails, then
\begin{equation}\label{C-discontinuous}
\sup_{\eps>0}\, C^-(P_\eps) < 1 < \inf_{\eps>0}\, C^+(P_\eps).
\end{equation}
\end{itemize}
\end{thm}

Theorem \ref{thm:continuity} will be proven in Section \ref{sec:continuity}.
A simple consequence of Theorem \ref{thm:continuity} is that the functions  $C^\pm(P)$
are continuous at $P\in \cA^*$ if and only if $(B)$ holds or $\theta=0$. We see that in the first and second rows of Table \ref{table:examples} when an asymptotic is obtained and $\gamma>1/2$, condition \eqref{eq:B} holds, and so an asymptotic formula follows from the Type I and Type II estimates for $P_\epsilon$ (for all $\epsilon>0$); this is essentially what was actually established in these papers.

When $\nu=0$ the value of $C^+(\gamma,\theta,\nu)$ is closely related to the upper bound function of the linear sieve, which satisfies a delay-differential equation. One might have hoped that a full theory of prime detecting sieves would correspondingly produce constants $C^{\pm}$ which satisfy similar relations. The presence of many discontinuities in the constants means that there is no simple generalization of the delay-differential equation when one incorporates Type II information. 

As far as we are aware, this is the first time that genuine discontinuities have been shown to exist (at least when $\gamma>1/2$). It has been observed in the past that Harman's sieve can produce bounds which are discontinuous when extra arithmetic information is also included (see, for example, the comments at the end of \cite{BI17}), but typically the sieve bounds of Harman's sieve are continuous when just Type I and Type II information is involved.

\begin{figure}[ht!]\label{figure-1}
\includegraphics[height=3.5in]{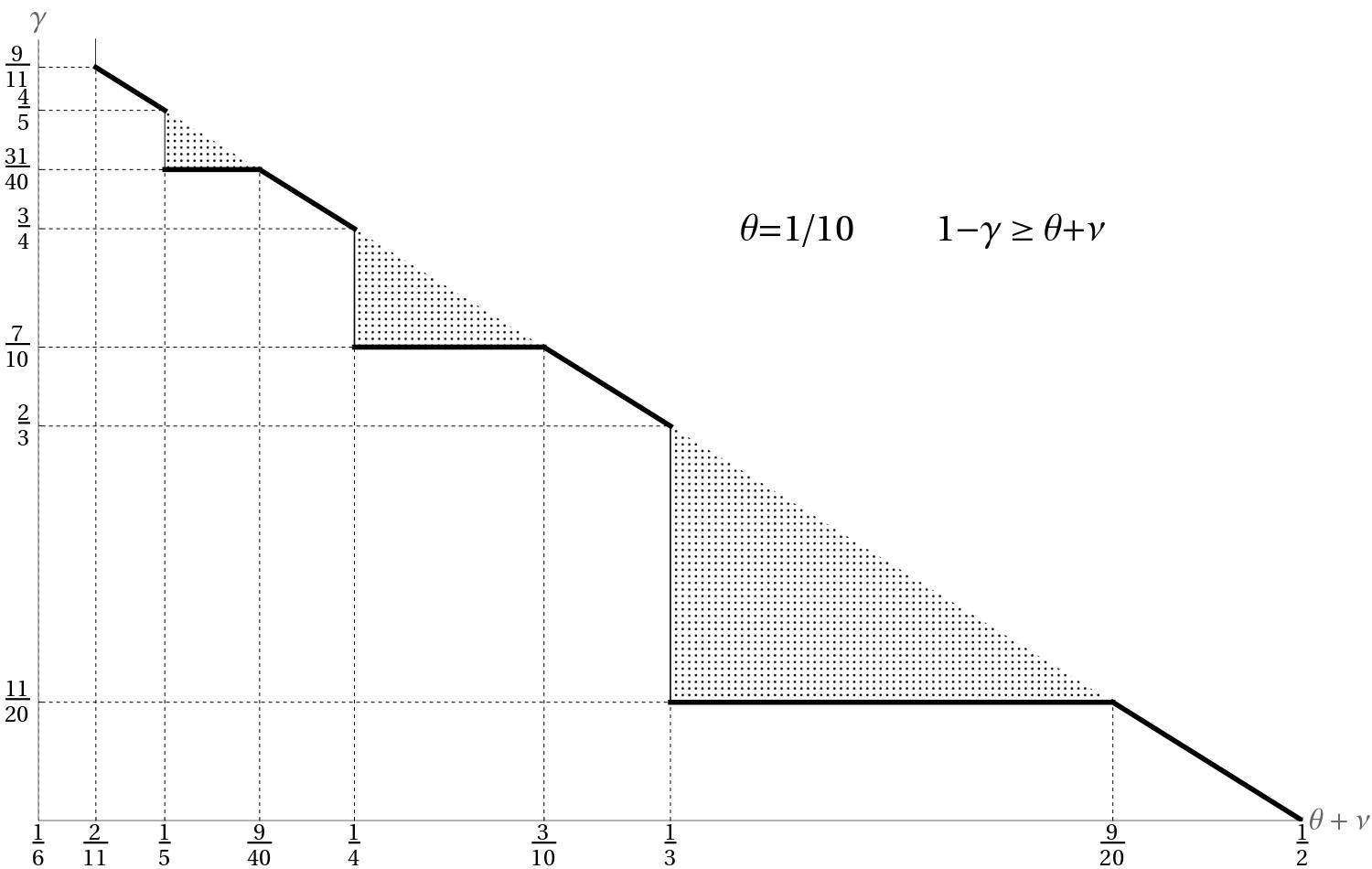}
\caption{Part of the set $\cA$ and $\cA^*$.}
\end{figure}

To illustrate Theorem \ref{thm:asymptotic} and Theorem \ref{thm:continuity}, Figure 1 shows $\cA$
and $\cA^*$ for the case $\theta=\frac{1}{10}$ and $1-\gamma \ge \theta+\nu$.
The dotted region is $\cA$ (thus $C^\pm=1$), the thick line segments are places where
$P\in \cA^*$ and (B) fails (thus $C^\pm$ are discontinuous there)
and the vertical boundary segments are where $P\in \cA^*$
and (B) holds (thus $C^\pm$ are continuous there).
At the corners such as $\gamma=\frac{11}{20}$, $\theta+\nu=\frac13$,
we also have (B) failing, while at corners such as $\gamma=\frac23$, $\theta+\nu=\frac13$, (B) holds.
Figure 1 illustrates a general principle, proven later in
Lemma \ref{lem:B-alternative}, that discontinuities occur 
at $P\in \cA$ if and only if one leaves $\cA$ by decreasing $\gamma$ while keeping $\theta,\nu$ fixed.

\medskip

Now we consider the case $\gamma=1/2$. Let $P=(1/2,\theta,\nu)\in \cQ$ with $C^\pm(P)=1$. If $\theta+\nu<1/2$, then condition \eqref{eq:A2} forces $\theta=0$ and so \eqref{eq:A1} forces $\nu\ge 1/3$. If instead $\theta+\nu=1/2$ then \eqref{eq:A2} holds automatically and \eqref{eq:A1} holds provided $\theta\le 1/3$. Thus we are interested in the set
\[
\cA_2^*:=\bigl\{(\tfrac{1}{2},0,\nu):\,\tfrac13\le \nu<\tfrac12\bigr\}\cup\bigl\{(\tfrac{1}{2},\theta,\tfrac12-\theta):\,0\le \theta\le \tfrac13\bigr\}.
\]

\begin{thm}[Continuity-discontinuity when $\gamma=1/2$]\label{thm:continuity-1/2}
Let $P\in \cA_2^*$.  Then, for $0<\eps \le \frac{1}{10}$, $C^-(P_\eps)=0$.
Furthermore, if $\theta=0$ then $C^+(P_\eps)=1+O(\eps)$.
\end{thm}

The first claim in Theorem \ref{thm:continuity-1/2} will be proven in Section \ref{sec:Setup} as part of Theorem \ref{thm:gamma<1/2},
while the second claim will be proven in Section \ref{sec:families}.

In some applications there is a natural barrier to obtaining \eqref{eq:TypeI} with 
$\gamma\ge\frac12$, and Theorem \ref{thm:continuity-1/2} (and the more general
Theorem \ref{thm:gamma<1/2}) shows that one cannot hope to obtain non-trivial lower bounds on primes without some additional assumptions (but one can still obtain strong upper bounds with sufficient Type II estimates when $\theta=0$). 
The proof that $C^-(P_\eps)=0$ for all $\eps>0$ utilizes constructions of sequences
which depend strongly on $\eps$ and take values $\gg 1/\eps$.
If, however, we impose the additional constraint that $w_n$ is divisor bounded,
 then we are able to recover good lower bounds for primes
 when $\eps \to 0$. This was exploited in Duke-Friedlander-Iwaniec \cite{DFI} in the case $\theta=0$, 
$\nu=\frac13$ and  Sarnak and Ubis \cite{SU} for the case $\nu=\frac15$.
To this end, we let $\CB^+(P;\varrho)$ and $\CB^-(P;\varrho)$ be the smallest and largest constants such that
\[
\Bigl(\CB^-(P,\varrho)+o(1)\Big)\sum_{p}b_p\le\sum_{p}a_p\le \Bigl(\CB^+(P,\varrho)+o(1)\Big)\sum_{p}b_p
\]
for any non-negative sequence $a_n$ with $w_n=a_n-b_n$ satisfying \eqref{eq:TypeI} and \eqref{eq:TypeII} with the added condition that $|w_n|\le \tau(n)^\varrho$.
Observe that this extra hypothesis automatically implies \eqref{w} if $\varpi$ is large enough in terms of $\varrho$ and $x$ is large enough.
 Precise technical definitions will be given in Section \ref{sec:Setup}.

\begin{thm}[Continuity for bounded sequences with $\gamma=1/2$]\label{thm:continuity-less-1/2}
Suppose $P=(\frac12,0,\nu)$ with $\frac13\le \nu < \frac12$, $0\le \eps \le \frac1{10}$
and $\varrho\ge 1$.  Then we have $\CB^-(P_\eps;\varrho)=1+O_\varrho(\eps)$ and 
$\CB^+(P_\eps;\varrho)=1+O_\varrho(\eps)$.
\end{thm}

Theorem \ref{thm:continuity-less-1/2} will be proven in Section \ref{sec:families}.
 Duke-Friedlander-Iwaniec \cite{DFI} essentially showed a version of Theorem
 \ref{thm:continuity-less-1/2} when $\rho=1$.

\medskip

\subsection{Analysis of two special families}

We consider two families of parameters $P=(\gamma,\theta,\nu)$
where $C^-(P)<1<C^+(P)$ which have been prominent in the literature and demonstrate that our framework (particularly Theorems \ref{thm:constructions} and
 \ref{thm: Main sieving}) is capable of establishing \textit{exact} formulas for
 $C^\pm(P)$ in some ranges.

The first family of parameters we consider is $P_\theta=(1-\theta,\theta,1-3\theta)$.
If $a_n$ is the normalized indicator function of a set $\cJ \subseteq (x/2,x]$ containing $x^{1-c}$ elements, then one can only hope for \eqref{eq:TypeI} to hold for $\gamma< 1-c$ and \eqref{eq:TypeII} for $\theta> c$.
Moreover, there is a natural barrier in attempting to establish \eqref{eq:TypeII} with $\theta+\nu\ge 1-2c.$\footnote{Showing one can take $\theta+\nu\ge 1-2c$ is closely related to estimating $\#\{n:nm_1,nm_2\in \cJ\}$ with an error term better than $O(1)$ on average over $m_1,m_2\sim x^{1-2c+\epsilon}$ (i.e. to show bilinear cancellation in the error term), which is typically very difficult outside of special situations.}
Thus the parameters $P_\theta$ often represent the best Type I and Type II estimates we can hope for when considering a set with $x^{1-\theta+o(1)}$ elements. 

This parameter family was intensively investigated by Harman \cite{Harman} in developing his sieve method. In particular, these ranges played a fundamental role in finding small fractional parts of $\alpha p$
where $\alpha$ is a given real number (see, e.g., Jia \cite{Jia} and Harman \cite[Chapters 3 and 5]{Harman}) and in work on prime values of incomplete norm forms (see \cite{NormForms}) because these cases handled sets containing $x^{1-\theta+o(1)}$ elements.

 When $\theta\le \frac14$, we have $C^\pm (P_\theta)=1$, 
which follows from Theorem \ref{thm:asymptotic} and was a critical part of
the work of Friedlander-Iwaniec \cite{FI98} on primes of the form $x^2+y^4$. As $\theta\rightarrow 1/3$ the Type II range becomes arbitrarily small, and so it follows from Theorem \ref{thm:TypeII-minimum} that $C^-(P_\theta)=0$ when $\theta$ is sufficiently close to 1/3. Thus the main region of interest is $1/4<\theta$ with $\theta$ not too close to $1/3$.

\begin{thm}\label{thm:1-parm theta family}
For $1/4 < \theta \le 2/7$, we have
\[
C^-(1-\theta,\theta,1-3\theta) = 1 - 2\mint{1-3\theta \le \beta_1 \le \beta_2 \le \theta \\ \b_1+\b_2 \ge 1/2}
\frac{d \b_1 d\b_2}{\b_1 \b_2 (1-\beta_1-\beta_2)} = 1 - 2\int_{1-2\theta}^{1/2}
\; \frac{\log \big(\frac{\theta}{1-\theta-\a}\big)}{\a(1-\a)}\, d\a.
\]
\end{thm}
For $\theta$ close to $\frac14$ we find
\[
1-C^-(1-\theta,\theta,1-3\theta) = 64 (\theta-1/4)^2 + O(\theta-1/4)^3.
\] 
Numerically,
\[
C^-(5/7,2/7,1/7) = 0.9214823\ldots.
\]
This improves upon the result of Harman \cite[Ch. 3,5]{Harman},
who showed weaker lower bounds on $C^-(P_\theta)$, in particular
\[
1-C^-(1-\theta,\theta,1-3\theta) \le 80(\theta-1/4)^2+O((\theta-1/4)^3),
\]
and  $C^-(5/7,2/7,1/7) \ge 0.9041$.

We have also determined the exact value of $C^-(P_\theta)$
in the wider range $\frac27 \le \theta \le \frac{3}{10}$, 
and established lower bounds on $C^-(P_\theta)$ for $\theta>\frac3{10}$,
but the details are much longer and this will appear in a future work.

\medskip

The second parameter family is $\gamma=\frac12$, $\theta=0$ and $0 \le \nu < \frac13$,
which played an important  role in the works of Duke, Friedlander and Iwaniec
\cite{DFI} and of Sarnak and Ubis \cite{SU}. 
Theorem \ref{thm:continuity-less-1/2} gives an asymptotic formula when $\nu \ge \frac13$, and  Theorem \ref{thm:TypeII-minimum} implies $C^-(P)=0$ if $\nu$ is small enough, so the main interest is in $\nu<1/3$ and $\nu$ not too small.
We state results for both $C^{\pm}(P)$ and $\CB^{\pm}(P_\eps;\varrho)$, as the methods we use to bound  $C^{\pm}(P)$ work equally well to bound $\CB^{\pm}(P_\eps;\varrho)$.  We do not have a proof that $C^{\pm}(P)=\lim_{\eps\to 0} \CB^{\pm}(P_\eps;\varrho)$ for all $\nu$, but we believe this is the case
for this special family. By some simple monotonicity principles (see Lemma \ref{prop:C-monotonicity}), we always have
\[
\sup_{\eps>0} \CB^-(P_\eps;\varrho) = \lim_{\eps\to 0^+} \CB^-(P_\eps;\varrho).
\]

\begin{thm}\label{thm:theta=0 gamma=1/2}
Fix $\varrho \ge 1$.  We have
\begin{itemize}
\item[(a)] For $\frac15 \le \nu < \frac13$, 
\[
\lim_{\eps\to 0^+} \CB^-(\tfrac12-\eps,\eps,\nu-2\eps;\varrho) = C^-(\tfrac12,0,\nu) = 1 - 2 \mint{\nu \le \beta_1 \le \beta_2 \le \beta_3 \le \frac12 \\
\beta_1+\beta_2+\beta_3=1} \frac{d\beta_1 d\beta_2}{\beta_1\beta_2\beta_3}.
\]
\item[(b)] For $\nu \ge 0.1663$, $C^-(\frac12,0,\nu)>0$ and
$\lim_{\eps\to 0^+} \CB^-(\frac12-\eps,\eps,\nu-2\eps;\varrho)>0$.
\item[(c)] For $\nu=0.1616$, $C^-(\frac12,0,\nu)=0=\CB^-(\frac12,0,\nu)$.
\end{itemize}
\end{thm}

In particular, $C^{-}(\frac12,0,\frac15) = \lim_{\eps\to 0} \CB^-(\frac12-\eps,\eps,\frac15-2\eps;\varrho) = 0.362\ldots$.  For comparison,
Duke, Friedlander and Iwaniec \cite{DFI} showed $\CB^-(\frac12-\eps,\eps,\frac15-2\eps;1) \ge 0.23$ for small enough $\eps>0$
and posed the problem to determine the infimum of $\nu$ with 
$\lim_{\eps\to 0^+} \CB^-(\frac12-\eps,\eps,\nu-2\eps;\varrho)>0$.
By Theorem \ref{thm:theta=0 gamma=1/2}, this threshold lies in $[0.1616,0.1663]$.
The main theorem of Sarnak and Ubis \cite{SU} uses the fact that 
$\CB^-(\frac12-\eps,\eps,\frac15-2\eps)>0$ for small $\eps>0$, 
however the authors' proof that the Type II bound \eqref{eq:TypeII} holds
is conditional on the Ramanujan/Selberg conjectures for $GL_2$. 
Sarnak and Ubis also showed unconditionally that \eqref{eq:TypeII}
holds for $x^{o(1)} \le m\le x^{3/19-o(1)}$.  This is not good enough
for an unconditional proof of their main result, however, since Theorem \ref{thm:theta=0 gamma=1/2} (c) implies $\CB^-(\frac12,0,\frac{3}{19})=0$.

Theorems \ref{thm:1-parm theta family} and  \ref{thm:theta=0 gamma=1/2} 
will be proven in Section \ref{sec:families}.

In future work we will apply our general framework in practice to obtain reasonable lower bounds on $C^-(\gamma,\theta,\nu)$ for fairly general range of parameters $\gamma,\theta,\nu$.

\bigskip

%
%
%
{\Large \section{Outline}\label{sec:outline}}
%
%
%

Before introducing our technical setup, we give a rough sketch of some of the key ideas in the paper, suppressing various technical issues. 

We wish to understand the sum $\sum_{p}a_p=\sum_{p}b_p+\sum_{p}w_p$ over primes $p$,
where we recall that our sequences are supported on $(x/2,x]$.
 Since we have a prime number theorem for $b_n$, it suffices to understand $\sum_{p}w_p$. By a combinatorial decomposition of primes such as the Heath-Brown identity, we have
\[
\sum_{p}w_p\approx \sum_{j} c_j \sum_{n=m_1\cdots m_j}w_{n}\beta_1(m_1)\cdots \beta_j(m_j)
\]
for some bounded constants $c_j$ and some coefficients $\beta_j$ with $\beta_j(m_j)\approx 1$ if $m_j>x^{1-\gamma}$. 
We can truncate all the variables $m_i$ to $m_i <x^{1-\gamma}$, since if $m_i \ge x^{1-\gamma}$ then $n/m_i \le x^{\gamma}$ and the contribution of such terms is small by our Type I assumption \eqref{eq:TypeI} and the simple behaviour of $\beta_i(m_i)$ in this range. 
Similarly, we can further restrict the summation to terms where $n$ has no divisor
in the Type II range $((x/2)^\theta,x^{\theta+\nu}]$ by our Type II assumption \eqref{eq:TypeII} and regularity properties of the $\beta_j$, e.g. in practice most of the
functions $\beta_j$ are multiplicative.

 If we assume that the $\beta_i$ are also concentrated on integers with no prime factors less than $x^\eps$, then by factorizing the $m_i$ we obtain an expression of the form
\begin{equation}
\sum_{p}w_p\approx \sum_{1/(1-\gamma)\le j<1/\eps}  \; \sum_{\substack{n=p_1\cdots p_j \\ x^{\eps} \le p_1\le \cdots \le p_j \\ \bv(n)\in \cR}}w_{n} \beta'(p_1,\ldots,p_j),
\label{eq:R1Decomposition}
\end{equation}
where $\bv(n)\in \mathbb{R}^{\Omega(n)}$ is the vector 
$\bigl(\frac{\log{p_1}}{\log{n}},\cdots,\frac{\log{p_j}}{\log{n}}\bigr)$, $\beta'$ is an explicit function of $p_1,\ldots,p_j$ and the set $\cR=\cR(\gamma,\theta,\nu)$ is given by
\[
\cR:=\bigcup_{j\ge 1}\Bigl\{\bx\in(0,1-\gamma)^j:\,\text{no subsums in }\two,\,\sum_{i=1}^jx_i=1\Bigr\}.
\]
By \emph{subsum} we mean a sum of a subset of the coordinates of $\bx$.

An immediate consequence of \eqref{eq:R1Decomposition} is that we obtain an asymptotic formula for $\sum_p a_p$ whenever $\cR$ is empty. It is introducing a general framework for studying $\sum_p w_p$ based on analysing the set $\cR$ and the decomposition \eqref{eq:R1Decomposition} which is the key contribution of the paper. 

An important observation is that the process of obtaining \eqref{eq:R1Decomposition} can also be reversed: if we specify $w_n$ for $n$ with $\mathbf{v}(n)\in \cR$ and set $w_n=0$ if $\bv(n)$ has a subsum in $\two$, then there is an essentially unique way of extending the definition of $w_n$ to all $n\in(x/2,x]$ such that \eqref{eq:TypeI} and \eqref{eq:TypeII} hold (which is given by combinatorial identities similar to above). This leads to a natural means of \textit{constructing} examples of sequences satisfying \eqref{eq:TypeI} and \eqref{eq:TypeII} and containing either many primes or few primes (provided one can understand the complicated combinatorial factors arising from this process).

With some understanding of the combinatorial structure of $\cR$, this allows us to show that whenever $\cR$ is non-empty, there are two sequences each satisfying the Type I and Type II estimates \eqref{eq:TypeI} and \eqref{eq:TypeII} but with different average weight on the primes,
thus demonstrating that
 the Type I and Type II estimates are not sufficient to deduce an asymptotic formula. 
In this way we see that $\cR$ being empty is a necessary and sufficient condition for an asymptotic formula.  Considerable combinatorial analysis shows that $\cR$ empty is
equivalent to both \eqref{eq:A1} and \eqref{eq:A2} holding, and thus we obtain 
Theorem \ref{thm:asymptotic}.

One application of our general construction method is to provide examples of sequences satisfying \eqref{eq:TypeI} and \eqref{eq:TypeII} but containing no primes at all (that is, $a_p=0$ on primes $p$) when the length $\nu$ of the Type II interval is small, thus giving
  Theorem \ref{thm:TypeII-minimum}. 
Given a weight $W_n$ for $n$ satisfying $\bv(n)\in\bigcup_{j\ge 1} [\epsilon,1-\gamma]^j$, we can extend $W_n$ as above to a sequence satisfying \eqref{eq:TypeI}, but this extension does not necessarily satisfy \eqref{eq:TypeII} since the support of $W_n$ may include $n$ with $\bv(n)$ having a subsum in $\two$.
 If we set $w_n=W_n$ whenever $\bv(n)\in\cR$, and $w_n=0$ when $\bv(n)$ has a subsum in $\two$, then our extension of $w_n$ will satisfy \eqref{eq:TypeI} and \eqref{eq:TypeII}, and moreover $w_n$ will have a similar size to $W_n$ for most integers $n$,
 provided that the Type II interval is short (i.e. if $\nu$ is sufficiently small compared with $\epsilon,1-\gamma$). 
If our original weight $W_n$ is a bit below $-1$ on primes and at least $-1$ on 
all other integers, then (again, if $\nu$ is small enough) by a slight rescaling
we can obtain a sequence $w_n$ with is $-1$ on primes, $\ge -1$ on all integers,
and satisfies \eqref{eq:TypeI} and \eqref{eq:TypeII}.
Taking $b_n=1$ for all $n$, this produces the required sequence $a_n=w_n+1$.
 Our choice of $W_n$ is based on variants of the Liouville function which are completely multiplicative, $-1$ on primes $p\in [x^\epsilon,x^{1-\gamma}]$, and $0$ on primes $p<x^\epsilon$.

To complement our constructions we also wish to use \eqref{eq:R1Decomposition} to produce matching sieve bounds. The right side of \eqref{eq:R1Decomposition} is difficult to estimate directly for a general sequence $w_n$ since we have no good control over the sign of the coefficients $\beta'(p_1,\ldots,p_j)$. 
However, with this decomposition in mind, we can keep control over positivity by introducing sieves $H^{\pm}(n)=\sum_{d|n}\lambda^\pm_d$ which are only required to have the correct sign on integers $n$ which can be non-trivially decomposed into vectors coming from $\cR$. More specifically, let $\cC(\cR)$ denote the set of all vectors whose components are subsums of a vector $\bx\in\cR$ according to some partition of the coordinates (we'll make this notion precise in the next section). 
Only for the composite integers $n$ satisfying $\bv(n)\in \cC(\cR)$ do we require the usual sieve inequalities
\[
H^-(n)\le 0\le H^+(n), \qquad H^{\pm}(1)=1.
\]
We then find that
\begin{align*}
\sum_{p}a_p\ge \sum_{\substack{\bv(n)\in \cC(\cR)}}a_n H^-(n)&= \sum_{\substack{\bv(n)\in \cC(\cR)}}b_n H^-(n)+ \sum_{\substack{\bv(n)\in \cC(\cR)}}w_n H^-(n)\\
&=\sum_{\substack{\bv(n)\in \cC(\cR)}}b_n H^-(n)+ \sum_{n}w_n H^-(n)- \sum_{\substack{\bv(n)\notin \cC(\cR)}}w_n H^-(n).
\end{align*}
Since $H^-$ is a short divisor sum, $\sum_{n}w_n H^-(n)\approx 0$ by our Type I assumption. By decompositions like \eqref{eq:R1Decomposition} we also find $\sum_{\substack{\bv(n)\notin \cC(\cR)}}w_n H^-(n)\approx 0$. We therefore obtain a lower bound
\begin{equation}
\sum_{p}a_p\gtrsim \sum_{\substack{p}}b_p+\ssum{n\text{ composite} \\ \substack{\bv(n)\in C(\cR)}}b_n H^-(n),
\label{eq:SieveBound}
\end{equation}
and we can evaluate the right hand side precisely since we have a generalized prime number theorem for $b_n$. This gives a general method for producing lower bounds (and a similar method produces upper bounds) for any sequence satisfying \eqref{eq:TypeI} and \eqref{eq:TypeII}.

In some ranges of parameters $P=(\gamma,\theta,\nu)$ there is a phase change where increasing the Type I or Type II ranges slightly causes $\cR(P)$ to suddenly gain a large amount of `mass', whereas in other ranges the `mass' of $\cR(P)$ varies smoothly. 
Utilizing our general sieve bounds and general constructions,
this behaviour of $\cR(P)$ allows us to establish the continuity or discontinuity of
the functions $C^{\pm}(\gamma,\theta,\nu)$ in particular ranges and gives
 Theorem \ref{thm:continuity} and Theorem \ref{thm:continuity-1/2}.
In particular, this shows that any general theory of primes in sequences satisfying
Type I/II sums will need to be sensitive to these arithmetic discontinuities.

We believe that both the constructions and the sieve bounds described above should be capable of producing essentially optimal results. 
In various common families of parameters $(\gamma,\theta,\nu)$ we are able to demonstrate a sieve bound which matches the constants obtained by a construction, thereby showing that both are best-possible (see Theorem \ref{thm:1-parm theta family} and Theorem \ref{thm:theta=0 gamma=1/2}). 
This typically shows that when we are not in the asymptotic region, previous techniques (such as those based on Harman's sieve) do not produce optimal constants since they do not fully exploit the arithmetic information available 
(this is reflected by the fact that there are certain `hidden symmetries' which allow for improvements in Harman's sieve). 

\medskip

\subsection{Schematic of the paper.}

Section \ref{sec:Setup} gives precise definitions of $C^-(\gamma,\theta,\nu)$ and $C^+(\gamma,\theta,\nu)$,
the specific hypotheses we require on the comparison sequence 
$(b_n)$, and the definitions of $\CB^-(\gamma,\theta,\nu)$
and $\CB^+(\gamma,\theta,\nu)$.

Section \ref{sec:Notation} contains notational conventions and basic results we need from prime number theory, prime decompositions and the geometry of certain regions of $\RR^k$.

Section \ref{sec:constructions} is devoted to a new general method of constructing sequences $(a_n)$, with $b_n=1$ for all $n$, which satisfy the Type I bounds 
\eqref{eq:TypeI} and Type II bounds \eqref{eq:TypeII} with $w_n=a_n-b_n$, and so give general upper bounds on $C^-(\gamma,\theta,\nu)$ and lower bounds on $C^+(\gamma,\theta,\nu)$.  The main result of this section is Theorem \ref{thm:constructions}.

A new sieve procedure is established in Section \ref{sec:sieving} which provides general lower bounds on $C^-(\gamma,\theta,\nu)$ and upper bounds on $C^+(\gamma,\theta,\nu)$.  The main result of this section is Theorem  \ref{thm: Main sieving}. We also give a criterion for when a construction in Section \ref{sec:constructions} and a sieve argument in Section \ref{sec:sieving}
are optimal, thus providing an exact value of $C^{\pm}(\gamma,\theta,\nu)$.

Section \ref{sec:families} analyzes $C^{\pm}(\gamma,\theta,\nu)$ for
two special 1-parameter families of triples $(\gamma,\theta,\nu)$ which have 
appeared in the literature. We establish Theorems \ref{thm:continuity-1/2}, \ref{thm:continuity-less-1/2}, \ref{thm:1-parm theta family} and \ref{thm:theta=0 gamma=1/2} here.

 In Section \ref{sec:TypeII-minimum}, we use the method
 from Section \ref{sec:constructions} to show that for
any $\gamma<1$ there is a positive $\nu_0$ so that whenever $\nu\le \nu_0$, and for 
any $\theta$, there are examples of sequences $a_n,b_n,w_n$ satisfying \eqref{eq:TypeI}
and \eqref{eq:TypeII} but with $a_p=0$ for all primes $p$, thus establishing Theorem \ref{thm:TypeII-minimum}.

In Section \ref{sec:asymptotic}, we determine precisely which triples $(\gamma,\theta,\nu)$ guarantee an asymptotic for $\sum_p a_p$ for any sequences satisfying
 \eqref{eq:TypeI} and \eqref{eq:TypeII}, and give a simple criterion
 for determining whether a given triple has this property. This gives Theorem \ref{thm:asymptotic}.
 
Section \ref{sec:continuity} is devoted to the study of the continuity and discontinuity of the 
functions $C^\pm$.  In particular, we show that there are many points $(\gamma,\theta,\nu)$ where these functions are discontinuous and establish Theorem \ref{thm:continuity}.

Sections \ref{sec:families}, \ref{sec:TypeII-minimum} and \ref{sec:asymptotic}
may be read independently of one another, while Section \ref{sec:continuity}
depends on the results of Section \ref{sec:asymptotic}.
Section \ref{sec:TypeII-minimum} requires Section \ref{sec:constructions}
and doesn't use Section \ref{sec:sieving}, while Sections \ref{sec:families},
\ref{sec:asymptotic} and \ref{sec:continuity} require the results of both
Sections \ref{sec:constructions} and \ref{sec:sieving}.

\bigskip


%
%
%
{\Large \section{Technical setup and reductions}\label{sec:Setup}}
%
%
%

\subsection{Notation for vectors}

As the mapping $n\to \bv(n)$ will be central to our analysis, we next
list some notational conventions for vectors which we use in this paper,
two of the most important being the concepts of \emph{fragmentations} 
and \emph{coagulations} of a given vector. 

\begin{defn}[Vector sizes, sums, concatenations]
For any vector $\bx=(\ssc{x}{1},\ldots,\ssc{x}{k})$, let $|\bx|$
be the sum of the components of $\bx$ (in our work, all components are non-negative,
so this is the $\ell_1$-norm), and let $\dim(\bx)=k$.
The empty vector is denoted $\emptyvec$, and we have
$\dim \emptyvec = 0$ and $|\emptyvec|=0$.
If $\bx=(\ssc{x}{1},\ldots,\ssc{x}{k})$ and $\by=(y_1,\ldots,y_\ell)$, we define
$(\bx,\by)=(\ssc{x}{1},\ldots,\ssc{x}{k},y_1,\ldots,y_\ell)$, in other words $(\bx,\by)$ is the
concatenation of $\bx$ and $\by$.
\end{defn}

\begin{defn}[Subvectors, subsums]
For $A\subseteq [k]:=\{1,2,\ldots,k\}$, $\ssc{\bx}{A} = (x_i:i\in A)$ is called
a \emph{subvector} of $\bx$, where the ordering of the $x_i$
is preserved, i.e. if $A=\{c_1,\ldots,c_m\}$ with 
$c_1<\cdots < c_m$, then $\ssc{\bx}{A} = ( x_{c_1},\ldots,x_{c_m} )$.  In particular, $\emptyvec$ is a subvector of any vector $\bx$.
We use the notation $\by \subseteq \bx$ to denote that $\by$ 
is a subvector of $\bx$.
For any $A$, $|\ssc{\bx}{A}|$
is called a \emph{subsum} of $\bx$, and if $0<|A|<k$ then $|\ssc{\bx}{A}|$
is a \emph{proper subsum} of $\bx$.
\end{defn}

We restate the formal definition of the set $\cR(\gamma,\theta,\nu)$ from the previous section.

\smallskip

\begin{defn}[The fundamental region $\cR$]\label{defn:R1}
For $P=(\gamma,\theta,\nu)$ satisfying \eqref{Q0}, let $\cR(P)$
be the set of all vectors $\bx$, or arbitrary dimension,
with components in $(0,1-\gamma)$ that sum to 1 and have no proper subsum in $[\theta,\theta+\nu]$.
\end{defn}

\begin{defn}[Decompositions]
Suppose that $\bx=(\ssc{x}{1},\ldots,\ssc{x}{k})$.
The notation $\bx_1 \sqcup \cdots \sqcup \bx_m = \bx$ means
that for some disjoint union $A_1 \sqcup \cdots \sqcup A_m = [k]$, $\bx_i=\bx_{A_i}$
for all $i$.  We call this a \emph{decomposition} of $\bx$.
A summation condition $\bx_1 \sqcup \cdots \sqcup \bx_m = \bx$
indicates a sum over all of the $m^k$ decompositions of $\bx$ into $m$ subvectors.
\end{defn}

\begin{defn}[Coagulations and fragmentations]\label{defn:Coagulations}
Given an vector $\bx$, we say that $\by$ is a \emph{coagulation} of $\bx$
if $\by$ is formed by joining together some of the components of $\bx$.  More specifically, there is a decomposition $\bx_1 \sqcup \cdots \sqcup \bx_h$ of $\bx$ 
with each vector $\bx_i$ nonempty and
such that
\[
\by  = \big( |\bx_1|,\ldots,|\bx_h| \big).
\]
Conversely, $\bx$ is called a \emph{fragmentation} of $\by$.

Given a set $\cR$ of vectors (of varying lengths), denote by $\cC(\cR)$  the set of all coagulations of all vectors in $\cR$.
\end{defn}

The set $\cC(\cR(P))$ will play a prominent role in our analysis,
and in particular appears in the hypotheses on $(b_n)$, which we list in 
the next subsection.

\subsection{Hypotheses on the sequence $(b_n)$ when $\cR(P)$ is nonempty}
\label{sec:bn-assumptions}
In order to have wide applicability, we impose very general conditions on the sequence
$(b_n)$.  Given a real number $\varpi \ge 1$, to ensure that the count of primes is larger than various error terms, we require 

\begin{align}
\sum_p b_p &\ge \frac{x}{(\log x)^{\varpi}}. \tag{b.1}\label{bp-sum}
\end{align}
We also require that $b_n$ satisfies a generalized prime number theorem, with an error term controlled by a constant $B$ (which will be taken to be the same constant appearing in \eqref{eq:TypeI} and \eqref{eq:TypeII}, and can be assumed to be sufficiently large in terms of the parameters $\gamma,\theta,\nu,\varpi$):
\begin{equation}\label{bn-f-sum}
\begin{split}
&\text{For every } 2\le k\le 1/\nu, \text{ convex set } \cT \subseteq \{\bx \in \RR^k \cap \cC(\cR(P)): \nu \le \ssc{x}{1}\le \cdots \le \ssc{x}{k} \}\text{ and} \\
&\text{Lipschitz continuous function } f \text{ on } \cT  \text{ with }
|f(\bx)|\le 1 \text{ and Lipschitz constant } \le 1,   \\
&\text{we have}\\
&\ssum{n=p_1\cdots p_k \\ p_1\le \cdots \le p_k} b_n f\Big(\frac{\log p_1}{\log n},\ldots, \frac{\log p_k}{\log n} \Big) = \bigg(  \sum_p b_p \bigg)
\Bigg(\! \mint{\cT} \frac{f(\bu)\, d\bu}{u_1\cdots u_k} + E \Bigg), \; \text{ with }\, |E| \le \frac{1}{B}.
\end{split}\tag{b.2}
\end{equation}

\begin{rmk}
In the present paper, the hypothesis 
\eqref{bn-f-sum} is used only in the proof of Lemma \ref{lem:h-sum-to-integral},
which is used to prove Theorem \ref{thm: Main sieving} and
Theorem \ref{thm:CB-lower}, the latter providing lower bounds on $\CB^-(P;\varrho)$.
 Although \eqref{bp-sum} and
\eqref{bn-f-sum} suffice for all of the results in the present paper, 
we anticipate that for certain ranges of $(\gamma,\theta,\nu)$, future applications of our methods may require
additional regularity conditions on $b_n$; this situation occurs when the set $\cG_2$, defined in \eqref{eq:VHZGdefs}, is nonempty; see Section \ref{sec:sieving}
for more commentary.  Any additional conditions will be satisfied for the constant
sequence $b_n=1$, which is what we use for our constructions in Section \ref{sec:constructions}, as well as for the sequences in Lemma \ref{lem:bn-PNT-examples}
below.
\end{rmk}

\textbf{Notational convention.}
Integrals over subsets of a hyperplane $\{\bx \in \RR^k : \ssc{x}{1}+\cdots+\ssc{x}{k}=z\}$, where $z\in \RR$,
 are with respect to the projection measure of the set onto the
first $k-1$ coordinates.  For such integrals, the choice of which $k-1$ coordinates
to project onto does not matter, as all projection measures are equivalent
for these special hyperplanes.
Likewise, when we refer to the \emph{measure} of such sets, we also mean
the projection measure.

These hypotheses \eqref{bp-sum} and \eqref{bn-f-sum} are satisfied for many natural sequences $(b_n)$ which have appeared in the literature, including  the constant sequence $b_n=1$
and for scaled indicator functions of $n$ in a short interval that are
coprime to a given $q$.

\begin{lem}\label{lem:bn-PNT-examples}
For any $\varpi,B>1$,
 if $x$ is large enough (in terms of $\gamma,\theta,\nu,\varpi,B$), $y\in [x^{1-\nu/11},x/2]$ and $q\le x^2$ is a positive integer,
 then the hypotheses \eqref{bp-sum} and \eqref{bn-f-sum} are satisfied for
\[
b_n=\frac{xq/2}{y\phi(q)}\one \big( \substack{x-y<n\le x\\ (n,q)=1} \big).
\]
\end{lem}
Lemma \ref{lem:bn-PNT-examples} follows quickly from the prime number theorem
for short intervals (e.g., Lemma \ref{PNT} below) and results connecting sums to
integrals (e.g., Lemma \ref{lem:sum_to_int} below).
We omit the proof.

%
%

\subsection{Definition of $C^-(\gamma,\theta,\nu)$ and $C^+(\gamma,\theta,\nu)$}
\label{sec:C-def}

\begin{defn}[Sequences under consideration]\label{def:Psi}
Let $(\gamma,\theta,\nu)$ satisfy \eqref{Q0}.
Let $\Psi(\gamma,\theta,\nu;B,\varpi,x)$ be the set of pairs of sequences
$((a_n),(b_n))$, each supported on $x/2<n\le x$, such that $a_n\ge 0$ and $b_n\ge 0$
for all $n$, \eqref{bp-sum} and \eqref{bn-f-sum} hold
and $w_n=a_n-b_n$ satisfies  \eqref{eq:TypeI}, \eqref{eq:TypeII} and \eqref{w}.
\end{defn}

\begin{defn}[The constants $C^\pm(\gamma,\theta,\nu)$]\label{def:Constants}
Define $C^-(\gamma,\theta,\nu)$ to be the supremum of of all constants $C$ such
that for any $\varpi>1$ there exists $B>0$ (depending on $\gamma,\theta,\nu,\varpi,C$) so that, whenever
$x$ is sufficiently large (in terms of $\gamma,\theta,\nu,B,C$) and 
$((a_n),(b_n))\in \Psi(\gamma,\theta,\nu;B,\varpi,x)$, we have
\[
\sum_{p} a_p  \ge C \sum_{p} b_p.
\]

Likewise, let $C^+(\gamma,\theta,\nu)$ be the infimum of of all constants $C$ such
that for any $\varpi>1$ there exists $B>0$ (depending on $\gamma,\theta,\nu,\varpi,C$) so that, whenever
$x$ is sufficiently large (in terms of $\gamma,\theta,\nu,B,C$) and 
$((a_n),(b_n))\in \Psi(\gamma,\theta,\nu;B,\varpi,x)$, we have
\[
\sum_{p} a_p  \le C \sum_{p} b_p.
\]
\end{defn}

For Theorems \ref{thm:continuity-less-1/2} and \ref{thm:theta=0 gamma=1/2},
 we also need a precise definition of $\CB^{\pm}(\gamma,\theta,\nu;\varrho)$.
 The definition is the same as that of $C^{\pm}(\gamma,\theta,\nu)$, except that
 the sequences under consideration are restricted to those satisfying 
 \begin{equation}\label{CBD}
 |w_n| \le \tau(n)^\varrho \;\; (x/2<n \le x), \qquad
 \sum_p b_p \ge \frac{x}{\varrho \log x}.
 \end{equation}
 Here we think of $\varrho$ as large and fixed.

\begin{defn}[The constants $\CB^\pm(\gamma,\theta,\nu,\varrho)$]\label{def:CB}
Fix $\varrho \ge 1$.
Let $\CB^-(\gamma,\theta,\nu;\varrho)$ be the the supremum of numbers $C$ so that 
 for any $\varpi>1$  there exists $B>0$  (depending on $\gamma,\theta,\nu,\varpi,C$, $\varrho$) so that, whenever
$x$ is sufficiently large (in terms of $\gamma,\theta,\nu,B,C,\varrho$), \eqref{CBD} holds and
$((a_n),(b_n))\in \Psi(\gamma,\theta,\nu;B,\varpi,x)$, 
we have
\[
\sum_{p} a_p \ge C  \sum_{p} b_p.
\]
Let $\CB^+(\gamma,\theta,\nu)$ be the the infimum of numbers $C$ so that  
 for any $\varpi>1$  there exists $B>0$  (depending on $\gamma,\theta,\nu,\varpi,C$, 
 $\varrho$) so that, whenever
$x$ is sufficiently large (in terms of $\gamma,\theta,\nu,B,C,\varrho$), \eqref{CBD} holds and
$((a_n),(b_n))\in \Psi(\gamma,\theta,\nu;B,\varpi,x)$, 
we have
\[
\sum_{p} a_p \le C  \sum_{p} b_p.
\]
\end{defn}

In the definition of $C^+(\gamma,\theta,\nu)$ it is possible that there is no such constant $C$; in this case we define $C^+(\gamma,\theta,\nu)=\infty$.
The same remark applies to $\CB^+(\gamma,\theta,\nu;\varrho)$.
We will show that $C^+(\gamma,\theta,\nu)<\infty$ whenever $\gamma\ge \frac12$
(see Corollary \ref{cor:C+finite}).
With additional hypotheses on the sequence $(b_n)$, similar to those in the
small sieve, one can use the small sieve to obtain finiteness of the
ratio $\sum_p a_p$ to $\sum_p b_p$; this is a consequence of the small sieve
and does not require any Type II information.

Taking $a_n=b_n=1$ for all $x/2<n\le x$,  we see from Lemma \ref{lem:bn-PNT-examples}
that for any $\varrho$,

\begin{equation}\label{eq:C-pm-trivial}
C^-(\gamma,\theta,\nu) \le  \CB^-(\gamma,\theta,\nu;\varrho) \le 1 
\le  \CB^+(\gamma,\theta,\nu;\varrho) \le C^+(\gamma,\theta,\nu).
\end{equation}

We have the expected monotonicity of the functions $C^{\pm}$ and $\CB^\pm$.

\begin{prop}[Monotonicity of $C^{\pm}$]\label{prop:C-monotonicity}
If $\varrho\ge 1$, $\gamma' \le \gamma$,
$[\theta',\theta'+\nu'] \subseteq [\theta,\theta+\nu]$, $P=(\gamma,\theta,\nu)$
and  $P'=(\gamma',\theta',\nu')$, then
\begin{align*}
C^-(P') &\le C^-(P), & \CB^-(P';\varrho) &\le \CB^-(P;\varrho),\\
C^+(P') &\ge C^+(P), & \CB^+(P';\varrho) &\ge \CB^+(P;\varrho).
\end{align*}
\end{prop}

\begin{proof}
Fix $\varpi,x,B$ and assume $(a_n),(b_n)$ satisfy (I), (II), \eqref{w},
\eqref{bp-sum} and \eqref{bn-f-sum} for $P'$.
It is clear that (I), (II) and \eqref{bp-sum} hold for $P$, since
$\nu'\le \nu$.  Since $\nu'\le \nu$ and $\cR(P') \supseteq \cR(P)$, \eqref{bn-f-sum} 
also holds for $P$.  Therefore,
\[
\Psi(P;B,\varpi,x) \subseteq  \Psi(P';B,\varpi,x)
\]
and Proposition \ref{prop:C-monotonicity} follows immediately.
\end{proof}

\begin{rmk}\noindent
\begin{itemize}
\item We write the strict inequality $(x/2)^{\theta} < m$ in (II) 
in order to have a meaningful statement when $\theta=0$, where the
term $m=1$ must be excluded; if the $m=1$ term is included in (II) then
(II) implies $\sum |w_n| \ll x/\log^B x$, and the analysis is trivial.

\item The standard ``small sieve'' setup has only \eqref{eq:TypeI}
as the main input, and it is known by Selberg's
examples \cite{Selberg} and Bombieri's work \cite{Bombieri} 
that there are sequences satisfying \eqref{eq:TypeI} for all $\gamma<1$ (and $x$ large depending
on $\gamma$), but still with $a_p=0$ on primes.

\item If $((a_n),(b_n))\in \Psi(P;B,\varpi,x)$, and 
each sequence $a_n,b_n$ is multiplied by $(\log x)^D$ for a constant $D$,
the new pair of sequences is in $\Psi(P;B-D,\varpi+|D|,x)$.  Thus,
our hypotheses for $C^{\pm}$ are not sensitive to logarithmic-sized rescalings of the sequences.
In contrast, the hypotheses for $\CB^{\pm}$ are very sensitive to unbounded rescalings.

\item For convenience we have used non-negative sequences $a_n$ in our setup, but we expect that there is no difference if one specializes attention to $a_n$ being the normalized indicator function of sets.  Trivially any bounds on $\sum_p a_p$ for general sequences apply to sets. All of our constructions in Section \ref{sec:constructions} produce extremal sequences with $a_n$ bounded, and so a random sampling argument would allow us to show the same properties hold for a sequence $a'_n$ which is the normalized indicator function of some set $\cA$.
\end{itemize}
\end{rmk}

The Type II bound \eqref{eq:TypeII} implies the same bound in
a complementary range by reversing the roles of $m$ and $n$, and the Type II
bound \eqref{eq:TypeII} implies \eqref{eq:TypeI} with the same range of $m$.
Specifically, we have the following.

\begin{prop}\label{prop:first-reduction}
Assume  \eqref{eq:TypeII}, where $x\ge 1000$.  Then
\begin{itemize}
\item[(a)] For any complex numbers $\xi_m, \kappa_n$ with $|\xi_m| \le \tau^B(m)$ and $|\kappa_n| \le \tau^B(n)$ for $m,n\in \NN$, we have
\[
\bigg|\sum_{\substack{ x^{1-\theta-\nu} \le m\le (x/2)^{1-\theta} \\ x/2<mn \le x}} \xi_m \kappa_n w_{mn} \bigg| \le \frac{x}{\log^{B}{x}}.
\]
\item[(b)] We have
\[
\sum_{m\in ((x/2)^\theta, x^{\theta+\nu}] \cup [x^{1-\theta-\nu},(x/2)^{1-\theta}]} \tau^B(m) \max_{\text{interval \cI}} \Big| \ssum{x/2<mn\le x \\ n\in \cI} w_{mn} \Big| \le \frac{2x}{\log^{B-1}{x}}.
\]
\end{itemize}
\end{prop}

\begin{proof}
If $x^{1-\theta-\nu} \le m\le (x/2)^{1-\theta}$ and $x/2<mn\le x$ then
$(x/2)^\theta < n\le x^{\theta+\nu}$ and hence part (a) follows from \eqref{eq:TypeII}.

Part (b) is proven using Fourier analysis.  For real $t$, let $K(t)=\min(x+1,(\pi |t|)^{-1},(\pi t)^{-2})$.  By Lemma 7.3 of Graham and Kolesnik \cite{GK},
for any interval $\cI \in [1,x]$ and any $m\in \NN$,
\[
\Big| \ssum{x/2<mn\le x \\ n\in \cI} w_{mn} \Big| \le \int_{-\infty}^\infty
K(t) \Big| \sum_\ell w_{\ell m} e(\ell t) \Big|\, dt,
\]
where $e(z)=e^{2\pi i z}$.
Thus, for $\cM =  ((x/2)^\theta, x^{\theta+\nu}] \cup [x^{1-\theta-\nu},(x/2)^{1-\theta}]$, we have
\[
\sum_{m\in \cM} \tau(m)^B  \max_{\text{interval \cI}} \Big| \ssum{x/2<mn\le x \\ n\in \cI} w_{mn} \Big| \le \int_{-\infty}^\infty K(t) \sum_{m\in \cM}
\tau(m)^B \phi_{m,t} \sum_\ell w_{m\ell} e(\ell t)\, dt,
\]
where, for each $m$ and $t$, $\phi_{m,t}$ is the complex number of modulus 1 defined by
\[
 \Big| \sum_\ell w_{\ell m} e(\ell t) \Big| = \phi_{m,t}  \sum_\ell w_{\ell m} e(\ell t).
\]
With $t$ fixed, the sum on $m$ is handled by \eqref{eq:TypeII} and part (a), with $\xi_m=\phi_{m,t} \tau(m)^B$ and $\kappa_\ell=e(\ell t)$.  Since $\int_{-\infty}^\infty K(t)\, dt \le \log x$ for $x\ge 1000$, part (b) follows.
\end{proof}

By Proposition \ref{prop:first-reduction} (a),
if $\theta+\nu\ge \frac12$, then (II) implies the same bounds with the
upper limit of $m$ increased to $(x/2)^{1-\theta}$, and if $\nu>1-2\theta$
then $1-\theta-\nu < \theta$ and so the lower limit of $m$ in \eqref{eq:TypeII}
may be lowered to $x^{1-\theta-\nu}$.  In other words,
\begin{equation}\label{eq:Type II-Cpm}
\begin{split}
C^{\pm}(\gamma,\theta,\tfrac12-\theta) &= C^{\pm}(\gamma,\theta,\nu)  \qquad\qquad \big( \tfrac12-\theta \le \nu < 1-2\theta \big), \\
C^\pm(\gamma,\theta,\nu) &= C^\pm(\gamma,\theta',\theta+\nu-\theta') 
\qquad\qquad (1-\theta-\nu < \theta' \le \theta).
\end{split}
\end{equation}

In a similar spirit, Proposition \ref{prop:first-reduction} (b) implies that
\begin{equation}\label{eq:Type-I-Cpm}
\begin{split}
C^{\pm}(\gamma,\theta,\nu) &= C^{\pm}(\theta+\nu,\theta,\nu) \qquad\qquad (\theta<\gamma\le \theta+\nu), \\
C^{\pm}(\gamma,\theta,\nu) &= C^{\pm}(\gamma',\theta,\nu) \qquad\qquad (1-\theta-\nu \le \gamma \le \gamma' < 1-\theta).
\end{split}
\end{equation}

When $\gamma<1/2$, the Type I information and Type II information is not
sufficient to detect primes; see Theorem \ref{thm:gamma<1/2} below.
In light of these reductions, in practice we need only 
consider triples $(\gamma,\theta,\nu)\in \cQ$ (recall the definition \eqref{Q2} of the set $\cQ$).

We have included the case $\theta+\nu=1/2$ in $\cQ$ even though the above reductions (almost) imply that $C^\pm(\gamma,\theta,1/2-\theta)=C^\pm(\gamma,\theta,1-2\theta)$. This is to cover all continuity cases when moving from $P$ to $P_\eps$. When moving from $P$ to $P_\eps$ for $\eps>0$, a Type II range of the form $[\theta,1-\theta]$ may shrink in one of two ways: (i) if we take $\nu=1-2\theta$ then the Type II range of $P_\eps$ is $[\theta+\eps,1-\theta-\eps]$
or (ii) if $\nu=\frac12-\theta$ then the Type II range for $P_\eps$ is $[\theta+\eps,\frac12-\eps]$, with complementary interval $[\frac12+\eps,1-\theta-\eps]$.

\medskip

\subsection{Asymptotic for primes, revisited}\label{subsec:A1A2}
As alluded to in Section \ref{sec:outline}, the proof of 
Theorem \ref{thm:asymptotic} has two parts,
an `analytic' argument showing that the asymptotic for $\sum_p a_p$
holds if and only if $\cR$ is empty, and a `combinatorial' part,
which shows that $\cR$ being empty is equivalent to both \eqref{eq:A1}
and \eqref{eq:A2} holding.  These two theorems will be proven in Section \ref{sec:asymptotic}.

\medskip

\begin{thm}\label{thm:asymptotic-R1}
Suppose that $P=(\gamma,\theta,\nu) \in \cQ$. If $\cR(P)$ is empty, then the conclusion of Theorem \ref{thm:asymptotic} (a)
holds.  If $\cR(P)$ is nonempty then the conclusion of Theorem \ref{thm:asymptotic} (b) holds. 
\end{thm}

\begin{thm}\label{thm:asymptotic-R1-A1A2}
Let $P=(\gamma,\theta,\nu) \in \cQ$. The $\cR(P)=\emptyset$ if and only if conditions \eqref{eq:A1} and \eqref{eq:A2} in Theorem
\ref{thm:asymptotic} both hold.
\end{thm}

Condition (A1) means that there are no points of the form $(\frac{1}{n},\ldots,\frac{1}{n})$ in $\cR(P)$, and condition \eqref{eq:A2} implies that the
point $(1-\gamma,\ldots,1-\gamma,1-M(1-\gamma))$ is not in $\cR(P)$.
The main part of the argument for Theorem \ref{thm:asymptotic-R1-A1A2}
is to show that if $\cR(P)$ is nonempty then it must contain a vector of one of these special types.
In either case, \eqref{eq:A1} failing or \eqref{eq:A2} failing, the existence 
of these special vectors can be used to produce examples
of sequences $(a_n)$ for which the asymptotic $\sum_p a_p \sim \sum_p b_p$ fails.
In the case where $(\frac1{n},\ldots,\frac1{n})\in \cR(P)$, the construction
focuses on integers whose prime factors are all close to $x^{a/n}$ for positive integers $a$,
and when $(1-\gamma,\ldots,1-\gamma,1-M(1-\gamma))\in \cR(P)$, the 
construction focuses on integers, all of whose 
prime factors but one are close to $x^{a(1-\gamma)}$ for positive integers $a$.

To get a feel for the combinatorial conditions \eqref{eq:A1} and \eqref{eq:A2}, we consider some examples.

\begin{itemize}
\item When $\gamma+\nu\ge 1$, \eqref{eq:A1} holds because all $n\ge M+1$ satisfy 
$\frac{1}{n} \le \nu$ and \eqref{eq:A2} holds since $1-\gamma \le \nu$.  Thus,
we have $C^\pm(P)=1$. This conclusion
has essentially been known to the experts.  
\item Take $\gamma =\frac79$, $\theta=\frac13$, $\nu=\frac{2}{21}$ so that $M=4$
and $\two = [\frac13, \frac37]$.  However, if $\gamma< \frac79$ with the same $\theta,\nu$ then \eqref{eq:A2} fails.
\item Take $\gamma=\frac35$, $\theta=\frac15$, $\nu=\frac2{15}$, so that $M=2$
and $\two = [\frac15,\frac13]$.
This is an example where we have a ``Type I gap'', namely we have good control of
$\sum_n w_{mn}$ when $m\le x^{3/5}$ and, using Proposition \ref{prop:first-reduction} (b), when $x^{2/3} \le m\le (x/2)^{4/5}$ as well.
However, if $\gamma< \frac35$ with the same $\theta,\nu$ then \eqref{eq:A2} fails.
\item 
Take $\gamma=\frac7{10}$, $\theta=\frac25$, $\nu=\frac15$, so that $M=3$
and $\two = [\frac25,\frac35]$.
\item By Lemma \ref{lem:GammaNu}, all such triples have $3\nu+\gamma\ge 1$,
with equality only for $(\frac35,\frac15,\frac2{15})$.
\item Take $\theta=0$, $\gamma=3/5$ and $\nu=\frac13-\frac15=\frac2{15}$.
\item If  $P=(\gamma,\theta,\nu)\in \cQ$ and $\gamma=\frac12>\theta+\nu$, then
we have $C^{-}(P)=C^+(P)=1$ if and only if $\theta=0$ and $\nu\ge \frac13$.
Indeed, since $\theta+\nu<\frac12$, \eqref{eq:A2}
implies that $\theta=0$.  But $M=2$ and hence \eqref{eq:A1} holds if and only if
$\nu\ge \frac13$.  The case $(\gamma,\theta,\nu)=(\frac12,0,\frac13)$ (actually a limiting version of it) was used by Duke-Friedlander-Iwaniec \cite{DFI}.
\end{itemize}

With the monotonicity properties of Proposition \ref{prop:C-monotonicity} we can justify the continuity claim made after the statement of Theorem \ref{thm:continuity}.

\begin{cor}\label{cor:continuity}
Let $P=(\gamma,\theta,\nu)\in \cA^*$.
If $\theta=0$ or \eqref{eq:B} holds, the functions
 $C^{\pm}$ are continuous at the point $P$, and otherwise
 both functions are discontinuous at $P$.
\end{cor}

\begin{proof}[Proof of Corollary \ref{cor:continuity} assuming Theorem \ref{thm:continuity}]
If \eqref{eq:B} fails and $\theta>0$, then it is clear that $C^\pm$
are discontinuous at $P$.
Assume next that (B) holds or $\theta=0$.
  By monotonicity (Proposition \ref{prop:C-monotonicity}), for any
$P'=(\gamma',\theta',\nu') \in \cQ$ with $|\gamma'-\gamma| \le \eps/2$, $|\theta'-\theta|\le \eps/2$ and $|\nu'-\nu|\le \eps/2$ we have $\gamma-\eps \le \gamma'$
and $[\theta+\eps,\theta+\nu-\eps] \subseteq [\theta',\theta'+\nu']$, thus
\[
C^-(P_\eps) \le C^-(P') \le 1 \le C^+(P') \le C^+(P_\eps).
\]
This gives the continuity of $C^{\pm}$ at $P_0$.
\end{proof}

\medskip

The failure of (B) implies
 that \eqref{eq:A2} holds for $P=P_0$ but fails for $P_\eps$ for sufficiently 
 small $\eps>0$.  In general (that is, for $P$ not necessarily
 from $\cA^*$), we believe that
  if \eqref{eq:A2} holds for $P$ but fails for $P_\eps$ when $\eps>0$
 is small, then $C^{\pm}$ are discontinuous at $P$.  This 
 stems from the fact that for such triples, $\cR(P_\eps)$ has a substantial subset, not present in $\cR(P)$, whose `mass' is independent of $\eps$.
 This will be taken up in a future work.

There is another characterization of when Hypothesis (B) holds, which 
is easy to see visually on graphs of $\cA$, such as in Figure 1.

\begin{lem}\label{lem:B-alternative}
Let $P=(\gamma,\theta,\nu)\in \cA^*$ with $\gamma> \frac12$.
Then Hypothesis (B) holds for $P$ if and only if for some $\eps>0$,
$(\gamma-\eps,\theta,\nu)\in \cA^*$.
\end{lem}

\begin{proof}
Suppose that (B) fails for $P$.  
Then there is an integer $h$ with $h(1-\gamma)\in \{\theta+\nu,
1-\theta\}$ and for all integers $h$, $h(1-\gamma)\notin [\theta,\theta+\nu)\cup
[1-\theta-\nu,1-\theta)$.  Hence there is an $\eps_0>0$ so that
for all $0<\eps \le \eps_0$, \eqref{eq:A2} fails for $(\gamma-\eps,\theta,\nu)$.
For such $\eps$, $(\gamma-\eps,\theta,\theta+\nu)\notin \cA$.

Now suppose that (B) holds for $P$ and let $M=\fl{1/(1-\gamma)}$.
 For some $\eps_0>0$, if $0 \le \eps \le \eps_0$,
\eqref{eq:A2} holds for $P' := (\gamma-\eps,\theta,\nu)$.
If $1-\gamma\ne \frac{1}{M}$, then clearly \eqref{eq:A1} holds
for $P'$ if $\eps_0$ is small enough.  Otherwise, if $1-\gamma=\frac{1}{M}$,
then \eqref{eq:A2} for $P$ implies that there is an integer $h$ with
$\frac{h}{M}\in [\theta,\theta+\nu]$, hence \eqref{eq:A1} holds for $P'$.
In both cases, we conclude that $P'\in \cA$.  But $P\in \cA^*$ and the monotonicity
of $C^{\pm}$ implies that $P'\in \cA^*$.
\end{proof}

\bigskip

\subsection{The case $\gamma<\frac12$}

In general, \eqref{eq:TypeI} with $\gamma<\frac12$ is not enough to detect primes,
provided that the Type II range in \eqref{eq:TypeII} doesn't \emph{essentially}
imply a larger Type I range via Proposition \ref{prop:first-reduction} (b).

\begin{thm}\label{thm:gamma<1/2}
Assume $\gamma<\frac12$ and $\gamma\not \in [\theta,\theta+\nu]$.
Then for any $B>0$ there are examples of sequences $(a_n),(b_n),w_n=a_n-b_n$ satisfying (I) and (II), with $b_n=1$ for all $n$ 
but with $w_p=-1$ for all primes.  In particular, $C^-(\gamma,\theta,\nu)=0$.
\end{thm}

\begin{proof}
There is a number $\a$ satisfying $\gamma<\a < 1/2$ and $\a \not\in  \two$.
This follows since either $\gamma<\theta$ or
$\theta+\nu < \gamma < 1/2$.  Fix $\eps>0$ small enough so that
the interval $[\a-\eps,\a+\eps]$ lies in $(\gamma,1/2)$ and has no intersection
with $\two$.  Let $x$ be large and take
\[
K = \int_{\a-\eps}^{\a+\eps} \frac{du}{u(1-u)} = \log\bigg( \frac{(\a+\eps)(1-\a+\eps)}{(\a-\eps)(1-\a-\eps)} \bigg).
\]
Now define $(w_n)_{x/2<n\le x}$ by taking $w_p=-1$ for primes $p$,
$w_{pq}=1/K$ if $p,q$ are primes with $x^{\a-\eps}\le p\le x^{\a+\eps}$,
and $w_n=1$ otherwise.  We see that (II) holds vacuously as $w_n=1$
if $n$ has a divisor in $[x^{\theta},x^{\theta+\nu}]$.
The terms in (I) with $m>1$ are also equal to zero, and the term $m=1$
equals
\[
\max_{[a,b] \subseteq [1/2,1]} \bigg| \sum_{ax <n\le bx} w_n  \bigg|
=\max_{a,b} \bigg| -(\pi(bx)-\pi(ax)) + K^{-1} \ssum{x^{\a-\eps}\le p\le x^{\a+\eps} \\ ax<pp'\le bx} 1   \bigg|
 = O_B(x/\log^B x)
\]
for any $B>1$ by the prime number theorem.
\end{proof}

\subsection{Generalizations and further arithmetic information}

Some results on primes, particularly those concerning primes in short intervals and primes in arithmetic progressions, incorporate additional arithmetic information which is \textit{not} of the form \eqref{eq:TypeI} or \eqref{eq:TypeII}, and so not covered by our setup. For example, the use of trilinear (and quadrilinear) estimates generalizing \eqref{eq:TypeII} plays an important role in the work of Bombieri-Friedlander-Iwaniec \cite{BFI} on primes in arithmetic progressions (and subsequent works such as  \cite{MaynardI,Polymath,Zhang} and many others), or in the work of Baker-Harman-Pintz \cite{BakerHarmanPintz} (and many earlier works) on prime in short intervals, or in the work of Pitt \cite{Pitt} on sums over primes with $a_p$ coming from the Fourier coefficients for holomorphic cusp forms.

Such trilinear estimates or higher estimates make use of \eqref{eq:TypeII} when the coefficients $\kappa_n$ have a particular convolution structure (either being a convolution of divisor-bounded sequences with specific support ranges, or with some coefficients being essentially constant; such estimates are referred to as `Type I/II' and `Type I$_j$' estimates by Harman \cite{Harman}, or `Type III' estimates by Zhang \cite{Zhang}). One would naturally like to have a generalization of the methods presented here to be able to incorporate such additional estimates.

Similarly, sometimes other additional assumptions have been used. In \cite{DFI} the fact that one was working with a set of positive density was vital (as seen in Theorem \ref{thm:continuity-1/2}). As mentioned in Section \ref{sec:examples-from-literature}, sometimes authors only establish a version of \eqref{eq:TypeI} or \eqref{eq:TypeII} for specific coefficient sequences (it is often only necessary to establish \eqref{eq:TypeII} when the coefficients $\xi_m,\kappa_n$ are the indicator function of certain types of prime factorization), although often the methods generalize to give \eqref{eq:TypeII}.

It would be naturally be desirable to have a theory which can incorporate such additional arithmetic information, or to generate new means to distinguish sets which contain primes from the examples produced here which do not. Given the arithmetic complexity of simply understanding the constants $C^\pm(\gamma,\theta,\nu)$, we have not attempted such a generalization. Give  a polytope $\cT$, an assumption of trilinear estimates of the form
\begin{align}\label{Type-III}
\bigg|\ssum{n=abc \\ (\frac{\log{a}}{\log{n}},\frac{\log{b}}{\log{n}},\frac{\log{c}}{\log{n}})\in \cT } \alpha_a \beta_b \zhu_c \ssc{w}{abc} \bigg| \le
\frac{x}{(\log x)^{B}},
\end{align}
for any choice of divisor-bounded sequences $\alpha_a, \beta_b, \zhu_c$ could be quite easily incorporated into the basic setup, although the subsequent analysis would naturally be more involved. Let $\cR'$ denote the set of vectors $\bx \in \cR$
which do not have a decomposition $\bx=\bx_1 \sqcup \bx_2 \sqcup \bx_3$
with $(|\bx_1|,|\bx_2|,|\bx_3|)\in \cT$. The arguments in Sections \ref{sec:constructions} and \ref{sec:sieving} show that
the analogs of Theorems \ref{thm:constructions} and \ref{thm: Main sieving}
hold with $\cR$ replaced by $\cR'$, that is, reducing further the support of
the functions $f$ and $g$ appearing there. The introduction of $k$-linear sums with arbitrary coefficients, as in
\eqref{Type-III}, with $k\ge 4$, will have a similar effect on reducing the region $\cR$.

\bigskip

%
{\Large \section{Notation and basic tools}\label{sec:Notation}}
%

\subsection{Notational conventions.}
\begin{itemize}
\item The symbol $p$, with or without subscripts, always denotes a prime.
\item $[k]$ denotes the set $\{1,\ldots,k\}$.
\item $A \sqcup B$ denotes the disjoint union; i.e., it is assumed that $A$ and $B$ 
are disjoint.  This is used frequently in summations.
\item $\omega(n)$ is the number of distinct prime factors of $n$.
\item $\Omega(n)$ is the number of prime power divisors of $n$; i.e., the number of
prime factors of $n$ counted with multiplicity.
\item $\tau_k(n)$ is the $k$-fold divisor function, the number of $k$-tuples of 
positive integers $(d_1,\ldots,d_k)$ with $d_1\cdots d_k=n$.
\item $P^-(n)$ is the smallest prime factor of $n$, with $P^-(1)$ defined to be $\infty$.
\item $P^+(n)$ is the largest prime factor of $n$, with $P^+(1)=1$.
\item $\emptyset$ is the empty set
\item $\varnothing$ is the empty vector, the unique `vector' of  dimension zero.
\item a sequence, set or function is `1-bounded' if all the elements/terms/values are complex numbers of modulus at most 1.
\item $\one(S)$ and $\one_S$ denote the indicator function of the statement $S$ being true.
\end{itemize}

\begin{defn}[Vectors of prime factors]\label{def:vn}
Let $n\in \NN$, $n\ge 2$ and $n=p_1\cdots p_k$ with $p_1\le \cdots \le p_k$.  Define
\[
\bv(n) = \bigg( \frac{\log p_1}{\log n}, \ldots, \frac{\log p_k}{\log n}\bigg),
\quad \bv(n;x) = \bigg( \frac{\log p_1}{\log x}, \ldots, \frac{\log p_k}{\log x}\bigg).
\]
\end{defn}

\medskip

\subsection{Combinatorial identities}

We begin with a modification of Heath-Brown's identity (\cite{H-B}; see also
\cite[Proposition 17.2]{Opera}).

\medskip

\begin{lem}[Modified Heath-Brown identity]\label{lem:HeathBrown}
For $1 \le n \le x$ and any positive integer $h$, we have
\[
(\log n)\one_{n\text{ prime}}=\sum_{j=1}^h (-1)^{j-1}\binom{h}{j}\sum_{r\le \frac{\log{x}}{\log 2}}\mu(r)\sum_{\substack{n=(e_1\cdots e_j f_1\cdots f_j)^r\\ e_i^r\le x^{1/h} \; (1\le i\le j)}}(\log{f_1})\mu(e_1)\cdots \mu(e_j).
\]
\end{lem}
\begin{proof}
Let $\zeta(s)$ be the Riemann zeta function.
We have that for $\Re(s)>2$
\[
-\sum_{r}\mu(r)\frac{\zeta'}{\zeta}(rs)=\sum_{r,p,\ell}\mu(r)\frac{\log{p}}{p^{\ell r s}}=\sum_{j}\frac{\log{p}}{p^{js}}\sum_{r|j}\mu(r)=\sum_{p}\frac{\log{p}}{p^s}.
\]
On the other hand,
\[
\frac{\zeta'}{\zeta}(rs)=\frac{\zeta'}{\zeta}(rs)\Bigl(1-\zeta(rs)M_{y}(rs)\Big)^h+\sum_{j=1}^h(-1)^{j-1} \binom{h}{j}M_y(rs)^j\zeta(rs)^{j-1}\zeta'(rs),
\]
where $M_y(s)=\sum_{m \le y}\mu(m)m^{-s}$. The first term has no coefficient of $n^{-s}$ if $n \le y^{rh}$. Therefore, taking $y=x^{1/(rh)}$ and equating coefficients of $n^{-s}$ for $n\in [2,x]$ gives
\[
(\log n)\one_{n\text{ prime}}
=\sum_{j=1}^h(-1)^{j-1}\binom{h}{j}\sum_{r}\mu(r)\sum_{\substack{n=(e_1\cdots e_j f_1\cdots f_j)^r\\ e_i^r\le x^{1/h}}}(\log{f_1})\mu(e_1)\cdots \mu(e_j).
\]
The innermost summand is nonzero only if $e_1\cdots e_j f_1\cdots f_j \ge 2$ and hence the inner sum is 
nonempty only when $r\le \frac{\log x}{\log 2}$.
\end{proof}

We next turn to a truncated, vector version of the function appearing in
 Linnik's identity.  Recall that Linnik's identity (\cite[0.6.13]{Linnik}; see
 also \cite[\S 17.2]{Opera}) states that
 \[
 \frac{\Lambda(n)}{\log n} = \sum_{j=1}^\infty \frac{(-1)^{j+1}}{j}
 \ssum{d_1\cdots d_j=n \\ d_i>1\, \forall i} 1.
 \]

For a vector $\bx= ( \ssc{x}{1},\ldots,x_r )$ and $c>0$ (we also allow $c=\infty$) we denote
\begin{equation}\label{Linnik-fcn}
\cyrL_c(\bx) := \sum_{j=1}^\infty \frac{(-1)^{j+1}}{j} \ssum{\by_1 \sqcup \cdots \sqcup
  \by_j = \bx \\ \dim(\by_i) \ge 1 \; \forall i \\ |\by_i| < c\; \forall i} 1.
\end{equation}

In Section \ref{sec:continuity}, we
 need an evaluation of $\cyrL_c$ for special types of vectors.

\begin{lem}\label{lem:tc}
Suppose that $\ell\ge 1$, $k\ge 0$, $c>0$ and $\ssc{x}{1},\ldots,x_\ell,\xi_1,\ldots,\xi_k$ satisfy 
\begin{align*}
x_i &< c \; \forall i \\
 x_i+x_j &\ge c \;\; (i\ne j), \\
x_i + \xi_j &\ge c \;\; (\text{all } i,j),\\
\xi_1+\cdots+\xi_k &< c.
\end{align*}
Then
\[
\cyrL_c(\ssc{x}{1},\ldots,x_\ell,\xi_1,\ldots,\xi_k) = (-1)^{\ell+k+1} (\ell-1)! \ell^k.
\]
\end{lem}

\begin{proof}
When partitioning the components $\ssc{x}{1},\ldots,x_\ell,\xi_1,\ldots,\xi_k$ into nonempty sets, 
each with sum $<c$, the components $x_i$ must be alone in singleton sets, and
the variables $\xi_1,\ldots,\xi_k$ may be placed arbitrarily in the other sets.
  Given that we are partitioning into $\ell+j$ sets,
the number of ways to place the variables $\ssc{x}{1},\ldots,x_\ell$ is $\frac{(\ell+j)!}{j!}$,
and so the number of ways to partition the variables equals
\[
\frac{(j+\ell)!}{j!} L_{j,k}, \quad L_{j,k} := \ssum{B_1 \sqcup \cdots \sqcup B_j=[k] \\ |B_i|\ge 1\; \forall i} 1 = \ssum{d_1+\cdots+d_j=k\\ d_i\ge 1\; \forall i} \frac{k!}
{d_1! \cdots d_j!}.
\] 
Therefore, $\cyrL_c(\bx,\bxi) = M_{\ell,k}$, where
\[
M_{\ell,k} := \sum_{j=0}^k (-1)^{j+\ell+1} \frac{(j+\ell-1)!}{j!} L_{j,k}.
\]
In particular, when $k=0$ only the term $j=0$ appears and 
we have $\cyrL_c(\bx) = M_{\ell,0}= (-1)^{\ell+1}(\ell-1)!$.

When $k\ge 1$ the term $j=0$ does not appear and we use generating functions.  Define
\[
F(y,z) = \sum_{\ell\ge 1} \sum_{k\ge 1} \frac{M_{\ell,k}}{k!(\ell-1)!} y^{\ell-1}z^k.
\]
Since $M_{\ell,k}=\cyrL_c(\bx,\bxi)\le \sum_{j\le \ell+k}(\ell+k)^j\ll (\ell+k)^{\ell+k}$, we see that this converges absolutely for $y,z$ sufficiently small. Then
\begin{align*}
F(y,z) &= \sum_{\ell\ge 1} \sum_{k\ge 1} \sum_{j=1}^k y^{\ell-1} (-1)^{j+\ell-1}
\binom{j+\ell-1}{j}  \ssum{d_1+\cdots+d_j=k\\ d_i\ge 1\; \forall i} \frac{z^{d_1+\cdots+d_j}}
{d_1! \cdots d_j!},
\end{align*}
which again converges absolutely for $y,z$ sufficiently small. Substitute $h=\ell-1$.  With $h,j$ fixed the sum over $k$ of the inner sum on 
$d_1,\ldots,d_j$ equals $(e^z-1)^j$.  Thus,
\begin{align*}
F(y,z) &= \sum_{j=1}^\infty (1-e^z)^j \sum_{h=0}^\infty (-y)^h \binom{h+j}{j}\\
&=\sum_{j=1}^\infty \frac{(1-e^z)^j}{(1+y)^{j+1}} = \frac{1-e^z}{(1+y)(y+e^z)}
=\frac{1}{y+e^z} - \frac{1}{y+1} = \frac{e^{-z}}{1+y e^{-z}} - \frac{1}{y+1}.
\end{align*}
Thus, $M_{\ell,k}$ is $(\ell-1)! k!$ times the coefficient of $y^{\ell-1} z^k$
in $e^{-z}/(1+y e^{-z})$.  Computing
\begin{align*}
\frac{e^{-z}}{1+y e^{-z}} &= e^{-z} \sum_{\ell=1}^\infty
(-y e^{-z})^{\ell-1} = \sum_{\ell=1}^\infty (-y)^{\ell-1}\sum_{k=0}^\infty \frac{(-\ell z)^k}{k!},
\end{align*}
the claimed formula for $M_{\ell,k}$ follows.
\end{proof}

%
%

The following identity will play a crucial role in the proof of our 
main construction result, Theorem \ref{thm:constructions}.

\begin{lem}\label{lem:t-ident}
Let $m,k$ be  positive integers and $0<c$.
If $\dim(\bx)=k$ and $|\bx| \ge mc$ then
\be\label{t-ident}
\sum_{r=1}^k \frac{m^r}{r!} \sum_{\bu_1 \sqcup \cdots \sqcup \bu_r = \bx} 
\cyrL_{c}(\bu_1)\cdots \cyrL_{c}(\bu_r) = 0.
\ee
\end{lem}

\begin{proof}
Fix $k\ge 1$ and a list $\bx= (\ssc{x}{1},\ldots,\ssc{x}{k})$.
Let $Y_k$ be the collection of all multisets with elements $1,2\ldots,k$. 
Let $\ssc{z}{1},\ldots,\ssc{z}{k}$ be complex numbers with $|z_i|  < \eps$ for a sufficiently small $\eps$ ($\eps$ may depend on $k$), and for $B \in Y_k$ define
\[
\bz^B = z_1^{\# \{ 1\in B \}} \cdots z_k^{ \# \{ k\in B \} }.
\]
Let
\[
V(\bz) = \ssum{A \subseteq [k] \\ |\ssc{\bx}{A}| < c} \prod_{i\in A} z_i.
\]
Then
\begin{equation}\label{VzB}
\log V(\bz) = \log(1 + (V(\bz)-1)) = \sum_{j=1}^\infty \frac{(-1)^{j+1}}{j}
(V(\bz)-1)^j = \sum_{B \in Y_k} \cyrL_c( \bx_i : i\in B  ) \bz^B,
\end{equation}
where $(\bx_i : i\in B )$ includes $h$ copies of $x_i$ if there are $h$ copies of 
$i$ in $B$.
Hence
\[
V(\bz)^m = \exp \bigg\{ m \sum_{B \in Y_k} \cyrL_c( \bx_i : i\in B  ) \bz^B \bigg\}.
\]
Since $|\bx| \ge mc$, the coefficient of $z_1\cdots z_k$ on the left is zero and the coefficient of $z_1\cdots z_k$ on the right side equals
\[
\sum_{r=1}^k \frac{m^r}{r!} \sum_{\bu_1 \sqcup \cdots \sqcup \bu_r = \bx} \cyrL_{c}(\bu_1)\cdots \cyrL_{c}(\bu_r). \qedhere
\]
\end{proof}

The following are more basic properties of $\cyrL_c(\bx)$.

\begin{lem}\label{lem:tc-basic}
(a) If some component $x_i$ of $\bx$ is $\ge c$, then $\cyrL_c(\bx)=0$.

(b) If $\bx=(\ssc{x}{1},\ldots,\ssc{x}{k})$ and $|\bx| < c$ then $\cyrL_c(\bx) = \one(k=1)$.
\end{lem}

\begin{proof}
Part (a) is immediate from \eqref{Linnik-fcn}.
Part (b) is the analog of Linnik's identity (\cite{Linnik}; see also \cite[Proposition 17.1]{Opera}), since $\cyrL_c(\bx)=\cyrL_\infty(\bx)$.
In the notation of the previous lemma, we have $V(\bz)=(1+z_1)\cdots (1+z_k)$
and part (b) follows by comparing the coefficients of $z_1\cdots z_k$ 
on the left and right sides of \eqref{VzB}.
\end{proof}

\subsection{Geometry of $\cR$ and $\cC(\cR)$}\label{sec:geometry}

Throughout this subsection, $P=(\gamma,\theta,\nu)$ is fixed and $\cR = \cR(P)$.

\begin{lem}\label{lem:R1-properties}
Suppose that $0<\gamma<1$ and $0\le \theta \le \theta+\nu \le 1$.
Suppose that $\by \in \cR$. Then 
\begin{enumerate}
\item[(i)]  $\nu < 1-\gamma$ and $y$ has a component bigger than $\nu$.
\item[(ii)] $\sum_{y_i\le \nu}y_i \le 1-\gamma-\nu$.
\end{enumerate}
\end{lem}

\begin{proof}
Suppose that no component $y_i$ is larger than $\nu$.
It is then clear that some subsum of the components $y_i$ lies in \two,
thus $\by\not\in \cR$.  Hence $\by$ has some component bigger than 
$\nu$ and so $\nu<1-\gamma$ since all components of vectors  in $\cR$
 are at most $1-\gamma$. This proves (i).

Now let $A$ be such that $|\by_A|$ is the smallest subsum of $\by$ which
is larger than $\theta+\nu$. If $m\in A$ then $|\by_A| - y_m < \theta$,  
since $\by$ has no subsum in $\two$ and the minimality of $|\by_A|$ means 
that $|\by_A|-y_m\le \theta+\nu$. In particular, we have $y_m > \nu$ for all $m\in A$.
Moreover, we must have
$|\by_A| - y_m + \sum_{y_i\le \nu}y_i < \theta$ as well, since otherwise $\by$ would have a subsum lying in \two.  It follows that
\begin{align*}
\sum_{y_i\le \nu}y_i &< \theta - (|\by_A| - y_m) \\
&< \theta + y_m - (\theta+\nu)\\
&< 1-\gamma-\nu.
\end{align*}
This completes the proof of (ii).
\end{proof}

The following results geometric lemmas concern convex polytopes.  
For this paper, we adopt the following definition.

\begin{defn}[Convex polytopes]
A \emph{convex polytope} is a bounded subset of $\RR^k$,
for some $k\ge 1$,
which is defined by a finite number of linear constraints on $(\ssc{x}{1},\dots,\ssc{x}{k})$ 
of the form
$\ssc{c}{1} \ssc{x}{1}+\cdots +\ssc{c}{k} \ssc{x}{k} < b$ or
$\ssc{c}{1} \ssc{x}{1}+\cdots +\ssc{c}{k} \ssc{x}{k} \le b$, where $\ssc{c}{1},\ldots,\ssc{c}{k},b$ are real numbers.
We also consider $\{ \emptyvec \}$ to be the ``trivial polytope'' of dimension 0.
\end{defn}

Frequently, our polytopes lie on the hyperplane $\ssc{x}{1}+\cdots+\ssc{x}{k}=1$.  This constraint
may be encoded using $\ssc{x}{1}+\cdots+\ssc{x}{k}\le 1$ and $-\ssc{x}{1}-\cdots-x_n \le -1$.

\begin{lem}\label{lem:CR1-union-polytopes}
Suppose that $\sigma \le \nu \le 1-\gamma$.
There is a constant $N\ll 1$, a collection of disjoint convex polytopes 
$\cT_1,\ldots,\cT_N$ (of variable dimension), a constant $J \ll 1$ and a collection of disjoint convex polytopes  $\cU_{1},\dots,\cU_{J}$ (depending only on $\nu,\theta,\gamma,\sigma$) such that for any $r$ and any $\by\in [0,\sigma)^r$ with $|\by|\le 1$ we have
\[
\big\{\bx : \,x_i \ge \sigma\, \forall i,\,|\bx|=1-|\by|,\,(\by,\bx)\in \cC(\cR)\big\}=
\Bigl(\bigsqcup_{j\le N}\cT_{j}\Bigr)\cap\{\bx : |\bx|=1-|\by|\}
\]
and
\[
\big\{\bx:\,x_i \ge \sigma\, \forall i,\,|\bx|=1-|\by|,\,(\by,\bx)\notin \cC(\cR)\big\}
=\Bigl(\bigsqcup_{j\le J}\cU_{j}\Bigr)\cap\{\bx: |\bx|=1-|\by|\}.
\]
\end{lem}
We note that since each $\cT_j,\cU_i$ only depend on $\gamma,\theta,\nu,\sigma$, the number of constraints defining them is also bounded in terms of $\gamma,\theta,\nu,\sigma$. 
We also note that the left hand side of the first display above is empty if $|\by|\ge\theta$, since then $\by$ will have a subsum in $\two$. This is encoded by the polytopes $\cT_j$ only containing vectors $\bx$ with $|\bx|>1-\theta$.

\begin{proof}
Let $\by\in[0,\sigma)^r$ with $|\by|\le 1$, and let $\bx\in\RR^k$ with $x_i \ge \sigma$ for all $i$ and 
$(\by,\bx)\in \cC(\cR)$. Clearly we have that $|\bx|=1-|\by|$ and $k\ll 1$ (we
 emphasize that implied constants may depend on $\sigma$, but not on $r$). Moreover, 
 $\bx$ is a coagulation of a vector $\bz \in \RR^\ell$ with
 $(\by,\bz)\in \cR$ and $z_i>0$ for all $i$. Here we use the fact the components of $\by$ are smaller than $\sigma$ and $\sigma\le \nu\le 1-\gamma$.

Since $\bx$ is a coagulation of $\bz$, we have that $x_j=|\bz_{\cI_j}|$ ($1\le j\le k$) for some
 partition $[\ell]=\cI_1\sqcup\dots\sqcup \cI_k$. If $z_{i_1}+z_{i_2}< 1-\gamma$ for 
 some distinct $i_1,i_2\in \cI_j$ (and some $j\in [k]$), then we may form a new vector
  $\bz'$ by replacing the components $z_{i_1},z_{i_2}$ of $\bz$ with a single component $z_{i_1}+z_{i_2}$. We clearly have $(\by,\bz')\in \cR$ and that $\bx$ is a 
  coagulation of $\bz'$. Therefore, we may assume without loss of generality that 
  $z_{i_1}+z_{i_2}\ge 1-\gamma$ for all distinct $i_1, i_2\in \cI_j$ and all $j$.
   In particular, this means 
that for each $j$, $(z_i)_{i\in \cI_j}$ can have at most one component smaller than $(1-\gamma)/2$, so that $\ell \le k + \lfloor\frac{2}{1-\gamma}\rfloor
 =: L_k$.  Also, none of the $z_i$ are equal to zero.

We may thus assume that $\ell \le L_k$ and $\bz \in (0,1-\gamma)^{\ell}$.
The condition $(\by,\bz)\in \cR$ is equivalent to the
simultaneous conditions $|\bz|+|\by|=1$, $\bz\in (0,1-\gamma)^\ell$ and that for every subset $J\subseteq[\ell]$,
\[
\text{either} \;\;\; |\ssc{\bz}{J}| > \theta+\nu\quad \text{or}\quad |\ssc{\bz}{J}|< \theta-|\by|.
\]
Since $|\bz|=1-|\by|$, $|\bz_J|<\theta-|\by|$ is equivalent to $|\ssc{\bz}{{[\ell]\setminus J}}|>1-\theta$.  Thus, the set of conditions
 $\{\bz\in (0,1-\gamma)^\ell,\,(\by,\bz)\in \cR\}$ is equivalent to
 $\bz$ lying in the union of sets 
$ \cQ_\ell(\cJ) \cap \{ \bz : |\bz| = 1-|\by| \},$
 where $\cJ$ runs over all collections of subsets of $[\ell]$ and we define the polytopes
 \[
\cQ_\ell(\cJ) :=  \big\{ \bz \in (0,1-\gamma)^\ell : |\ssc{\bz}{J}| > \theta+\nu \; (J\in \cJ);\; |\ssc{\bz}{{[\ell]\setminus J}}| > 1-\theta \; (J\not\in \cJ) \big\}.
 \]
Many of the sets $\cQ_\ell(\cJ)$ are empty, for example if $J\in \cJ$ and $J'\not\in \cJ$
for some sets $J,J'$ with $J$ a proper subset of $J'$. 
 We see that this implies
 \begin{align*}
 \big\{\bx:&\,x_i \ge \sigma\, \forall i,\,|\bx|=1-|\by|,\,(\by,\bx)\in \cC(\cR)\big\}\\
 &=
\bigcup_{1\le k\le 1/\sigma} \; \bigcup_{k\le \ell \le L_k} \cC \bigg(  \bigcup_{\cJ} \cQ_\ell(\cJ) \bigg)\cap [\sigma,1]^k\cap\{\bx:|\bx|=1-|\by|\}\\
&= \biggl[\bigcup_{1\le k\le 1/\sigma}\; \bigcup_{k\le \ell \le L_k} \bigcup_{\cJ} 
\cC(\cQ_\ell(\cJ)) \cap [\sigma,1]^k\biggr]\cap\{\bx:|\bx|=1-|\by|\},
 \end{align*}
where the set inside the large brackets is the union of $O(1)$ convex polytopes involving $O(1)$ linear constraints, since the map replacing some coordinates by their sum is a linear projection sending a convex polytope involving a bounded number of constraints to a convex polytope involving a bounded number of constraints.  This proves the first claim.

Next,
\[ 
\big\{\bx:\,x_i \ge \sigma\, \forall i,\,|\bx|=1-|\by|,\,(\by,\bx)\notin \cC(\cR)\big\}=
\{ \bx : x_i\ge \sigma\, \forall i,|\bx|=1-|\by|\} \setminus  \bigg( \bigsqcup_{i\le N} \cT_i \bigg).
\]
Since the difference of two polytopes involving a bounded number of constraints can be written as
a disjoint union of convex polytopes each involving a bounded number of linear
constraints, and the intersection of a finite number of convex polytopes is itself a
convex polytope, we can write the right side
as a finite union of disjoint convex polytopes $\cU_{1},\dots,\cU_{J}$,
each intersected with $\{\bx: x_i \ge \sigma\, \forall i,\,|\bx|=1-|\by|\}$.
This proves the second claim.
\end{proof}

\subsection{Covering the boundary of convex regions with boxes.}

\begin{lem}\label{lem:covering}
For each positive integer $k$ there is a constant $c_k$ such that
for any convex region $\cT$ in $(0,1]^k$ and
any $N\in \NN$, the number of boxes of the form $\cB(\bd):=(\frac{d_1-1}{N},\frac{d_1}{N}] \times \cdots \times (\frac{d_k-1}{N}, \frac{d_k}{N}]$ with $d_1,\ldots,d_k$ integers,
which intersect the boundary of $\cT$ is at most $c_k N^{k-1}$.
\end{lem}

This is essentially known.
We thank Denka Kutzarova for finding the first proof, which is based on
properties of certain mappings 
between the boundary $\delta \cT$ of $\cT$ and the boundary of
another convex set $\cU$ with $\cT \subseteq \cU$, see for example
 Br\'ezis \cite[Ch. 5]{Brezis}.  These arguments show that the
the number of balls of radius $\eps$ needed to cover $\delta \cT$ is
at most a constant, depending on $k$, multiple of
the number of such balls required to cover $\delta \cU$; in our case, $\cU=[0,1]^k$.
We subsequently found that Lemma \ref{lem:covering} also follows quickly from
a much more general theorem of Lassak \cite[Corollary 1]{Lassak},
which compares the number of tiles, defined by a family of hyperplanes,
intersecting $\delta \cT$ with the number of tiles  intersecting $\delta \cU$.
Both of these proofs produce a constant $c_k$ in Lemma \ref{lem:covering} that is exponential in $k$.

Here, we provide a short, self-contained proof
with a constant $c_k$ of quadratic growth.

\begin{proof}
For $\bx \in \delta \cT$,
let $\bu(\bx)$ be an outer normal vector of $\cT$ at $\bx$.
Such a vector exists by the supporting hyperplane theorem
(see, e.g., \cite[Cor. 11.6.1]{Rockafellar}), and need not be unique.
Then for any other point $\bx' \in \delta T$,
\begin{equation}\label{orthog}
\bu(\bx) \cdot (\bx'-\bx) \le 0.
\end{equation}
By scaling, we may assume without loss of generality that the maximum of the absolute values of the components of $\bu(\bx)$ equals 1. Denoting the $j-$th component of
$\bu(\bx)$ by $\bu(\bx)_j$, we define
\[
A_j= \{ \bx \in \delta \cT : \bu(\bx)_j =1 \}, \;\; A_{k+j} =  \{ \bx \in \delta \cT : \bu(\bx)_j = -1\}
\quad (1\le j\le k), 
\]
so that $\delta \cT = A_1 \cup \cdots \cup A_{2k}$.
Fix $j\in [k]$, fix $d_i \in [0,N]\cap \ZZ$ for each $i\ne j$ and assume that there
are $k+1$ integers $d_j$ so that $\cB(\bd) \cap A_j \ne \emptyset$.
Then there are two points $\bx,\bx'\in A_j$ so that their components satisfying
\[
x_j'-x_j > \frac{k-1}{N}, \qquad |x_i-x_i'| < \frac{1}{N} \;\; (i\ne j).
\]
We then have
\[
\bu(\bx) \cdot (\bx'-\bx) > \frac{k-1}{N} -\sum_{i\ne j} \frac{1}{N} = 0,
\]
a contradiction.  Therefore, the number of boxes $\cB(\bd)$ which intersect $A_j$
is at most $k\cdot N^{k-1}$.  The same argument works if $j>k$, here showing
that $\bu(\bx')\cdot (\bx-\bx') >0$ if there are $k+1$ choices for $d_{j-k}$
with the other $d_i$ fixed. Summing over all $j$ we see that the total number
of such boxes intersecting $\delta \cT$ is at most $2k^2 N^{k-1}$.
\end{proof}

\subsection{Prime number sums}
%
%
\begin{lem}\label{PNT}
Fix $\eps\in (0,\frac1{100})$.  Uniformly for $x\ge 100$ and $x^{7/12+\eps}\le y\le x/2$, we have
\[
\pi(x)-\pi(x-y) = \int_{x-y}^x\; \frac{dt}{\log t} + O_{\eps}\Big( y e^{-(\log x)^{1/4}} \Big).
\]
\end{lem}
\begin{proof}
When $y>x e^{-(\log x)^{1/3}}$, this follows from the classical prime number theorem estimate
$\pi(x)=\int_2^x dt/\log t + O(xe^{-c\sqrt{\log x}})$ with a constant $c>0$.  
Otherwise apply
Huxley's prime number theorem for short intervals, with a more explicit
error term.  By Theorem 1.1 of \cite{TZ}, for a constant $c(\eps)>0$,
\[
\sum_{x-y<p\le x} \log p = y + O_\eps\Big( y e^{-c(\eps)(\log x)^{1/3}(\log\log x)^{-1/3}} \Big) = y+ O_\eps \Big( y e^{-(\log x)^{1/4}} \Big).
\]
For $x-y<u\le x$ we have $\log u = \log x + O(y/x)$
and the result follows for $y\le x e^{-(\log x)^{1/3}}$.
Alternatively, one may follow the argument in section 12.5 of \cite{Ivic}, 
taking $T=x^{5/12}$ and $u=\frac12$ there.
\end{proof}

 The next lemma is a standard type
result relating sums over vectors of primes to multiple integrals,
although we need a version with some of the prime factors fixed
and valid for short intervals.

\medskip

\begin{lem}\label{lem:sum_to_int}
Fix $0<\eta\le \frac14$ and integers $k$ and $t$ with $0\le t\le k-1$.
Let $\cT$ be a  convex polytope with
\[
\cT \subseteq \{ (\ssc{x}{1},\ldots,\ssc{x}{k}) : \ssc{x}{1}+\cdots+\ssc{x}{k}=1, \eta \le \ssc{x}{1} < \cdots < x_t, \eta \le x_{t+1} < \cdots < \ssc{x}{k} \}.
\]
Let $f:\cT\to\RR$ with $|f|\le 1$ on $\cT$,
and $f$ Lipschitz continuous on $\cT$ with Lipschitz constant $1$.
Let $m$ be a positive, squarefree integer with $m \le x^{1-\eta}$, 
and prime factorization $m=p_1\cdots p_t$, where $p_1 < p_2 < \cdots < p_t$.
Let $p_j=x^{u_j}$ for each $j$ and $\lambda = u_1+\cdots+u_t$, so that $m=x^{\lambda}$.

Then, uniformly for any $\frac12 \le c \le d \le 1$ and any $m$, we have
\begin{multline*}
\ssum{cx/m < r \le dx/m \\ (\bv(m;mr),\bv(r;mr)) \in \cT \\ \mu^2(r)=1,\, \omega(r)=k-t \\ (r,m)=1} 
f \big( \bv(m;mr),\bv(r;mr) \big) = 
\int_{1-\lambda+\frac{\log c}{\log x}}^{1-\lambda+\frac{\log d}{\log x}}
x^w \mint{u_{t+1}+\cdots+u_k=w \\ \bu/(\lambda+w)\in \cT} \frac{f(\frac{u_1}{\lambda+w},\ldots,\frac{u_k}{\lambda+w})}{u_{t+1}\cdots u_k}\, d \bu \, dw\, + \\
+ O_{\eta}\bigg(  \frac{x^{1-\eta/3}}{m}+\frac{x(d-c)}{m} e^{-(\log x)^{1/5}}\bigg).
\end{multline*}
When $t=k-1$ we interpret the right side as a single integral over $w$ only,
with 
\[
u_k=w-(u_1+\cdots+u_{k-1}).
\]
\end{lem}

\begin{proof}
 Write each $r$ as $r=p_{t+1}\cdots p_k$, $p_{t+1}< \cdots < p_k$.
We first show that we may remove the condition $(r,m)=1$ from the summation
with only a small error.  Indeed, the conditions on the sum imply
that $p_{t+1} \ge (x/2)^{\eta}$, hence the terms in the sum with 
$(r,m)>1$ total
\begin{align*}
\ll \ssum{p|m \\ p \ge (x/2)^\eta} \ssum{x/2<n\le x \\ pm|n} \ll \frac{x^{1-\eta}}{\eta m}.
\end{align*}

Let $\delta_1=\frac{-\log c}{\log x}$, $\delta_2 = \frac{-\log d}{\log x}$,
 so that $x^{1-\lambda-\delta_1}<r\le x^{1-\lambda-\delta_2}$.
Since $\ssc{x}{1}+\cdots+\ssc{x}{k}=1$, we may assume that $e_0=0$.
Let $\cU$ denote the set of vectors $(u_{t+1},\ldots,u_k)$ such that
\[
1-\lambda-\delta_1 < u_1+\cdots+u_k \le 1-\lambda-\delta_2, \qquad
\bigg( \frac{u_1}{u_1+\cdots+u_k},\dots,\frac{u_k}{u_1+\dots +u_k}\bigg) \in \cT.
\]
Since $\cT$ is a convex polytope, it follows that $\cU$ is also a convex polytope, and that
the simultaneous conditions
\[
(\bv(m;mr),\bv(r;mr))\in \cT, \quad cx/m<r\le dx/m
\] 
are equivalent to $\by := (\frac{\log p_{t+1}}{\log x},\ldots,\frac{\log p_k}{\log x}) \in \cU$. Writing $g(r)=f\big( \bv(m;mr),\bv(r;mr) \big)$,
it thus remains to show that
\begin{multline}\label{eq:sum-to-int-reduction}
\ssum{r=p_{t+1}\cdots p_k \\ \by \in \cU} g(r) = 
\mint{(u_{t+1},\ldots,u_k)\in \cU}  \frac{f\big(\frac{(u_1,\ldots,u_k)}{\lambda+u_{t+1}+\cdots+u_k}\big)x^{u_{t+1}+\cdots+u_k}}
{u_{t+1}\cdots u_k} du_{t+1}\cdots du_k + \\
+O_{\eta}\bigg( \frac{x^{1-\eta/3}}{m}+\frac{x(d-c)}{m} e^{-(\log x)^{1/5}}\bigg).
\end{multline}
If $\cU$ is empty, there is nothing to prove and thus we suppose that $\cU$ is nonempty.
  Define
\be\label{Nc1}
N := \big\lfloor x^{\eta/3} \big\rfloor.
\ee
For an $(k-t)-$tuple $\bd=(d_{t+1},\ldots,d_k)\in [N]^{k-t}$, let
\[
B(\bd) := \(\frac{d_{t+1}-1}{N},\frac{d_{t+1}}{N}\right] \times \cdots \times
\(\frac{d_k-1}{N},\frac{d_k}{N}\right].
\]
Let $\cD$ be the set of $\bd$ such that $\cB(\bd) \subseteq \cU$,
and let $\cD^*$ denote the set of $\bd$ with $\cB(\bd)$ intersecting $\cU$
but not contained in $\cU$; for such $\bd$, $\cB(\bd)$ intersects the boundary 
of $\cU$. In particular, if $\bd \in \cD$, then
$d_{t+1}<\cdots < d_k$.
Since $|g(r)|\le 1$ for all $r$, 
\begin{equation}\label{eq:sum-to-int-DD*}
\ssum{r=p_{t+1}\cdots p_k \\ \by \in \cU} g(r) = \sum_{\bd\in \cD}\;
\ssum{r=p_{t+1}\cdots p_k \\ \by \in \cB(\bd)} g(r) +
O\bigg( \sum_{\bd\in \cD^*} \;\ssum{r=p_{t+1}\cdots p_k \\ \by \in \cB(\bd)} 1 \bigg).
\end{equation}
For each fixed $\bd$, $g(r)$ is almost constant over $r$ with $\by\in \cB(\bd)$,
 and thus we are left to estimate the size of the sets
\[
Q(\bd) := \{ r=p_{t+1}\cdots p_k : \by \in \cB(\bd) \}.
\]
If $\bd\in \cD \cup \cD^*$ then necessarily
\begin{equation}\label{eq:dtdk}
\begin{split}
N(1-\lambda-\delta_1) &\le d_{t+1}+\cdots+d_k \le N(1-\lambda-\delta_2)+k-t,\\
d_j  & \ge N\eta \quad (t+1\le j\le k).
\end{split}
\end{equation}
Since $N\ge \log x$ for large enough $x$,
\begin{equation}\label{eq:y-lower-dj}
x^{d_j/N}-x^{(d_j-1)/N} \order x^{d_j/N}\cdot \frac{\log x}{N} \gg \frac{x^{d_j/N}}{N} \ge x^{(2/3)d_j/N}.
\end{equation}
Thus, Lemma \ref{PNT} with $\eps=1/20$ implies that for all $\bd\in \cD$,
\begin{align*}
|Q(\bd)| &= \prod_{j=t+1}^k \( \pi(x^{d_j/N})-\pi(x^{(d_j-1)/N}) \) \\
&= \prod_{j=t+1}^k \bigg( \int_{x^{(d_j-1)/N}}^{x^{d_j/N}}\; \frac{dt}{\log t} + 
  O\Big( \frac{x^{d_j/N}\log x}{N}  e^{-(\eta\log x)^{1/4}} \Big) \bigg) \\
&= \prod_{j=t+1}^k \bigg( \int_{(d_j-1)/N}^{d_j/N} \frac{x^u}{u}\, du + 
O_\eta \Big(\frac{x^{d_j/N}}{N}  e^{-(\log x)^{1/5}} \Big) \bigg).
\end{align*}
Crudely, by \eqref{eq:dtdk} and \eqref{eq:y-lower-dj},
\[
\int_{(d_j-1)/N}^{d_j/N} \frac{x^u}{u}\, du \ll\frac{x^{d_j/N}}{\eta N} .
\]
Again using \eqref{eq:dtdk}, we find that for $d\in \cD$,
\begin{equation}\label{Qa}
\begin{split}
|Q(\bd)| &=  \mint{B(\bd)} \frac{x^{u_{t+1}+\cdots+u_k}}{u_{t+1}\cdots u_k}\, d\bu +
O_\eta \Big(\frac{x/m}{N^{k-t}}  e^{-(\log x)^{1/5}} \Big).
\end{split}
\end{equation}
Since the right side of \eqref{Qa} is a valid
upper bound for $|Q(\bd)|$ when $\bd\in \cD^*$, we have 
\be\label{Qa-bound}
|Q(\bd)| \ll_\eta \frac{x/m}{N^{k-t}} \qquad (\bd \in \cD \cup \cD^*).
\ee
It follows from Lemma \ref{lem:covering} that $|\cD^*| \ll (k-t)^2 N^{k-t-1}
\ll \eta^{-2} N^{k-t-1}$ and hence
\be\label{Qstar-r}
\sum_{\bd\in \cD^*} |Q(\bd)|  \ll_{\eta} \frac{x}{m N}
\ee
and
\be\label{Qstar-Ba}
\sum_{\bd\in \cD^*} \text{meas} (B(\bd)) \ll_{\eta} \frac{1}{N}.
\ee
In particular, the big-$O$ term in \eqref{eq:sum-to-int-DD*} is 
$O_{\eta}(x/(mN))$.

Now fix $\bd\in \cD$ and let
\[
\bv_0 = \bv_0(\bd) := \bigg( \frac{u_1}{w},\ldots,\frac{u_t}{w}, \frac{d_{t+1}/N}{w},
\ldots, \frac{d_k/N}{w} \bigg), \quad w = u_1 + \cdots + u_t + \frac{d_{t+1}+\cdots+d_k}{N}. 
\]
Then $\bv_0\in \cT$ and  $|(\bv(m;mr),\bv(r;mr))-\bv_0| \ll 1/N$ for every $r\in Q(\bd)$.  Thus, 
\[
|f(\bv(m;mr),\bv(r;mr))-f(\bv_0)|\ll 1/N \qquad (r\in Q(\bd)).
\]
Combining this with \eqref{Qa} and \eqref{Qa-bound}, and then recalling \eqref{Nc1}, we see that for $\bd\in \cD$,
\begin{align*}
\sum_{r\in Q(\bd)} g(r) &= f(\bv_0) |Q(\bd)| + O\pfrac{|Q(\bd)|}{N} \\
&=f(\bv_0) \mint{B(\bd)} \frac{x^{u_{t+1}+\cdots+u_k}}{u_{t+1}\cdots u_k}\,du_{t+1}\cdots du_k +
O_{\eta} \( \frac{x e^{-(\log x)^{1/5}}}{m N^{k-t}} \)\\
&=\mint{B(\bd)} \frac{f\big(\frac{(u_1,\ldots,u_k)}{\lambda+u_{t+1}+\cdots+u_k}\big)
x^{u_{t+1}+\cdots+u_k}}{u_{t+1}\cdots u_k}
du_{t+1}\cdots du_k +O_{\eta} \( \frac{x e^{-(\log x)^{1/5}}}{m N^{k-t}} \).
\end{align*}
By \eqref{eq:dtdk},
\[
|\cD| \ll N^{t-k+1} + (\delta_1-\delta_2)N^{t-k} \ll N^{t-k-1}+N^{t-k}(d-c).
\]
Therefore,
\begin{multline*}
\sum_{\bd\in \cD} \; \sum_{r\in Q(\bd)} g(r) =
\sum_{\bd\in \cD}\;  \mint{B(\bd)} \frac{f\big(\frac{(u_1,\ldots,u_k)}{\lambda+u_{t+1}+\cdots+u_k}\big)x^{u_{t+1}+\cdots+u_k}}
{u_{t+1}\cdots u_k} du_{t+1}\cdots du_k + \\
+O_{\eta}\Big( \frac{x}{mN} +  (d-c) \frac{x}{m} e^{-(\log x)^{1/5}} \Big).
\end{multline*}
Inserting this into \eqref{eq:sum-to-int-DD*}, and completing the
multiple integral with the missing pieces of $\cU$ using \eqref{Qstar-Ba},
we obtain \eqref{eq:sum-to-int-reduction}, as desired.
\end{proof}

\bigskip

%
%
{\Large \section{Constructions}\label{sec:constructions}}
%
%

In this section, we present a general method of constructing sequences $w_n$
satisfying \eqref{eq:TypeI} and \eqref{eq:TypeII}, which we use to find such sequences
with small or large weight on primes.
This then gives upper bounds on $C^-(P)$ and lower bounds on $C^+(P)$.

Our sequences satisfy \eqref{eq:TypeII} trivially, as $w_n$ is supported
on integers with no divisor in $((x/2)^\theta,x^{\theta+\nu}]$,
and thus the main technical issue is engineering \eqref{eq:TypeI}.
We conjecture that this class of sequences is sufficient to determine
$C^{\pm}(P)$, and give heuristic justification for this claim at
the end on this section.
In fact, this class of functions suffices for establishing $C^\pm(P)$
in the special cases considered in Theorem 
\ref{thm:1-parm theta family} and Theorem
\ref{thm:theta=0 gamma=1/2}.

Let $\cQ_0$ be the set of parameters $(\gamma,\theta,\nu)$ satisfying \eqref{Q0}
\begin{equation}
\cQ_0:=\{(\gamma,\theta,\nu):\, 0<\gamma<1, \quad 0\le \theta < \tfrac12,  \quad 0 < \nu \le 1-\theta\}.
\label{eq:Q0Def}
\end{equation}

We recall Definition \ref{defn:R1} for $\cR(P)$, and the vector notation of Section \ref{sec:Notation}. Throughout, we consider $P=(\gamma,\theta,\nu)\in\cQ_0$ fixed, and let $\cR=\cR(P)$. 
\medskip

\begin{defn}[The set $\cS$ of symmetric functions on vectors] Let $\cS$ denote the set of all real-valued functions
$f$ on variable-length vectors, such that, for all $k > 0$, the restriction of $f$ to $\RR^k$ is symmetric in all variables.
\end{defn}

\begin{defn}[The set $\sF_{\eta}$ of Type I-compatible functions]\label{defn:FetaP}
Let $\sF_{\eta}$ be the set of functions $f \in \cS$ such that
\begin{itemize}
\item[(a)] $f$ is supported on vectors in $\cC(\cR)$ with all components $\ge \eta$, and is bounded;
\item[(b)] For each $k$, the restriction of $f$ to $\RR^k$ is a finite sum of functions $f_{k,j}$, where each $f_{k,j}$ is supported on
 a convex polytope and is Lipschitz continuous on the polytope;
\item[(c)] The following property holds:
\begin{equation}
\label{TypeI-f}
\text{for all } r\ge 0 \, \text{ and } \,  \xi_1+\cdots+\xi_r \le \gamma, \text{ we have }
\sum_{k\ge r+1}  \;\;\; \mint{\xi_{r+1} < \cdots < \xi_k \\ \xi_1+\cdots+\xi_k=1} \frac{f(\xi_1,\ldots,\xi_{k})}{\xi_{r+1}\cdots \xi_{k}} = 0.
\end{equation}
\end{itemize}
\end{defn}

The relations \eqref{TypeI-f} are integral analogs of the Type I bounds \eqref{eq:TypeI}.
We note that any function $f\in \sF_\eta$ is supported on
vectors with dimension at most $1/\eta$ (since the components are at least $\eta$ and sum to 1), and that
any real linear combination of functions in $\sF_{\eta}$ is also in $\sF_{\eta}$.
Roughly speaking, we will build sequences $(a_n),(b_n),(w_n)$ from $f\in \sF_{\eta}$ via the relation $w_n=f(\bv(n))$, $b_n=1$.  The restriction of the support of $f$ to $\cC(\cR)$
ensures that \eqref{eq:TypeII} holds, and relation \eqref{TypeI-f} implies that \eqref{eq:TypeI} holds.
Also, the singleton value $f(1)$ represents the weight $w_p$ on primes.
Our object is to find $f$ to optimize $f(1)$, subject to 
$f(\bx)\ge -1$ for all $\bx$.

%
%

\begin{thm}[Constructions of sequences with small/large prime counts]\label{thm:constructions}
Let $P=(\gamma,\theta,\nu)\in \cQ_0$ and $f\in \sF_\eta$ with $0 < \eta < 1-\gamma$. Then:
\begin{itemize}
\item[(a)] For any $\delta>0$, there is a constant $z=z(\delta)$ with $|z-f(1)|\le \delta$
such that for all $B>0$ and large enough $x$ (in terms of $B,\delta,\gamma,\theta,\nu,\eta$), there is a bounded sequence $(w_n)$
satisfying \eqref{eq:TypeI} and \eqref{eq:TypeII} and with:
\begin{itemize}
\item[(i)] $w_p=z$ for all primes $p\in (x/2,x]$,
\item[(ii)]  $w_n \ge \min \Big( 0, \min\limits_{\dim(\bu)\ge 2} f(\bu) - \delta \Big)$ for $n\in(x/2,x]$ with $\Omega(n)\ge 2$. 
\end{itemize}
\item[(b)] If  $f(\bu)\ge -1$ for all $\bu$, then 
we have 
\[
C^-(P) \le 1+f(1) \le C^+(P)
\]
and for sufficiently large $\varrho$ (in terms of $f,\gamma,\theta,\nu,\eta$) we have
\[
\CB^-(P;\varrho)\le 1+f(1)\le \CB^+(P;\varrho).
\] 
Moreover, the sequences $(a_n),(b_n),(w_n)$ used to prove this all
satisfy $b_n=1$ for all $n$ and have $(a_n)$ bounded.
\end{itemize}
\end{thm}

Although \eqref{TypeI-f} has a clean formulation, some massaging leads to 
a recursive formula for $f$, whereby if some $\xi_i\ge 1-\gamma$, then $f(\bxi)$
is uniquely determined by the values of $f$ with all arguments $<1-\gamma$.
To state this result, we associate to $f$ the restrictions $f_{s,\ell}$,
where each $f_{s,\ell}$ is the restriction of $f$ to $(s+\ell)-$dimensional vectors
with $s$ components in $[\eta,1-\gamma)$ and $\ell$ components which are $ \ge 1-\gamma$.
Recall the definition of the `truncated Linnik function' $\cyrL_c(\bx)$ from \eqref{Linnik-fcn}.

%
%

\begin{thm}\label{thm:construction-recursive}
Let $P=(\gamma,\theta,\nu)\in \cQ_0$ and suppose that $f\in \cS$ satisfies axioms (a) and (b) in the definition of $\sF_\eta$.
Then \eqref{TypeI-f} is equivalent to the following statement, which we
refer to as the ``fragmentation relation'':
for all $s\ge 0,\ell\ge 1$, all $\beta_1,\ldots,\beta_s \in [\eta,1-\gamma)$ and $\alpha_1,\ldots,\alpha_{\ell} 
\ge 1-\gamma$ with $\beta_1+\cdots+\beta_s+\alpha_1+\cdots+\alpha_\ell=1$, we have
\begin{multline}\label{fsl}
f_{s,\ell}(\beta_1,\ldots,\beta_{s},\alpha_1,\ldots,\alpha_\ell) := \alpha_1\cdots \alpha_\ell\, \times \\
 \times \, \ssum{k_1,\ldots,k_\ell\ge 2 \\ k:=k_1+\cdots+k_\ell+s} \;\;\; \mint{\alpha_j=u_{j,1}+\cdots+u_{j,k_j} \\ \eta \le u_{j,1} < \cdots < u_{j,k_j} \\ 1\le j\le \ell}
\frac{\cyrL_{1-\gamma}(\bu_1)\cdots \cyrL_{1-\gamma}(\bu_\ell) f_{k,0}(\bbeta,\bu_1,\ldots,\bu_\ell)}{\prod_{j=1}^\ell \prod_{h=1}^{k_j}\, u_{j,h}} d\bu_1 \cdots d\bu_\ell,
\end{multline}
where $\bu_j = ( u_{j,1},\ldots,u_{j,k_j} )$ for $1\le j\le \ell$.
\end{thm}

Roughly speaking, each component that is $\ge 1-\gamma$ is fragmented into components
which are smaller than $1-\gamma$ on the right side of \eqref{fsl}.
Although \eqref{TypeI-f} has a cleaner formulation than \eqref{fsl},
\eqref{fsl} can be used to easily construct functions satisfying \eqref{TypeI-f}.
Indeed, one can define $f$ arbitrarily on vectors with all components $<1-\gamma$,
use \eqref{fsl} to define $f$ for other vectors, and then $f$ will automatically
satisfy \eqref{TypeI-f}.  This is, in fact, how we shall proceed.

%
%
%
%

%
%
%
%
%
\subsection{Proof of Theorem \ref{thm:construction-recursive}}
We first show that the fragmentation relation is equivalent to
the following, more symmetric, ``alternative fragmentation relation'':
for all $m\ge 1$ and  $\bxi \in \cC(\cR(P))$ with $\xi_i\ge \eta$ for all $i$, we have
\begin{multline}
\label{fragmentation-alternative}
\frac{f(\xi_1,\ldots,\xi_m)}{\xi_1\cdots \xi_m} = 
\ssum{k_1,\ldots,k_m\ge 1} \;\;
 \mint{\xi_j=u_{j,1}+\cdots+u_{j,k_j} \\ \eta \le u_{j,1}<\cdots<u_{j,k_j} \\ 1\le j\le m}\!\!
\frac{\cyrL_{1-\gamma}(\bu_1)\cdots \cyrL_{1-\gamma}(\bu_m) f(\bu_1,\ldots,\bu_m)}{u_{1,1} \cdots u_{m,k_m}}
 d\bu_1 \cdots d\bu_m,
\end{multline}
where $\bu_j = ( u_{j,1},\ldots,u_{j,k_j} )$ for $1\le j\le m$.
Indeed, this follows quickly from Lemma \ref{lem:tc-basic} (a) and (b).
In \eqref{fragmentation-alternative},
if $\xi_i < 1-\gamma$, then $\cyrL_{1-\gamma}(\bu_i)=\one(k_i=1)$
and therefore the only nonzero integrands occur when $u_{i,1} = \xi_i$.
Also, if $\xi_i \ge 1-\gamma$ and $k_i=1$ then $\cyrL_{1-\gamma}(\xi_i)=0$,
thus the nonzero integrands on the right side require $k_i\ge 2$.
Furthermore, if some $u_{j,h} \ge 1-\gamma$ then $\cyrL_{1-\gamma}(\bu_j)=0$.

It remains to show the equivalence of \eqref{TypeI-f} and the alternative
fragmentation relation \eqref{fragmentation-alternative}.
We  first assume \eqref{TypeI-f} and deduce \eqref{fragmentation-alternative},
beginning with the observation that $\cyrL_{\infty}(\ssc{x}{1},\ldots,\ssc{x}{k}) = \one(k=1)$
for any $(\ssc{x}{1},\ldots,\ssc{x}{k})$,
which follows from Lemma \ref{lem:tc-basic} (b).
Hence, for $\bxi \in \cC(\cR(P))$ with all components $\ge \eta$ and $m\ge 1$,
\[
\frac{f(\xi_1,\ldots,\xi_m)}{\xi_1\cdots \xi_m} =  \sum_{k_1,\ldots,k_m\ge 1} 
 \;\;\; \mint{\xi_j=u_{j,1}+\cdots+u_{j,k_j} \\ \eta \le u_{j,1}<\cdots< u_{j,k_j}  \\
 (1\le j\le m)} \frac{\cyrL_\infty(\bu_1)\cdots \cyrL_\infty(\bu_m) f(\bu_1,\ldots,\bu_m)}{u_{1,1}\cdots u_{m,k_m}}\, d \bu_1 \cdots d\bu_{m}.
\]
Our strategy is to use \eqref{TypeI-f} to replace each factor $\cyrL_\infty(\bu_i)$
with $\cyrL_{1-\gamma}(\bu_i)$, which gives \eqref{fragmentation-alternative}.
As observed earlier, if $\xi_i<1-\gamma$ then $\cyrL_{1-\gamma}(\bu_i) = \one(k_i=1) = \cyrL_{\infty}(\bu_i)$, so we may immediately make this replacement for all $i$
with $\xi_i < 1-\gamma$.  For the components
with $\xi_i \ge 1-\gamma$, we shall replace $\cyrL_{\infty}(\bu_i)$
by $\cyrL_{1-\gamma}(\bu_i)$ one at a time.
It thus suffices to show, for any $\xi \in [ 1-\gamma,1]$
and any vector $\by$ with $|\by| = 1-\xi$, that
\be\label{Tzero}
T(\by) := \sum_{k\ge 1} \frac{1}{k!} \;\;\; \mint{\xi=u_1+\cdots+ u_k \\ \eta \le u_i\; \forall i}
 \;\; \frac{(\cyrL_{\infty}(\bu)-\cyrL_{1-\gamma}(\bu)) f(\by,\bu)}{u_1\cdots u_k}\, d\bu = 0.
\ee
It is convenient here to leave the variables $u_1,\ldots,u_k$ unordered,
and introduce the factor $1/k!$ to compensate.
We also observe that the terms $k> 1/\eta$ in \eqref{Tzero} are zero, as
$f$ is supported on vectors with components $\ge \eta$.
From \eqref{Linnik-fcn}, we have
\[
\cyrL_{\infty}(\bu)-\cyrL_{1-\gamma}(\bu) = \sum_{j\ge 1} \frac{(-1)^{j-1}}{j} 
\ssum{\bu_1 \sqcup \cdots \sqcup \bu_j = \bu \\ \forall i: \dim(\bu_i) \ge 1 \\ \exists\, i: |\bu_{i}|\ge 1-\gamma} \;\; 1.
\]
\newcommand{\aone}{\ssc{a}{1}}
Insert this into \eqref{Tzero}, and re-organize $T(\by)$, by first fixing $j$
and  using the substitutions
$d_i=\dim(\bu_i)$ and $\phi_i = |\bu_{i}|$, so that $\phi_1+\cdots+\phi_j=\xi$,
$\phi_i\ge 1-\gamma$ for some $i$  and $d_1+\cdots + d_j=k$.
With $k,j$ and $d_1,\ldots,d_j$ fixed,
there are $\binom{k}{d_1\; d_2\; \cdots \; d_j}$ choices for $\bu_1,\ldots,\bu_j$
such that $\bu_1 \sqcup \cdots \sqcup \bu_j = \bu$.  
Since the integrand in \eqref{Tzero} is symmetric in $u_1,\ldots,u_k$, we obtain
\begin{align*}
T(\by) &= \ssum{j\ge 1 \\ k\ge 1} \frac{(-1)^{j-1}}{j k!} \!\!
\ssum{d_1,\ldots,d_j\ge 1 \\ d_1 + \cdots+d_j=k} \!\!\! \binom{k}{d_1\; d_2\; \cdots \; d_j}\;\;
\mint{\xi=\phi_1+\cdots + \phi_j \\ \eta \le \phi_i\; \forall i \\ \exists \, i: \phi_i\ge 1-\gamma} \;\;\;\; 
\mint{\phi_i=u_{i,1}+\cdots+u_{i,d_i} \\ 1\le i\le j}
\frac{f(\by,\bu)}{u_{1,1}\cdots u_{j,d_j}}\, d\bu_1\cdots d\bu_j \, d\bphi \\
&= \sum_{j\ge 1} \frac{(-1)^{j-1}}{j}\mint{\xi=\phi_1+\cdots + \phi_j \\ \eta \le \phi_i\; \forall i \\ \exists \, i: \phi_i\ge 1-\gamma}  \;\;\;
\sum_{d_1,\ldots,d_j\ge 1}\frac{1}{d_1! \cdots d_j!} \;\;\;  \mint{\phi_i=u_{i,1}+\cdots+u_{i,d_i} \\ 1\le i\le j}
\frac{f(\by,\bu)}{u_{1,1}\cdots u_{j,d_j}}\, d\bu_1\cdots d\bu_j \, d\bphi.
\end{align*}
Fix $j$ and $(\phi_1,\ldots,\phi_j)$,
and let $m$ be an index with $\phi_m \ge 1-\gamma$.
Fix all of the $d_i$ except for $d_m$ and fix the variables
$u_{i,h}$ for $i\ne m, 1\le h\le d_i$.
Let $\by'$ be the concatenation of the vector $\by$ with all of 
$u_{i,h}$ for $i\ne m, 1\le h\le d_i$, and let $g$ be the sum of the components
of $\by'$, so that $g=1-\phi_m\le \gamma$.   What remains 
is
\[
\sum_{d_m\ge 1} \frac{1}{d_m!} \;\;\; \mint{\phi_m=u_{m,1}+\cdots+u_{m,d_m}} 
\frac{f(\by',\bu_m)}{u_{1,m}\cdots u_{m,d_m}} \, d\bu_m,
\]
which equals zero by \eqref{TypeI-f}.  Therefore, \eqref{Tzero} holds
and the proof of \eqref{fragmentation-alternative} is complete.

 \medskip
 
 Now we deduce \eqref{TypeI-f} from \eqref{fragmentation-alternative}.
 Let $\bxi=(\xi_1,\ldots,\xi_r)$ with components $\ge \eta$ and $|\bxi| = g\le \gamma$.
  Then, by \eqref{fragmentation-alternative} and the symmetry of $f$,
  \begin{align*}
  \sum_{s\ge 1} \frac{1}{s!} \mint{\psi_1+\cdots+\psi_s=1-g} \frac{f(\bxi,\bpsi)}{\psi_1\cdots \psi_h} &= \sum_{k_1,\ldots,k_r\ge 1} \frac{\xi_1\cdots \xi_r}{k_1! \cdots k_r!} \mint{\xi_j=u_{j,1}+\cdots + u_{j,k_j} \\ 1\le j\le r}
  \frac{A(\bu_1,\ldots,\bu_r)}{u_{1,1}\cdots u_{r,k_r}}\, d\bu_1\cdots d\bu_r,
  \end{align*}
  where
  \begin{multline*}
  A(\bu_1,\ldots,\bu_r) = \sum_{s\ge 1} \frac{1}{s!} \;\;\; \mint{\psi_1+\cdots+\psi_s=1-g} \;\; \sum_{h_1,\ldots,h_s\ge 1} \frac{1}{h_1! \cdots h_s!} \times \\
  \times \mint{\psi_j=v_{j,1}+\cdots + v_{j,h_j} \\ 1\le j\le s}
  \frac{\cyrL_{1-\gamma}(\bv_1) \cdots \cyrL_{1-\gamma}(\bv_s) f(\bu_1,\ldots,\bu_r,\bv_1,\ldots,\bv_s)}{v_{1,1}\cdots v_{s,h_s}}\, d\bv_1 \cdots d\bv_s.
  \end{multline*}
  Now let $m=h_1+\cdots+h_s$ and relabel the variables $v_{1,1},\ldots,v_{s,h_s}$
  as $\beta_1,\ldots,\beta_m$, where
  \[
  \beta_{h_1+\cdots+h_{j-1}+i} = v_{j,i} \qquad (1\le j\le s, 1\le i\le h_j).
  \]
  The function $\frac{f(\bu_1,\ldots,\bu_r,\bbeta)}{\beta_1\cdots \beta_m}$ is symmetric in
  $\beta_1,\ldots,\beta_m$. Thus, with $h_1,\ldots,h_s$ and $\bbeta$ fixed, we may replace the factor $\cyrL_{1-\gamma}(\bv_1) \cdots \cyrL_{1-\gamma}(\bv_s)$ by its symmetric average
  \[
  \binom{m}{h_1\; h_2\; \cdots \; h_s}^{-1} \ssum{\bz_1 \sqcup \cdots \sqcup \bz_s= \bbeta \\ |\bz_i|=h_i\; (1\le i\le s)} \cyrL_{1-\gamma}(\bz_1)\cdots \cyrL_{1-\gamma}(\bz_{s}).
  \]
It follows that
\[
 A(\bu_1,\ldots,\bu_r) = \sum_{m\ge 1} \frac{1}{m!} \mint{\beta_1+\cdots+\beta_m=1-g}
 \frac{f(\bu_1,\ldots,\bu_r,\bbeta)}{\beta_1 \cdots \beta_m} \;
 \sum_{s=1}^m \frac{1}{s!} \ssum{\bz_1 \sqcup \cdots \sqcup \bz_s=\bbeta \\ \forall i: |\bz_i|\ge 1} \!\! \cyrL_{1-\gamma}(\bz_1)\cdots \cyrL_{1-\gamma}(\bz_s)\, d\bbeta.
\]
By Lemma \ref{lem:t-ident}, the sum on $s$ equals zero, and thus
$A(\bu_1,\ldots,\bu_r) = 0$.  This completes the proof of \eqref{TypeI-f}.

\subsection{Proof of Theorem \ref{thm:constructions}}
We need to handle an annoying technicality, that $\log n$ is not equal to $\log x$
when $x/2 < n\le x$.  So, even if $\bv(n)$ avoids the Type II interval $\two$,
it may be that $n$ itself has a divisor very close to the boundary of the Type II
range (II) (recall Definition \ref{def:vn} of $\bv(n)$.) Furthermore, $n^\gamma$ is slightly smaller than $x^{\gamma}$,
another issue that must be dealt with.  Fortunately, $f$ is bounded and
so we may easily sidestep these issues because the measure of affected
vectors is very small.
Fix a small $\eps>0$.  We will define a tweak of $f$, namely $h$, as follows.
Let $\cI_\eps = [\theta+\nu,\theta+\nu+\eps] \cup [\gamma-\eps,\gamma+\eps] $.
For all $k$ and $\beta_1,\ldots,\beta_k < 1-\gamma$, define 
 \[
 h(\beta_1,\ldots,\beta_k) = f(\beta_1,\ldots,\beta_k) \cdot 
 \one \big( \bbeta \text{ has no subsum in } \cI_\eps \big),
\] 
and then define $h$ for other vectors by the analog of  \eqref{fsl} with $h$
replacing $f$ everywhere.
By Theorem \ref{thm:construction-recursive}, \eqref{TypeI-f} holds for $h$;
that is, $h\in \sF_\eta$.
Now let $\bbeta=(\beta_1,\ldots,\beta_s)$ and $\balpha=(\alpha_1,\ldots,\alpha_\ell)$
with $\ell\ge 1$, $\beta_i<1-\gamma\le \alpha_j$ for all $i,j$, and $|\bbeta|+|\balpha|=1$.  
If $(\bbeta,\balpha)$ has a subsum in $\cI_\eps$ then clearly 
$(\bbeta,\bu_1,\ldots,\bu_\ell)$ does also for any choice of $\bu_1,\ldots,\bu_\ell$
in \eqref{fsl} and therefore $h(\bbeta,\balpha)=0$.
We will show that
\begin{equation}\label{fh1}
|f(\bbeta,\balpha) - h(\bbeta,\balpha)| \ll \eps \qquad (\text{if }(\bbeta,\balpha)\text{ has no subsum in }\cI_\eps),
\end{equation}
the implied constant depending on $\sup_{\bu} |f(\bu)|$ and $\eta$ only.
In \eqref{fsl}, $s$ and $\ell$ are bounded, there
are a bounded number of vectors $(k_1,\ldots,k_\ell)$,
each factor $\cyrL_{1-\gamma}(\bu_i)$ is bounded, and we have $u_{j,h} \ge \eta$
for all $j,h$, thus the integrand is bounded.
With $k_1,\ldots,k_\ell$ all fixed, we claim that the 
$(k_1+\cdots+k_\ell-\ell)$-dimensional measure
of $\bu_1,\ldots,\bu_\ell$ for which $(\bbeta,\bu_1,\ldots,\bu_\ell)$
has a subsum in $\cI_\eps$ is $O(\eps)$.  This will clearly give \eqref{fh1}.
Such a subsum must have the form
\[
\sum_{i\in \cD} \beta_i + \sum_{j=1}^\ell \sum_{i\in U_j} u_{j,i},
\]
where $\cD \subseteq [s]$, $U_j\subseteq [k_j]$ for $1\le j\le \ell$
and there is some $j'$ for which $1 \le |U_{j'}| \le k_{j'}-1$.
Fix $\cD,U_1,\ldots,U_\ell$ and
fix $i' \in U_{j'}$ and $i''\in [k_{j'}]\setminus U_{j'}$.
For each fixed choice of $\beta_i, 1\le i\le s$ 
and all of the variables $u_{j,i}$ except for $j=j',i\in \{i',i''\}$,
the sum $u_{j',i'}+u_{j',i''}$ is fixed and the measure of the
set of $u_{j',i'}$ for which the above subsum is in $\cI_\eps$
is at most $4\eps$.  This proves the claim.

Now we define, for $x/2<n\le x$ the weights
\begin{equation}\label{wf}
w_n = h(\bv(n)) \qquad (\mu^2(n)=1).
\end{equation}
By definition, if $w_n\ne 0$ then $\bv(n)$ has no subsum in $[\theta,\theta+\nu+\eps]$.
Thus, all of the divisors of $n$ are either $< n^{\theta} \le x^{\theta}$ or
$> n^{\theta+\nu+\eps} > (x/2)^{\theta+\nu+\eps} > x^{\theta+\nu}$.  Therefore, (II) holds trivially.
For (I) we will in fact show more, that for \emph{all} $m\le x^{\gamma}$ and $\frac12 \le c\le d\le 1$, we have
\begin{equation}\label{construction-TypeI}
S_m(c,d) := \sum_{cx/m < r \le dx/m} w_{mr} \ll_{f,\eps} \frac{x e^{-(\log x)^{1/5}}}{m}.
\end{equation}
From this, (I) follows easily, for any $B$ and for $x$ large enough in terms of $B$.
Fix $m \le x^{\gamma}$, with $m=p_1\cdots p_t$, $p_1<\cdots < p_t$,
$p_i = x^{u_i}$ for $1\le i\le t$, and let $\lambda = u_1+\cdots+u_t \le \gamma$,
so that $m=x^{\lambda}$.
If $m$ has a prime factor $< (x/2)^{\eta}$ then by the support of $f$
and \eqref{wf}, $w_{mr}=0$ for all $r$ and $S_m(c,d)=0$.
If $m > x^{\gamma-\eps}$ then for any integer $r$ with $x/2<mr\le x$ 
the vector $\bv(mr)$ has a subsum in $[\gamma-\eps,\gamma+\eps]$
and we also have $S_m(c,d)=0$.
Now suppose all $p_i \ge (x/2)^{\eta}$ and that $m \le x^{\gamma-\eps}$,
that is, $\lambda \le \gamma-\eps$.
Then,
\begin{equation}\label{Smcd}
S_m(c,d) = \ssum{cx/m < r \le dx/m \\ (r,m)=1 \\ \mu^2(r)=1} h(\bv(m;mr),\bv(r;mr)).
\end{equation}
Recall that $f = \sum_{k,j} f_{k,j}$, where each $f_{k,j}$ is supported on
a convex polytope 
\[
P_{k,j} \subseteq \{ \bx\in \RR^k : \ssc{x}{1}+\cdots+\ssc{x}{k}=1, 0\le \ssc{x}{1}\le \cdots \le \ssc{x}{k} \}
\]
and is Lipschitz continuous on $P_{k,j}$. 
Removing from each $P_{k,j}$ the vectors with a subsum in $\cI_\eps$
leaves a region which is a bounded union of polytopes, and thus
we may write  $h=\sum_{k,j} h_{k,j}$, each $h_{k,j}$
supported and Lipschitz continuous on a convex polytope $Q_{k,j}$.
Now fix $k>t$ and $j$, and consider the terms in the sum in \eqref{Smcd} 
corresponding to $r=p_{t+1}\cdots p_k$ with $\omega(r)=k-t$ and $p_{t+1}<\cdots < p_k$.  
Now fix one of the $O_k(1)$ 
orderings of the primes $p_1,\ldots,p_k$, which come from possible meshings
of the two ordered vectors $(p_1,\ldots,p_t)$ and $(p_{t+1},\ldots,p_k)$.
Such an ordering has the form 
\begin{equation}\label{eq:p-sigma}
p_{\sigma(1)} < \cdots < p_{\sigma(k)},
\end{equation}
where $\sigma$ is a permutation of $[k]$, and then
\[
\bv(mr) = \bigg( \frac{\log p_{\sigma(1)}}{\log (mr)},\ldots,\frac{\log p_{\sigma(k)}}{\log (mr)} \bigg).
\]
With $k,j$ and $\sigma$ fixed, the corresponding summands in \eqref{Smcd} 
are those with $\bv(mr) \in Q_{k,j}$ and \eqref{eq:p-sigma} holding.
Since $(\bv(m;mr),\bv(r;mr)) = (\frac{\log p_1}{\log mr},\ldots,\frac{\log p_k}{\log mr})$,
this in turn is equivalent to 
\[
(\bv(m;mr),\bv(r;mr)) \in 
\cT_{k,j,\sigma} := \big\{ \bu \in \RR^k : u_{\sigma(1)}<\cdots < u_{\sigma(k)},\,
 \big( u_{\sigma(1)},\ldots,u_{\sigma(k)}\big) \in Q_{k,j}\big\}.
\]
Again, for some $\ell$ the inequality $\sum_h \ssc{e}{\ell,h} x_h \ge 0$ may be 
replaced by a corresponding strict inequality.
This is then a sum of
the type in Lemma \ref{lem:sum_to_int}.  Summing over the $O(1)$ choices
for $k,j,\sigma$, it follows from this lemma that
\begin{align*}
\ssum{cx/m < r \le dx/m \\ (r,m)=1,\mu^2(r)=1} &h(\bv(mr)) =
\sum_{k,j,\sigma} \;\;\ssum{cx/m < r \le dx/m \\ (r,m)=1,\mu^2(r)=1 \\ \bv(mr)\in Q_{k,j}} h(\bv(mr))\\
&= \int_{1-\lambda+\frac{\log c}{\log x}}^{1-\lambda+\frac{\log d}{\log x}}
x^w \;\sum_{k,j,\sigma}\;\;\; \mint{u_{t+1}+\cdots+u_k=w \\ \bu/(\lambda+w)\in \cT_{k,j,\sigma}} \frac{h_{k,j}(\frac{u_1}{\lambda+w},\ldots,\frac{u_k}{\lambda+w})}{u_{t+1}\cdots u_k}\, d \bu \, dw\, 
 +O_{f,\eps}\Big( \frac{x}{m} e^{-(\log x)^{1/5}} \Big)\\
&=
\int_{1-\lambda+\frac{\log c}{\log x}}^{1-\lambda+\frac{\log d}{\log x}}
x^w\sum_{k\ge t+1}  \;\;\; \mint{u_{t+1}+\cdots+u_k=w \\ u_{t+1}<\cdots<u_k} \frac{h(\frac{u_1}{\lambda+w},\ldots,\frac{u_k}{\lambda+w})}{u_{t+1}\cdots u_k}\, d \bu \, dw\, 
 +O_{f,\eps}\Big(  \frac{x}{m} e^{-(\log x)^{1/5}} \Big).
\end{align*}
On the right side, 
\[
\frac{u_1+\cdots + u_t}{\lambda+w} \le \frac{\gamma-\eps}{1 - \frac{\log 2}{\log x}}
< \gamma
\]
for large enough $x$.
Therefore, \eqref{TypeI-f} implies that for each $w$ we have
\[
\sum_{k\ge t+1}\;\;\; \mint{u_{t+1}+\cdots+u_k=w \\ u_{t+1} < \cdots < u_k} \frac{h(\frac{u_1}{\lambda+w},\ldots,\frac{u_k}{\lambda+w})}{u_{t+1}\cdots u_k}\, d \bu  = 0,
\]
and this proves \eqref{construction-TypeI}.
Finally, for all primes $p\in (x/2,x]$ we have $w_p = h(1)$.
By \eqref{fh1} and the fact that $h(\bu)=0$ if $\bu$ has a subsum in $\cI_\eps$, 
part (a) of the theorem follows upon
letting $\eps$ be small enough in terms of $\delta$.

\medskip

To prove part (b), assume that $f(\bu) \ge -1$ for all $\bu$.
Let $B>0$ be arbitrary and $\delta>0$ be arbitrarily small.
By (a) there is a number $z$,  depending only on $\delta$,
with $|z-f(1)|\le \delta$, and such that for sufficiently
large $x$, there is a sequence $(w_n)$ satisfying \eqref{eq:TypeI}, \eqref{eq:TypeII}, 
$w_p=z$ for all primes $p\in (x/2,x]$, and such that $w_n \ge -1-\delta$
for all $n\in (x/2,x]$.  For each $n\in (x/2,x]$ define
\[
w_n' = \frac{w_n}{1+\delta}, \qquad b_n=1, \qquad a_n=w_n'+b_n.
\]
Then $w_p'=\frac{z}{1+\delta}$ for all primes $p$,
and $w_n' \ge -1$ for all $n$, hence $a_n\ge 0$ for all $n$.
Moreover, $(w_n)$ is bounded and hence \eqref{w} holds for 
any $\varpi>1$ and $x$ large enough.
 Furthermore, \eqref{eq:TypeI} and \eqref{eq:TypeII}
trivially hold with $w_n$ replaced by $w_n'$.  Since
\[
\sum_p a_p = \Big(\frac{z}{1+\delta}+1\Big) \big( \pi(x)-\pi(x/2) \big) \sim 
\Big( \frac{z}{1+\delta}+1\Big)\cdot   \frac{x/2}{\log x} \qquad (x\to \infty),
\]
it follows that $C^-(P) \le \frac{z}{1+\delta}+1 \le C^+(P)$.  Letting $\delta\to 0$
proves the first part of (b).   Since $f$ is bounded, by $F$ say, so is $w_n$, thus \eqref{CBD} holds for large enough $\varrho$ and with $\varpi=2$.
We see that
for any $\delta>0$, $\CB^-(P;\varrho) \le \frac{z}{1+\delta}+1 \le \CB^+(P;\varrho)$
and the second claim follows.

\bigskip

\subsection{Heuristic justification for considering only special types
of functions $f$. }

As mentioned in the beginning of this section, we believe that for the purposes of calculating the constants $C^\pm(P)$ (at least when $P$ is a continuity point), it should be sufficient to only consider sequences $(a_n)$, $(b_n)$ with $w_n=0$ whenever $n$ has a divisor in $[(x/2)^\theta,x^{\theta+\nu}]$. (i.e. we can assume that $w_n$ satisfies  \eqref{eq:TypeII} trivially). As a heuristic justification, we sketch how we should be able to pass from a sequence $w_n$ satisfying \eqref{eq:TypeI} and \eqref{eq:TypeII} to a sequence $w_n^{(3)}$ satisfying a (slightly weakened version) of \eqref{eq:TypeI} and \eqref{eq:TypeII}, supported on integers with no divisor in the Type II range and with $\sum_p w_p\approx \sum_p w^{(3)}_p$. First we set
\[
w^{(1)}_n=w_n\one_{P^-(n)\ge z}
\]
for $z= x^{1/(\log\log{x})^3}$. Since $\one_{P^-(n)\ge z}$ is multiplicative, $(w_n^{(1)})$ satisfies \eqref{eq:TypeII} since $(w_n)$ does. By the fundamental lemma of sieve theory, $\one_{P^-(n)\ge z}\approx \sum_{d|n}\lambda_d$ for suitable sieve weights $\lambda_d$ supported on $d\le x^{1/\log\log{x}}$. Thus $w_n^{(1)}$ will satisfy \eqref{eq:TypeI} with $\gamma$ replaced by $\gamma-1/\log\log{x}$ by expanding the sieve and using the fact that $w_n$ satisfies \eqref{eq:TypeI}. Secondly, for $n=p_1\cdots p_k$ we set
\[
w^{(2)}_{n}=\mathbb{E}_{p_1'\in [p_1,p_1+p_1/\log^{A}{x}]}\cdots \mathbb{E}_{p_k'\in [p_k,p_k+p_k/\log^{A}{x}]}w^{(1)}_{p_1'\cdots p_k'},
\]
where $\mathbb{E}$ indicates an average over primes $p_i'$. Since $(w_n^{(1)})$ satisfies \eqref{eq:TypeII}, by swapping the order of summation we see that $w_n^{(2)}\ll \log^{-A+O(1)}{x}$ whenever $n$ has a divisor in the interval $[x^{\theta'},x^{\theta'+\nu'}]$ where $\theta'=\theta+\log^{-A}{x}$ and $\nu'=\nu-2\log^{-A}{x}$ (so $w_n^{(2)}$ satisfies a pointwise version of \eqref{eq:TypeII} if $A$ is chosen large enough). Similarly, by swapping the order of summation, we see that whenever $w_n^{(2)}$ satisfies \eqref{eq:TypeI} with $\gamma$ replaced by $\gamma'=\gamma-2/\log\log{x}$. Finally, we set
\[
w_n^{(3)}=\begin{cases}
0,\qquad &\exists d|n\text{ s.t. }d\in [x^{\theta'},x^{\theta'+\nu'}],\\
w_n^{(2)},&\text{otherwise.}
\end{cases}
\]
Since $w_n^{(2)}$ is small whenever $n$ has a divisor in $[x^{\theta'},x^{\theta'+\nu'}]$, $(w_n^{(3)})$ also satisfies \eqref{eq:TypeI} and \eqref{eq:TypeII} (with $\gamma,\theta,\nu$ replaced by $\gamma',\theta', \nu'$) but is supported on numbers with no divisor in type II interval. Moreover, it is easy to check that $\sum_p w_p\approx \sum_p w_p^{(1)}\approx \sum_p w_p^{(2)}\approx \sum_p w^{(3)}_p$, so if $w_n$ is close to extremal for $(\gamma,\theta,\nu)$ then $w_n^{(3)}$ is close to extremal for $(\gamma',\theta',\nu')$ provided $(\gamma,\theta,\nu)$ is a point of continutity of the functions $C^-$
and $C^+$.

\bigskip
%
%
%
{\Large \section{Sieving}\label{sec:sieving}}
%
%
%

Our main goal in this section is to develop a sieve method which establishes upper 
and lower bounds on $\sum_pa_p$ for any sequence satisfying the Type I and Type 
II estimates, thereby producing a lower bound for $C^-(\gamma,\theta,\nu)$ and 
an upper bound for $C^+(\gamma,\theta,\nu)$. We do this by constructing good 
\emph{sieve weights} $\Lambda^\pm_d$ which exploit both the Type I information 
and Type II information. Throughout this section, we consider $P=(\gamma,\theta,\nu) \in \cQ_0$
 fixed (recalling the definition \eqref{eq:Q0Def} of $\cQ_0$)
 and let $\cR = \cR(P)$ (recalling the Definition \ref{defn:R1} of $\cC(\cR)$).
  All constants implied by $O$ 
and $\ll$ symbols may depend on $\gamma,\theta,\nu$. Any other dependencies 
will be indicated by subscripts to the  $O$ and $\ll$ symbols.

We begin by expanding upon the outline of the general sieve method given
in Section \ref{sec:outline}.
Let $\cN$ be the set of composite integers $n$ in $(x/2,x]$ such that $\bv(n)\in \cC(\cR)$ (recall Definition \ref{def:vn} of $\bv(n)$).
In particular, such integers have no divisor in $(n^{\theta},n^{\theta+\nu}]$.
We choose the weights so that $\Lambda^{\pm}_1=1$ and $\Lambda^\pm_d$ are
 supported on integers $d\le x^{\gamma}$. We define $H^\pm(n) := \sum_{d|n} \Lambda^{\pm}_d$.  We want for each $n\in \cN$
\[
H^-(n) \le 0 \le H^+(n).
\]
These functions $\Lambda^\pm$ resemble the sieve weights that are used when
there is no Type II information (see e.g. Chapter 5 of \cite{Opera}), but now we 
only require $H^-(n) \le 0\le H^+(n)$ for the special set $\cN$, rather than for 
all integers $n>1$ with $P^+(n)\le z$, where $z$
is the sifting limit.
Let $\cP$ be the set of primes in $(x/2,x]$.
Focusing on the lower bound, the fact that $H^-(p)=1$ for primes $p\in \cP$ and 
that $w_n\ge -b_n$, we have
\begin{align*}
\sum_{p\in \cP} w_p = \sum_{p\in \cP} (w_p+b_p)-b_p &\ge -\sum_{p\in \cP} b_p + \sum_{n\in \cP \cup \cN} (w_n+b_n)H^-(n) \\
&= \sum_{n\in \cN} b_n H^-(n) + \sum_{n\in \cP \cup \cN} w_n H^-(n)\\
&=  \sum_{n\in \cN} b_n H^-(n) + \sum_{x/2<n\le x} w_n H^-(n) -  \ssum{x/2<n\le x
\\ n\not\in \cP \cup \cN} w_n H^-(n).
\end{align*}
Using the Type I bound \eqref{eq:TypeI}, the second sum on the right satisfies
\[
\bigg| \sum_{x/2<n\le x} w_n H^-(n)\bigg|  = \bigg|\sum_{d\le x^{\gamma}} \Lambda^-_d \sum_{d|n} w_n\bigg| \le \frac{x}{\log^B x},
\]
provided that $\Lambda^-_d$ is divisor bounded (true in practice),
and we can use the Type I and Type II information together to show that the third sum
on the right is also small (this is the most complicated and longest part of the argument).  We then obtain an estimate
\[
\sum_{p\in \cP} w_b \gtrsim \sum_{n\in \cN} b_n H^-(n),
\]
and it remains to make a good choice for $\Lambda^-$.  

When $\cR$ is empty, the analysis becomes much simpler since then $\cN$
is also empty, there is no need to choose $\Lambda_d^\pm$, and we conclude that $\sum w_p \approx 0$, confirming that $\sum a_p \sim \sum b_p$.

When $\cR$ is nonempty,
it turns out to be useful to define $\Lambda_d^-$ in terms of the canonical 
factorization of $d$ as $d=d_1 d_2$ with (roughly) $P^-(d_1)\ge x^{\nu} > P^+(d_2)$.
To make all of this precise, we define the quantities
\begin{equation}
\begin{split}
\cH = \cH(P) &:= \{ \bx \in \cC(\cR) : x_i \ge  \nu\; \forall i, \text{ at least two components}\}, \\
\cZ = \cZ(P) &:= \{ \by: \by\text{ is a subvector of a vector } \bx\in \cC(\cR) \}\\
&\:= \{ \ssc{\bx}{A} : \bx\in \cC(\cR), A\subseteq [\dim(\bx)] \},\\
\psi(\bx) = \psi(\bx;P) &:= \sup \{ |\bxi| : \xi_i < \nu\; \forall i, (\bx,\bxi)\in \cZ    \}, \\
\cG_1 = \cG_1(P) &:= \{ \bx \in \cZ : \bx_i\ge \nu\; \forall i; |\bx|+ \psi(\bx) \le \gamma \}, \\
\cG_2 = \cG_2(P) &:= \{ \bx \in \cZ : \bx_i\ge \nu\; \forall i; |\bx|+ \psi(\bx) > \gamma, |\bx| \le \gamma \}.
\end{split}
\label{eq:VHZGdefs}
\end{equation}

 If $\cR$ is nonempty then $\cZ$ contains the empty vector $\emptyvec$.  If $\cR$ is nonempty then it is clear that $\psi(\emptyvec)\le \theta$,
since $\bxi$ contains a subsum in every interval of length $\nu$ that is contained
in $[0,|\bxi|]$.  It follows that if $\cR$ is nonempty and $\gamma \ge \theta$
then $\cG_1$ contains $\emptyvec$.  On the other hand, if $\gamma<\theta$
(which implies that $\gamma<\frac12$), then $\cG_1$ is empty (that is, does not
even contain the empty vector $\emptyvec$).  To see this,
for any $\bx\in \cZ$ with all components
$\ge \nu$ and $|\bx| \le \gamma$, and $y\in (\gamma-|\bx|,\theta-|\bx|)$,
$(\bx,y,1-y-|\bx|)\in \cR$.  Hence,
by breaking up $y$ into pieces which are less than $\nu$ (in an arbitrary manner) we see that $\psi(\bx)=\theta-|\bx| > \gamma-|\bx|$
and so $\bx\in \cG_2$.
Also, if $\bx \in \cZ$ then $\psi(\bx)\ge 0$ since we include the empty vector $\bxi = \emptyvec$ in the supremum.

The set $\cH$ is the vector version of $\cN$, restricted to vectors with all
components $\ge \nu$, and $\cZ$ is the set of vectors corresponding to divisors of
elements of $\cN$.  The set $\cG := \cG_1 \cup \cG_2$ is
the domain of $\bv(d_1;d)$.
Given $d_1$, $x^{\psi(\bv(d_1;x))}$ is an upper bound on the possible
values of $d_2$ such that $d_1d_2$ is a divisor of an element of $\cN$. 
Thus, $\cG_1$ corresponds to the set of $d_1$ for which $d_1d_2 \le x^{\gamma}$ for \emph{any} choice of $d_2$, and in this case a good all-purpose choice is
$\Lambda^-_d = g(d_1) \mu(d_2)$, for some function $g$ which is piecewise smooth
on $\bv(d_1;x)\in\cG_1$.  This choice makes $\cH^\pm(n)=0$ if $P^-(n) \le x^{\nu}$.
In other words, we may use the Legendre sieve on $d_2$.
When $\bx \in \cG_2$, there are possible values of $d_2$ with $d_1 d_2 > x^{\gamma}$,
and we must use a less-efficient sieve weight in place of $\mu(d_2)$,
one supported on $[1,x^{\gamma}/d_1]$.

The main theorem of this section, Theorem \ref{thm: Main sieving}, 
represents a general method of constructing sieves that utilize $\cG_1$
but not $\cG_2$, that is, utilizing a general choice of $g$.
  It is often the case that $\cG_2$ is empty
and Theorem \ref{thm: Main sieving} can produce optimal bounds, that is,
gives $C^\pm(\gamma,\theta,\nu)$ exactly with the right choice of $g$. 
By comparison, standard uses of Harman's sieve can be viewed in this language,
and would correspond to a sieve which exploits $\cG_1$ (but the corresponding choice of $g$ would typically not do this optimally), 
but does not exploit $\cG_2$.

Consider those triples $(\gamma,\theta,\nu)$ with $\gamma=1-\theta$.
As mentioned in the introduction, this is a common case that occurs when
trying to detect primes in a thin set, for example.
In this case it is easy to see that $\cG_2$ is empty.
Indeed, if $\bx\in \cZ$ with all components $\ge \nu$ and $|\bx| \le \gamma=1-\theta$,
then in fact $|\bx| < 1-\theta-\nu$.  If $\bxi$ has all  components $<\nu$ and 
$(\bx,\bxi)\in \cZ$, then $|\bx|+|\bxi| < 1-\theta-\nu$ as well (otherwise
$(\bx,\bxi)$ would have a subsum in $[1-\theta-\nu,1-\theta]$, a contradiction)
and thus $|\bx| + \psi(\bx) \le 1-\theta-\nu < \gamma$. 
In this situation, we expect that Theorem \ref{thm: Main sieving} below is
capable of producing optimal bounds for the right choice of $g$,
and we confirm this in some special cases (see Theorem \ref{thm:1-parm theta family}).

Consider now another family of parameters given by 
$\theta=2\delta, \nu=\frac12-3\delta$ and $\frac12<\gamma<\frac12+\delta$,
where $0<\delta<\frac1{10}$.  We claim that $\cG_1 = \{ \emptyvec \}$
and that $\cG_2$ is nonempty.
 Since $\nu > 2\delta$, $\cR$ consists of vectors
of the form $(\ssc{x}{1},\ssc{x}{2},\bxi)$ where $\ssc{x}{1},\ssc{x}{2}\in (\frac12-\delta,1-\gamma)$,
$\ssc{x}{1}+\ssc{x}{2}+|\bxi|=1$ and $|\bxi| < 2\delta$.  Thus, $\cC(\cR)$ has two types of vectors, those of the 
form $(\ssc{x}{1},\bxi)$ with $1-2\delta<\ssc{x}{1}\le 1$ and $|\bxi|<2\delta$ and those of
the form  $(\ssc{x}{1},\ssc{x}{2},\bxi)$ where $\ssc{x}{1},\ssc{x}{2}\in (\frac12-\delta,\frac12+\delta)$
and $|\bxi| < 2\delta$.  Thus, if $\bx\in \cZ$ with at least one component
and all components $\ge \nu$, then $\bx$ has just one component  $\ssc{x}{1}\in(\frac12-\delta,\gamma]$ and $\psi((\ssc{x}{1})) = \frac12+\delta-\ssc{x}{1}$.  Thus, $\ssc{x}{1}+\psi((\ssc{x}{1}))=\frac12+\delta > \gamma$, $\cG_1 = \{\emptyvec\}$ and $\cG_2 = \{ (\ssc{x}{1}) : \frac12-\delta < \ssc{x}{1} \le 1-\gamma\}$.  For this family, it is necessary to work with $\cG_2$ since there is no non-trivial choice of $\Lambda_d^\pm$ of the form $g(d_1)\mu(d_2)$. In this case our main sieving Theorem \ref{thm: Main sieving} does not produce any lower bounds on $C^-(\gamma,\theta,\nu)$. In a future work, we will address the situation when $\cG_2$ is nonempty
and how to choose $\Lambda^{\pm}$ in this case.

\medskip

\begin{defn}[The convolution operation $\bone\, \star$] For a function $g$ on arbitrary length vectors (including the dimension 0 vector $\emptyvec$), we define the vector convolution
\[
(\bone \star g)(\bx) = \sum_{\by \subseteq \bx} g(\by),
\]
the sum over all $2^{\dim \bx}$ subvectors $\by$ of $\bx$. 
\end{defn}

The functions $g$ under consideration will all be symmetric, i.e., in $\cS$.

\medskip

\begin{defn}[The set $\sG_1$ of functions on $\cG_1$]\label{def:sG1}
Let $\sG_1$ denote the set of all vector functions in $\cS$, supported on $\cG_1$
that are finite sums of functions which are each bounded,
 supported on a convex polytope which lies in $\{ \bx\in \RR^k : 
 \,0\le \ssc{x}{1}\le \cdots \le \ssc{x}{k}\}$ for some $k$,
and with bounded, continuous first order partial derivatives on the interior of the polytope.
\end{defn}
These are natural conditions, since $\cG_1$ is the union of polytopes
(Lemma \ref{lem:G1-polytopes} below).

\begin{thm}[Sieve bounds, $\cG_1$ only version]\label{thm: Main sieving}
Suppose that $(\gamma,\theta,\nu)\in \cQ_0$ and that $\cR$ is nonempty.
Let $g \in \sG_1$  satisfy $g(\emptyvec)=1$.

(a) If $(\bone \star g)(\bx) \le 0$ for all $\bx\in \cH(P)$, then
\[
C^-(\gamma,\theta,\nu) \ge 1 + \sum_{k=2}^{\fl{1/\nu}} \mint{\bx \in\cH(P)\cap \RR^k \\ \ssc{x}{1} \le \cdots \le \ssc{x}{k}} \frac{(\bone \star g)(\bx)}{\ssc{x}{1}\cdots \ssc{x}{k}} d \bx.
\]

(b)
If $(\bone \star g)(\bx) \ge 0$ for all $\bx\in \cH(P)$, then
\[
C^+(\gamma,\theta,\nu) \le 1 + \sum_{k=2}^{\fl{1/\nu}} \mint{\bx \in \cH(P)\cap \RR^k \\ \ssc{x}{1} \le \cdots \le \ssc{x}{k}} \frac{(\bone \star g)(\bx)}{\ssc{x}{1}\cdots \ssc{x}{k}} d \bx.
\]
\end{thm}

\medskip

Combining the results of Theorems \ref{thm:constructions} and \ref{thm: Main sieving}, we derive a simple sufficient condition for a choice of $f,g$ to be optimal.
The idea behind this comes from linear programming, whereby the optimal solution
of the original problem and optimal solution of the dual problem satisfy a
 `slackness' property. 

\begin{thm}[Duality between constructions and sieve bounds]\label{thm:duality}
Suppose that $(\gamma,\theta,\nu)\in \cQ_0$ and that $\cR$ is nonempty.
Let $0<\eta<1-\gamma$. Assume that $f\in \mathscr{F}_\eta$ and $f(\bx)\ge -1$ for all $\bx$.

(a) If $g$ satisfies the hypotheses of Theorem \ref{thm: Main sieving} (a), then 
\[
f(1) \ge  \sum_{\ell\ge 2} \mint{\cH \cap \RR^\ell \\ \ssc{x}{1} \le \cdots \le x_\ell }  \frac{(1\star g)(\bx)}{\ssc{x}{1}\cdots x_\ell} d \bx,
\]
with equality if and only if $(1+f(\bx))(1\star g)(\bx)=0$ for all $\bx\in \cH$
outside a set of measure zero.  In case of equality, we have $C^-(P)=1+f(1)$.

(b)  If $g$ satisfies the hypotheses of Theorem \ref{thm: Main sieving} (b), then 
\[
f(1) \le  \sum_{\ell\ge 2} \mint{\cH \cap \RR^\ell \\ \ssc{x}{1} \le \cdots \le x_\ell}  \frac{(1\star g)(\bx)}{\ssc{x}{1}\cdots x_\ell} d\bx,
\]
with equality if and only if $(1+f(\bx))(1\star g)(\bx)=0$ for all $\bx\in \cH$
outside a set of measure zero.  In case of equality, we have $C^+(P)=1+f(1)$.
\end{thm}

\begin{proof}
Under the hypotheses of either part (a) or part (b), \eqref{TypeI-f} implies that
\begin{equation}\label{fg-ident}
\begin{split}
0 &= \mint{|\bxi|\le \gamma \\ \xi_1<\cdots<\xi_h} \frac{g(\bxi)}{\xi_1\cdots \xi_h}
\sum_{k\ge 1}\; \mint{\xi_1'<\cdots<\xi_k' \\ } \frac{f(\bxi,\bxi')}{\xi_1'\cdots \xi_k'}d\bxi' d\bxi
= \sum_{\ell\ge 1} \mint{\bx\in \cH \cap \RR^\ell \\ \ssc{x}{1} \le \cdots \le x_\ell} \frac{f(\bx)(1\star g)(\bx)}{\ssc{x}{1}\cdots x_\ell}\, d\bx\\
&= f(1) - \sum_{\ell\ge 2} \mint{\bx \in \cH \cap \RR^\ell \\ \ssc{x}{1} \le \cdots \le x_\ell}  \frac{(1\star g)(\bx)}{\ssc{x}{1}\cdots x_\ell}d\bx+\sum_{\ell\ge 2}  \mint{\bx\in \cH \cap \RR^\ell \\ \ssc{x}{1} \le \cdots \le x_\ell} \frac{(1+f(\bx))(1\star g)(\bx)}{\ssc{x}{1}\cdots x_\ell}d\bx.
\end{split}
\end{equation}
The claims follow immediately.
\end{proof}

Recalling the remarks after \eqref{eq:VHZGdefs}, if $\gamma\ge \theta$
the $\cG_1$ contains the empty vector $\emptyvec$.  Thus, taking $g(\emptyvec)=1$
and $g(\bx)=0$ for other $\bx\in \cG_1$, we obtain the following immediate
corollary.

\begin{cor}\label{cor:C+finite}
If $(\gamma,\theta,\nu)\in \cQ_0$ and $\gamma\ge \theta$,
then $C^+(\gamma,\theta,\nu)$ is finite.  In fact,
\[
C^+(\gamma,\theta,\nu) \le 1 + \sum_{k=2}^{\fl{1/\nu}} \mint{\bx \in \cH\cap \RR^k \\ \ssc{x}{1} \le \cdots \le \ssc{x}{k}} \frac{1}{\ssc{x}{1}\cdots \ssc{x}{k}} d \bx.
\]
In particular, if $(\gamma,\theta,\nu)\in\cQ_0$ with $\gamma\ge \frac12$, then $C^+(\gamma,\theta,\nu)$ is finite.
\end{cor}

The proof of Theorem \ref{thm: Main sieving} breaks naturally into two cases, $\nu>1-\gamma$ and $\nu\le 1-\gamma$, the former being much simpler.  
We first show that $\cG_1$ is the union of polytopes, which justifies
the restriction of $g$ to $\sG_1$.

%
%
%
%
\begin{lem}\label{lem:G1-polytopes}
For any $k\ge 1$, $\cG_1 \cap \RR^k$ is either empty or the disjoint union 
of convex polytopes, each of which is determined by a bounded (in terms of $\gamma,
\theta,\nu$ only) number of linear constraints.
\end{lem}

\begin{proof}
Throughout the proof, we use the symbols $\bx,\bx'$ to denote vectors with all components in $[\nu,1]$ and the symbols $\by,\by'$ to denote vectors with all components in $(0,\nu)$.
By Lemma \ref{lem:CR1-union-polytopes}, there is a collection of disjoint convex polytopes
$\cT_1,\ldots,\cT_N$, each defined by $O(1)$ linear constraints, and such that
for any $\by$ we have
\[
\big\{\bx : \,(\by,\bx)\in \cC(\cR)\big\}=
\Bigl(\bigsqcup_{j\le N}\cT_{j}\Bigr)\cap\{\bx : |\bx|=1-|\by|\}.
\]
Recall from \eqref{eq:VHZGdefs} that $\cZ$ is the set of subvectors of $\cC(\cR)$. It follows that $(\by,\bx)\in \cZ$ if and only if
there is a vector $\bx'$ and a $j$ so that $|(\bx,\bx',\by)|\le 1$
and $(\bx,\bx')\in \cT_j$.
To see this, observe that $(\by,\bx)\in \cZ$ if and only if there are vectors
$\bx',\by'$  so that $(\bx,\bx',\by,\by')\in \cC(\cR)$, and this occurs if and only if
there is a choice of $(\bx',\by')$ such that
$(\bx,\bx')$ lies in one of the polytopes $\cT_j$ and $|(\bx,\bx',\by,\by')|=1$. If $|(\bx,\bx',\by)|\le 1$ then there is always a $\by'$ such that $|(\bx,\bx',\by,\by')|=1$.

Since the linear projection of a convex polytope is a convex polytope,
there is a finite set of polytopes $\cW_1,\ldots,\cW_L$ such that
any $\bx \in [\nu,1]^k$ lies in $\cZ$ if any only if $\bx$ belongs to one
of the $\cW_i$.  Then
\[
\cG_1 = \bigcup_{\ell} \big\{ \bx\in \cW_\ell: |\bx|\le \gamma, \,\forall \by\text{ with }
(\bx,\by)\in \cZ\text{ we have }|\by|\le \gamma-|\bx| \big\}.
\]
The condition that $|\by|\le \gamma-|\bx|$ for all $\by$
with $(\bx,\by)\in \cZ$ is equivalent to the condition that
for all $\bx'$ with $(\bx,\bx')\in \cZ$, we have $|\bx'|\ge 1-\gamma$.
Thus, $\cG_1$ is the union of the empty vector and
\begin{align*}
\bigcup_\ell \bigg( \big\{ \bx\in \cW_\ell: |\bx|\le \gamma \big\} \setminus
\Big( \bigcup_j \big\{ \bx\in \cW_\ell: \exists \bx' \text{ with } (\bx,\bx')\in \cT_j\text{ and } |\bx'| <1-\gamma  \big\} \Big) \bigg).
\end{align*}
For each $\ell$, $\cV_\ell :=\{\bx\in \cW_\ell : |\bx|\le \gamma \}$
is a convex polytope.
For each $\ell$ and $j$, the set
\[
\{ (\bx,\bx')\in \cT_j : \bx\in \cW_\ell, \,|\bx'|<1-\gamma \}
\]
is a convex polytope, and it follows that
\[
\cD_{\ell,j} := \big\{ \bx\in \cW_\ell: \exists \bx' \text{ with } (\bx,\bx')\in \cT_j\text{ and } |\bx'| <1-\gamma  \big\}
\]
is also a convex polytope.  It follows that for each $\ell$, 
$\cV_\ell \setminus (\cup_j \cD_{\ell,j})$ is the disjoint union of boundedly many
convex polytopes, each defined by  a bounded number of linear constraints.
\end{proof}

\medskip

\subsection{Preparatory lemmas for Theorem \ref{thm: Main sieving}.}
%
%
%
%
\begin{lem}\label{lem:tauk-tau}
For positive integers $k,n$ we have $\tau_k(n) \le \tau(n)^{k-1}$.
\end{lem}

\begin{proof}
It suffices to check the inequality when $n=p^m$, $p$ prime.  We have
\[
\tau_k(p^m)=\binom{k+m-1}{k-1}=\pfrac{k+m-1}{k-1}\cdots \pfrac{m+1}{1} \le (m+1)^{k-1}=\tau(p^m)^{k-1}. \qedhere
\]
\end{proof}

The next lemma is needed in the proof of Lemma \ref{lem:PrimeDecomposition} below.

\begin{lem}\label{lem:Squarefull}
Assume $(w_n)$ satisfies (I), $r\in \NN$ and $C\le B-1$.  Then
\[
\sum_{d \le x^\gamma}\tau(d)^{C}\sup_{\text{Interval $\cI$}}\Biggl|
\sum_{\substack{f\in \cI \\ \text{ if } r\ge 2 \text{ then }df\le x^{\gamma}}}w_{f^rd}\Biggr|\ll \frac{x}{(\log{x})^{B}} \tag{i}
\]
and
\[
\sum_{d \le x^\gamma}\tau(d)^{C}\sup_{\text{Interval $\cI$}}\Biggl|
\sum_{\substack{f\in \cI \\\text{ if } r\ge 2 \text{ then }df\le x^{\gamma} }}w_{f^rd} \frac{\log f}{\log x} \Biggr|\ll \frac{x}{(\log{x})^{B}}. \tag{ii}
\]
\end{lem}
\begin{proof}
If $r=1$ then (i) is immediate from (I). 
For (ii) when $r=1$ and for any interval $\cI = (a,b]$, partial summation gives
\[
\sum_{a<f\le b} w_{fd} \log f = (\log b) \sum_{a<f\le b} w_{df} -
\int_{a}^b \frac{1}{t} \sum_{a<f\le t} w_{df}\, dt.
\]
Therefore, by the triangle inequality,
\[
\sup_{\text{Interval $\cI$}}\Biggl|\sum_{\substack{f\in \cI}}w_{fd} \log f \Biggr|
\le 2(\log x) \sup_{\text{Interval $\cI$}}\Biggl|\sum_{\substack{f\in \cI}}w_{fd} \Biggr|,
\]
and the estimate (ii) follows from (i).

If $r\ge 2$ then in case (i) let $\lambda_f=1$ for all $f$, and for (ii) let $\lambda_f=\frac{\log f}{\log x}$ for all $f$.
With $d$ fixed, we use the crude bound
\[
\sup_{\text{Interval $\cI$}}\Biggl|\sum_{\substack{f\in \cI \\ df\le x^{\gamma}}}w_{f^rd}\lambda_f\Biggr| \le \ssum{x/2 < f^rd\le x \\ df\le x^{\gamma}} |w_{f^rd}|.
\]
Thus, writing $e=d f$ so that $e\le x^{\gamma}$ and $f|e$, we have
\begin{align*}
\sum_{d \le x^\gamma}\tau(d)^{C}\sup_{\text{Interval $\cI$}}\Biggl|
\sum_{\substack{ f\in \cI \\ df\le x^{\gamma}}}w_{f^rd} \lambda_f \Biggr| &\ll 
\sum_{e\le x^{\gamma}} \tau(e)^C \sum_{f|e} |w_{ef^{r-1} }| \\
&\ll 
\sum_{e\le x^{\gamma}} \tau(e)^{C+1} \sup_{\cI} \Bigg| \sum_{g\in \cI} w_{eg} \Bigg| \\
&\ll \frac{x}{\log^{B} x}
\end{align*}
by (I).  This shows (i) and (ii).
\end{proof}

%
%
%
%
%
\begin{lem}[Separation of variables in inequalities]\label{lem:Integration}
Let $f(n),g(m)$ be positive real functions with $|f(n)|$, $|g(m)|$ and $|f(n)-g(m)|$
all lying in $[1/x^{20},x^{20}]$ for all $n,m\le x$.  Let $\alpha_{m,n}$ be a complex sequence with $|\alpha_{n,m}|\le x^2$ for all $n,m\le x$. Then we have
\[
\sum_{\substack{n,m\le x\\ f(n)> g(m)}}\alpha_{n,m}\ll (\log{x})\sup_{t\in \RR}\Bigl|\sum_{n,m\le x}\alpha_{n,m}f(n)^{it}g(m)^{-it}\Bigr|+x^{-100}.
\]
\end{lem}
\begin{proof}
By Perron's formula (see \cite[Proposition 5.54]{IwaniecKowalski}) we have for $f(n)\ne g(m)>0$ and $c,T>0$
\[
\frac{1}{2\pi i}\int_{c-i T}^{c+i T} f(n)^{s}g(m)^{-s}\frac{d s}{s}=\one_{f(n)>g(m)}+O\Bigl(\frac{f(n)^c g(m)^{-c}}{T|\log(f(n)/g(m))|}\Bigr).
\]
We take $T=x^{200}$ and $c=1/T$. Our assumptions on $f,g$ mean that the big-$O$ term is $O(x^{-160})$. We now multiply by $\alpha_{n,m}$ and sum over $n,m\le x$. This gives
\begin{align*}
\sum_{\substack{n,m\le x\\ f(n)> g(m)}}\alpha_{n,m}=&\frac{1}{2\pi i}\int_{c-i T}^{c+i T} \Bigl(\sum_{n,m\le x}\alpha_{n,m} f(n)^{s}g(m)^{-s}\Bigr)\frac{ds}{s}+O(x^{-150})\\
&\ll \Bigl(\int_{c-iT}^{c+i T}\frac{|ds|}{|s|}\Bigr)\sup_{t\in [-T,T]}\Bigl|\sum_{n,m\le x}\alpha_{n,m} f(n)^{c+it}g(m)^{-c-it}\Bigr|+x^{-150}.
\end{align*}
The integral of $|ds|/|s|$ is $O(\log{x})$. Finally, noting that 
\[
(f(n)/g(m))^{c+it}=(f(n)/g(m))^{it}+O(x^{-150})
\]
for all $m,n\le x$, we  obtain the result of the lemma.
\end{proof}

Frequently, we apply Lemma \ref{lem:Integration} where
 one of $f(n),g(m)$ is integer valued, and the other function
is always $\frac12$ plus an integer, from which it follows that $|f(n)-g(m)| \ge \frac12$
for all $n,m$.

%
%
%
%
%
%
We next record a variant of the previous method, useful for encoding
conditions coming from polytopes.
Recall that the notation $a \sim x$ stands for $a\in (x/2,x]$, and that the symbol $p$
always denotes a prime.

\begin{lem}[Encoding a polytope condition]\label{lem:polytope-encode-one-condition}
Assume that $(w_n)$ satisfies \eqref{w}.
Suppose that $k\ge 1$, $\cM$ is a nonempty subset of $[k]$ and that 
$\ssc{c}{1},\ldots,\ssc{c}{k}$ are real numbers
with the numbers $\ssc{c}{j}$ for $j\in \cM$ all equal, the common value 
being $\ge 1$ or $\le -1$.  Suppose also that
for a $k-$tuple of positive integers $\bn=(\ssc{n}{1},\ldots,\ssc{n}{k})$ with product in $(x/2,x]$, $\xi_{\bn}$ is a complex number with
$|\xi_{\bn}| \le 1.$  Then
\begin{align*}
\ssum{n=\ssc{n}{1}\cdots \ssc{n}{k}\sim x \\ \prod_{j\in \cM}\ssc{n}{j} \ge x^{\nu/3} 
\\ n_1^{\ssc{c}{1}}\cdots n_k^{\ssc{c}{k}} < 1} \xi_{\bn} w_n
&\ssc{\ll}{k} \; (\log x) \sup_{\Re s=x^{-2}}
\bigg| \ssum{n=\ssc{n}{1}\cdots \ssc{n}{k}\sim x \\ \prod_{j\in \cM}\ssc{n}{j} \ge x^{\nu/3}}
\xi_{\bn} w_n n_1^{-\ssc{c}{1}s}\cdots n_k^{-\ssc{c}{k}s} \bigg|\\
& +(\log x) \sup_{\Re s=x^{-2}}
\bigg| \ssum{n=\ssc{n}{1}\cdots \ssc{n}{k}\sim x \\ \prod_{j\in \cM}\ssc{n}{j} \ge x^{\nu/3}}
|\xi_{\bn}| w_n n_1^{-\ssc{c}{1}s}\cdots n_k^{-\ssc{c}{k}s} \bigg|
+x^{1-\nu/10}.
\end{align*}
The same bound holds if the condition $n_1^{\ssc{c}{1}}\cdots n_k^{\ssc{c}{k}} < 1$
is replaced by the nonstrict inequality 
 $n_1^{\ssc{c}{1}}\cdots n_k^{\ssc{c}{k}} \le 1$.
\end{lem}

\begin{proof}
By relabeling, we may assume that $\cM = [j]$ for some $j$ with $1\le j\le k$.
For $0<\delta<1$ let $f_\delta(x):[0,\infty)\to [0,1]$ which equals 1 for 
$0\le x\le 1-\delta$, equals 0 for $x\ge 1$ and has continuous 2nd order derivative.
We can choose such an $f$ with $f_{\delta}^{(j)}(x) \ll \delta^{-j}$ for $j=1,2$.
Define
\[
g_\delta(x) := \one_{1-\delta \le x\le 1} (1-f_\delta(x)) + \one_{x>1} f_\delta(x-\delta),
\]
so that $g$ is supported on $[1-\delta,1+\delta]$, $g_{\delta}^{(j)}(x)\ll \delta^{-j}$ for $j=1,2$ and $f_\delta(x)+g_\delta(x)=1$ for $0\le x\le 1$.
Let
\[
F_\delta(x) := \int_0^\infty f_\delta(u) u^{s-1}\, du, \qquad
G_\delta(x) := \int_0^\infty g_\delta(u) u^{s-1}\, du
\]
be the Mellin transforms of $f_\delta$ and $g_\delta$, respectively.
Using integration by parts, and the fact that $f_\delta'(x)=f_\delta''(x)=g_{\delta}'(x)=g_\delta''(x)=0$
for $x\not\in (1-\delta,1+\delta)$, we see that
\[
|F_\delta(s)| + |G_{\delta}(s)| \ll \min \bigg( \frac{1}{|s|}, \frac{1/\delta}{|s|^2}\bigg) \qquad (0<\Re s \le 1).
\]
In particular,
\begin{equation}\label{eq:FG-delta}
\int_{\delta-i\infty}^{\delta+i\infty} \Big( |F_\delta(s)|+|G_\delta(s)|\Big)\, |ds| 
\ll 1+\int_\delta^{1/\delta}\frac{dt}{t}+\int_{1/\delta}^\infty \frac{dt}{\delta t^2}
\ll 1+\log(1/\delta).
\end{equation}
Since $\one_{x<1} = f_{\delta}(x) +  O(g_\delta(x))$, we have
\[
\ssum{n=\ssc{n}{1}\cdots \ssc{n}{k}\sim x \\ \ssc{n}{1}\cdots \ssc{n}{j} \ge x^{\nu/3} 
\\ n_1^{\ssc{c}{1}}\cdots n_k^{\ssc{c}{k}} < 1} \xi_{\bn} w_n =
\ssum{n=\ssc{n}{1}\cdots \ssc{n}{k}\sim x \\  \ssc{n}{1}\cdots\ssc{n}{j} \ge x^{\nu/3}}
 \xi_{\bn} w_n f_\delta\big( n_1^{\ssc{c}{1}}\cdots n_k^{\ssc{c}{k}} \big) +
  O\bigg( \ssum{n=\ssc{n}{1}\cdots \ssc{n}{k}\sim x \\  \ssc{n}{1}\cdots\ssc{n}{j} \ge x^{\nu/3}}
  |\xi_{\bn} w_n| g_\delta\big( n_1^{\ssc{c}{1}}\cdots n_k^{\ssc{c}{k}} \big)  \bigg).
\]
Denote the right hand side by $S_1+O(S_2)$, and let $\delta = x^{-2}$.
By Mellin inversion and \eqref{eq:FG-delta},
\begin{align*}
S_1 &= \frac{1}{2\pi i}\int_{\delta-i\infty}^{\delta+i\infty} F_\delta(s)
 \ssum{n=\ssc{n}{1}\cdots \ssc{n}{k}\sim x \\  \ssc{n}{1}\cdots\ssc{n}{j} \ge x^{\nu/3}} 
 \xi_{\bn} w_n n_1^{-\ssc{c}{1}s} \cdots n_k^{-\ssc{c}{k}s} \, ds\\
 &\ll (\log x) \sup_{\Re s=\delta} \bigg| 
 \ssum{n=\ssc{n}{1}\cdots \ssc{n}{k}\sim x \\  \ssc{n}{1}\cdots\ssc{n}{j} \ge x^{\nu/3}} 
 \xi_{\bn} w_n n_1^{-\ssc{c}{1}s} \cdots n_k^{-\ssc{c}{k}s} \bigg|.
\end{align*}
By \eqref{w}, $|w_n| \le w_n + 2x^{\nu/10}$.
Hence $S_2 \le S_3 + 2 S_4$, where
\begin{align*}
S_3 &:= \ssum{n=\ssc{n}{1}\cdots \ssc{n}{k}\sim x \\  \ssc{n}{1}\cdots\ssc{n}{j} \ge x^{\nu/3}}
  |\xi_{\bn}| w_n g_\delta\big( n_1^{\ssc{c}{1}}\cdots n_k^{\ssc{c}{k}} \big), \\
  S_4 &:=  x^{\nu/10} \ssum{n=\ssc{n}{1}\cdots \ssc{n}{k}\sim x \\  \ssc{n}{1}\cdots\ssc{n}{j} \ge x^{\nu/3}} 
 g_\delta\big( n_1^{\ssc{c}{1}}\cdots n_k^{\ssc{c}{k}} \big).
\end{align*}
By \eqref{eq:FG-delta} again,
\begin{align*}
S_3&= \frac{1}{2\pi i}\int_{\delta-i\infty}^{\delta+i\infty} G_\delta(s)
 \ssum{n=\ssc{n}{1}\cdots \ssc{n}{k}\sim x \\  \ssc{n}{1}\cdots\ssc{n}{j} \ge x^{\nu/3}} 
 |\xi_{\bn}| w_n n_1^{-\ssc{c}{1}s} \cdots n_k^{-\ssc{c}{k}s}  \, ds\\
 &\ll (\log x) \sup_{\Re s=\delta} \bigg| 
 \ssum{n=\ssc{n}{1}\cdots \ssc{n}{k}\sim x \\  \ssc{n}{1}\cdots\ssc{n}{j} \ge x^{\nu/3}} 
 |\xi_{\bn}| w_n n_1^{-\ssc{c}{1}s} \cdots n_k^{-\ssc{c}{k}s} \bigg|.
\end{align*}
The $g_\delta$ factor is nonzero only when
the argument is in $[1-\delta,1+\delta]$ which implies that
\[
\sum_{h=1}^k \ssc{c}{h} \log \ssc{n}{h} = O(\delta).
\]
Recall that $\delta=x^{-2}$, $\ssc{c}{1}=\cdots=\ssc{c}{j}$ and 
$|\ssc{c}{1}|\ge 1$.
With positive integers $\ssc{n}{j+1},\ldots,\ssc{n}{k}$ all fixed, the number 
\[
\ssc{c}{1}(\log \ssc{n}{1}+\cdots+\log \ssc{n}{j}) =\ssc{c}{1} \log(\ssc{n}{1}\cdots\ssc{n}{j})
\]
 lies in an interval of length $O(\delta)$, and hence
 $\log(\ssc{n}{1}\cdots\ssc{n}{j})$ lies in an interval of length $O(\delta/|\ssc{c}{1}|)$.  Since  $n_1\cdots n_j\le x$, it follows that the product $s=\ssc{n}{1}\cdots \ssc{n}{j}$ is unique.
Hence there are at most $\max_{s\le x} \tau_j(s) = O_k(x^{\nu/100})$ choices for the tuple  $(\ssc{n}{1},\ldots,\ssc{n}{j})$.
 Therefore, writing $m=\prod_{h> j}\ssc{n}{h}$,
we have $m\le x^{1-\nu/3}$ and hence
\begin{align*}
S_4 & \ll \;\; x^{\nu/10} \sum_{m\le x^{1-\nu/3}} x^{\nu/100}\tau_{k-j}(m) \ll x^{1-\nu/10},
\end{align*}
again using the divisor bound $\tau_{k-j}(m)\ll_k x^{\nu/100}$.
Together with the earlier bounds on $S_1$ and $S_3$, this completes the proof.
Replacing the strict inequality $<1$ with $\le 1$ has no effect on the above 
argument.
\end{proof}

\begin{lem}[Encoding many polytope conditions]\label{lem:polytope-encode}
Assume that $(w_n)$ satisfies \eqref{w}.
Let $k$, $s$ and $\ell$ be  positive integers and $D>1$.
Let $\cM_1,\ldots,\cM_s$ be non-empty subsets of $[k]$
and $c_{i,j} \in [-D,-1]\cup\{0\} \cup [1,D]$ for $1\le i\le \ell$, $1\le j\le k$.
For $1\le j\le s$ let $\sM_j$ denote the condition $\prod_{h\in \cM_j} \ssc{n}{h} \ge x^{\nu/3}$.  For $1\le i\le \ell$, let $j_i\in [s]$ and let  $\sL_i$ be either the condition
that for some $j$, the numbers $c_{i,m}$ for $m\in \cM_{j_i}$ are all equal and nonzero and 
 $n_1^{c_{i,1}}\cdots n_k^{c_{i,k}}\le 1$, or the variant with
 $n_1^{c_{i,1}}\cdots n_k^{c_{i,k}} <1$ in place of $n_1^{c_{i,1}}\cdots n_k^{c_{i,k}}\le 1$.
 Suppose also that for every $k$-tuple of positive integers $\bn=(\ssc{n}{1},\ldots,\ssc{n}{k})$
with product in $(x/2,x]$, $\phi_\bn$ is a complex number with
$|\phi_{\bn}| \le 1$.  Then
\begin{align*}
\ssum{n=\ssc{n}{1}\cdots \ssc{n}{k}\sim x \\ \sM_1,\ldots,\sM_s \\ \sL_1,\ldots,\sL_\ell} \phi_\bn w_n \;&\ssc{\ll}{k,\ell,D} \; (\log x)^{\ell}
\sup_{\substack{\zhu_1,\ldots,\zhu_k \\ 1-\text{bounded}}}\; \bigg| 
\ssum{n=\ssc{n}{1}\cdots \ssc{n}{k}\sim x \\ \sM_1,\ldots,\sM_s} \phi_\bn w_n \zhu_1(\ssc{n}{1})\cdots \zhu_k(\ssc{n}{k}) \bigg|\\
&+ (\log x)^{\ell} \sup_{\substack{\zhu_1,\ldots,\zhu_k \\ 1-\text{bounded}}} \; \bigg| 
\ssum{n=\ssc{n}{1}\cdots \ssc{n}{k}\sim x \\ \sM_1,\ldots,\sM_s} |\phi_\bn|\, w_n \zhu_1(\ssc{n}{1})\cdots \zhu_k(\ssc{n}{k}) \bigg|+x^{1-\nu/20}.
\end{align*}
\end{lem}

\begin{proof}
We iterate Lemma \ref{lem:polytope-encode-one-condition}.  
For $0\le i \le \ell$ let
\begin{align*}
Y_i &:= \sup_{\substack{\zhu_1,\ldots,\zhu_k \\ 1-\text{bounded}}}\; \bigg| 
\ssum{n=\ssc{n}{1}\cdots \ssc{n}{k}\sim x \\ \sM_1,\ldots,\sM_s \\ \sL_1,\ldots,\sL_i} \phi_\bn w_n \zhu_1(\ssc{n}{1})\cdots \zhu_k(\ssc{n}{k}) \bigg|+\\
&\qquad + \sup_{\substack{\zhu_1,\ldots,\zhu_k \\ 1-\text{bounded}}} \; \bigg| 
\ssum{n=\ssc{n}{1}\cdots \ssc{n}{k}\sim x \\ \sM_1,\ldots,\sM_s \\ \sL_1,\ldots,\sL_i} |\phi_\bn|\, w_n \zhu_1(\ssc{n}{1})\cdots \zhu_k(\ssc{n}{k}) \bigg|.
\end{align*}
Now let $1\le i \le \ell$.  To estimate $Y_i$,  we apply 
Lemma \ref{lem:polytope-encode-one-condition} twice, each with
$\ssc{c}{j}=\ssc{c}{i,j}$ for $1\le j\le k$ and $\cM=\cM_{j_i}$,
one application with
\[
\xi_{\bn} = \phi_\bn \zhu_1(\ssc{n}{1})\cdots \zhu_k(\ssc{n}{k}) \one_{\sM_1,\ldots,\sM_s,\sL_1,\ldots,\sL_{i-1}}
\]
and a second application with
\[
\xi_{\bn} = \big| \phi_\bn \big| \zhu_1(\ssc{n}{1})\cdots \zhu_k(\ssc{n}{k}) \one_{\sM_1,\ldots,\sM_s,\sL_1,\ldots,\sL_{i-1}}.
\]
For all $j$ and $\Re s = x^{-2}$, $n_j^{-c_j s}$ is a bounded function
of $\ssc{n}{j}$ (the bound depends on $D$), and it quickly follows that
\[
Y_i \;\ssc{\ll}{k,D} \; (\log x) Y_{i-1} + x^{1-\nu/10}.
\]
Iterating this relation gives the claimed bound.
\end{proof}
%
%
%
%
%
%

%
%
\begin{lem}[Prime Decomposition]\label{lem:PrimeDecomposition}
Suppose that 
 $|\ssc{\psi}{u,v}|\le 1$ for all pairs $(u,v)$
 of positive integers,  and let
 $\ell=6\lceil 1 /(1-\gamma) \rceil$.  Let $k$ be a positive integer and for each $(u,v)$, let $\cT_{u,v}$ be a convex region in $[0,1]^k$, which may depend on $x$. Assume $(w_n)$ satisfies \eqref{w} and $\eqref{eq:TypeI}$ with $B$ sufficiently large in terms of
$k,\gamma$.  Let $\mathscr{S}$ be given by
\[
\mathscr{S}=\sum_{u,v}\ssc{\psi}{u,v}\ssum{n=uv p_1\cdots p_k\sim x\\ \big( \frac{\log p_1}{\log x},\ldots, \frac{\log p_k}{\log x}\big)\in \cT_{u,v}}   w_n\cdot \frac{\log p_1}{\log x} \cdots \frac{\log p_k}{\log x}.
\]
Then we have
\begin{align*}
\mathscr{S} \ll_k  (\log x)^{k\ell} \sup_{\substack{\beta_{1,1},\ldots,\beta_{k,\ell} \\ 1-\text{bounded}}} \;
\Biggl|\sum_{u,v}\ssc{\psi}{u,v}\ssum{n=uv m_1\cdots m_k\sim x\\ 
\big(\frac{\log m_1}{\log x},\ldots,\frac{\log m_k}{\log x}\big)\in \cT_{u,v} \\
\ssc{m}{h}=d_{h,1}\cdots \ssc{d}{h,\ell} \; (1\le h\le k)\\ 
d_{h,j}<n/x^\gamma\,\forall h,j} w_{n}  \prod_{h=1}^k \prod_{j=1}^{\ell}\beta_{h,j}(d_{h,j})\Biggr| +O\pfrac{x}{(\log x)^{B-k}}.
\end{align*}
\end{lem}
\begin{proof}
We begin by rewriting $\mathscr{S}$ as
\[
\mathscr{S} = \sum_{u,v} \psi_{u,v} \ssum{n=uv m_1\cdots m_k\sim x\\
\big( \frac{\log \ssc{m}{1}}{\log x}, \ldots,\frac{\log \ssc{m}{k}}{\log x}\big)\in \cT_{u,v}} w_n \prod_{h=1}^k \frac{(\log m_h)\one_{m_h\text{ prime}}}{\log x}.
\]
Let $J := 2\lceil1/(1-\gamma)\rceil$. By Lemma \ref{lem:HeathBrown}, for each $1\le h\le k$,
\begin{multline*}
(\log m_h)\one_{m_h\text{ prime}}
= \sum_{j_h=1}^{J} (-1)^{j_h-1}\binom{J}{j_h}
 \sum_{\ssc{r}{h}\le \frac{\log{x}}{\log 2}} \mu(\ssc{r}{h}) \times \\
 \times \ssum{\ssc{m}{h} = (e_{h,1}\cdots e_{h,j_h} f_{h,1}\cdots f_{h,j_h})^{\ssc{r}{h}} \\ 
 e_{h,j}^{\ssc{r}{h}}\le x^{1/J}\; (1\le j\le j_h)}
 (\log{f_{h,1}})\mu(e_{h,1})\cdots \mu(e_{h,j_h}).
\end{multline*}
We substitute this into our expression for $\mathscr{S}$, and take the maximum over all
$j_1,\ldots,j_k\le J$ and $r_1,\ldots,r_k\le (\log{x})/(\log{2})$.  Thus
\[
\mathscr{S} \ll_k (\log x)^{k} \max_{\substack{j_1,\ldots,j_k\le J \\ r_1,\ldots,r_k}}
\Bigg| \sum_{u,v} \psi_{u,v} \ssum{\eqref{eq:mj-cons}} w_n \prod_{h=1}^k \frac{\log f_{h,1}}{\log x}
 \sprod{1\le h\le k \\ 1\le j \le j_h} \mu(e_{h,j}) \Bigg|,
\]
where the summation is restricted by the conditions
\begin{equation}\label{eq:mj-cons}
\begin{split}
& \bigg( \frac{\log \ssc{m}{1}}{\log x}, \ldots, \frac{\log \ssc{m}{k}}{\log x} \bigg) \in \cT_{u,v}, \\
 n&=\ssc{m}{1}\cdots \ssc{m}{k} uv \sim x,\\
 \ssc{m}{h} &= (e_{h,1}\cdots e_{h,j_h} f_{h,1}\cdots f_{h,j_h})^{r_h}\quad (1\le h\le k), \\
  e_{h,j}^{r_h}&\le x^{1/J}\quad (1\le h\le k,\,1\le j\le j_h).
\end{split}
\end{equation}
We now split the sum according to the set $\cJ$ of indices $(h,j)$ for which $\max(f_{h,j},f_{h,j}^{r_h-1})\ge n/x^{\gamma}$.  
This gives
\begin{equation}\label{eq:S1bb}
\mathscr{S} \ll (\log x)^{k} \max_{\substack{j_1,\ldots,j_k\le J \\ r_1,\ldots,r_k }}
\sum_{\cJ}\Bigg| \sum_{u,v} \psi_{u,v} \ssum{\eqref{eq:mj-cons} \\
 \max(f_{h,j},f_{h,j}^{r_h-1})\ge n/x^{\gamma} \; \forall (h,j)\in \cJ \\
 \max(f_{h,j},f_{h,j}^{r_h-1})< n/x^{\gamma} \; \forall (h,j)\notin \cJ } 
 w_n \prod_{h=1}^k \frac{\log f_{h,1}}{\log x}
 \sprod{1\le h\le k \\ 1\le j \le j_h} \mu(e_{h,j}) \Bigg|.
\end{equation}
First we consider the contribution to the sum in \eqref{eq:S1bb} when $\cJ\ne \emptyset$, and so $\cJ$ contains
 some element $(h',j')$. We write $f=f_{h',j'}$, $r=r_{h'}$ and
 $d=n/f^r$, so that $d\le x^{\gamma}$ if $r=1$ and
 $df\le x^{\gamma}$ when $r\ge 2$. 
Also write $\lambda_{f}=1$ or $\lambda_{f}=\log{f}$, according to whether $j'>1$ or $j'=1$.
With $u,v$ and all of the variables $e_{h,j},f_{h,j}$ fixed except for $f=f_{h',j'}$,
 the conditions on the summation place $f$ in an interval which depends on the
 other variables (here we use the fact that $\cT_{u,v}$ is convex). 
  Also, given $d$ there are at most $\ssc{\tau}{2Jk+1}(d)$
 choices for the variables $u,v$ and the variables $e_{h,j}$ and $f_{h,j}$ for $(h,j)\ne (h',j')$.
Thus, by Lemma \ref{lem:tauk-tau}, the contribution to $\mathscr{S}$ from such $\cJ$ is
\[
\ll (\log{x})^k
\sum_{d \le x^{\gamma}} \tau(d)^{2Jk} \sup_{\text{Interval $\cI$}}\Biggl|
\ssum{ f\in \cI \\ \text{ if } r\ge 2 \text{ then } df\le x^{\gamma}} w_{d f^r}\lambda_{f}\Biggr|.
\]
By Lemma \ref{lem:Squarefull}, this is $O_k(x/(\log{x})^{B-k})$ provided that $B\ge 2Jk+1$.

Thus we are left to consider the sum in \eqref{eq:S1bb} with $\cJ=\emptyset$, so that
$\max(f_{h,j},f_{h,j}^{r_h-1}) < n/x^{\gamma}$ for all $h,j$.
A minor tweak is needed in order to mold our expression into the type required by the lemma.  We define
$d_{h,j} = f_{h,j}^{r_h-1}$ so that $f_{h,j}^{r_h} = f_{h,j} d_{h,j}$ and, for the purpose of using Lemma \ref{lem:Integration}, we
encode this as
\begin{equation}\label{fijdij}
f_{h,j}^{r_h-1} - \tfrac12 < d_{h,j} < f_{h,j}^{r_h-1} + \tfrac12.\end{equation}
By our choice of $J$, we have $x^{1/J} < (x/2)/x^{\gamma}$.
Let $e_{h,j}' = e_{h,j}^{\ssc{r}{h}}$ for all $h,j$ and
define the 1-bounded functions
\begin{align*}
E_r(e') &= (\one_{e'=e^r\text{ for some }e\in \NN}) \mu((e')^{1/r}) \one_{e' \le x^{1/J}}.
\end{align*}
Then the $\cJ=\emptyset$ term in \eqref{eq:S1bb} equals
\[
\sum_{u,v} \psi_{u,v} \ssum{\eqref{fijdij}, \eqref{eq:mj-cons2}}
 w_n \prod_{h=1}^k \frac{\log f_{h,1}}{\log x} \sprod{1\le h\le k\\1\le j\le j_h}
 E_{\ssc{r}{h}}(e'_{h,j}).
\]
where the summation is restricted by the conditions
\begin{equation}\label{eq:mj-cons2}
\begin{split}
& \bigg(\frac{\log \ssc{m}{1}}{\log x},\ldots,\frac{\log \ssc{m}{k}}{\log x}\bigg)\in \cT_{u,v}, \\
 n&=\ssc{m}{1}\cdots \ssc{m}{k} uv \sim x,\\
 \ssc{m}{h} &= d_{h,1} \cdots d_{h,j_h}e'_{h,1}\cdots e'_{h,j_h} f_{h,1}\cdots f_{h,j_h}\quad (1\le h\le k), \\
  d_{h,j},e'_{h,j},f_{h,j}&< n/x^\gamma \quad (1\le h\le k,\,1\le j\le j_h).
\end{split}
\end{equation}

The conditions \eqref{fijdij} may be encoded using at most
$2kJ$ applications of  Lemma \ref{lem:Integration}.  The error terms are all
$O(x^{-90})$ using the bound $|w_n| \ll x^{1.1}$ that follows from the first part of \eqref{w}.
  This introduces a factor 
$(\log x)^{2kJ}$ together with factors
\[
\prod_{h,j} d_{h,j}^{i \ssc{t}{h,j} -i t'_{h,j}} (f_{h,j}^{r_h-1}+1/2)^{i t'_{h,j}} (f_{h,j}^{r_h-1}-1/2)^{-it_{h,j}},
\]
where $\ssc{t}{h,j}, t_{h,j}'$ are real numbers. Therefore,
\begin{multline*}
\mathscr{S} \ll  (\log x)^{k+2kJ} \sup_{\substack{\alpha_{h,j},\beta_{h,j},\gamma_{h,j} \;1-\text{bounded}\\ (1\le h\le k,\,1\le j\le j_h)}} \;
\Biggl|\sum_{u,v}\ssc{\psi}{u,v}\ssum{\eqref{eq:mj-cons2}} w_{n}  \prod_{h=1}^k \prod_{j=1}^{j_h}\alpha_{h,j}(d_{h,j})\beta_{h,j}(e'_{h,j})\gamma_{h,j}(f_{h,j})\Biggr| \\
+O\pfrac{x}{(\log x)^{B-k}}.
\end{multline*}
Finally, let $\ell=3J=6\cl{1/(1-\gamma)}$. To obtain the expression of the lemma we introduce extraneous variables $e'_{h,j},f_{h,j},d_{h,j}$ for each $1\le h\le k$ and $j_h < j\le J$,
each weighted by the indicator function of the variable equalling 1. This allows us to bound $\mathscr{S}$ with the same expression where $j_h=J$ for all $h$. Finally, we relabel the variables $d_{h,1},\dots,d_{h,J},e'_{h,1},\dots,e'_{h,J},f_{h,1},\dots,f_{h,J}$ as $d_{h,1},\dots,d_{h,\ell}$. This gives
\begin{multline*}
\mathscr{S} \ll  (\log x)^{k\ell} \sup_{\substack{\beta_{1,1},\ldots,\beta_{k,\ell} \\ 1-\text{bounded}}} \;
\Biggl|\sum_{u,v}\ssc{\psi}{u,v}\ssum{n=uv m_1\cdots m_k\sim x\\ 
(\frac{\log{m_1}}{\log{x}},\dots,\frac{\log{m_k}}{\log{x}})\in \cT_{u,v} \\
\ssc{m}{h}=d_{h,1}\cdots \ssc{d}{h,\ell} \; (1\le h\le k)\\ 
d_{h,j}<n/x^\gamma\,\forall h,j} w_{n}  \prod_{h=1}^k \prod_{j=1}^{\ell}\beta_{h,j}(d_{h,j})\Biggr| +O\pfrac{x}{(\log x)^{B-k}},
\end{multline*}
as required.
\end{proof}
\begin{lem}[Removing box conditions]\label{lem:BoxConditions}
Let $r\ge 1$ and $0\le c_i<d_i\le 1$ for $1\le i\le r$.
For positive integers $\ssc{m}{1},\ldots,\ssc{m}{r}$, suppose that
$|\alpha_{\ssc{m}{1},\dots,\ssc{m}{r}}| \le x^3$.  Then
\[
\Biggl|\ssum{\ssc{m}{1},\cdots,\ssc{m}{r}\\ x^{c_i} < m_i \le x^{d_i}\; \forall i}
 \alpha_{\ssc{m}{1},\dots,\ssc{m}{k}}\Biggr| \ll_{r} \; (\log{x})^{2r} 
\sup_{\substack{t_1,\ldots,t_r\in\RR}}
\biggl|\sum_{\ssc{m}{1},\cdots,\ssc{m}{r}}\alpha_{\ssc{m}{1},\dots,\ssc{m}{r}} m_1^{it_1}\cdots m_r^{it_r} \biggr|+x^{-95}.
\]
\end{lem}
\begin{proof}
The set of conditions $x^{c_j} < m_j \le x^{d_j}$ for each $1\le j\le r$, is equivalent to
\[
\fl{x^{c_j}}+\tfrac12 < m_j < \lfloor x^{d_j}\rfloor +\tfrac12 \qquad (1\le j\le r),
\]
which is convenient for the purpose of using Lemma \ref{lem:Integration}.
Thus the result follows from at most $2r$ successive applications of Lemma \ref{lem:Integration},
giving an extra factor which is at most $(\log x)^{2r}$ together with factors
\[
\prod_{j=1}^k (\lfloor x^{d_j} \rfloor + \tfrac12)^{it_j}m_j^{-it_j} 
(\lfloor x^{c_j} \rfloor + \tfrac12)^{-it'_j} m_j^{it'_j},
\]
where the $t_j$ and $t_j'$ are real numbers. 
\end{proof}
%
%

\begin{lem}\label{lem:PrimeSplit}
Suppose that 
 $|\ssc{\psi}{u,v}|\le 1$ for all pairs $(u,v)$
 of positive integers with $u\le x^{\gamma}$, let $k$ be a positive integer,  and let
 $\ell=6\lceil 1 /(1-\gamma) \rceil$.
For each such pair $(u,v)$, let $\cT_{u,v}$ be a convex region in $[0,1]^k$.
Suppose that $0<\sigma <1$.
Let $g$ be a real function supported on a convex region $\cU$, Lipschitz continuous on $\cU$
with Lipschitz constant and $|g|$ bounded above by $K$. 
Suppose that for all $u,v,p_1,\ldots,p_k$ with $n=uvp_1\cdots p_k\sim x$,
$P^+(uv)\le n^{\sigma}<p_1\le \cdots \le p_k$
  and $\bv(p_1\cdots p_k;x)\in \cT_{u,v}$,
 we have $\bv(p_1\cdots p_k;n)\in \cU$.
 Let $D\ge 1$.
Assume $(w_n)$ satisfies \eqref{w} and $\eqref{eq:TypeI}$ with $B$ sufficiently large in terms of $D,k,\gamma$.
Then, 
\begin{multline*}
\sum_{u,v}\psi_{u,v} \ssum{n=uv p_1\cdots p_k\sim x\\ P^+(uv)\le n^{\sigma}< p_1\le \cdots \le p_k\\
 \bv(p_1\cdots p_k;x)\in \cT_{u,v}}
 w_n \; g \big( \bv(\ssc{p}{1}\cdots \ssc{p}{k};n) \big)
\ll \frac{x}{(\log x)^{D-\varpi}}
+(\log x)^{k\ell+(k+2)(D+2)}  \, \times \\
\times \sup_{\substack{\beta_{1,1},\ldots,\beta_{k,\ell} \\ 1-\text{bounded}}} \;\;
\sup_{t,t'\in \RR} \;
\Biggl| \sum_{u,v} \ssc{\psi}{u,v} 
\ssum{ n=\ssc{m}{1}\cdots \ssc{m}{k} uv \sim x \\
P^+(uv) \le n^\sigma <  m_1 \le \cdots \le m_k \\
  \big(\frac{\log \ssc{m}{1}}{\log x},\ldots,\frac{\log \ssc{m}{k}}{\log x}\big)\in\cT_{u,v}\\ 
  \ssc{m}{h}=d_{h,1}\cdots d_{h,\ell}\; (1\le \ell\le k)\\ 
  d_{h,j}<n/x^\gamma\,\forall h,j}
w_n u^{it} v^{it'} \prod_{h=1}^k \prod_{j=1}^{\ell} \beta_{h,j}(d_{h,j})\Biggr|,
\end{multline*}
where the constant implied by $\ll$ may depend on $D,K,k,\sigma$.
\end{lem}
\begin{proof}
Let $M=\fl{(\log x)^D}$.
We will partition $(0,1]^{k+2}$ into small boxes.
For each tuple $\bd=(d_1,\ldots,d_{k+2})\in [M]^{k+2}$, let
\[
\cB(\bd) := \prod_{i=1}^{k+2} \bigg( \frac{\ssc{d}{j}-1}{M}, \frac{\ssc{d}{j}}{M} \bigg].
\]

We fix $\bd$ such that
there is at least one $(k+2)$-tuple $(p_1,\ldots,p_k,u,v)$ satisfying
\begin{equation}\label{eq:fixed-bb}
\begin{split}
\bv(\ssc{p}{1} \cdots \ssc{p}{k};x) &\in \cT_{u,v}, \\
 \bigg( \bv(\ssc{p}{1} \cdots \ssc{p}{k}; x),
 \frac{\log u}{\log x}, \frac{\log v}{\log x} \bigg)
 &\in \cB(\bd), \\
P^+(uv)\le n^\sigma <\ssc{p}{1} \le \cdots \le \ssc{p}{k}, \;\; n=\ssc{p}{1}\cdots \ssc{p}{k} uv &\sim x.
\end{split}
\end{equation}
The third line in \eqref{eq:fixed-bb} implies that $\ssc{p}{j} \ge (x/2)^\sigma$ 
for all $j$, and hence $d_j \ge (\sigma/2) M$ for all $j$.
 Also, $\ssc{p}{1}\cdots \ssc{p}{k} uv > x/2$
and thus   $\sum d_j \ge M/2$.
For one such choice of $(k+2)-$tuple 
\[
(\ssc{p}{1},\ldots,\ssc{p}{k},u,v) = (p_1',\ldots,p_k',u',v')
\]
satisfying \eqref{eq:fixed-bb} (it does not matter which tuple), let
\[
g_0(\bd) := g \bigg( \frac{\log p_1'}{\log n'},\ldots,\frac{\log p_k'}{\log n'} \bigg),
\quad n'=p_1'\cdots p_k' u'v'.
\]
For any tuple   $(p_1,\ldots,p_k,u,v)$ satisfying \eqref{eq:fixed-bb}, 
\[
\frac{\log p_j}{\log n} = \frac{\frac{d_j+O(1)}{M}}{\frac{d_1+\cdots+d_{k+2}+O(1)}{M}} = 
 \frac{d_j}{d_1+\cdots + d_{k+2}} + O\pfrac{1}{M}
\qquad (1\le j\le k),
\]
and it follows that
\begin{align*}
g \bigg( \frac{\log p_1}{\log n},\ldots,\frac{\log p_k}{\log n} \bigg) &= g_0(\bd) 
+ O_{k,K} \pfrac{1}{M}\\
&= \frac{(\log p_1)\cdots (\log p_k) g_0(\bd)}{\prod_{j=1}^k ((d_j/M)\log x)} + 
O_{k,K,\sigma} \pfrac{1}{M}.
\end{align*}
(We introduce the logarithm factors in order to apply Lemma \ref{lem:PrimeDecomposition}).
For each $\bd$, it follows that
\begin{align*}
\sum_{u,v}\ssc{\psi}{u,v} &\ssum{\eqref{eq:fixed-bb}} w_{n} g \bigg( \frac{\log \ssc{p}{1}}{\log n},\ldots,\frac{\log \ssc{p}{k}}{\log n} \bigg)
\ll_{D,K,k,\sigma} \\
& \qquad\qquad  \Bigg|\sum_{u,v}\ssc{\psi}{u,v}
\ssum{\eqref{eq:fixed-bb}}   w_n\cdot \frac{\log p_1}{\log x} \cdots \frac{\log p_k}{\log x}\Bigg| +
 \frac{1}{M} \sum_{u,v}|\ssc{\psi}{u,v}| \ssum{\eqref{eq:fixed-bb}} |w_{n}|.
\end{align*}
Write the right side as $|S_1(\bd)| + M^{-1} S_2(\bd)$.

We first estimate $S_1(\bd)$.  With $u,v$ fixed,
the conditions \eqref{eq:fixed-bb} can be written as 
\[
\bv(\ssc{p}{1} \cdots \ssc{p}{k};x) \in \cT'_{u,v}
\]
for some convex region $\cT'_{u,v}$. Since $(w_n)$ satisfies \eqref{eq:TypeI} with $B$ sufficiently large in terms of $k,\gamma$, it then follows from Lemma \ref{lem:PrimeDecomposition} that
\[
S_1(\bd) \ll  (\log x)^{k\ell} 
\sup_{\substack{\beta_{1,1},\ldots,\beta_{k,\ell} \\ 1-\text{bounded}}} \;
\Biggl| \sum_{u,v} \ssc{\psi}{u,v} \!\! \ssum{(\star) \\
\ssc{m}{h}=d_{h,1}\cdots d_{h,\ell}\\ \;(1\le \ell\le k)\\ d_{h,j}<n/x^\gamma\,\forall h,j}
 w_n \prod_{h=1}^k \prod_{j=1}^{\ell} \beta_{h,j}(d_{h,j})\Biggr|+O\pfrac{x}{(\log x)^{B-k}},
\]
where $(\star)$ is the version of \eqref{eq:fixed-bb} with each variable $p_j$
replaced by $m_j$.
We may remove the condition
$(\frac{\log m_1}{\log x},\ldots,\frac{\log m_k}{\log x},\frac{\log u}{\log x}, \frac{\log v}{\log x})\in \cB(\bd)$ 
using Lemma \ref{lem:BoxConditions}.
This introduces a factor $(\log x)^{2(k+2)}$ together with factors
$m_1^{it_1}\cdots m_k^{it_k}$ (which can be absorbed into the functions
$\beta_{h,j}$) and  $u^{it_{k+1}} v^{it_{k+2}}$, for real numbers
$t_1,\ldots,t_{k+2}$.
Summing over the $M^{k+2} \le (\log x)^{D(k+2)}$ choices for $\bd$ and recalling that $B$ is assumed to be sufficiently large in terms of $D$ gives
\begin{multline*}
\sum_{\bd} |S_1(\bd)| \ll (\log x)^{k\ell+(D+2)(k+2)} \times \\ \times
 \sup_{\substack{\beta_{1,1},\ldots,\beta_{k,\ell} \\ 1-\text{bounded}}} \;\;
\sup_{t,t'\in \RR} \;
\Biggl| \sum_{u,v} \ssc{\psi}{u,v} 
\ssum{ n=\ssc{m}{1}\cdots \ssc{m}{k} uv \sim x \\
P^+(uv) \le n^\sigma <  m_1 \le \cdots \le m_k \\
  \big(\frac{\log \ssc{m}{1}}{\log x},\ldots,\frac{\log \ssc{m}{k}}{\log x}\big)\in\cT_{u,v}\\ 
  \ssc{m}{h}=d_{h,1}\cdots d_{h,\ell}\; (1\le \ell\le k)\\ 
  d_{h,j}<n/x^\gamma\,\forall h,j}
w_n u^{it} v^{it'} \prod_{h=1}^k \prod_{j=1}^{\ell} \beta_{h,j}(d_{h,j})\Biggr|.
\end{multline*}

We now estimate the $S_2(\bd)$ terms.  Given $n$, there is at most one tuple
 $(\ssc{n}{1},p_1,\ldots,p_k,\bd)$ with $\ssc{n}{1} p_1\cdots p_k=n$ 
 and $P^+(\ssc{n}{1})\le n^{\sigma} < p_1\le \cdots \le p_k$. Since $u|\ssc{n}{1}$
 and $u\le x^{\gamma}$, then by \eqref{w}, 
\begin{align*}
\sum_{\bd} S_2(\bd) \le \sum_{n} |w_{n}|\sum_{u|n, u\le x^{\gamma}} 1 
\le \sum_n |w_n| \tau(n) \le  x (\log x)^{\varpi}.
\end{align*}

Recalling that $M=\fl{(\log x)^D}$, we conclude that 
\begin{equation}\label{eq:S2b}
\frac{1}{M} \sum_{\bd} S_2(\bd) \ll \frac{x}{(\log x)^{D-\varpi}}.
\end{equation}
This gives the result.
\end{proof}

%
%
%
%

In a future work we will need to analyze more general constructions than those
we use in \eqref{N-def}--\eqref{h-def} below in order to handle parameter ranges where 
$\cG_2$ becomes important.
To facilitate this we introduce a hypothesis, which we call \emph{splittable}, on a function $f(d)$. 

Given a positive integer $d$ and parameter $L\ge 2$ we may canonically decompose $d$
as  follows: if $d=1$ then set $s=1$ and $d_s=1$.  If $d>1$, the $d$ 
 can be decomposed uniquely as
\begin{equation}\label{eq:diprimes}
\begin{split}
d&=d_1\cdots d_s, \quad P^-(d_i) \ge  P^+(d_{i+1}) \quad  (1\le i\le s-1),\\
d_j &\in \big(L,P^-(d_j)L\big] \quad  (1\le j\le s-1),\qquad\qquad d_s \in \big(1, P^-(d_s)L\big],
\end{split}
\end{equation}
by successively adding the largest unused prime to 
$d_j$ until the product is $>L$ or we run out of primes.
 Specifically, if $d$ has prime factorization $d=q_1\cdots q_r$ with $q_1 \ge \cdots \ge q_r$, then either $d\le L$, in which case $s=1$ and $d_1=d$, or $d>L$, in which case
 $d_1 = q_1 \cdots q_{r_1}$ where $r_1$ is the minimum integer so that
$d_1 > L$; in particular, if $q_1>L$ then $r_1=1$. 
Then either $d/d_1\le L$, in which case $s=2$ and $d_2=d/d_1$, or
$d/d_1>L$, in which case $d_2=q_{r_1+1}\cdots q_{r_2}$, where $r_2$
is the smallest integer such that $d_2>L$, and so on. 
Only the final factor $d_s$ may be $\le L$, so certainly $s$ is finite.

When $d$ is known to be squarefree, we may similarly decompose $d>1$
uniquely  in the form 
\begin{equation}\label{eq:diprimes-sqfree}
\begin{split}
d&=d_1\cdots d_s, \quad P^-(d_i) >  P^+(d_{i+1}) \quad  (1\le i\le s-1),\\
d_j &\in \big(L,P^-(d_j)L\big] \quad  (1\le j\le s-1),\qquad\qquad d_s \in \big(1, P^-(d_s)L\big].
\end{split}
\end{equation}

\medskip

\begin{defn}[Splittable functions]\label{def:splittable}
 Given $L\ge 2$, we
say that a function $f:\NN\to \CC$ is \emph{splittable with respect
to $L$} if $|f(n)|\le 1$ for all $n\in \NN$ and there are functions $f_{i,j}$ 
with $|f_{i,j}(e)|\le 1$ for all $i,j,e$ and with either
\begin{itemize}
\item[(a)] For all $d=d_1\cdots d_s>1$ with \eqref{eq:diprimes},
then $f(d)=f_{s,1}(d_1)\cdots f_{s,s}(d_s)$; or
\item[(b)] $f$ is supported on squarefree integers and for all squarefree $d=d_1\cdots d_s>1$ with \eqref{eq:diprimes-sqfree},
then $f(d)=f_{s,1}(d_1)\cdots f_{s,s}(d_s)$.
\end{itemize}
\end{defn}

\begin{rmk}  When $s=1$ we may always take $f_{1,1}=f$, and this also covers
the case $d=1$, as then $f(1)=f_{1,1}(1)$.
\end{rmk}

\begin{lem}\label{lem:splittable-modified-by-mult-fcn}
We have 
\begin{enumerate}
\item[(i)] For any $L\ge 2$ and any completely multiplicative functions $f$ with $|f(d)|\le 1$ for all $d$, $f$ is splittable with respect to $L$;
\item[(ii)]  For any $L\ge 2$ and any multiplicative functions $f$ supported on squarefree integers and with $|f(d)|\le 1$ for all $d$, $f$ is splittable with respect to $L$;
\item[(iii)] For any $L\ge 2$, any function $f$ splittable with respect to $L$,
and any completely multiplicative function $f'$ with $|f'(d)|\le 1$ for all $d$,
$ff'$ is  splittable with respect to $L$.
\end{enumerate}
\end{lem}

\begin{proof}
For (i) and (ii), take $f_{i,j}=f$ for all $i,j$.
To prove (iii), suppose that
 $f_{i,j}$ are the functions associated to $f$, satisfying (a).
Then (a) holds with $f$ replaced by $f'$ and $f_{i,j}$
replaced by $f' f_{i,j}$ for all $i,j$.  The same argument works if $f$ satisfies (b).
\end{proof}

For the purposes of our argument, the only property of splittable functions which
we need is embodied in Lemma \ref{lem:SmoothSplit} below.
In a later work we will need that various sieve weights satisfy the conclusion
of Lemma \ref{lem:SmoothSplit}, which is rather easy when the sieve weights 
are in `well-factorable' form (see Chapter 12.7 in \cite{Opera}).
It is also easy to show that such sieve weights are also linear combinations of splittable functions, each with respect to a different value of $L$.
One may certainly obtain the conclusion of Lemma \ref{lem:SmoothSplit} for other
functions, for example by decomposing 
integers by a scheme different from those in \eqref{eq:diprimes} and \eqref{eq:diprimes-sqfree}.   However our restriction to splittable functions
 will suffice for the applications to prime-producing sieves.

\begin{lem}[Spitting up smooth numbers]\label{lem:SmoothSplit}
Let $x\ge 100$, $0<\sigma<1$, $x^{\sigma/3} \le L \le (x/2)^{\sigma/2}$
and suppose that $f$ is splittable with respect to $L$. 
For any complex numbers $\ssc{\alpha}{d,m}$ with $|\ssc{\alpha}{d,m}|\le x^2$ for all $d,m$, we have
\begin{multline*}
\bigg|\! \ssum{dm \sim x \\ P^+(d)\le (dm)^{\sigma}} \!\!f(d) \ssc{\alpha}{d,m} \bigg|
\ll_{\sigma} (\log{x})^{3/\sigma} \sum_{s=1}^{\fl{3/\sigma}+2} 
\sup_{\beta_1,\ldots,\beta_s} \Biggl|\sum_{\substack{d_1\cdots d_s m\sim x\\ d_j \le (d_1\cdots d_s m)^\sigma\,\forall j \le s}}\!\!\!\!\alpha_{d_1\cdots d_s,m}\prod_{j=1}^s \beta_j (d_j)\Biggr|+x^{-90},
\end{multline*}
the supremum over all 1-bounded functions $\beta_1,\dots \beta_s$ 
such that $\beta_s$ is supported on $[1,x^{\sigma/3}]$ and $\beta_j$ is supported on
 $(x^{\sigma/3},x^\sigma]$ for $j<s$. 
\end{lem}

\begin{proof}
We consider the case where $f$ satisfies part (a) in Definition \ref{def:splittable};
the case where part (b) holds is a nearly identical argument and we indicate where
changes must be made.
Consider a pair $(d,m)$ with $dm \sim x$ and decompose $d=d_1\cdots d_s$
according to \eqref{eq:diprimes}.
We have $L^2\le (x/2)^{\sigma} \le (d_1\cdots d_s m)^\sigma$ and observe 
that the condition $P^+(d) \le (dm)^{\sigma}$
  is equivalent to the collection of statements $P^+(d_j)\le (d_1\cdots d_s m)^\sigma$
 ($1\le j\le s$).  If $d_j$ is not prime then $P^-(d_j)\le L$ and hence $P^+(d_j)\le d_j \le L^2 \le (d_1\cdots d_s m)^\sigma$, so the stronger condition $d_j \le (d_1\cdots d_s m)^\sigma$ holds automatically.
Therefore, if \eqref{eq:diprimes} holds, then
 the condition $P^+(d)\le (d m)^\sigma$ is equivalent to 
 the collection of conditions $d_j\le (d_1\cdots d_s m)^\sigma$ ($1\le j\le s$).
 Also, since $d_i>L\ge x^{\sigma/3}$ for $i \le s-1$, $s\le 1 + 3/\sigma$.
 Thus we see that
\[
\sum_{\substack{d m\sim x\\ P^+(d)\le (d m)^\sigma}}f(d)\alpha_{d,m}=\sum_{s\le 1+\lfloor 3/\sigma\rfloor}\sum_{\substack{d m\sim x\\ d=d_1\cdots d_s\\ d_j\le (d m)^\sigma\,\forall j\\ \eqref{eq:diprimes}}}f_{s,1}(d_1)\cdots f_{s,s}(d_s) \alpha_{d,m}.
\]
With $s$ fixed, we use $s-1$ successive applications of Lemma \ref{lem:Integration} to separate the dependencies between $d_i$ and $d_{i+1}$ in the first line \eqref{eq:diprimes}, rewriting each inequality as
$P^-(d_i)+\frac12 \ge P^+(d_{i+1})$; when we use part (b)
of Definition \ref{def:splittable}, we write each condition in the first line of 
\eqref{eq:diprimes-sqfree} as  $P^-(d_i)-1/2 \ge P^+(d_{i+1})$.
  This gives
\begin{multline*}
\sum_{\substack{d m\sim x\\ d=d_1\cdots d_s\\ d_j\le (d m)^\sigma\,\forall j\\ \eqref{eq:diprimes}}}
f_{s,1}(d_1)\cdots f_{s,s}(d_s)\alpha_{d,m} \ll_\sigma \; x^{-99} + \\
\sup_{t_1,\dots,t_{s-1}\in \RR} \;\; (\log{x})^{s-1}\Biggl|\sum_{\substack{d m\sim x\\ d=d_1\cdots d_s\\ d_j\le (d m)^\sigma\,\forall j\\ (\star)}}
f_{s,1}(d_1)\cdots f_{s,s}(d_s)\alpha_{d,m}\prod_{j=1}^{s-1}\biggl(\frac{P^-(d_j) + \frac{1}{2}}{P^+(d_{j+1})}\biggr)^{it_j} \Biggr|,
\end{multline*}
where $(\star)$ is the collection of conditions in the second line of \eqref{eq:diprimes}.
With $t_1,\ldots,t_{s-1}$ fixed, for each positive integer $e\le x^{\sigma}$ and $1\le j\le s$ define
the 1-bounded functions
\[
\xi_j(e):=\begin{cases}
f_{s,j}(e)\one_{e\in (L,P^-(e)L]}(P^-(e) + 1/2)^{it_j},\qquad &j=1\text{ and }s\ge 2,\\
f_{s,j}(e)\one_{e\in (L,P^-(e)L]}(P^-(e) + 1/2)^{it_j}P^+(e)^{-it_{j-1}}, &2\le j\le s-1,\\
f_{s,j}(e)\one_{e\in [1,P^-(e)L]}P^+(e)^{-it_{j-1}},&j=s\text{ and } s\ge 2, \\
f_{s,j}(e)\one_{e\in [1,P^-(e)L]}, & j=s=1.
\end{cases}
\]
Thus, we then see that 
\[
\sum_{\substack{d m\sim x\\ d=d_1\cdots d_s\\ d_j\le (d m)^\sigma\,\forall j\\
 (\star)}} 
 f_{s,1}(d_1)\cdots f_{s,s}(d_s)\alpha_{d,m}\prod_{j=1}^{s-1}\Bigl(\frac{P^-(d_j)+\frac{1}{2}}
{P^+(d_{j+1})}\Bigr)^{it_j}=\sum_{\substack{d m\sim x\\ d=d_1\cdots d_s\\ d_j\le (d m)^\sigma\,\forall j}}\alpha_{d,m}\prod_{j=1}^{s}\xi_j(d_j).
\]
For $1\le j\le s-1$, $\xi_j$ is supported on integers greater than $L$, and thus
greater than $x^{\sigma/3}$.
Anticipating some minor future technicalities, we wish to be careful as to whether $d_s$ is bigger or smaller than $x^{\sigma/3}$.
 We introduce new functions
\begin{align*}
\xi'_s(e) &:= \xi_s(e) \one_{e\le x^{\sigma/3}}, \\
\xi''_s(e) &:= \xi_s(e) \one_{e > x^{\sigma/3}}, \\
\xi_{s+1}(e) &:= \one_{e=1}.
\end{align*}
In this way, $\xi_s(e) = \xi'_s(e) + \xi''_s(e)$, and hence
\begin{multline*}
\ssum{d m\sim x\\ d=d_1\cdots d_s\\ d_j\le (d m)^\sigma\,\forall j}\alpha_{d,m}\prod_{j=1}^{s}\xi_j(d_j) =
\ssum{d m\sim x\\ d=d_1\cdots d_s\\ d_j\le (d m)^\sigma\,\forall j}\alpha_{d,m}
 \xi_1(d_1)\cdots \xi_{s-1}(d_{s-1}) \xi'_s(d_s) + \\
+\ssum{d m\sim x\\ d=d_1\cdots d_{s+1}\\ d_j\le (d m)^\sigma\,\forall j}\alpha_{d,m}
 \xi_1(d_1)\cdots \xi_{s-1}(d_{s-1}) \xi''_s(d_s) \xi_{s+1}(d_{s+1}).
\end{multline*}
Each sum on the right side has the required form.
\end{proof}

%
%
%

%

\subsection{Proof of Theorem \ref{thm: Main sieving}}
Since $\cR$ is nonempty, we may assume that $\nu<1-\gamma$.
We recall the vector notation of Section \ref{sec:Notation},  Definition \ref{defn:R1} of our fundamental region $\cR=\cR(\gamma,\theta,\nu)$ and the quantities defined at the beginning of Section \ref{sec:sieving}. Define
\begin{align}
\cN &:= \big\{ x/2 < n \le x : n\text{ composite, } \bv(n;n) \in \cC(\cR) \big\}, \label{N-def} \\
G(m;n) &:= \big( \one_{m \le n^{\gamma}} \big) \mu(\ssc{m}{1}) g(\bv(\ssc{m}{2};n)),\quad (m=\ssc{m}{1}\ssc{m}{2},\,P^+(\ssc{m}{1})<n^\nu\le P^-(\ssc{m}{2})), \label{g-def}\\
H(n) &:= \sum_{d|n}G(d;n). \label{h-def}
\end{align}
Note that $|\bv(n;n)|=1$. The above definition of $H$ is equivalent to $H(n)=(\bone\star g_1)(\bv(n;n))$, where $g_1(\bx)=(-1)^k g(\bx_1)\one_{|\bx| \le \gamma}$ if $\bx$ has $k$ components less that $\nu$ and $\bx_1$ is the vector formed from the components of $\bx$ which are $\ge \nu$.  The most important properties of the function $H(n)$ are given
in the next lemma.  Essentially, we need $G(m;n)$ to behave like a 
sieve weight in \eqref{h-def}, but only when restricted to $n\in \cN$.
 In particular, when $g$ is supported in $\cG_1$, the factor $\mu(m_1)$ in \eqref{g-def} effectively sifts out
 numbers with a prime factor below $n^{\nu}$, analogous to the Legendre sieve.

\begin{lem}\label{lem:h-onesign}
Let $g\in \sG_1$ and define $H$ by \eqref{g-def} and \eqref{h-def}.
We have
\begin{itemize}
\item[(a)]  For all $n\in \cN$, $H(n)=(1\star g)(\bv(n))\one_{P^-(n)\ge n^{\nu}}$.
\item[(b)]  For all $n\in \cN$, $|H(n)| \ll 1$.
\item[(c)] If $(\bone \star g)(\bx) \le 0$ for all $\bx \in \cH$ then for all 
$n\in \cN$, $H(n) \le 0$; If  $(\bone \star g)(\bx) \ge 0$ for all $\bx \in \cH$ 
then for all $n\in \cN$, $H(n) \ge 0$.
\end{itemize}
\end{lem}
\begin{proof}
(a): If $n\in\cN$ and $d=\ssc{m}{1}\ssc{m}{2}$ divides $n$ with $P^+(\ssc{m}{1})<n^\nu\le P^-(\ssc{m}{2})$ then $\bv(\ssc{m}{2};n)$ has all coordinates $\ge \nu$ and is a subvector of an element of $\cC(\cR)$ (i.e. $\bv(\ssc{m}{2};n)\in\cZ$). By assumption $g$ is supported on $\cG_1$, so we have
 $g(\bv(\ssc{m}{2};n))=0$ unless 
 \[
 |\bv(\ssc{m}{2};n)|\le \gamma-\psi(\bv(\ssc{m}{2};n))\le \gamma-|\bv(\ssc{m}{1};n)|,
 \]
where we used that $\bv(m_1m_2;n)\in\cZ$. Thus we may drop the constraint $d \le n^\gamma$ in the definition of $G(d;n)$ whenever $n\in\cN$, so we have that 
\[
H(n)=\sum_{\substack{\ssc{m}{2}|n\\ n^\nu\le P^-(\ssc{m}{2})}}g(\bv(\ssc{m}{2};n))\sum_{\substack{\ssc{m}{1}|n\\ P^+(\ssc{m}{1})<n^\nu}}\mu(\ssc{m}{1}).
\]
Clearly the inner sum is 0 unless $P^-(n)\ge n^\nu$, in which case the outer
sum is $(1\star g)(\bv(n))$.

(b): If $P^-(n)\ge n^\nu$ then $n$ has $O(1)$ divisors, and the result follows from the fact $g\ll 1$.

(c): If $n\in \cN$, then $\bv(n)$ has at least two components and thus
$\bv(n)\in \cH$.  The conclusion then follows from part (a).
\end{proof}

\begin{prop}\label{prop:sieve-usingIandII}
Suppose that $(\gamma,\theta,\nu)\in \cQ_0$ and $\nu\le 1-\gamma.$
Let $g\in \sG_1$ and define $H$ by \eqref{g-def} and \eqref{h-def}.
Let $A\ge 1$, $\varpi\ge 1$ and $B$ be sufficiently large in terms of $A,\varpi,\gamma,\theta,\nu,g$.
Assume $(w_n)$ satisfies \eqref{w}, \eqref{eq:TypeI} and \eqref{eq:TypeII}.
 Then
\[
\ssum{x/2 < n \le x \\ n\not\in \cP \cup \cN} w_n H(n) \ll_g \frac{x}{(\log x)^A}.
\]
\end{prop}

The proof of Proposition \ref{prop:sieve-usingIandII} is rather long, and so we defer it to the end of the section.

%
%
%
%
%

\begin{lem}\label{lem:h-sum-to-integral}
Let $(\gamma,\theta,\nu)\in \cQ_0$, $g\in \sG_1$ and define $H(n) = (1\star g)(\bv(n))$.
Let $\cU=\cU_1 \sqcup \cdots \sqcup \cU_j$, where $\cU_i$ is a convex polytope
for $1\le i\le j$.
Suppose that $\varpi>0$ and $(b_n)$ is a sequence satisfying \eqref{bn-f-sum}.
Then for sufficiently large $x$ (in terms of $\varpi,\nu$) we have
\[
\ssum{\bv(n) \in \cU \\ P^-(n) \ge n^{\nu}} b_n H(n) = 
 \bigg(  \sum_p b_p \bigg) \Bigg( \sum_{k=2}^{\fl{1/\nu}} \mint{\bx \in \cU\cap \RR^k \\ \ssc{x}{1} \le \cdots \le \ssc{x}{k}} \frac{(\bone \star g)(\bx)}{\ssc{x}{1}\cdots \ssc{x}{k}}d\bx + O_{g,j}\pfrac{1}{B}\Bigg).
 \]
\end{lem}
%
%

\begin{proof}
If $P^-(n)\ge n^{\nu}$ then $n$ has at most $1/\nu$ prime factors, counted with multiplicity.  Thus,
\[
\ssum{\bv(n) \in \cU \\ P^-(n) \ge n^{\nu}} b_n H(n) = \sum_{2\le k \le \lfloor 1/\nu\rfloor}\; \sum_{\cJ\subseteq[k]}
\ssum{x/2<n\le x \\ n=p_1\cdots p_k \\ \bv(n)\in \cU\\ n^{\nu}\le p_1\le \dots\le p_k } b_{n}g\Big(\bv\big(\prod_{j\in \cJ}p_j;n\big)\Big).
\]
Since $g\in \sG_1$, we may write $g(\bx)=\sum_i \one_{\cT_i}(\bx)g_i(\bx)$ for a collection of $O_g(1)$ convex polytopes $\cT_i$, and Lipschitz-continuous functions $g_i$ each with Lipschitz constant $O_g(1)$ each bounded by $O_g(1)$.
For each $k,\cJ,h,i$ the collection of summation conditions 
$\bv(\prod_{j\in \cJ}p_j;n)\in \cT_h \cap \cU_i\cap \RR^k$,
$n^{\nu}\le p_1$, and $\bv(n)\in \cU_j$ is equivalent to $\bv(n)$ lying in a convex polytope $\cU_{k,\cJ,h,i}\subseteq[\nu,1]^k$.
 It follows that the triple sum above is
\[
\sum_{2\le k \le \lfloor 1/\nu\rfloor}\; \sum_{\cJ\subseteq[k]}\sum_{h,i} \;
\ssum{x/2<n=p_1\cdots p_k\le x  \\ \bv(n)\in \cU_{k,\cJ,h,i} \\ p_1\le \cdots \le p_k } b_{n}f_{k,\cJ,h,i}(\bv(n))
\]
for a Lipschitz-continuous function $f_{k,J,h,i}$ with Lipschitz constant $O_g(1)$ bounded by $O_g(1)$. Thus for each $k,\cJ,h,i$ we may apply \eqref{bn-f-sum}, which gives
\[
\ssum{x/2<n=p_1\cdots p_k\le x  \\ \bv(n)\in \cU_{k,\cJ,h,i} \\ p_1\le \cdots \le p_k } b_{n}f_{k,\cJ,h,i}(\bv(n)) = 
\bigg(  \sum_p b_p \bigg) \Bigg(\;\; \mint{\cU_{k,\cJ,h,i}} \frac{f_{k,\cJ,h,i}(u_1,\ldots,u_k)}{u_{1}\cdots u_k}\, d \bu \,+ O_g\pfrac{1}{B} \Bigg).
\]
Summing over all $h,i$ with $k$ fixed yields
\begin{align*}
\sum_{h,i}\ssum{x/2<n=p_1\cdots p_k\le x  \\ \bv(n)\in \cU_{k,\cJ,h,i} \\ p_1\le \cdots \le p_k } b_{n}f_{k,\cJ,j}(\bv(n))&=\bigg(  \sum_p b_p \bigg)\Bigg(\;\; \mint{\bu\in \cU\cap\RR^k \\ \nu\le u_1\le \dots \le u_k} \frac{g(\ssc{\bu}{\cJ})}{u_{1}\cdots u_k}\, d \bu \,+ O_{g,j}\pfrac{1}{B} \Bigg).
\end{align*}
Thus, summing over over $\cJ\subseteq [k]$ gives
\begin{align*}
 \sum_{\cJ\subseteq[k]}
\ssum{x/2<n\le x \\ n=p_1\cdots p_k \\ \bv(n)\in \cU \\ n^{\nu}\le p_1\le \cdots \le p_k } b_{n}g\Big(\bv\big(\prod_{j\in \cJ}p_j;n\big)\Big)
&= \bigg(  \sum_p b_p \bigg)\Bigg(\;\mint{\cU\cap \RR^k \\ \nu\le u_1 \le \cdots \le u_k} \frac{(\bone \star g)(\bu)}{u_{1}\cdots u_k}\, d \bu + + O_{g,j}\pfrac{1}{B} \Bigg).
\end{align*}
We get the lemma upon summing over $2\le k\le \lfloor 1/\nu\rfloor$.
\end{proof}

\begin{rmk}
Lemma \ref{lem:h-sum-to-integral} is the only part of the proof of Theorem \ref{thm: Main sieving} that requires hypothesis \eqref{bn-f-sum}.
\end{rmk}

\medskip

\begin{lem}\label{lem:htypeI}
Suppose that $(w_n)$ satisfies the Type I bound \eqref{eq:TypeI} for some $B>1$.  
Suppose that $g\in \sG_1$ and $H$ is defined by
\eqref{g-def} and \eqref{h-def}.  Then
\[
\sum_{x/2<n\le x}w_nH(n)\ll_g \frac{x}{(\log{x})^B}.
\]
\end{lem}

\begin{proof}
We recall that $g\in \sG_1$, so $g(\bx)=\sum_j g_{j}(\bx)$ for some finite collection of bounded functions $g_{j}$ each supported on a convex polytope $\cT_j$ with continuous, bounded first partial derivatives on the interior of $\cT_j$. Then 
\[
H(n)=\sum_j \sum_{\substack{\ssc{m}{1}\ssc{m}{2}|n\\ \ssc{m}{1}\ssc{m}{2} \le n^\gamma\\ P^+(\ssc{m}{1})< n^\nu\le P^-(\ssc{m}{2})}}\mu(\ssc{m}{1})g_{j}(\bv(\ssc{m}{2};n)).
\]
Given $\ssc{m}{2}$, $\bv(\ssc{m}{2};n)$ will lie in the convex polytope $\cT_j$ provided $n\in \cI_{\ssc{m}{2}}$ for some interval depending only on $\ssc{m}{2}$. Similarly the conditions $\ssc{m}{1}\ssc{m}{2} \le n^\gamma$, $P^+(\ssc{m}{1})<n^\nu\le P^-(\ssc{m}{2})$ will hold provided $n$ lies in some interval depending only on $\ssc{m}{1},\ssc{m}{2}$. Let $\cI_{\ssc{m}{1},\ssc{m}{2}}$ denote the intersection of this interval with $\cI_{\ssc{m}{2}}$. We see that since $\frac{\partial}{\partial n}\bv(m;n)=-\bv(m;n)/(n\log{n})$, we have
\[
\frac{\partial}{\partial t}g_{j}(\bv(\ssc{m}{2};t))=\frac{-1}{t(\log{t})}\bv(\ssc{m}{2};t)\cdot \nabla g_{j}(\bv(\ssc{m}{2};t)).
\]
We let $\tilde{g}_j(m;t):=\bv(m;t)\cdot \nabla g_{j}(\bv(m;t))$, which is
continuous as a function of $t$ with $m$ fixed, and bounded due to the smoothness of $g_{j}$.
Thus,
\[
g_{j}(\bv(\ssc{m}{2};n))=g_j(\bv(\ssc{m}{2};x))+\int_{n}^x\frac{1}{t\log{t}} \tilde{g}_j(\bv(\ssc{m}{2};t))dt
\]
for $n\in\cI_{\ssc{m}{2}}$. Substituting this into our definition of $H$, and summing over $x/2<n\le x$ weighted by $w_n$ gives
\begin{align*}
\sum_{x/2<n\le x}w_nH(n)&=\sum_j \sum_{\substack{\ssc{m}{1}\ssc{m}{2} \le x^\gamma\\ P^+(\ssc{m}{1})< P^-(\ssc{m}{2})}}\mu(\ssc{m}{1})g_{j}(\bv(\ssc{m}{2};x))\sum_{\substack{x/2<n\le x\\ n\in \cI_{\ssc{m}{1},\ssc{m}{2}} \\ \ssc{m}{1}\ssc{m}{2}|n}}w_n\\
&+\sum_j \sum_{\substack{\ssc{m}{1}\ssc{m}{2} \le x^\gamma\\ P^+(\ssc{m}{1})< P^-(\ssc{m}{2})}}\mu(\ssc{m}{1})\int_{x/2}^{x} \frac{\tilde{g}_{j}(\bv(\ssc{m}{2};t))}{t\log{t}}\sum_{\substack{t<n\le x\\ n\in \cI_{\ssc{m}{1},\ssc{m}{2}} \\ \ssc{m}{1}\ssc{m}{2}|n}}w_n\; dt.
\end{align*}
We see that $n$ is restricted to the interval $\cI_{\ssc{m}{1},\ssc{m}{2}} \cap (t,x]$
 and that $\tilde{g}_j(\ssc{m}{2};t)/(t\log{t})\ll_g 1/x$. Thus by \eqref{eq:TypeI} we have that
\[
\sum_{x/2<n\le x}w_nH(n)\ll_g \sum_{m \le x^\gamma}\tau(m)\sup_{\text{interval $\cI$}}\;\;\Biggl|\sum_{\substack{x/2<mn\le x\\ n\in \cI}}w_{mn}\Biggr|\ll \frac{x}{\log^B{x}},
\]
as required.
\end{proof}

With these lemmas we may complete the proof of Theorem \ref{thm: Main sieving} quickly.

\begin{proof}[Proof of Theorem \ref{thm: Main sieving}] 
Let $a_n,b_n$ be non-negative sequences and $w_n=a_n-b_n$.
By Lemma \ref{lem:h-onesign} $(c)$, the fact that $H(p)=1$ for primes $p\in (x/2,x]$, and $w_n\ge -b_n$, under the hypotheses in Theorem \ref{thm: Main sieving} (a) we have
\begin{align*}
\sum_{p\in \cP} w_p = \sum_{p\in \cP} (w_p+b_p)-b_p &\ge -\sum_{p\in \cP} b_p + \sum_{n\in \cP \cup \cN} (w_n+b_n)H(n) \\
&= \sum_{n\in \cN} b_n H(n) + \sum_{n\in \cP \cup \cN} w_n H(n),
\end{align*}
and  under the hypotheses in Theorem \ref{thm: Main sieving} (b) we have
\begin{align*}
\sum_{p\in \cP} w_p = \sum_{p\in \cP} (w_p+b_p)-b_p &\le -\sum_{p\in \cP} b_p + \sum_{n\in \cP \cup \cN} (w_n+b_n)H(n) \\
&= \sum_{n\in \cN} b_n H(n) + \sum_{n\in \cP \cup \cN} w_n H(n).
\end{align*}
Let 
\[
C := \sum_{k=2}^{\fl{1/\nu}} \mint{\bx \in \cH\cap \RR^k \\ \ssc{x}{1} \le \cdots \le \ssc{x}{k}} \frac{(\bone \star g)(\bx)}{\ssc{x}{1}\cdots \ssc{x}{k}}d\bx.
\]
By Lemma \ref{lem:CR1-union-polytopes}, $\cH$ is the union of a finite
number of convex polytopes.
  By Lemma \ref{lem:h-onesign} (a), $H(n)\ne 0$
implies that $P^-(n)\ge n^{\nu}$, and in this case $n\in \cN$ is equivalent
to $\bv(n)\in \cH$.
Lemma \ref{lem:h-sum-to-integral} implies that
\[
\sum_{n\in \cN} b_n H(n) = \Big( C+ O\Bigl(\frac{1}{B}\Bigr)
\Big) \sum_p b_p.
\]

Let $\varpi>0$ be arbitrary, assume that $A\ge 1+\varpi$ and that
$B$ is sufficiently large in terms of $A,\gamma,\theta,\nu,g$,
$x$ sufficiently large in terms of $\gamma,\theta,\nu,g,\varpi,A$,
and $((a_n),(b_n)) \in \Psi(\gamma,\theta,\nu;B,\varpi,x)$.
 Proposition \ref{prop:sieve-usingIandII} and Lemma \ref{lem:htypeI} imply that
\begin{equation}\label{eq:sieve-sum-wnHn}
\sum_{n\in \cP \cup \cN} w_n H(n) = \sum_{x/2 < n\le x} w_n H(n) - \ssum{x/2<n\le x
\\ n\not\in \cP \cup \cN} w_n H(n)\ll_g \frac{x}{(\log x)^{A}}.
\end{equation}
By \eqref{bp-sum},
\[
\sum_{n\in \cN} b_n H(n) + \sum_{n\in \cP \cup \cN} w_n H(n) =
\Bigl(C+O\Bigl(\frac{1}{B}\Bigr)
 +O\Bigl(\frac{1}{(\log{x})^{A-\varpi}\Bigr)}\Bigr) \sum_p b_p.
\]
Therefore, taking $A=1+\varpi$, the lower bound $C^-(\gamma,\theta,\nu)\ge 1+C$ in part $(a)$ and the upper bound
$C^+(\gamma,\theta,\nu)\le 1+C$ in part $(b)$ of Theorem \ref{thm: Main sieving} follow on considering $B$ arbitrarily large.

\end{proof}

\subsection{Proof of Proposition \ref{prop:sieve-usingIandII}}

Anticipating future applications, where we also use the set $\cG_2$,
we prove a more general version of Proposition \ref{prop:sieve-usingIandII}.

\begin{prop}\label{prop:general}
Suppose that
\begin{equation}\label{eq:prop-general-param-range}
(\gamma,\theta,\nu)\in \cQ_0, \qquad 0 < \sigma \le \nu \le 1-\gamma.
\end{equation}
Let $g$ be a vector function in $\cS$ supported on vectors with all components
$\ge \sigma$ and sum of components at most $\gamma$, and such that $g$ is a finite sum of
functions which are each in $\cS$, supported on a convex polytope and Lipschitz continuous on this polytope.  
Let $x^{\sigma/3} \le L \le (x/2)^{\sigma/2}$ and suppose that $\lambda(d)$
is splittable with respect to $L$.
 Define
\begin{align}
G(m;n) &= \big( \one_{m \le n^{\gamma}} \big) \lambda({\ssc{m}{1}}) g(\bv(\ssc{m}{2};n)),\quad (m=\ssc{m}{1}\ssc{m}{2},\,P^+(\ssc{m}{1})<n^\sigma\le P^-(\ssc{m}{2})) 
\label{g-def-gen}\\
H(n) &= \sum_{d|n}G(d;n). \label{h-def-gen}
\end{align}
Let $A>0$, $\varpi\ge 1$ and
assume $B$ is sufficiently large in terms of $\gamma,\theta,\nu,g,A,\varpi,\sigma$.
Assume $(w_n)$ satisfies \eqref{w}, \eqref{eq:TypeI}, \eqref{eq:TypeII}.
Define $\cN$ by \eqref{N-def}, and let $\cP$ be the set of primes in $(x/2,x]$.
Then
\[
\ssum{x/2<n\le x \\ n\not\in \cN \cup \cP} H(n) w_n = \ssc{O}{g,A,\sigma}
 \pfrac{x}{(\log x)^{A}}.
\]
\end{prop}

Proposition \ref{prop:sieve-usingIandII} follows immediately upon taking
$\lambda(d) = \mu(d)$ and $\sigma=\nu$ and invoking Lemma \ref{lem:SmoothSplit} (ii).

%
%
%
\begin{proof}[Proof of Proposition \ref{prop:general}]
%
%
%

Consider $n\sim x$ with $n\not \in \cP \cup \cN$.  Equivalently, $x/2<n\le x$,
$\Omega(n)\ge 2$ and $\bv(n) \not\in \cC(\cR)$.
We write $n=\ssc{n}{1}\ssc{n}{2}$ with $P^+(\ssc{n}{1})\le n^\sigma <P^-(\ssc{n}{2})$. 
Thus $\ssc{n}{2}$ has at most $1/\sigma$ prime factors, and we consider separately for 
each $k\in \{0,\dots,\lfloor1/\sigma\rfloor\}$ the contribution when $\ssc{n}{2}=p_1\cdots p_k$
with $p_1\le \cdots \le p_k$. By Lemma \ref{lem:CR1-union-polytopes}, we must have that
 $\bv(\ssc{n}{2};n)\in\cT$ for one of a bounded number of convex polytopes 
 $\cT\subseteq \RR^k$. Moreover, each $\cT$ has the property that for any 
 $\bt\in\cT$ and $\bt'\in [0,1-\gamma]^r$ (for some $r$) with $|\bt'|=1-|\bt|$,
  we have $(\bt,\bt')\notin \cC(\cR)$ and that $\cT$ only involves a bounded number of constraints.
   We also see that the condition $\Omega(n)\ge 2$ is automatically satisfied if $k\ge 2$, whereas if $k \le 1$,
  then $\Omega(n)\ge 2$ is equivalent to
   $\ssc{n}{1}>1$. Thus it suffices to show, for each such polytope $\cT$ 
   and each such $k$, that
\[
\sum_{\substack{n=\ssc{n}{1}p_1\cdots p_k\sim x\\ \ssc{n}{1}>1\text{ or } k\ge 2\\
 P^+(\ssc{n}{1})\le n^\sigma<p_1\le p_2\dots\le p_k\\ \bv(p_1\cdots p_k;n)\in \cT}}
H(n)w_n \ll_{k,A} \frac{x}{(\log x)^{A}}.
\]
We expand out the definition of 
\[
H(n)=\sum_{u|\ssc{n}{1}} \sum_{\cJ \subseteq [k]}
\one \Big( u\prod_{j\in \cJ} p_j \le n^\gamma \Big) \lambda(u) g\Big( \bv\big(\prod_{j\in \cJ} p_j;n\big) \Big).
\] 
 Since $g\in\sG_1$, when restricted to $\RR^{|\cJ|}$, 
\[
g(\bx)=\sum_{i}\big( \one_{\bx \in \cG_i} \big) \tilde{g}_i(\bx)
\]
 is a sum of $O_g(1)$ Lipschitz continuous functions $\tilde{g}_i$ bounded by $O_g(1)$, each multiplied by the indicator function of a convex polytope $\cG_i$ involving $O_g(1)$ constraints.  
 We include the `trivial polytope' $\cG_i$ of dimension zero 
 which contains the single vector $\emptyvec$ and corresponds to the term $\bx=\emptyvec$.
 Thus it suffices to show, for each $k$, $\cT$, $\cJ\subseteq[k]$, convex polytope $\cG$ involving a bounded number of constraints, and Lipschitz continuous function $\tilde{g}$ that
\[
\sum_{\substack{n=\ssc{n}{1}\ssc{p}{1}\cdots \ssc{p}{k}\sim x\\ \ssc{n}{1}>1\text{ or }k\ge 2\\ P^+(\ssc{n}{1})\le n^\sigma<\ssc{p}{1}\le \dots \le \ssc{p}{k}\\ \bv(\ssc{p}{1}\cdots \ssc{p}{k};n)\in \cT \\ \bv(\prod_{j\in \cJ} \ssc{p}{j};n)\in \cG}\;\;}\sum_{\substack{u|\ssc{n}{1}\\ u \prod_{j\in \cJ} \ssc{p}{j} \le n^\gamma}}
\tilde{g}\bigg( \bv\Big(\prod_{j\in \cJ}p_j;n\Big)\bigg) 
\lambda(u) w_n  \ll_{k,A,\tilde{g}}\; \frac{x}{(\log x)^{A}}.
\]
Here the case $\cJ=\emptyset$ corresponds only to the trivial polytope $\cG$.

The conditions $\bv(p_1\cdots p_k;n)\in \cT$, $n^\sigma < p_1\le \dots\le p_k$ and 
$\bv(\prod_{j\in \cJ}p_j;n)\in \cG$ are all systems of linear inequalities in the
components of $\bv(p_1\cdots p_k;n)$, and so may be combined into a condition 
$\bv(p_1\cdots p_k;n)\in \cT'$ for some convex polytope $\cT'\subseteq \cT$. We note that the number of linear
 inequalities defining the polytope $\cT'$ is bounded, since $\cT,\cG$ involve a bounded number of constraints and $k\le 1/\sigma$. Since $\cT'\subseteq \cT$, if  $\bt\in\cT'$ and $\bt'\in [0,1-\gamma]^r$ (for some $r$) with $|\bt'|=1-|\bt|$, we have $(\bt,\bt')\notin \cC(\cR)$. 
 Hence, it suffices to show for each such convex polytope 
 $\cT'\subseteq[\sigma,1]^k$ and Lipschitz continuous $\tilde{g}$ and $\cJ\subseteq[k]$ that
\[
\ssum{n= u v \ssc{p}{1}\cdots \ssc{p}{k}\sim x\\ 
uv>1\text{ or }k\ge 2\\ P^+(uv)\le n^\sigma \\ 
\bv(\ssc{p}{1} \cdots \ssc{p}{k};n)\in \cT' \\ 
u \prod_{j\in \cJ} \ssc{p}{j} \le n^\gamma}\;\;
\tilde{g}\bigg( \bv\Big(\prod_{j\in \cJ}p_j;n\Big)\bigg) 
\lambda(u) w_n  \ll_{k,A,\tilde{g}}\; \frac{x}{(\log x)^A}.
\]
With $(u,v)$ fixed the collection of  conditions $P^+(uv)\le n^{\sigma}$,
$\bv(\ssc{p}{1} \cdots \ssc{p}{k};n)\in \cT'$ and 
$u \prod_{j\in \cJ} \ssc{p}{j} \le n^\gamma$ is equivalent to 
$\bv(\ssc{p}{1}\cdots \ssc{p}{k};x)\in \cT_{u,v}$
for some convex polytope $\cT_{u,v}$, and that it is nonempty
only if $u\le x^{\gamma}$.
With this notation, the left side above equals the left side in Lemma \ref{lem:PrimeSplit}.  We apply this lemma with $D=A+\varpi$
 and $\ell=6\cl{1/(1-\gamma)}$.

It therefore suffices to show that for any $k\ge 0$, real $t$ and $t'$, all choices of convex 
$\cT'$, and $1$-bounded functions $\alpha_{h,j}$, we have
\begin{equation}\label{eq:after-prime-decomp}
\ssum{n=u v \ssc{m}{1}\cdots \ssc{m}{k}\sim x\\ uv>1\text{ or } k\ge 2\\
 P^+(uv)\le n^\sigma\\ 
 \ssc{m}{h}=\prod_{j=1}^\ell d_{h,j}\; (1\le h\le k) \\  
   \by\in \cT' \\
 \\ d_{h,j} < n/x^{\gamma}\,\forall h,j \\
  u \prod_{h\in \cJ} \ssc{m}{h} \le n^\gamma} 
\lambda(u) w_n u^{it} v^{it'} \prod_{h=1}^k\prod_{j=1}^\ell \alpha_{h,j}(d_{h,j})
\ll_A \; \frac{x}{(\log x)^{k\ell+(k+2)(D+2)+A}},
\end{equation}
where
\[
\by:=\Bigl(\frac{\log{\ssc{m}{1}}}{\log n},\dots,\frac{\log{\ssc{m}{k}}}{\log n}\Bigr).
\]
When $k\le 1$, the condition $uv>1$ may be encoded as $uv \ge \frac32$ and removed using Lemma \ref{lem:Integration}, which gives extra factors $(uv)^{it''}(\log x)$
for real $t''$.  It thus suffices to prove a version of \eqref{eq:after-prime-decomp}
with the condition ``$uv>1$ or $k\ge 2$'' removed from the summation and the right side multiplied by $(\log x)^{-1}$.

By hypothesis, $\lambda(u)$ is splittable with respect to $L$.
By Lemma \ref{lem:splittable-modified-by-mult-fcn} (iii), $\lambda(u) u^{it}$
is also splittable with respect to $L$.
Therefore, taking
 $m=v\ssc{m}{1}\cdots \ssc{m}{k}$ in Lemma \ref{lem:SmoothSplit}, 
 we bound the left side of \eqref{eq:after-prime-decomp} by $x^{-90}$ plus a sum where
 $u$ is replaced by a product $\ssc{u}{1}\cdots \ssc{u}{s}$ with $1\le s\le 3/\sigma+2$, we twist by 1-bounded functions $\beta_1(u_1),\ldots, \beta_s(u_s)$
 satisfying the support conditions in Lemma \ref{lem:SmoothSplit},
and we gain an extra factor $(\log x)^{3/\sigma}$.

We then apply Lemma \ref{lem:SmoothSplit} (i),  which shows that
 $v^{it'}$ is a splittable function with respect to $L$, and apply Lemma \ref{lem:SmoothSplit} again, this time with
 $m=\ssc{u}{1}\cdots \ssc{u}{s} \ssc{m}{1}\cdots \ssc{m}{k}$.  
Then we further bound the left side of \eqref{eq:after-prime-decomp} by $O(x^{-80})$
plus a sum where $v$ is 
replaced by $\ssc{v}{1}\cdots \ssc{v}{r}$, $r\le 3/\sigma+2$,
we twist by 1-bounded functions $\beta'_1(v_1),\ldots, \beta'_r(v_r)$
and we gain an extra factor $(\log x)^{3/\sigma}$.
 Also, the quantity $\alpha_{d,m}$
appearing in Definition \ref{def:splittable} is bounded above by a power of $\tau(dm)$
multiplied by $w_{dm}$,
and hence $|\alpha_{d,m}|\le x^2$ for large enough $x$ 
using the crude bound $|w_n|\ll  x^{1.1}$ which follows from \eqref{w}.
Lemma \ref{lem:SmoothSplit} also implies that the functions $\beta_1,\ldots,\beta_{s-1},\beta'_1,\ldots,\beta'_{r-1}$
are supported on integers $>x^{\sigma/3}$, and that $\beta_s,\beta'_r$
are supported on $[1,x^{\sigma/3}]$.
Thus, we may add to the summation the conditions
$u_j > x^{\sigma/3}$ for $j<s$, $v_j > x^{\sigma/3}$ for $j<r$, $u_s \le x^{\sigma/3}$
and $v_r \le x^{\sigma/3}$.
We may also add the conditions $\ssc{m}{i}>(x/2)^\sigma$ since in the polytope $\cT'$,
every coordinate is at least $\sigma$ (this comes from the encoding of $n^{\sigma}\le p_1 \le \cdots  \le p_k$ that is part of the definition of $\cT'$).

We see that it suffices to show that for some $A'$ which is sufficiently large in terms of $P=(\gamma,\theta,\nu)$, $\sigma$, $\varpi$, and $A$,
\[
\ssum{n=\ssc{u}{1}\cdots \ssc{u}{s} \ssc{v}{1}\cdots \ssc{v}{r}\ssc{m}{1}\cdots \ssc{m}{k}\sim x\\ 
 \ssc{m}{h}=\prod_{j=1}^\ell d_{h,j}\; \forall h\le k \\ \eqref{eq:collection}\\
 \\ d_{h,j} < n/x^{\gamma}\,\forall h,j \\
 }
 w_n \cI(\bu,\bv,\bmm) \sprod{1\le h\le k \\ 1\le j \le \ell} \alpha_{h,j}(d_{h,j})\prod_{j=1}^{s}\beta_j(\ssc{u}{j}) \prod_{j=1}^{r}\beta'_j(\ssc{v}{j})
 \ll \frac{x}{(\log{x})^{A'}}
\]
for any 1-bounded functions $\beta_j,\beta'_j,\alpha_{h,j}$ and 
$s,r\le 3+2/\sigma$, where
\begin{equation}\label{eq:collection}
\begin{split}
\by\in \cT', \qquad \ssc{u}{j} &\le n^{\sigma} \;\; (1\le j\le s-1),\\
 \ssc{u}{1}\cdots \ssc{u}{s}\prod_{h\in \cJ} m_h \le n^{\gamma}, \qquad 
\ssc{v}{j} &\le n^\sigma \;\; (1\le j\le r-1), 
\end{split}
\end{equation}
and $\cI(\bu,\bv,\bmm)$ is the indicator function of the simultaneous conditions
$u_j>x^{\sigma/3}$ for $j<s$, $v_j>x^{\sigma/3}$ for $j<r$, $u_s\le x^{\sigma/3}$,
$v_r\le x^{\sigma/3}$, and $m_h > (x/2)^{\sigma}$ for $1\le h\le k$.

Since $d_{h,j}< n/x^\gamma\le n^{1-\gamma}$ for all $h,j$ and also $\ssc{u}{j},\ssc{v}{j}\le n^\sigma\le n^{1-\gamma}$ for all $j$ (using $\sigma \le \nu \le 1-\gamma$), we have that
\[
\bz:=\Bigl(\frac{\log{\ssc{u}{1}}}{\log{n}},\dots,\frac{\log{\ssc{u}{s}}}{\log{n}},\frac{\log{\ssc{v}{1}}}{\log{n}},\dots,\frac{\log{\ssc{v}{r}}}{\log{n}},\frac{\log{d_{1,1}}}{\log{n}},\dots,\frac{\log{d_{k,\ell}}}{\log{n}}\Bigr)\in [0,1-\gamma]^{s+r+k\ell}
\]
is a fragmentation of the vector $\by'=(\by,\frac{\log u}{\log{n}},\frac{\log v}{\log{n}})$. Moreover, since $\by\in\cT'$ and $\frac{\log \ssc{u}{i}}{\log{n}},\frac{\log \ssc{v}{i}}{\log{n}}\le \sigma \le \nu \le 1-\gamma$ with $|\by|=1-\frac{\log u}{\log{n}}-\frac{\log v}{\log{n}}$ we have $\by'\notin \cC(\cR)$, and so $\bz\notin \cR$. Since all components of $\bz$ are bounded by $1-\gamma$, this implies there is a subsum of $\bz$ which lies in the interval $[\theta,\theta+\nu]$, and so certainly lies in the slightly larger interval
\[
\cK_n:=\Bigl[\theta\frac{\log (x/2)}{\log{n}},(\theta+\nu)\frac{\log{x}}{\log{n}}\Bigr].
\]
Write $\bz = (z_1,\ldots,z_{k\ell+s+r})$.
  By inclusion-exclusion on the set of subsums of $\bz$ that lie in $\cK_n$, it suffices to show for any 
  non-empty collection $\cE$ of non-empty subsets $E\subseteq[k\ell+s+r]$, we have
  \[
\sum_{\substack{n=\ssc{u}{1}\cdots \ssc{u}{s} \ssc{v}{1}\cdots \ssc{v}{r}\ssc{m}{1}\cdots \ssc{m}{k}\sim x\\ 
 \ssc{m}{h}=\prod_{j=1}^\ell d_{h,j}\; \forall h\le k \\ \eqref{eq:collection}\\
 d_{h,j} < n/x^{\gamma}\,\forall h,j \\
  |\ssc{\bz}{E}| \in \cK_n\,\forall E\in \cE
 }}
 w_n \cI(\bu,\bv,\mathbf{m}) \sprod{1\le h\le k \\ 1\le j \le \ell} \alpha_{h,j}(d_{h,j})\prod_{j=1}^{s}\beta_j(\ssc{u}{j}) \prod_{j=1}^{r}\beta'_j(\ssc{v}{j})
 \ll \frac{x}{(\log{x})^{A'}}.
\]
The constraint $d_{h,j} < n/x^{\gamma}$ is equivalent to 
$\lfloor x^\gamma\rfloor + 1/2 < n/d_{h,j}$. Therefore we can use Lemma \ref{lem:Integration} to remove each such constraint at the cost of twisting our sum by
a factor $(n/d_{h,j})^{it}=\prod_{(h',j')\ne (h,j)} d_{h',j'}^{it}$ with $t\in \RR$ (which can be absorbed into the 1-bounded functions), an additional factor $\log x$, and a negligible $O(x^{-90})$ error term.  After applying this for each $h,j$, we see that
it suffices to prove, for any 1-bounded functions $\alpha_{h,j},\beta_j,\beta_j'$, that
 \begin{equation}\label{eq:reduction-before-encoding-polytope}
\ssum{n=\ssc{u}{1}\cdots \ssc{u}{s}\ssc{v}{1}\cdots \ssc{v}{r} 
\ssc{m}{1}\cdots \ssc{m}{k}\sim x\\
 \ssc{m}{h}=\prod_{j=1}^\ell d_{h,j}\; \forall h\le k \\ \eqref{eq:collection}\\
   |\ssc{\bz}{E}| \in \cK_n\,\forall E\in \cE
 }
 w_n \cI(\bu,\bv,\mathbf{m}) \sprod{1\le h\le k \\ 1\le j \le \ell} \alpha_{h,j}(d_{h,j})\prod_{j=1}^{s}\beta_j(\ssc{u}{j}) \prod_{j=1}^{r}\beta'_j(\ssc{v}{j})
 \ll \frac{x}{(\log{x})^{A'+k\ell}}.
\end{equation}

We next remove the conditions in \eqref{eq:collection} from the summation in 
\eqref{eq:reduction-before-encoding-polytope}.
There is one special 
case we must dispense with first, and that is the case where one of the linear
constrains defining $\cT'$ is $\ssc{x}{1}+\cdots+\ssc{x}{k}\ge 1$.
Since $\ssc{\by}{\cJ} \in \cG$, and in any nontrivial $\cG$ the sum of coordinates is $\le \gamma<1$, $\cG$ must be the trivial polytope and $\cJ=\emptyset$.
Moreover, $\by\in \cT'$ implies that $m_1\cdots m_k=n$, which can only happen
if $r=s=1$ and $\ssc{u}{1}=\ssc{v}{1}=1$.  Thus, \eqref{eq:collection}
is either never satisfied or equivalent to $r=s=1$ and  $\ssc{u}{1}=\ssc{v}{1}=1$.  We encode the latter
as $\ssc{u}{1} < \frac32$ and $\ssc{v}{1} < \frac32$, and then two applications
of Lemma \ref{lem:Integration} encodes these conditions with extra factors
which are absorbed into the functions $\beta_1(\ssc{u}{1})$ and $\xi_1(\ssc{v}{1})$.
In this special case, we see that it suffices to prove a version of \eqref{eq:reduction-before-encoding-polytope} with the condition \eqref{eq:collection} removed,
and an extra factor $(\log x)^{-2}$ on the right side.

When  $\ssc{x}{1}+\cdots+\ssc{x}{k}\ge 1$ is not one of the linear constraints defining $\cT'$, we use Lemma \ref{lem:polytope-encode}.  Let $N=k\ell+r+s$ and relabel the
variables $\ssc{u}{1},\ldots,\ssc{u}{s},\ssc{v}{1},\ldots,\ssc{v}{r}$ and all of the $d_{h,j}$ as
$\ssc{n}{1},\ldots,\ssc{n}{N}$:
\begin{align*}
\ssc{n}{i} &= \ssc{u}{i} \qquad (1\le i\le s),\\
\ssc{n}{i+s} &= \ssc{v}{i} \qquad (1\le i\le r), \\
\ssc{n}{r+s+(h-1)\ell+j} &= d_{h,j} \qquad (1\le h\le k,1\le j\le \ell).
\end{align*}
For $1\le h\le k$, let $\cM_h = \{r+s+(h-1)\ell+1,\ldots,r+s+h\ell\}$
(this corresponds to the variables $\ssc{n}{j}$ whose product is $\ssc{m}{h}$),
for $k+1\le h\le k+s-1$ 
 let $\cM_h = \{h-k\}$ (this corresponds to the variable $\ssc{u}{h-k}$) and for $k+s \le h\le k+s+r-2$ let $\cM_h = \{ h-k+1 \}$ (this corresponds to the variable $\ssc{v}{h-k-s+1}$).  
 In this way, for each $h$, $\prod_{j\in \cM_h} \ssc{n}{j} > x^{\sigma/3}$
 whenever $\cI(\bn)=1$, where $\cI(\bn)=\cI(\bu,\bv,\mathbf{m})$.

 Also, for
$\bn=(\ssc{n}{1},\ldots,\ssc{n}{N})$ let
\[
\zhu(\bn) := \sprod{1\le h\le k \\ 1\le j \le \ell} \alpha_{h,j}(d_{h,j})\prod_{j=1}^{s}\beta_j(\ssc{u}{j}) \prod_{j=1}^{r}\beta'_j(\ssc{v}{j})
\]
so that the summand in \eqref{eq:reduction-before-encoding-polytope} equals
$w_n \cI(\bn) \zhu(\bn).$

We claim that for some bounded $D$,
the conditions \eqref{eq:collection} may be written as
the intersection of conditions, each of the form
\[
n_1^{c_1}\cdots n_{N}^{c_N} \le 1,
\]
where the inequality may be strict, 
every $c_i \in [-D,-1]\cup \{0\} \cup [1,D]$, and
there is some set $\cM_{h}$ so that the numbers $c_{i}$ for $j\in \cM_{h}$
are equal and nonzero.
Indeed, each condition $u_j\le n^{\sigma}$, for $j<s$, is equivalent to
$n_j^{1-\sigma}\prod_{i\ne h} n_i^{-\sigma}\le 1$. 
Raising both sides to a bounded power yields all exponents $\le -1$ or $\ge 1$,
then we may take $h=k+j$.
The same analysis holds for the conditions $v_j\le n^{\sigma}$, for $j<r$.
The condition $\ssc{u}{1}\cdots \ssc{u}{s} \prod_{h\in \cJ} m_h \le n^{\gamma}$
is equivalent to
\[
\prod_{j\in \cL} n_j^{1-\gamma} \prod_{j\in [N]\setminus \cL} n_j^{-\gamma}\le 1 
\]
where $\cL=\{1,\ldots,s\} \cup (\cup_{h\in \cJ} \cM_h)$.
If $s\ge 2$ then $\cL$ contains $\cM_{k+1}$ (corresponding to $\ssc{u}{1}$), if $r\ge 2$ then $[N]\setminus \cL$ contains $\cM_{k+s}$ (corresponding to $\ssc{v}{1}$), and if
$k>0$ then either $\cL$ or $[N]\setminus \cL$ contains $\cM_1$.
The case $r=s=1$ and $k=0$ is not possible since then $n=\ssc{u}{1}\ssc{v}{1}\le x^{2\sigma/3}<x^{2/3}$.
Now consider one of the linear inequalities defining $\cT'$, in the case $k\ge 1$,
which we write as
\begin{equation}\label{eq:U'-condition}
\ssc{e}{1} \ssc{x}{1} + \cdots + \ssc{e}{k} \ssc{x}{k} \le \ssc{e}{0},
\end{equation}
where at least one of $\ssc{e}{1},\ldots,\ssc{e}{k}$ is nonzero.
This corresponds to $m_1^{e_1}\cdots m_k^{e_k} \le n^{e_0}$, equivalently
\[
\big( \ssc{n}{1}\cdots \ssc{n}{r+s} \big)^{-\ssc{e}{0}} \prod_{h=1}^k \bigg( \prod_{j\in \cM_h} \ssc{n}{j} \bigg)^{\ssc{e}{h}-\ssc{e}{0}} \le 1.
\]
If $\ssc{e}{h}\ne \ssc{e}{0}$ for some $h$, then raising both sides to some bounded power
gives the desired conclusion.  
Now suppose that $\ssc{e}{h}=\ssc{e}{0}\ne 0$ for all $h$.
If $\ssc{e}{0}\ge 0$ then the above is always satisfied and this condition may be omitted from $\cT'$.  If $\ssc{e}{0}<0$ then the condition in \eqref{eq:U'-condition} is
equivalent to $\ssc{x}{1}+\cdots+\ssc{x}{k}\ge 1$, which we assumed is not one of
the constrains defining $\cT'$.
The analysis for a version of \eqref{eq:U'-condition} with strict inequality
is similar, but now if  $\ssc{x}{1}+\cdots+\ssc{x}{k} > 1$ is one such constraint,
it corresponds to $\ssc{m}{1}\cdots \ssc{m}{k}>n$, which is impossible
and hence this constraint may be omitted from $\cT'$.
This completes the proof of the claim.

By the claim, we may use Lemma \ref{lem:polytope-encode} to effectively remove
the conditions \eqref{eq:collection} from the sum on the left side of \eqref{eq:reduction-before-encoding-polytope}. We take
\[
\phi_\bn = \cI(\bn) \zhu(\bn) \one\big(|\ssc{\bz}{E}|\in \cK_n\, \forall E\in \cE\big),
\]
and the number $\ell'$ of constraints we remove is equal to $r+s-1$
plus the number of linear constraints defining $\cT'$ (which is $O_g(1)$). It thus suffices to 
prove a version of \eqref{eq:reduction-before-encoding-polytope}
with the conditions \eqref{eq:collection} removed
and $A'$ replaced by $A'+\ell'+2$.

If we order the components $z_j$ of $\bz$ so that $z_j = \frac{\log n_j}{\log n}$
for all $j$, then each condition $E\in \cE$ is equivalent to $\prod_{e\in E} n_e \in ((x/2)^\theta,x^{\theta+\nu}]$.  It therefore
suffice to prove that if $A'$ is sufficiently large (in terms of $P$, $A$, and $\sigma$), for any collection $\cE$ of nonempty subsets of $[N]$,
1-bounded functions $\ssc{\zhu}{1},\ldots,\ssc{\zhu}{N}$, 
we have
\[
\ssum{n=\ssc{n}{1}\cdots \ssc{n}{N}\sim x\\
\prod_{e\in E} n_e \in ((x/2)^\theta,x^{\theta+\nu}]\; \forall E\in \cE}
 w_n \cI(\bn) \ssc{\zhu}{1}(\ssc{n}{1})\cdots \ssc{\zhu}{N}(\ssc{n}{N})
 \ll \frac{x}{(\log{x})^{A'}}.
\]
Each of the $r+s+k$ constraints in $\cI(\bn)$ states that a particular product
$n(J):= \prod_{j\in J} \ssc{n}{j}$ is either $>y$ or $\le y$ for some $y$ which depends
only on $x,\sigma$.  Any condition $n(J) >y$ is equivalent to $n(J) > \fl{y}+\frac12$
and the condition $n(J) \le y$ is equivalent to $n(J) \le \fl{y}+\frac12$.
Thus, applying Lemma \ref{lem:Integration} to each constraint, we introduce
a factor $(\log x)^{r+s+k}$ and additional factors which may be absorbed into
the functions $\ssc{\zhu}{j}$.  In addition, if we fix one $E\in \cE$,
we may encode all of the conditions
\[
\prod_{e\in E'} n_e \in ((x/2)^\theta,x^{\theta+\nu}] \qquad (E'\in \cE, E'\ne E)
\]
by similar applications of Lemma \ref{lem:Integration}.  It then suffices to
prove that for some sufficiently large $A'$  (in terms of $P$, $A$, and $\sigma$),
for any nonempty $E\subseteq [N]$, 1-bounded functions $\ssc{\zhu}{1},\ldots,\ssc{\zhu}{N}$, we have
\begin{equation}\label{eq:sieve-final}
\ssum{n=\ssc{n}{1}\cdots \ssc{n}{N}\sim x\\
\prod_{e\in E} n_e \in ((x/2)^\theta,x^\theta+\nu]}
 w_n \ssc{\zhu}{1}(\ssc{n}{1})\cdots \ssc{\zhu}{N}(\ssc{n}{N})
 \ll\, \frac{x}{(\log{x})^{A'}}.
\end{equation}
Let
\[
Y_1(n') := \ssum{n'=\prod_{e\in E} n_e} \; \prod_{e\in E}
\ssc{\zhu}{e}(n_e), \qquad
Y_2(n'') := \ssum{n''=\prod_{e\notin E} n_e} \; \prod_{e\notin E}
\ssc{\zhu}{e}(n_e).
\]
Thus, the left side of \eqref{eq:sieve-final} equals
\[
\ssum{n= n'n'' \sim x \\ n'\in  ((x/2)^\theta,x^{\theta+\nu}]}
Y_1(n') Y_2(n'') w_{n'n''}.
\]
For any 1-bounded functions $\ssc{\zhu}{j}$, Lemma \ref{lem:tauk-tau} implies that
$|Y_1(n')| \le \ssc{\tau}{|E|}(n') \le \tau(n')^{|E|}$
and
$|Y_2(n'')| \le \ssc{\tau}{N-|E|}(n'') \le \tau(n'')^{N-|E|}$.
Inequality \eqref{eq:sieve-final} then follows from the Type II bound \eqref{eq:TypeII}
if $B$ is sufficiently large. This completes the proof of Proposition \ref{prop:general} (and hence Proposition \ref{prop:sieve-usingIandII}).
\end{proof}

\bigskip


%
%
%
{\Large \section{Analysis of two special families}\label{sec:families}}
%
%
%

In this section we prove Theorems \ref{thm:continuity-1/2}, \ref{thm:continuity-less-1/2}
and \ref{thm:theta=0 gamma=1/2}, which are about the special family with 
$\theta=0$ and either $\gamma=\frac12$ or $\gamma$ just below $\frac12$,
and prove Theorem \ref{thm:1-parm theta family} about another 1-parameter
family with $\gamma=1-\theta$ and $\nu=1-3\theta$.
We are able to prove exact formulas for $C^{\pm}(P)$, $\CB^{\pm}(P;\varrho)$
and $\lim_{\eps\to 0^+} \CB^{\pm}(P_\eps;\varrho)$
in some ranges, in particular by applying Theorem \ref{thm:duality} (a).
We begin with the latter family as the details are simpler and give a flavor
of the main ideas.
\medskip

%
%
\subsection{The special family $\gamma=1-\theta$, $\nu=1-3\theta$}
%
%

The proof of Theorem \ref{thm:1-parm theta family} uses the theory
of integral equations.  The integral equation
\be\label{Volterra}
\int_a^t H(t,s) u(s)\, ds = w(t),
\ee
with given functions $w,H$ and unknown function $u$,
is called a Volterra integral equation of the first kind.
The following is Theorem 2.2.1 in \cite{Brunner}.

\begin{thm}\label{thm:Volterra}
Let $I=[a,b]$ be a finite interval,
$D=\{ (s,t)\in \RR^2 : a \le s\le t\le b \}$ and $m\ge 0$.
Suppose further that
\begin{itemize}
\item[(a)] $w\in C^{m+1}(I)$; $w(a)=0$;
\item[(b)] $H \in C^{m+1}(D)$; $H(t,t)\ne 0$ for $t\in I$.
\end{itemize}
Then the equation \eqref{Volterra} has a unique solution
$u : I \to \RR$, and moreover $u\in C^m(I)$.
\end{thm}

\medskip

Recall the definitions of $\cH$ and $\cG_1$ in
\eqref{eq:VHZGdefs}.

\begin{proof}[Proof of Theorem \ref{thm:1-parm theta family}]
For $P=P_\theta=(1-\theta,\theta,1-3\theta)$ and $\frac14 < \theta \le \frac27$,
$\cR$ consists of all vectors with components in $(0,\theta)$, sum of components 1
 and with no subsum in $[\theta,1-2\theta]$. If $\bx\in \cH$ then all components are $\ge 1-3\theta$ and $\bx\ne (1)$, and so
 all components $\ssc{x}i$ lie in $J_1 \cup J_2 \cup J_3$, where
 $J_1=[1-3\theta,\theta)$, $J_2=(1-2\theta,2\theta)$ and $J_3=(1-\theta,3\theta)$. Moreover, all subsums of the $\ssc{x}i$ lie in $\{0\}\cup J_1\cup J_2\cup J_3\cup\{1\}$.
 As $\theta \le \frac27$, the sum of any two components from $J_1$ must lie in $J_2$, and so $\dim \bx \le 4$. Moreover, any $\bx\in \cH$ has one of four forms:
 \begin{itemize}
 \item[(a)] one component in $J_1$, one component in $J_3$;
 \item[(b)] two components in $J_2$;
 \item[(c)] two components in $J_1$ and one component in $J_2$;
 \item[(d)] four components in $J_1$.
 \end{itemize}
In fact, all vectors of these forms (a)--(d) lie in $\cH$.
This follows from the fact that for all $\beta_1,\ldots,\beta_4\in J_1$
 with sum 1, $\bbeta\in \cR$ since $\beta_i+\beta_j < 2\theta$ for all $i,j$.
 Thus, for $\bx$ of one of the forms (a)--(d) above, replacing each
 component $\alpha \in J_2$ with two copies of $\alpha/2$, and replacing
 any component $\alpha \in J_3$ with three copies of $\alpha/3$, produces
 a vector in $\cR$.
 
 We see that $\cG_1$ consists of vectors with components in $J_1\cup J_2$
 and sum of components $<2\theta$, since $\psi(\bx) \le 2\theta-|\bx|$.
 Define $g:\cG_1\to \RR$ by $g(\emptyvec)=1$ and
 \begin{align*}
 g(x) &= -\one(x\in J_1\cup J_2 : x\le 1/2), \\
 g(\ssc{x}{1},\ssc{x}{2}) &= \one(\ssc{x}{1}\in J_1, \ssc{x}{2}\in J_1, \ssc{x}{1}+\ssc{x}{2} < 1/2).
 \end{align*}
 Consider $\bx\in \cH$.  If $\bx$ has type (a), then $(\bone\star g)(\bx)=0$.
 If $\bx$ has type (b), then $(\bone\star g)(\bx)=0$ unless $\bx=(\frac12,\frac12)$
 in which case $(\bone\star g)(\bx)=-1$.  If $\bx$ has type (d), then $(\bone\star g)(\bx)=0$ unless $\ssc{x}i+\ssc{x}j=\frac12$ for some $i,j$ and in this case $(\bone\star g)(\bx)<0$.
 If $\bx$ has type (c) then  $(\bone\star g)(\bx)=0$ if the two components in $J_1$
 have sum $<\frac12$, and otherwise $(\bone\star g)(\bx)=-2$.
In all cases, $(\bone\star g)(\bx)\le 0$, and aside from a set of measure zero,
 $(1\star g)(\bx)=0$ unless $\bx$ has type (c) with the two components in $J_1$
 having sum at least $\frac12$.
 
We next construct a function $f$ by first
finding a function $f_{4,0}(\bx)$ 
and defining $f$ on the remainder of $\cH \cup \{1\}$ via \eqref{fsl}.
In light of Theorem \ref{thm:duality} (a), we need $f_{2,1}(\beta_1,\beta_2,\alpha)=-1$
when $\beta_1+\beta_2 \ge \frac12$ and $\beta_1,\beta_2\in J_1$ and $f(\bx)\ge -1$ for all $\bx\in\cH\cup\{1\}$.

\medskip

\textbf{Claim.}
There is a non-negative function $f_{4,0}$ on $\cH \cap \RR^4$ so that
when $\beta_1,\beta_2\in J_1$, $\alpha\in J_2$ and $(\beta_1,\beta_2,\alpha)\in \cH$,
we have
$f_{2,1}(\b_1,\b_2,\a)=-1$ when $\b_1+\b_2 \ge 1/2$ and
$f_{2,1}(\b_1,\b_2,\a)\ge -1$ otherwise.

\medskip

For now we assume the claim, and construct $f$ assuming it. By Lemma \ref{lem:tc} (or direct calculation from the definition \eqref{Linnik-fcn}), we have
\begin{align*}
\cyrL_{\theta}(\ssc{x}{1},\ssc{x}{2}) &= -1 \qquad (\ssc{x}{1}\in J_1, \ssc{x}{2}\in J_1, \ssc{x}{1}+\ssc{x}{2}\in J_2),\\
\cyrL_\theta(\ssc{x}{1},\ssc{x}{2},\ssc{x}{3}) &= 2 \qquad (\ssc{x}{1}\in J_1, \ssc{x}{2}\in J_1, \ssc{x}{3}\in J_1, \ssc{x}{1}+\ssc{x}{2}+\ssc{x}{3}\in J_3).
\end{align*}
Since $f_{4,0}$ is non-negative, plugging these into \eqref{fsl} 
implies that $f_{0,2}(\a_1,\a_2) \ge 0$ and $f_{1,1}(\b,\a)\ge 0$.  The Claim
then implies that $f(\bx) \ge -1$ for all $\bx\in \cH\cup \{1\}$.
By Theorem \ref{thm:duality} (a), 
\[
C^-(P) = 1+f(1) = 1 + \mint{\cH \cap \RR^3} \frac{(1\star g)(\bx)}{\ssc{x}{1}\ssc{x}{2}\ssc{x}{3}}d\bx = 
1 - 2 \mint{1-3\theta \le \beta_1 \le \beta_2 \le \theta \\ \beta_1+\beta_2\ge 1/2}
\frac{d\beta_1 d\beta_2}{\beta_1 \beta_2(1-\beta_1-\beta_2)}.
\]

\textbf{Proof of the Claim.}
Suppose $\frac14 < \theta \le \frac27$, $\b_1,\b_2\in J_1$, $\alpha\in J_2$, with
$(\beta_1,\beta_2,\alpha)\in \cH$.  By \eqref{fsl},
\begin{align*}
f_{2,1}(\b_1,\b_2,\alpha) &= -\a \int_{\a-\theta}^{\a/2} \; \frac{f_{4,0}(\b_1,\b_2,\b_3,\a-\b_3)}{\b_3(\a-\b_3)}\, d\b_3 \\
&= -\int_{\a-\theta}^{\a/2} \(\frac{1}{u}+ \frac{1}{\a-u}\) f_{4,0}(\b_1,\b_2,u,\a-u)\, du.
\end{align*}
For $\bbeta\in \cH \cap \RR^4$, we take
\[
 f_{4,0}(\bbeta) = \begin{cases}
K, & \text{ if } 1/2-\theta \le \b_1,\beta_2,\beta_3,\beta_4 \le \theta, \\
h(u), &
\text{ if some } \b_i= u < 1/2-\theta,
\end{cases}
\]
where $K$ is a constant to be determined, and
$h: [1-3\theta,1/2-\theta] \to \RR$  is a smooth function to be determined.
This $f_{4,0}$ is well-defined and symmetric, as there can be at most one
$\b_i < 1/2-\theta$ because $\b_i+\b_j\in J_2$ 
for any $i\ne j$.

Firstly, assume that $\alpha \le \frac12 \le \beta_1+\beta_2$.
Then $\b_1,\b_2 \ge 1/2-\theta$, thus $f_{4,0}(\bbeta)=K$ unless $\b_3 < 1/2-\theta$.
The desired equation $f_{2,1}(\b_1,\b_2,\a) = -1$ is equivalent to
\be\label{eq:h-eq}
1 =  \int_{\a-\theta}^{1/2-\theta} \(\frac{1}{u} +
\frac{1}{\a-u}\) h(u)\, du + K \log\pfrac{\a-1/2+\theta}{1/2-\theta}.
\ee
Setting $\a=1/2$, we see that
\be\label{eq:K-def}
K = \frac{1}{\log\big(\frac{\theta}{1/2-\theta}\big)}.
\ee
Equation \eqref{eq:h-eq} is a Volterra integral equation of the first kind.
By Theorem \ref{thm:Volterra}, there is a unique
solution $h\in C^\infty([1-3\theta,1/2-\theta])$.
Differentiating \eqref{eq:h-eq} with respect to $\a$ gives
\[
h(\a-\theta) \(\frac{1}{\a-\theta} + \frac{1}{\theta}\) + \int_{\a-\theta}^{1/2-\theta} \frac{h(v)}{(\a-v)^2}\, dv =  \frac{K}{\a+\theta-1/2}.
\]
Set $u=\a-\theta$ and solve for $h(u)$.  This gives
\be\label{th27-recurs}
h(u) = \frac{\theta u}{u+\theta} \left[ \frac{K}{u+2\theta-1/2} - 
\int_u^{1/2-\theta} \frac{h(v)}{(u+\theta-v)^2}\, dv \right].
\ee
Now let
\[
C:= \min_{1-3\theta\le u\le 1/2-\theta} h(u), \qquad
D:=\max_{1-3\theta \le u\le 1/2-\theta} h(u).
\]
Setting $u=1/2-\theta$ we see that $h(1/2-\theta)=(1-2\theta)K$, and in particular
$D > 0$.
For all $1/4<\theta\le 2/7$ and $1-3\theta\le u\le 1/2-\theta$ we have
\begin{align*}
\frac37 &\le \frac{\theta u}{(u+\theta)(u+2\theta-1/2)} \le  \frac12, \\
 0 &\le \frac{\theta u}{u+\theta} \(\frac{1}{u+2\theta-1/2}-\frac{1}{\theta}\) \le \frac19.
\end{align*}
Thus, from \eqref{th27-recurs} we have
\[
C \ge \frac37 K - \frac19 D, \qquad D \le \frac12K + \frac19 \max(0,-C).
\]
If $C<0$ then 
\[
C \ge \frac37K - \frac19 \(\frac12 K - \frac19 C \) = \frac{47}{126} K + \frac{C}{81},
\]
a contradiction.  Thus, $C\ge 0$ and hence $D\le K/2$.  That is,
\[
0 \le h(u) \le K/2 \qquad (1-3\theta \le u \le 1/2-\theta).
\]
In particular, $f_{4,0}(\bbeta)\ge 0$ for all $\bbeta\in \cH \cap \RR^4$, as required.
Now assume that $\a > 1/2$.  Since $f_{4,0}(\bbeta)\le K$ for all $\bbeta$,
\begin{align*}
f_{2,1}(\b_1,\b_2,\a) &\ge  -K \int_{\a-\theta}^{\a/2}
\(\frac{1}{u} + \frac1{\a-u}\) \, du 
= -K \log \pfrac{\theta}{\a-\theta}  > -1
\end{align*}
on account of \eqref{eq:K-def}.  This completes the proof of the Claim.
\end{proof}

\bigskip

%
\subsection{The special family with $\theta=0$ and $\gamma$ near $1/2$.}
%

\begin{proof}[Proof of Theorem \ref{thm:continuity-1/2}]
The claim $C^-(P_\eps)=0$ follows from Theorem \ref{thm:gamma<1/2}.

When $P=(\frac12,0,\nu)$ with $\frac13 \le \nu < \frac12$,
the claim $C^+(P_\eps)=1+O(\eps)$ follows from the special case $\nu=\frac13$.
Let $\eps$ be very small, $P_\eps=(\frac12-\eps,\eps,\frac13-2\eps)$.
Adopt the notation from Section \ref{sec:sieving}.
We see that $\cR(P_\eps)$ consists of two types of vectors.  
One type has two components in $(\frac12-2\eps,\frac12+\eps)$ with the remaining components having total $<\eps$, and the
other type has three component in $(\frac13-\eps,\frac13+2\eps)$ and the remaining components
having sum $<\eps$.  It follows that $\cH(P_\eps)$ consists of three types of vectors,
one type has two components $\frac12+O(\eps)$, a second type  with two components, one $\frac13+O(\eps)$ and the other $\frac23+O(\eps)$, and a third type with three 
components all $\frac13+O(\eps)$.  Define the function $g\in \sG_1$ by
$g(\emptyvec)=1$ and $g(\bx)=0$ otherwise.  Then $(\bone \star g)(\bx)=1$
identically on $\cH$.  By Theorem \ref{thm: Main sieving} (b),
\[
C^+(P_\eps) \le 1 + \sum_{k=2}^3 \;\; \mint{\bx\in \cH \cap \RR^k \\ \ssc{x}{1}\le \cdots \le \ssc{x}{k}}
\frac{d\bx}{\ssc{x}{1}\cdots \ssc{x}{k}} = 1 + O(\eps).\qedhere
\]
\end{proof}

\medskip

When $\gamma=\frac12$ and $\theta=0$, we need a version of Theorem 
\ref{thm: Main sieving} (a) which applies to $\CB^{-}(P;\varrho)$,
and which has a weaker hypothesis on $g$.

\begin{thm}\label{thm:CB-lower}
Let $P=(\gamma,\theta,\nu)\in \cQ_0$ with $\cR=\cR(P)$ nonempty.
Partition $\cH(P)$ into two sets $\cH_1,\cH_2$, each of which is a finite
union of convex polytopes.  Let $g\in \sG_1$ with $g(\emptyvec)=1$
and $(1\star g)(\bx)\le 0$ for $\bx\in \cH_1$.  Define $H$ by
\eqref{g-def} and \eqref{h-def}.  Fix $\varrho\ge 1$.  Then
\[
\CB^{-}(P;\varrho) \ge 1+\sum_k \mint{\cH_1 \cap \RR^k} 
\frac{(1\star g)(\bx)}{\ssc{x}{1}\cdots \ssc{x}{k}}\,d\bx + O_{g,\varrho,\nu}\bigg(\sum_k \mint{\cH_2\cap \RR^k} \frac{d\bx}{\ssc{x}{1}\cdots \ssc{x}{k}}\bigg).
\]
\end{thm}

\begin{proof}
Fix $\varpi\ge 1$ and let $B$ be sufficiently large in terms of $P,\varpi$,
and $x$ sufficiently large in terms of $P,\varpi,B$.
Let $\cN_1$ be the set of $n\in \cN$ with $P^-(n)\ge n^{\nu}$ and $\bv(n)\in \cH_1$
and let $\cN_2 = \cN \setminus \cN_1$. Let $\cP$ be the set of primes in $(x/2,x]$.
Suppose $((a_n),(b_n))\in \Psi(P;B,\varpi,x)$ with
\eqref{CBD} holding.  By hypothesis, we have
\begin{align*}
\sum_p w_p &\ge - \sum_p b_p + \sum_{n\in \cP \cup \cN_1} (w_n+b_n)H(n)\\
&=\sum_{n\in \cN_1} b_n H(n) + \sum_{n\in \cP \cup \cN} w_n H(n) - 
\sum_{n\in \cN_2} w_n H(n).
\end{align*}
By Lemma \ref{lem:h-sum-to-integral}, 
\[
\sum_{n\in \cN_1} b_n H(n) = \bigg(\sum_p b_p \bigg)
\Bigg[\sum_k \mint{\cH_1 \cap \RR^k} 
\frac{(1\star g)(\bx)}{\ssc{x}{1}\cdots \ssc{x}{k}}\,d\bx +  O_{g}\pfrac{1}{B} \Bigg] .
\]
Combining Proposition \ref{prop:sieve-usingIandII} with Lemma \ref{lem:htypeI}, we get
\[
\sum_{n\in \cP \cup \cN} w_n H(n) = \sum_n w_n H(n) - \sum_{n\not\in \cP \cup \cN} w_n H(n) \ll \frac{x}{\log^2 x}.
\]
Finally, for $n\in \cN_2$, $|w_n|\le \tau(n)^{\varrho} \ll_{\varrho,\nu} 1$, thus
by Lemma \ref{lem:h-onesign} (a) and Lemma \ref{lem:h-sum-to-integral},
\[
\sum_{n\in \cN_2} w_n H(n) \ll_{g,\varrho,\nu} \ssum{n\in \cN_2\\P^-(n)\ge n^{\nu}} 1
= \sum_{\bv(n)\in\cH_2} 1 \ll \frac{x}{\log x} \sum_k \mint{\cH_2\cap \RR^k} \frac{d\bx}{\ssc{x}{1}\cdots \ssc{x}{k}}.
\]
Combining these estimates, taking $B$ arbitrarily large and recalling \eqref{CBD}, the proof is 
complete.
\end{proof}

\begin{thm}\label{thm:CB-lower-theta=0}
Let $P=(\frac12,0,\nu)$ where $0<\nu \le \frac13$, and let $0\le \eps\le \nu/100$.
  Let $g\in \sG_1$ with $g(\emptyvec)=1$
and $(1\star g)(\bx)\le 0$ for $\bx\in \cH(P)$.  Define $H$ by
\eqref{g-def} and \eqref{h-def}.    Then, for any $\varrho\ge 1$,
\[
\CB^-(P_\eps;\varrho) \ge 1 + \sum_k \mint{\cH(P) \cap \RR^k} 
\frac{(1\star g)(\bx)}{\ssc{x}{1}\cdots \ssc{x}{k}}\,d\bx + O_{g,\varrho,\nu}(\eps).
\]
\end{thm}

We begin with an explicit description of $\cH(P)$.

\begin{lem}\label{lem:gamma=1/2-H}
Let $0<\nu \le \frac13$ and $P=(\frac12,0,\nu)$.  Then
\[
\cH(P) = \big\{\bx : \dim\, \bx \ge 2, |\bx|=1, x_i\in (\nu,1-2\nu)\cup(2\nu,1-\nu)\; \forall i \big\}. 
\]
\end{lem}

\begin{proof}
$\cR=\cR(P)$ is the set of vectors with components in $(\nu,\frac12)$
and sum 1.
Consider $\bx\in \cH(P)$ with $\dim(\bx)=h \ge 2$.
If $h\ge 3$, then all components of $\bx$ are $<1-2\nu$.  If $h=2$ then some component 
$x$ of $\bx$ is at least $\frac12$, and therefore equal to the sum of two or more
numbers in $(\nu,\frac12)$ (since $\cH(P)\subseteq \cC(\cR)$).  
In this case all other components are $<1-2\nu$.
Conversely, suppose that $|\bx|\ge 2$ and has all components in
$(\nu,1-2\nu)\cup (2\nu,1-\nu)$.  It suffices to show that every
component $x$ which is $\ge \frac12$ is the sum of numbers in $(\nu,\frac12)$.
If $x>2\nu$, $x=\frac12x+\frac12x$ has the required form.
If $x\le 2\nu$ then $x<1-2\nu$ as well, contradicting $x\ge \frac12$.
\end{proof}

\begin{proof}[Proof of Theorem \ref{thm:CB-lower-theta=0}]
Let $g_\eps(\bx) = g(\bx) \one(|\bx|\le \frac12-2\eps)$ and let
\begin{align*}
\cH_1 &= \{ \bx\in \cH(P): \text{ no component in } (\tfrac12-2\eps,\tfrac12] \},\\
\cH_2 &= \cH(P_\eps) \setminus \cH_1.
\end{align*}
The Type II range for $P_\eps$ is $[\eps,\nu-\eps]$ and hence
$\cR(P_\eps)$ is the set of vectors $\by$ of the form
$(\bxi,\by)$ with $|\bxi|+|\by|=1$, $|\bxi|<\eps$ and $\nu-\eps<y_i<\frac12+\eps$
for each component $y_i$ of $\by$.  Clearly $\by$ has at least two components.
Therefore, the elements of $\cZ(P_\eps)$ have the form $(\bxi,\bz)$, where
$|\bxi|<\eps$ and the components of $\bz$ are $>\nu-\eps$.  Consequently,
$\psi(\bx;P_\eps)\le \eps$ for any $\bx\in \cZ$ with all components $>\nu-\eps$.
  As $\cZ(P)\subseteq \cZ(P_\eps)$, $\cG_1(P_\eps)$ includes all vectors
  $\bz\in \cZ(P)$ with $|\bz| \le \frac12-2\eps$, and it follows that $g_\eps$
  is supported on $\cG_1(P_\eps)$.
Also, if $\bx\in \cH_1$ then $\bx\in \cH(P)$ and
hence  $(1\star g_\eps)(\bx) = (1\star g)(\bx)\le 0$.

We claim that any vector $\bx\in \cH_2$ has a component
equal to one of $\frac12+u$, $\nu+u$, $2\nu+u$, $1-\nu+u$
or $1-2\nu+u$ for some $|u|\le 6\eps$.  It follows that
\[
\sum_k \mint{\cH_2 \cap \RR^k} \frac{d\bx}{\ssc{x}{1}\cdots \ssc{x}{k}} \ll \eps.
\]
Therefore, by Theorem \ref{thm:CB-lower},
\begin{align*}
\CB^-(P_\eps;\varrho) &\ge 1 + \sum_k \mint{\cH_1\cap \RR^k} \frac{(1\star g_\eps)(\bx)}{\ssc{x}{1}\cdots \ssc{x}{k}}\, d\bx + O_{g,\varrho,\nu}(\eps)\\
&=1 + \sum_k \mint{\cH\cap \RR^k} \frac{(1\star g)(\bx)}{\ssc{x}{1}\cdots \ssc{x}{k}}\, d\bx + O_{g,\varrho,\nu}(\eps),
\end{align*}
as desired.

To prove the claim, suppose $\bx\in \cH_2$ has no such component of the claimed type. 
It suffices to show that $\bx\in \cH(P)$ in order
to reach a contradiction, for then $\bx \in \cH_1$.
Since $\bx\in \cH(P_\eps)$, all components of $\bx$
are $>\nu+6\eps$ and avoid $[\frac12-6\eps,\frac12+6\eps]$.
Thus, if $\nu<\frac14+3\eps$ then $\bx\in \cH(P)$ by Lemma \ref{lem:gamma=1/2-H}.
If $\nu \ge \frac14+3\eps$, then $\bx$ has 2 or 3
components.  If $\dim(\bx)=3$ then each component is at most $1-2(\nu+6\eps)<1-2\nu$.  
Lemma \ref{lem:gamma=1/2-H} again implies that $\bx\in \cH(P)$.  
Now suppose $\bx=(\ssc{x}1,\ssc{x}2)$ with $\ssc{x}1>\frac12+6\eps$. Since $\bx\in \cC(\cR(P_\eps))$ and by the above characterization of $\cR(P_\eps)$, $\ssc{x}1> 2(\nu-2\eps)$ and thus
$\ssc{x}1 > 2\nu+6\eps > 2\nu$, and it also follows that $\bx\in \cH(P)$.
\end{proof}

\begin{proof}[Proof of Theorem \ref{thm:theta=0 gamma=1/2} (a)]
When $\frac15 \le \nu < \frac13$ and $P=(\frac12,0,\nu)$, $\cR=\cR(P)$
consists of vectors with three or four components, each in $(\nu,1-\gamma)$
and with sum 1.  For $\bbeta \in \cR$, we define
\[
f(\beta_1,\beta_2,\beta_3)=-1,\qquad f(\beta_1,\beta_2,\beta_3,\beta_4)=0,
\]
and define $f(\bx)$ for the remainder of $\cC(\cR)=\cH \cup \{1\}$ using \eqref{fsl}.
In particular, writing $\a>\frac12 \ge \b_i$ for each $i$,
we have $f_{1,1}(\beta_1,\alpha) \ge 0$ since if $\alpha = \beta_2+\beta_3$
and $(\beta_1,\beta_2,\beta_3)\in \cR$ then $\cyrL_{1/2}(\beta_1,\beta_2)=-1$, and also $f_{2,1}(\beta_1,\beta_2,\alpha)=0$.
Also, for $(\beta_1,\beta_2,\beta_3)\in \cR$, $\cyrL_{1/2}(\beta_1,\beta_2,\beta_3)=2$ and thus
\[
f(1) =  -2 \;\mint{\nu \le \beta_1 \le \beta_2 \le \beta_3 \le \frac12 \\
\beta_1+\beta_2+\beta_3=1} \frac{d\bbeta}{\beta_1\beta_2\beta_3},
\]
where we used that $\cH \cap \RR^3 = \{ (\beta_1,\beta_2,\beta_3) : |\bbeta|=1, \nu<\beta_1 \le \beta_2 \le \beta_3 \}$ from
Lemma \ref{lem:gamma=1/2-H}.

We next construct a function $g$ on $\cG_1$ such that
$(\bone \star g)(\bx)\le 0$ for $\bx \in \cH$. With Theorem \ref{thm:duality} in mind, 
we wish to choose $g$ such that $(1\ast g)(\bx)=0$ whenever $\bx\in\cH$ is such that $f(\bx)>-1$.
By Lemma \ref{lem:gamma=1/2-H} and the fact that $\psi(\bx)=0$ for all $\bx\in \cZ$,
\[
\cG_1 = \emptyvec \cup \{ (x):x\in (\nu,\tfrac12]\cap (\nu,1-2\nu) \} \cup \{ (\ssc{x}{1},\ssc{x}{2}) : \nu<\ssc{x}{1}, \nu<\ssc{x}{2}, \ssc{x}{1}+\ssc{x}{2} \le \tfrac12 \}.
\]
Set $g(\emptyvec)=1$, $g(x)=-1$ for all $x\in (\nu,\tfrac12]\cap (\nu,1-2\nu)$, 
$g(\ssc{x}{1},\ssc{x}{2})=1$ if $(\ssc{x}{1},\ssc{x}{2})\in \cG_1$ and $\ssc{x}{1}+\ssc{x}{2}<\frac12$, and let $g(\bx)=0$ otherwise.
Now consider $(\bbeta,\balpha)\in \cH$ with $\beta_i < \frac12 \le \alpha_j$
for all $i,j$.  In particular, $|\bbeta|+|\balpha|=1$. 
By Lemma \ref{lem:gamma=1/2-H}, $\b_i < 1-2\nu$ and
$\alpha_j > 2\nu$ as well.
 We have the following convolution identities:
\begin{itemize}
\item $(\bone\star g)(\beta_1,\beta_2,\beta_3) = -2$.
\item $(\bone\star g)(\beta_1,\beta_2,\beta_3,\beta_4) = 0$ unless
$\beta_i+\beta_j=\frac12$ for some $i,j$, in which case $(\bone\star g)(\bbeta)<0$. 
\item $(\bone\star g)(\frac12,\frac12)=-1$.  (This only occurs if $\nu < \tfrac14$.)
\item $(\bone\star g)(\beta_1,\beta_2,\alpha_1)=0$ unless $\beta_1+\beta_2=\frac12=\alpha_1$, in which case $(\bone\star g)(\beta_1,\beta_2,\alpha_1)=-2$.
\item $(\bone\star g)(\beta_1,\alpha_1)=0$ since $\beta_1<\frac12<\alpha_1$.
\end{itemize}
It follows that $(\bone\star g)(\bbeta,\balpha) \le 0$ in all cases.
In addition, $(\bone\star g)(\bx)=0$ whenever $f(\bx)>-1$ ,
aside from a set of $\bx$ of measure zero.  Thus, by Theorem \ref{thm:duality} (a),
$C^-(P)=1+f(1)$.  Moreover, by Theorem \ref{thm:constructions} (b), 
$\CB^-(P) \le 1+f(1)$, and Theorem \ref{thm:CB-lower-theta=0} (a) implies that
$\lim_{\eps\to 0^+} \CB^-(P_\eps) \ge 1+f(1)$, where we used Theorem \ref{thm:duality} (a) again (the criterion for equality).  Thus, $\lim_{\eps\to 0^+} \CB^-(P_\eps)=1+f(1)$ as well. 
\end{proof}

%
%
%

\begin{proof}[Proof of Theorem \ref{thm:theta=0 gamma=1/2} (b)]
Suppose that $\nu=0.16623$.
 By Lemma \ref{lem:gamma=1/2-H}, $\cH=\cH(P)$ is the set of vectors with sum 1 and 
 all components in $(\nu,1-\nu)$; such vectors have at most 6 components.  As $\psi(\bx)=0$ for all $\bx\in \cZ$,
$\cG_1(P)$ is the set of vectors with components $>\nu$ and sum of components at most $\frac12$; there can be at most three components.   Set $g(\emptyvec)=1$ and
\begin{equation}\label{eq:theta0g}
\begin{split}
g(x) &= -\one(x\le \tfrac12), \\
g(\ssc{x}{1},\ssc{x}{2}) &= \one \big( \ssc{x}{1}+\ssc{x}{2}<\tfrac12 \big) - \one \big(\ssc{x}{1}+ \ssc{x}{2} \le \tfrac{1-\nu}2\big) = \one\big( \tfrac{1-\nu}{2}<\ssc{x}{1}+\ssc{x}{2}<\tfrac12  \big), \\
g(\ssc{x}{1},\ssc{x}{2},\ssc{x}{3}) &= - \one \big( \ssc{x}{1}+\ssc{x}{2}+\ssc{x}{3} \le \tfrac12 \big).
\end{split}
\end{equation}
Now consider $\bx=(\ssc{x}{1},\ldots,\ssc{x}{k})\in \cH$ with $\ssc{x}{1}\le \cdots \le \ssc{x}{k}$.  For brevity, define $\cH_k := \cH \cap \RR^k$.
 When $k=2$, $(1\star g)(\bx)=0$ if $\ssc{x}{1}<\ssc{x}{2}$ and $(1\star g)(\frac12,\frac12)=-1$.  
 When $k=3$, at most one coordinate is $>\frac12$, thus
 \begin{align*}
 (1\star g)(\bx) &= 1-2-\one(\ssc{x}{3}\le \tfrac12) + \one(\ssc{x}{1}+\ssc{x}{2}<\tfrac12)-\one(\ssc{x}{1}+\ssc{x}{2} \le \tfrac{1-\nu}2) \\
 &= -2\cdot \one(\ssc{x}{1}+\ssc{x}{2}\ge \tfrac12) - \one(\ssc{x}{1}+\ssc{x}{2} \le 
 \tfrac{1-\nu}2) \le 0 \qquad\qquad (\bx\in \cH_3).
 \end{align*}
 
When $k=4$, there are exactly three pairs of coordinates from $\bx$
 with sum less than $\frac12$, except for a set of measure zero (this occurs when $x_i+x_j=\frac12$ for some $i\ne j$) in which case there are fewer than three such pairs.  Also, it is not possible to have $\ssc{x}{1}+\ssc{x}{4} \le \frac{1-\nu}2$, for then
$\ssc{x}{2}\le \ssc{x}{3} \le \ssc{x}{4} \le \frac{1-3\nu}{2}$ and 
$1=\ssc{x}{1}+\cdots+\ssc{x}{4} \le \frac{3-7\nu}{2}<1$, a contradiction.
Therefore, since for every $i$ either $\ssc{x}{i}\le \frac12$ or $|\bx|-\ssc{x}{i} \le \frac12$ (with both occuring only if $i=4$ and $x_4=\tfrac12$),
\begin{align*}
(1\star g)(\bx) =
- \sum_{1\le i < j\le 3} \one \big(\ssc{x}{i}+\ssc{x}{j} \le \tfrac{1-\nu}2\big) - \one(\ssc{x}{4}=\tfrac12)
- \tfrac12 \sum_{1\le i<j\le 4} \one(\ssc{x}{i}+\ssc{x}{j}=\tfrac12) \quad  (\bx \in \cH_4).
\end{align*}
In particular, $(1\star g)(\bx)\le 0$ for $\bx\in \cH_4$.
Also, for a set $\cH_4'\subseteq \cH_4$ of full measure,
\begin{align*}
(1\star g)(\bx) = - \sum_{1\le i < j\le 3} \one \big(\ssc{x}{i}+\ssc{x}{j} \le \tfrac{1-\nu}2 \big) \qquad (\bx\in \cH_4').
\end{align*}

Now let $k=5$ and $\bx\in \cH_5$. We have that $\ssc{x}{i}<1-4\nu<1/2$ for all $i$, and if $A\subseteq[5]$ with $|A|=3$, then $|\ssc{\bx}{A}|\le 1/2$ is equivalent to $|\ssc{\bx}{{[5]\setminus A}}|\ge 1/2$. 
Therefore, we find
\begin{align*}
(1\star g)(\bx) &=  1 - 5 +\sum_{i<j} \one(\ssc{x}{i}+\ssc{x}{j}<\tfrac12) - \sum_{i<j} \one\big(\ssc{x}{i}+\ssc{x}{j} \le \tfrac{1-\nu}{2}\big)
-\sum_{i<j}\one(\ssc{x}{i}+\ssc{x}{j} \ge \tfrac12)\\
&= 6 - \ssum{i<j} \one(\ssc{x}{i}+\ssc{x}{j} \le \tfrac{1-\nu}2)
-2 \ssum{i<j} \one(\ssc{x}{i}+\ssc{x}{j}\ge \tfrac12)\qquad\qquad (\bx\in \cH_5).
\end{align*}
We have either $\ssc{x}{2}+\ssc{x}{5} \le \frac{1-\nu}2$ or
$\ssc{x}{3}+\ssc{x}{4} \le \frac{1-\nu}2$, for otherwise
$1-\nu<\ssc{x}{2}+\cdots+\ssc{x}{5}=1-\ssc{x}{1}$, a contradiction.
 It follows that there are at least 6 pairs $(i,j)$ with $i<j$
 and $\ssc{x}{i}+\ssc{x}{j} \le \frac{1-\nu}{2}$, hence
  $(1\star g)(\bx)\le 0$ when $k=5$.
  
Finally, when $k=6$, all pairs have sum $\le 1-4\nu < \frac{1-\nu}{2}$,
and there are exactly 10 triples $(\ssc{x}{i},\ssc{x}{j},\ssc{x}{k})$ with sum $\le \frac12$, except on a set of
measure zero where there are more than 10 such triples (this occurs when $|\ssc{\bx}{A}|=\tfrac12$ for some 3-tuple $A$).  Thus,
$(1\star g)(\bx)=1-6-10=-15$ except on a set of measure zero where $(1\star g)(\bx) < -15$.
  
Thus $(\bone\star g)(\bx) \le 0$ for all $\bx\in \cH$.

It follows from Theorem \ref{thm: Main sieving} (a)
and Theorem \ref{thm:CB-lower-theta=0} that $C^-(P)\ge C$ and $\CB^-(P_\eps;\varrho)\ge C-O(\eps)$, where
\[
C = 1 + I_3 + I_4 + I_5 + I_5'+I_6,
\]
\begin{align*}
I_3 &= -2\;\;\mint{\nu\le \ssc{x}1\le \ssc{x}2\le \ssc{x}3\le 1/2 \\ \ssc{x}1+\ssc{x}2+\ssc{x}3=1} \frac{d\bx}{\ssc{x}{1}\ssc{x}{2}\ssc{x}{3}} -  \;\; \mint{\nu\le \ssc{x}1\le \ssc{x}2\le \ssc{x}3 \\ \ssc{x}1+\ssc{x}2+\ssc{x}3=1 \\ \ssc{x}1+\ssc{x}2\le (1-\nu)/2} \frac{d\bx}{\ssc{x}{1}\ssc{x}{2}\ssc{x}{3}},\\
I_4 &=- \sum_{1\le i<j\le 3} \;\;\mint{\nu\le \ssc{x}1\le \cdots\le \ssc{x}4 \\ \ssc{x}1+\ssc{x}2+\ssc{x}3+\ssc{x}4=1 \\ \ssc{x}i+\ssc{x}j\le (1-\nu)/2}
\frac{d\bx}{\ssc{x}{1}\ssc{x}{2}\ssc{x}{3}\ssc{x}{4}},
\end{align*}
and, setting $S=\{ (3,4),(1,5),(2,5),(3,5),(4,5) \}$,
\begin{align*}
I_5 &= \mint{\nu\le \ssc{x}1\le \cdots \le \ssc{x}5 \\ \ssc{x}1+\cdots+\ssc{x}5=1}
\frac{h(\bx)}{\ssc{x}{1}\cdots \ssc{x}{5}}\, d\bx, \quad h(\bx)=1- \sum_{(i,j)\in S} \one \big(\ssc{x}i+ \ssc{x}j\le \tfrac{1-\nu}2 \big), \\
I_5' &= -2 \sum_{1\le i<j\le 5} \;\;\; \mint{\nu\le \ssc{x}1\le \cdots \le \ssc{x}5 \\ \ssc{x}1+\cdots+\ssc{x}5=1
\\ \ssc{x}i+\ssc{x}j\ge 1/2} \frac{d\bx}{\ssc{x}{1}\cdots \ssc{x}{5}},\\
I_6&=-15\mint{\nu\le \ssc{x}1\le \cdots \le \ssc{x}6 \\ \ssc{x}1+\cdots+\ssc{x}6=1
} \frac{d\bx}{\ssc{x}{1}\cdots \ssc{x}{6}}.
\end{align*}
The integral $I_5'$ is tiny, since every $\ssc{x}{\ell}$ is close to $\frac16$
 for $\ell\not\in \{i,j\}$.  Also, the multiple integral is zero for pairs $(i,j)$ with
 $i\le 2$ and $j\le 4$, and  for all pairs $(i,j)$ we have
 $\nu \le \ssc{x}{i} \le \frac{x_i+x_j}{2} \le \frac{1-3\nu}{2}$. Thus,
\begin{align*}
|I_5'| &\le 10 \mint{\nu\le \ssc{x}1\le \cdots \le \ssc{x}5 \\ \ssc{x}1+\cdots+\ssc{x}5=1
\\ \ssc{x}4+\ssc{x}5\ge 1/2} \frac{d\bx}{\ssc{x}{1}\cdots \ssc{x}{5}} \\
&\le \frac{10}{\nu^4(1-4\nu)} \Big( \frac{1-3\nu}{2}-\nu \Big) \text{Vol} \big\{
\nu\le \ssc{x}1 \le \ssc{x}2 \le \ssc{x}3: \ssc{x}1+\ssc{x}2 + \ssc{x}3 \le \tfrac12 \big\}\\
&= \frac{5(1-5\nu)}{\nu^4(1-4\nu)}\cdot \frac{(1/2-3\nu)^3}{36} < 3\cdot 10^{-7}.
\end{align*}
Similarly, the integral $I_6$ is tiny as all $\ssc{x}{i}$
 are close to $\tfrac16$.  We have
\[
|I_6|\le \frac{15}{\nu^6}\mint{\nu\le \ssc{x}1\le \cdots \le \ssc{x}6 \\ \ssc{x}1+\cdots+\ssc{x}6=1
} d\bx=\frac{15(1-6\nu)^5}{5!6!\nu^6}<10^{-12}.
\]
The integrals $I_3,I_4,I_5$ were computed with Mathematica and produce\footnote{We find $I_3=-0.92205199\dots, I_4=-0.07714894\dots, I_5=-0.00079222\dots$ and so $I_3+I_4+I_5=-0.99999316\dots$.} $C\ge 0.000006$
when $\nu=0.16623$.  This proves Theorem \ref{thm:theta=0 gamma=1/2} (b). 
\end{proof}


\begin{proof}[Proof of Theorem \ref{thm:theta=0 gamma=1/2} (c)]
Let $\nu=0.16169$.
Motivated by Theorem \ref{thm:duality}, and recalling
the notation $f_{s,\ell}(\bx)$ from \eqref{fsl}, 
for $\bx=(\ssc{x}{1},\ldots,\ssc{x}{k})$ with $\nu<\ssc{x}{1}\le \cdots \le \ssc{x}{k}$ 
and $|\bx|=1$ we take
\begin{align*}
f_{3,0}(\bx) &= -1, \qquad f_{4,0}(\bx) = -1, \qquad f_{6,0}(\bx)=0, \\
f_{5,0}(\bx) &= \ssc{x}{1}\cdots \ssc{x}{5}\one(\ssc{x}{4}+\ssc{x}{5}<\tfrac12)
h(\ssc{x}{3}+\ssc{x}{4}+\ssc{x}{5}),
\end{align*}
for some decreasing, non-negative function $h$ on $[\frac35,1-2\nu]$, to be chosen later.
The main purpose of $f_{5,0}$ is to make $f_{2,1}(\b_1,\b_2,\a)$ close to $-1$
when $\b_1+\b_2 \le \frac{1-\nu}{2}$, which is the region where
$(1\star g)(\b_1,\b_2,\a)<0$.

Our goal is to choose $h$ so that $f_{s,\ell}(\bbeta,\a)\ge -1$ for all $s\ge 1$
and $\a\ge \frac12$, where $f_{s,\ell}$ is defined by \eqref{fsl}, and also
that $f(1)<-1$.
Then replacing $f(\bx)$ by $(-1/f(1)) f(\bx)$ for all $\bx$ gives
$f(1)=-1$ and $f(\bx)\ge -1$ for all $\bx$.  By Theorem \ref{thm:constructions} (b),
we obtain $C^-(\frac12,0,\nu)=0=\CB^-(\frac12,0,\nu)$.

 We have $f_{3,1}(\bbeta,\a)=0$, since in the support of
$f_{5,0}$, the sum of any two variables is $<\frac12$.  Thus, it remains to
show $f_{1,1}(1-\a,\a)\ge -1$ and $f_{2,1}(\b_1,\b_2,\a)\ge -1$.  The former is easy,
in fact it holds for any choice of $h$, which we now show.
In \eqref{fsl}, let $k\ge 2$ and $\ssc{u}1+\cdots+\ssc{u}k=\a$.  
Using \eqref{Linnik-fcn}, if $k=2$ then
$\cyrL_{1/2}(\bu)=-1$, if $k=3$ then $\cyrL_{1/2}(\bu)=2-\#\{(i,j):\,i<j,\ssc{u}i+\ssc{u}j< 1/2\} \le 2$, and if $k=4$ and $f_{5,0}(1-\a,\bu)\ne 0$ then $\cyrL_{1/2}(\bu)=3$.
Thus, $f_{5,0}(1-\a,\bu)\cyrL_{1/2}(\bu) \ge 0$ for all $\bu$.
By \eqref{fsl}, $f_{1,1}(1-\a,\a) \ge F_3(\a)+F_4(\a)$, where
\[
F_3(\a)=\a \mint{\nu \le \ssc{u}1\le \ssc{u}2 \le 1/2 \\ \ssc{u}1+\ssc{u}2=\a} \frac{1}{\ssc{u}1\ssc{u}2} = \log \bigg(
\frac{\a}{\max(\nu,\a-1/2)}-1\bigg).
\]
where we considered separately the cases $\a<\frac12+\nu$
and $\a\ge \frac12+\nu$, and
\begin{align*}
F_4(\a) \ge -2\a \mint{\nu\le \ssc{u}1\le \ssc{u}2\le \ssc{u}3 \\ \ssc{u}1+\ssc{u}2+\ssc{u}3=\a} \frac{1}{\ssc{u}1 \ssc{u}2 \ssc{u}3}
&\ge -\frac{2\a}{\nu^2(\a-2\nu)} \text{meas} \big\{ \nu\le \ssc{u}1\le \ssc{u}2\le \ssc{u}3, |\bu|=\a \big\}\\
&= -\frac{2\a(\a-3\nu)^2}{12 \nu^2 (\a-2\nu)}.
\end{align*}
Therefore, for any $\a\in [\frac12,1-\nu]$,
\[
f_{1,1}(1-\a,\a) \ge \log \bigg(
\frac{\a}{\max(\nu,\a-1/2)}-1\bigg) -\frac{\a(\a-3\nu)^2}{6 \nu^2 (\a-2\nu)}
\ge -0.911.
\]

We now bound $f_{2,1}(\b_1,\b_2,\a)$. By \eqref{fsl}, for $\a\ge \frac12$ we have 
$f_{2,1}(\b_1,\b_2,\a)=F_4(\a)-F_5(\b_1,\b_2)$, where
\begin{align*}
F_4(\a) = \a \mint{\nu\le \b_3\le \b_4\le 1/2 \\ \b_3+\b_4=\a} \frac{1}{\b_3 \b_4} =
\log\bigg( \frac{\a}{\max(\nu,\a-1/2)}-1\bigg)
\end{align*}
and
\begin{align*}
F_5(\b_1,\b_2) &=  \a \b_1 \b_2  \mint{\nu\le \b_3\le \b_4 \le \b_5  \\ \b_3+\b_4+\b_5=\a} 
h(\ssc{x}{3}+\ssc{x}{4}+\ssc{x}{5})\one(x_4+x_5<1/2)\, d\bbeta,
\end{align*}
where $(\ssc{x}{3},\ssc{x}{4},\ssc{x}{5})$ are the three largest components of $\bbeta$.
Now $\ssc{x}{3}+\ssc{x}{4}+\ssc{x}{5} \ge \max(\frac35,\a)$, $\b_1\b_2 \le (\frac{1-\a}{2})^2$. The condition $\ssc{x}{4}+\ssc{x}{5}<\frac12$ implies $\b_4+\b_5<\frac12$, which 
is equivalent to $\b_3 > \a-1/2$.
Therefore, $h(\ssc{x}{3}+\ssc{x}{4}+\ssc{x}{5})\le h\big(\max(\frac35,\a)\big)$ and
\begin{align*}
F_5(\b_1,\b_2) &\le \frac{\a(1-\a)^2}{4} h\big( \max(\tfrac35,\a) \big)
\cdot \text{Vol} \big\{ \max(\nu,\a-\tfrac12)\le \b_3\le \b_4\le \b_5: \a=\b_3+\b_4+\b_5 \big\} \\
&=\frac{\a(1-\a)^2}{48} h\big( \max(\tfrac35,\a) \big) \big(\a-3\max(\nu,\a-\tfrac12)\big)^2.
\end{align*}
Let $\ell(\a) = \frac{1}{48} \a(1-\a)^2 \big(\a-3\max(\nu,\a-\tfrac12)\big)^2$.
One can check that $(1+F_4(\a))/\ell(\a)$ is decreasing on $[\frac12,\frac12+\nu]$
and increasing on $[\frac12+\nu,1-2\nu]$.  Therefore, taking
\begin{equation}\label{eq:theta0-h}
h(u) = \begin{cases}
\frac{1+F_4(u)}{\ell(u)} & \text{ if } \tfrac12 \le u\le \tfrac12+\nu, \\
\frac{1+F_4(1/2+\nu)}{\ell(1/2+\nu)} & \text{ if } \tfrac12+\nu <u\le 1-2\nu,
\end{cases}
\end{equation}
we see that $h(u)$ is decreasing, and $h(u) \le (1+F_4(u))/\ell(u)$ for all
$u\ge \frac35$.  It follows immediately that $F_4(\a)-F_5(\b_1,\b_2) \ge -1$
for $\frac35 \le \a\le 1-2\nu$.  For $\frac12 \le \a < \frac35$, we have
\[
F_4(\a)-F_5(\b_1,\b_2) \ge F_4(\a)-\ell(\a)h(3/5) \ge F_4(\a)-\ell(\a)h(\a)=-1.
\]

Finally, with the choice of $h$ given by \eqref{eq:theta0-h}, we show that $f(1)<-1$.
By \eqref{Linnik-fcn}, if $\bx \in \cH_3$ with all components $<\frac12$ then $\cyrL_{1/2}(\bx)=2$, if $\bx\in\cH_4$ with all components $<\frac12$ and no pair of components has sum equal to $\frac12$ then $\cyrL_{1/2}(\bx)=0$,
and if $\bx\in \cH_5$ with every pair of components having sum less than $\frac12$
then $\cyrL_{1/2}(\bx) = -6$.  Thus, $f(1)= I_3 + I_5$, where
\[
I_3=-2 \mint{\nu\le \ssc{x}1\le \ssc{x}2\le \ssc{x}3 \le 1/2 \\ |\bx|=1} 
\frac{d\bx}{\ssc{x}1 \ssc{x}2 \ssc{x}3} = -2 \int_{\nu}^{1/3} \;\frac{\log\big( \frac{1-x}{\max(x,1/2-x)}-1\big)}{x(1-x)}\, dx
\]
and
\[
I_5 = -6 \mint{\nu \le \ssc{x}1\le \cdots \le \ssc{x}5 \\ |\bx|=1} \frac{f_{5,0}(\bx)}{\ssc{x}{1}\cdots \ssc{x}{5}} = -6 
\mint{\nu \le \ssc{x}1\le \cdots \le \ssc{x}5 \\ \ssc{x}{4}+\ssc{x}{5} < 1/2 \\|\bx|=1}
h(\ssc{x}{3}+\ssc{x}{4}+\ssc{x}{5})\, d\bx.
\]
Computations with Mathematica give $I_3+I_5=-1.000015\ldots$.
\end{proof}

\medskip

%
%
%
%
{\Large \section{Minimal Type II range needed to detect primes}
\label{sec:TypeII-minimum}}
%
%
%
%

In this section we prove Theorem \ref{thm:TypeII-minimum}.
Recall the definition of the sets $\sF_\eta(P)$ from Definition \ref{defn:FetaP}.
Our basic strategy is to
construct a weight $w_n$ satisfying the Type I condition
\eqref{eq:TypeI}, ignoring the Type II condition \eqref{eq:TypeII},
 with $w_p$ a constant less than $-1$
at primes $p$ and $w_n\ge -1$ for other $n$, using
Theorem \ref{thm:constructions} (a), and then tweak the function
to give a desired construction satisfying \eqref{eq:TypeII}
 when $0<\nu \le \ssc{\nu}{0}$ for sufficiently small $\ssc{\nu}{0}$.

We define analogs of $\cR$ and $\sF_\eta$ in the `$\nu=0$' case.
Let $\cR^*(\gamma)$ be the set of vectors, of arbitrary dimension,
with components in $(0,1-\gamma)$ and with sum 1,
and let $\sF_\eta^*(\gamma)$ be the set of bounded functions $f\in \cS$
supported on the subset of $\cC(\cR^*(\gamma))$ with all components
$\ge \eta$, and satisfying conditions (b) and (c) in Definition \ref{defn:FetaP}.

\medskip

\begin{thm}\label{thm:fsl}
Fix $\gamma \in [\frac12,1]$.
Suppose that there is a $\eta>0$ and function $f \in \sF^*_{\eta}(\gamma)$  so that
\begin{itemize}
\item $f(1) < -1$ \; (contribution to primes);
\item for all $\bbeta=(\beta_1,\ldots,\beta_k)$ with $k\ge 2$, we have
$f(\bbeta) \ge -1$.
\end{itemize}
Then Theorem \ref{thm:TypeII-minimum} holds for this $\gamma$, that is, there
is a $\ssc{\nu}{0}$, depending only on $\gamma$, so that whenever
$P=(\gamma,\theta,\nu)\in \cQ_0$, $\nu\le \nu_0$, and $B>0$, if $x$ is large enough
then there is a non-negative sequence $(a_n)$
such that if $b_n=1$ and $w_n=a_n-1$ for all $n\in (x/2,x]$
then $w_n$
satisfies \eqref{eq:TypeI} and \eqref{eq:TypeII}, and also
and $a_p=0$ for all primes $p$.  In particular, we have $C^-(P)=0$.
\end{thm}

\begin{proof}
Define $\delta>0$ by
\[
f(1) = -1 - \delta.
\]
Let $\ssc{\nu}{0}$ be a small enough positive constant, $0 \le \nu \le \ssc{\nu}{0}$,
$0\le \theta \le 1/2$ and $P=(\gamma,\theta,\nu)$.
We now construct a function $\ft \in \sF_{\eta}(P)$ from $f$ that takes into account the Type II restrictions \eqref{eq:TypeII}.

For each $k\in \NN$ let $\cD_{k,\eta}(P)$ be the set of vectors in $\cC(\cR^*(\gamma))\cap \RR^k$ with all components $\ge \eta$ and 
no proper subsum (that is, a subsum which is not zero and not one) in 
$[\theta,\theta+\nu]$.  Define, for each $k$ the restriction
\be\label{ftilde}
\tilde{f}_{k,0}(\bxi) = f_{k,0}(\bxi) \one(\bxi \in \cD_{k,\eta}(P)).
\ee
Complete the definition of $\tf$ by applying \eqref{fsl} with the functions 
$\tf_{k,0}$ on the right side in place of $f_{k,0}$.
Then $\tf \in \sF_{\eta}(P)$ since $\cD_{k,\eta}(P)$ is a finite union of 
convex polytopes (by Lemma \ref{lem:CR1-union-polytopes}) and by assumption $f\in \sF_\eta(\gamma)$.
We claim that for all $k\ge 2$ and $\bxi\in \cD_{k,\eta}(P)$ we have
\be\label{fbetalower}
\ft(\bxi) \ge -1 - \frac{\delta}{3},
\ee
and furthermore, we have 
\be\label{fone}
\ft(1) \le -1-\frac{2\delta}{3}.
\ee
From \eqref{fbetalower} and \eqref{fone}, we quickly deduce the desired conclusion.
Indeed, by Theorem \ref{thm:constructions} (a), for any $B>0$ and large enough $x$,
there is a constant $z$ with $|z-\tf(1)| \le \delta/3$ and a sequence $(v_n)_{x/2<n\le x}$ such that
\[
v_p=z \;\; \text{ for prime }p\in (x/2,x], \qquad
v_n \ge \min_{\dim{\bu}\ge 2} \tf(\bu) - \delta/3 \;\; \text{ for } x/2<n\le x, n\text{ composite}.
\]
By \eqref{fone}, $z \le -1 - \delta/3$. 
Then, by \eqref{fbetalower}, the sequence $(w_n)$ given by $w_n = -v_n/z$ for all $n$, satisfies
\[
w_p=-1 \;\; (\text{prime }p\in (x/2,x]), \qquad
w_n \ge (-1/z) (-1-\delta/3) \ge -1 \;\; (x/2<n\le x, n\text{ composite}).
\]

It remains to show \eqref{fbetalower} and \eqref{fone}.
Fix $r\ge 1$ and $\bxi \in \cC(\cR^*(\gamma)) \cap \RR^r$ with all components
$\ge \eta$.
If $\bxi \not \in \cD_{r,\eta}(P)$ (that is, $\bxi$ has a proper subsum in $[\theta,\theta+\nu]$), then clearly $\tf(\bxi)=0$.
Otherwise, we claim that
\be\label{f-ftilde}
|\tf(\bxi) - f(\bxi)| \le \delta/3.
\ee
If $\nu_0$ is small enough, and this will establish \eqref{fbetalower} and \eqref{fone}.
Let $\bxi\in \cD_{r,\eta}$ and write $\bxi=(\bbeta,\balpha)$, where 
\[
\eta \le \beta_1,\ldots,\beta_s<1-\gamma \le \alpha_1,\ldots,\alpha_\ell, \quad  r=s+\ell.
\] 
  If $\ell=0$ then $\tf(\bxi)=f(\bxi)$ by \eqref{ftilde}.
If $\ell \ge 1$,
observe that on the right side of \eqref{fsl}, $k_1,\ldots,k_\ell$ are bounded,
the functions $\cyrL_{1-\gamma}(\bu_i)$ are bounded and hence the integrands are bounded
(these bounds depend on $\eta$).
Fix $k_1,\ldots,k_\ell\le 1/\eta$, set $k=k_1+\cdots+k_\ell+s$
and consider one of the $2^k-2$
proper subsums of  $(\bbeta,\bu_1,\ldots,\bu_\ell)$, 
\[
A = \sum_{i\in I} \beta_i + \sum_{j=1}^\ell \sum_{i\in L_j} u_{j,i},
\]
where $I\subseteq [s]$ and $L_j \subseteq [k_j]$ for each $j$.
To show \eqref{f-ftilde} for sufficiently small $\nu_0$, it suffices to show that 
for each choice of $I,L_1,\ldots,L_\ell$,
the  $(k_1+\cdots+k_\ell-\ell)$-dimensional measure
of the set of $(\bu_1,\ldots,\bu_\ell)$ for which $A\in [\theta,\theta+\nu]$
is at most $\nu$.

If, for all $j$, $L_j$ is empty or $L_j=[k_j]$, then $A$ is a proper subsum
of $\bxi$ and hence always avoids $[\theta,\theta+\nu]$.
Now assume there is some $j'$ for which $L_{j'}$ is nonempty and not the whole of $[k_{j'}]$.
Pick $i',i''$ so that $i'\in L_{j'}$ and $i''\in [k_{j'}] \setminus L_{j'}$, and let
$A' = A - u_{j',i'}$.  With all of the variables $u_{j,i}$ fixed except for 
$u_{j',i'}$ and $u_{j',i''}$, $A'$ is fixed and  $u_{j',i'}+u_{j',i''}$ is fixed (since $\alpha_j = u_{j,1} + \cdots + u_{j,k_j}$).  Hence the measure of $u_{j',i'}$ with $A\in [\theta,\theta+\nu]$ is $\le \nu$, and the proof is complete.
\end{proof}

It remains to find a function satisfying the conditions of Theorem \ref{thm:fsl}.

\subsection{Modified Liouville functions}
Our construction depends properties of `modified Liouville functions' which are completely multiplicative, supported
on $x^{\eta}-$rough numbers and satisfy the Type I bounds \eqref{eq:TypeI} (that is, has 
`level of distribution' $x^{\gamma}$); in the case
$\eta=0$ the ordinary Liouville function has these properties, and 
this is the basis
for the famous Selberg examples which show that $C^-(\gamma,\theta,0)=0$ 
for all $\gamma<1$.
We define our functions in the vector setting, which neatly sidesteps
various messy issues in the integer setting.

Throughout, we assume that $\eta,c$ are fixed and satisfy 
\[
0 < \eta < c\le \frac12.
\]
Define, for $\a> 0$ the function 
\be\label{M-def}
M^{(c,\eta)}(\a) := \a \sum_{k\ge 1}\frac{(-1)^k}{k!} 
\mint{\a=\beta_1+\cdots+\beta_k \\ \eta \le \beta_i \; \forall i}
\frac{\cyrL_{c}(\bbeta)}{\beta_1\cdots \beta_k}\, d\bbeta.
\ee
(Recall the definition \eqref{Linnik-fcn} of $\cyrL_c$.) This formula comes from the fragmentation identity in \eqref{fsl}.
We now define a function $\tl^{(c,\eta)}$ on vectors of positive real numbers (of arbitrary length) by
\be\label{lambda-def}
\tl^{(c,\eta)}(\xi_1,\ldots,\xi_k) = M^{(c,\eta)}(\xi_1) \cdots M^{(c,\eta)}(\xi_k).
\ee
Like the Liouville function, this vector function is ``completely multiplicative'' in
the sense that
\[
\tl^{(c,\eta)}(\balpha,\bbeta) = \tl^{(c,\eta)}(\balpha)\tl^{(c,\eta)}(\bbeta) \qquad (\text{all }\balpha,\bbeta).
\]
Suppose that $\a \ge c$.  By Lemma \ref{lem:tc-basic} (a), if any $\beta_i\ge c$
then $\cyrL_c(\bbeta)=0$.  In particular, the $k=1$ term in \eqref{M-def} is zero
and hence
\be\label{M-s}
M^{(c,\eta)}(\a) = \a \sum_{k\ge 2}\frac{(-1)^k}{k!} 
\mint{\a=\beta_1+\cdots+\beta_k \\ \eta \le \beta_i < c \; \forall i}
\frac{\cyrL_{c}(\bbeta)}{\beta_1\cdots \beta_k}\, d\bbeta \qquad (\a \ge c).
\ee
If $\a < c$ then Lemma \ref{lem:tc-basic} (b) implies that
 $\cyrL_{c}(\bbeta)=\one_{k=1}$
and thus
\be\label{M-tiny}
M^{(c,\eta)}(\a) = \begin{cases}
-1 \qquad &\eta\le \alpha < c,\\
0,&\alpha<\eta.
\end{cases}
\ee
Thus, if $\eta \le \beta_1 \le \cdots \le \beta_s < c \le \alpha_1 \le \cdots \le \alpha_\ell$, then
\[
\tl^{(c,\eta)}(\bbeta,\balpha) = (-1)^s M^{(c,\eta)}(\alpha_1) \cdots M^{(c,\eta)}(\alpha_\ell).
\]

In the integer setting, this corresponds to a completely multiplicative function
$\tl$ with $\tl(p)=0$ for $p<x^\eta$,
 $\tl(p)=-1$ if $x^{\eta} \le p <x^{c}$ and if $p\approx x^{\a}$ for $\a\ge c$ then
$\tl(p)=M^{(c,\eta)}(\a)$.  We will show below that $M^{(c,\eta)}(\a)$ is very close to $-1$ and thus $\tl^{(c,\eta)}$ behaves similarly to the Liouville function.

Using \eqref{M-s}, these functions satisfy the analog of
\eqref{fsl}.  In fact, the formula \eqref{M-def} was derived from \eqref{fsl}
by setting $f(\ssc{x}{1},\ldots,\ssc{x}{k})=(-1)^k$ when all components are $<c$.
 We give another proof below, which applies to a more general type
 of function, those with $f(x)=-m$ for small $x$, where $m\in\NN$.
\medskip

\begin{lem}\label{lem:TypeI}
For a positive integer $m$, real $w \ge mc$ and real $c > \eta >0$, 
\[
\sum_r \frac{m^r}{r!} \;\; \mint{\xi_1+\cdots+\xi_r=w \\ \eta \le \xi_i\; \forall i} 
\frac{\tl^{(c,\eta)}(\bxi)}{\xi_1\cdots\xi_r}\, d\bxi = 0.
\]
\end{lem}

\begin{proof}
Denote by $L$ the left side in the lemma.
 By \eqref{lambda-def} and \eqref{M-def}, 
\[
L=\sum_r \frac{m^r}{r!} \mint{w=\xi_1+\cdots+\xi_r} \sum_{k_1,\ldots,k_r}
\prod_{i=1}^r \Bigg[ \frac{(-1)^{k_i}}{k_i!} \mint{\xi_i=\beta_{i,1}+\cdots+\beta_{i,k_i} \\
\eta \le \beta_{i,j}\; \forall j} \frac{\cyrL_{c}(\beta_{i,1},\ldots,\beta_{i,k_i})}{\beta_{i,1}\cdots \beta_{i,k_i}}\, d\bbeta_i  \Bigg] d\bxi.
\]
Now reorganize this, firstly fixing $k=k_1+\cdots+k_r$ and the vector of $k$
components $\beta_{i,j}$.  Also relabel these components as $\phi_1,\ldots,\phi_k$, where
\[
\phi_{k_1+\cdots+k_{i-1}+j} = \beta_{i,j} \qquad (1\le i\le r, 1\le j\le k_i)
\]
and let $\bphi_i = (\phi_{k_1+\cdots+k_{i-1}+1},\ldots,\phi_{k_1+\cdots+k_i})$
for $1\le i\le r$.  Then
\[
L = \sum_{k\ge 1} (-1)^k \mint{w=\phi_1+\cdots+\phi_k \\ \eta \le \phi_i\; \forall i} \frac{1}{\phi_1\cdots \phi_k} \sum_r \frac{m^r}{r!} \ssum{k_1+\cdots+k_r=k \\ k_i\ge 1\; \forall i} \frac{1}{k_1!\cdots k_r!} \cyrL_{c}(\bphi_1)\cdots \cyrL_{c}(\bphi_r)\; d\bphi.
\]
Since the region of integration is symmetric in all variables $\phi_i$ and
$\phi_1\cdots \phi_k$ is symmetric in the $\phi_i$, we may replace
$\cyrL_{c}(\bphi_1)\cdots \cyrL_{c}(\bphi_r)$ by its symmetric average
\[
\binom{k}{k_1\; k_2\; \cdots k_r}^{-1} \ssum{A_1 \sqcup \cdots \sqcup A_r=[k] \\ |A_i|=k_i\; \forall i} \cyrL_{c}(\bphi_{A_1})\cdots \cyrL_{c}(\bphi_{A_r}).
\]
This gives
\[
L = \sum_{k\ge 1} \frac{(-1)^k}{k!} \mint{w=\phi_1+\cdots+\phi_k \\ \eta \le \phi_i\; \forall i} \frac{1}{\phi_1\cdots \phi_k} \sum_{r\ge 1} \frac{m^r}{r!}
\sum_{A_1 \sqcup \cdots \sqcup A_r = [k]} \cyrL_{c}(\bphi_{A_1})\cdots \cyrL_{c}(\bphi_{A_r})
\,d \bphi,
\]
where we have dropped the condition $|A_i| \ge 1$ since $\cyrL_c(\emptyvec)=0$ by 
the definition \eqref{Linnik-fcn}.
Since $w\ge mc$, Lemma \ref{lem:t-ident} implies that the sum 
on $r$ equals zero.  Hence, $L=0$ as desired.
\end{proof}

\begin{lem}\label{lem:M-bounds}
For $\a>0$ we have
\be\label{M-eq}
M^{(c,\eta)}(\a) = - \a \sum_{j=1}^{\fl{\a/\eta}} \frac{1}{j} \;\;\;\; \mint{\a=\phi_1+\cdots+\phi_j \\ \eta \le \phi_i < c\; \forall i}\;\;
\prod_{m=1}^j \pfrac{-\rho'(\phi_m/\eta)}{\eta}\, d \bphi
\ee
Moreover, for $\alpha\ge \eta$
\be\label{M-ineq}
-1 \le M^{(c,\eta)}(\a) \le -1 + \frac{\a^2}{c\eta}\; \rho\( \frac{c}{\eta} - 1 \),
\ee
where $\rho$ is the Dickman function.
\end{lem}

The Dickman function is the unique function which satisfies (i) $\rho(u)=1$ for
$0\le u\le 1$, (ii) $\rho$ is continuous on $[0,\infty]$, (iii) $\rho$ is differentiable
 for $u>1$ and (iv) $\rho$ satisfies the differential-delay equation
  $u \rho'(u)=-\rho(u-1)$ for $u>1$.
It is known that $\rho$ is positive, decreasing and
has decay $\rho(u) \ll e^{-u\log u}$
(see, e.g. \cite[Chapter III.5]{Tenenbaum-book}),
thus the upper bound tends very rapidly to
-1 as $\eta\to 0$.

We also observe that the $j=1$ term in \eqref{M-eq} is 1 for $\eta\le \a<c$
and $0$ for $\a\ge c$ or $0<\a<\eta$, and for each
 $j\ge 2$, the multiple integral has bounded derivative on the interior
 of its support, namely for  $\alpha \in (j\eta,jc)$.  As $j$ is bounded, we see that
 $M(\a)$ is piecewise differentiable with a bounded derivative.

To prove Lemma \ref{lem:M-bounds}, we need an integral version of a result about summing the
Liouville function over rough integers (see, e.g. \cite[Lemma 12.2]{Opera}).

\begin{lem}\label{Phi-star}
For positive $\a,v$ satisfying $\a v \ge 1$, we have
\[
\sum_{k=1}^{\fl{\a v}} (-1)^k \;\;\; \mint{\a=\beta_1+\cdots+\beta_k \\ \frac1{v} \le \beta_1 \le \cdots \le \beta_k} \frac{d\bbeta}{\beta_1\cdots \beta_k}
= v\, \rho'(\a v) = - \frac{\rho(\a v-1)}{\a}.
\]
\end{lem}

\begin{proof}
Let
\[
F(u) = - \sum_{k\ge 1} (-1)^k \mint{1=\beta_1+\cdots+\beta_k \\ \frac{1}{u} \le \beta_1 \le \cdots \le \beta_k} \frac{d\bbeta}{\beta_1\cdots \beta_k}.
\]
The sum is finite, as the multiple integral is zero for $k>u$.  Also, it is clear
that $F(u)$ is continuous for $u\ge 1$, differentiable for $u>a$ and that
$F(u)=1$ for $1\le u\le 2$.  When $u> 2$, 
\begin{align*}
F(u) &= 1 - \sum_{2\le k\le u} (-1)^k \int_{1/u}^{1/k} \frac{1}{\beta_1} \;\;
\mint{1-\beta_1=\beta_2+\cdots+\beta_k \\ \beta_1\le \beta_2\le \cdots \le \beta_k} \frac{1}{\beta_2\cdots \beta_k}\, d(\beta_2,\ldots,\beta_k)\, d\beta_1.
\end{align*}
Making the change of variables $\beta_j=(1-\beta_1)\beta_j'$ for $j\ge 2$, we obtain
\begin{align*}
F(u)&=1 - \sum_{2\le k\le u} (-1)^k \int_{1/u}^{1/k} \frac{1}{\beta_1(1-\beta_1)} 
\mint{1=\beta_2'+\cdots+\beta_k' \\ \frac{\beta_1}{1-\beta_1}\le \beta_2'\le \cdots \le \b'_k} \frac{1}{\b'_2\cdots \b'_k}\, d\bbeta'\, d\beta_1 \\
&= 1 - \int_{1/u}^{1/2} \frac{1}{\beta_1(1-\beta_1)} F\pfrac{1-\beta_1}{\beta_1}\, d\beta_1.
\end{align*}
Differentiating this gives
\[
F'(u) = \frac{-1/u^2}{(1/u)(1-1/u)} F\pfrac{1-1/u}{1/u} = -\frac{F(u-1)}{u-1}.
\]
Together with the initial conditions $F(u)=1$ for $1\le u\le 2$ and the fact
that $F$ is continuous, this differential-delay equation uniquely determines $F$.
Comparing with the differential-delay equation for the Dickman function, we
conclude that $F(u)=\rho(u-1)$ for $u\ge 1$.  This proves the desired formula
when $\a=1$.  The formula for general $\a$ follows by making the change 
of variables $\beta_j = \a \beta_j'$ for $1\le j\le k$.
\end{proof}

\begin{proof}[Proof of Lemma \ref{lem:M-bounds}]
Recall the definition \eqref{M-def} of $M^{(c,\eta)}(\a)$.  Inserting
the definition \eqref{Linnik-fcn} of $\cyrL_{c}(\bbeta)$, we get
\begin{align*}
M^{(c,\eta)}(\a) &= \alpha \sum_{j=1}^{\fl{\alpha/\eta}}  \frac{(-1)^{j+1}}{j} \sum_{k=j}^{\fl{\alpha/\eta}} \frac{(-1)^k}{k!}
\ssum{A_1 \sqcup \cdots \sqcup A_j = [k] \\ k_m := |A_m| \ge 1 \; \forall m} \;\;\; \mint{\a=\beta_1+\cdots + \beta_k \\ \eta \le 
\beta_i < c\; \forall i \\ \eta \le |\bbeta_{A_i}| < c\; \forall i} \frac{1}{\beta_1\cdots \beta_k} \, d\bbeta.
\end{align*}
The multiple integral depends only on $k_1,\ldots,k_j$, and with $k_1,\ldots,k_j$ fixed there are $\binom{k}{k_1\; k_2 \; \cdots k_j}$
choices for $A_1,\ldots,A_j$.  Also
\[
\frac{(-1)^{k+j}}{k!} \binom{k}{k_1\; k_2 \; \cdots k_m} = \prod_{m=1}^j \frac{(-1)^{k_m+1}}{k_m!}.
\]
Let $\phi_j = |\bbeta_{A_j}|$, so that $\a = \phi_1+\cdots + \phi_j$.  Then
\[
M^{(c,\eta)}(\a) = - \alpha \sum_{j=1}^{\fl{\alpha/\eta}}   \frac{1}{j}  \;\;\; \mint{\a=\phi_1+\cdots+\phi_j \\ \eta \le \phi_i < c\; \forall i}
\;\;\; \prod_{m=1}^j \Bigg[ \sum_{k_m\ge 1} (-1)^{k_m+1} \;\; \mint{\phi_m=\beta_{m,1}+\cdots+\beta_{m,k_m} \\ \eta \le \beta_{m,1}\le \cdots \le \beta_{m,k_m}}
\;\; \frac{d\bbeta_m}{\beta_{m,1}\cdots \beta_{m,k_m}} \Bigg]\, d\bphi.
\]
Here the factor $1/k_m!$ has been removed by imposing the ordering $\beta_{m,1}\le \cdots \le \beta_{m,k_m}$.
By Lemma \ref{Phi-star}, the expression in brackets equals $-\rho'(\phi_m/\eta)/\eta$ and this completes the proof of \eqref{M-eq}.

To show \eqref{M-ineq},
 we may assume that $\a \ge c$ by \eqref{M-tiny}.
 We will first show that relaxing the conditions $\phi_i < c$ makes only a small change to 
right side of \eqref{M-eq}.  Define
\[
M^*(\a) =  \a \sum_{k\ge 1}\frac{(-1)^k}{k!} 
\mint{\a=\beta_1+\cdots+\beta_k \\ \eta \le \beta_i \; \forall i}
\frac{\cyrL_{\infty}(\bbeta)}{\beta_1\cdots \beta_k}\, d\bbeta.
\]
On the one hand, by Lemma \ref{lem:tc-basic} (b), $M^*(\a)=-1$.  Following the
above proof leading to \eqref{M-eq}, we also see that
\[
M^*(\a) = - \a \sum_{j=1}^{\fl{\a/\eta}} \frac{1}{j} \;\;\;\; \mint{\a=\phi_1+\cdots+\phi_j \\ \eta \le \phi_i \; \forall i}\;\;
\prod_{m=1}^j \pfrac{-\rho'(\phi_m/\eta)}{\eta}\, d\bphi.
\]
Moreover, by forcing $\phi_j$ to be the largest, we remove the factor $1/j$ appearing
above and in \eqref{M-eq}, and therefore
\be\label{MM}
\begin{split}
1+M^{(c,\eta)}(\a) &= M^{(c,\eta)}(\a)-M^*(\a)\\
&= \a \sum_{j=1}^{\fl{\a/\eta}}  \;\;\;\; \mint{\a=\phi_1+\cdots+\phi_j \\ \eta \le \phi_i \le \phi_j \; \forall i \\ \phi_j \ge c}\;\;
\prod_{m=1}^j \pfrac{-\rho'(\phi_m/\eta)}{\eta}\, d\bphi.
\end{split}
\end{equation}
Since $\rho(u)$ is decreasing,  the right side of \eqref{MM} is non-negative
and also
\[
\frac{-\rho'(\phi_j/\eta)}{\eta} = \frac{\rho(\phi_j/\eta-1)}{\phi_j}
\le \frac{\rho(c/\eta-1)}{c}. 
\]
Therefore,
\[
0 \le M^{(c,\eta)}(\a) +1 \le   \frac{\a \rho(\frac{c}{\eta}-1)}{c}  \sum_{j=1}^{\fl{\a/\eta}}
\Bigg[ \int_\eta^\infty \frac{-\rho'(\phi/\eta)}{\eta}\, d\phi \Bigg]^{j-1}.
\]
The integral on the right side equals $\rho(1)=1$ and this completes the proof
of \eqref{M-ineq}.
\end{proof}

\begin{lem}\label{lem:lambda-Fdelgam}
Suppose $m\in \NN$, $0<\eta < 1-\gamma$, and $0<c \le \frac{1-\gamma}{m}$.
For $\ssc{x}{1},\ldots,\ssc{x}{k}\ge \eta$ with sum 1, define 
\[
f(\bx) = m^k \tl^{(c,\eta)}(\bx).
\]
Then $f\in \sF_{\eta}^*(\gamma)$.
\end{lem}
We recall the set $\sF_\eta^*(\gamma)$ is defined at the beginning of the section a the `$\nu=0$' analolg of $\sF_\eta$ from Definition \ref{defn:FetaP}.
\begin{proof}
By \eqref{lambda-def}, $f(\ssc{x}{1},\ldots,\ssc{x}{k})$ is symmetric in $\ssc{x}{1},\ldots,\ssc{x}{k}$
and supported on vectors with $\ssc{x}i\ge \eta$ for all $i$.
Also, since $M^{(c,\eta)}$ is piecewise differentiable, so is $f$.
By \eqref{M-ineq}, $f$ is bounded.
If we fix $\xi_1,\ldots,\xi_r \ge \eta$ with sum $g\le \gamma \le 1-mc$,
then
\[
\sum_{k\ge 1} \frac{1}{k!} \;\;\; \mint{\xi_{r+1}+\cdots+\xi_{r+k}=1-g \\ \eta \le \xi_i\; \forall i} \frac{m^{k+r} \tl^{(c,\eta)}(\bxi)}{\xi_{r+1}\cdots \xi_{r+k}}\, 
d (\xi_{r+1},\ldots, d\xi_{r+k}) = m^r \tl^{(c,\eta)}(\xi_1,\ldots,\xi_r) Z,
\]
where, by Lemma \ref{lem:TypeI} and $1-g \ge mc$ we have
\[
Z := \sum_{k\ge 1} \frac{m^k}{k!} \mint{\xi_{r+1}+\cdots+\xi_{r+k}=1-g \\
\eta \le \xi_i\, \forall i} \frac{\tl^{(c,\eta)}(\xi_{r+1},\ldots,\xi_{r+k})}{\xi_{r+1}\cdots \xi_{r+k}}\,d (\xi_{r+1},\ldots, d\xi_{r+k})  = 0.
\]
Thus, \eqref{TypeI-f} holds and therefore, $f\in \sF_{\eta}^*(\gamma)$.
\end{proof}

%
%
\subsection{The proof of Theorem \ref{thm:fsl}}
%
%

Fix $1/2<\gamma<1$ and suppose that
\[
c = \frac{1-\gamma}{2}, \qquad 0 < \eps < c,
\]
with $\eps$ small enough so that
\[
\frac{1}{c\eps} \rho\( \frac{c}{\eps} -1 \) \le \frac14.
\]
Such $\eps$ exists by the rapid decay $\rho(u) \ll e^{-u\log u}$.
 By \eqref{M-ineq}, for all $\a>0$,
\be\label{Mac-ineq}
-1 \le M^{(c,\eps)}(\a) \le -\frac{3}{4}\one(\a\ge \epsilon).
\ee
For $\bx=(\ssc{x}{1},\ldots,\ssc{x}{k})$ with sum 1 and each $\ssc{x}i\ge \eps$, define
\[
g(\bx) = \big( 1 - 2^{k-3} \big) \tl^{(c,\eps)}(\bx).
\]
Two applications of Lemma \ref{lem:lambda-Fdelgam}, one with $m=1$ and the other with $m=2$, shows that $g\in \sF_{\eps}^*(\gamma)$.
Note also that $g$ is supported on vectors of dimension $\le 1/\eps$.
Let
\[
\ssc{g}{0} = 2^{1/\eps}
\]
so that for all $\bx$, $|g(\bx)|\le \ssc{g}{0}$.
Using \eqref{Mac-ineq}, we have
\be\label{g1}
g(1) = (3/4)M^{(c,\eps)}(1) \le -\frac9{16} < -\frac12.
\ee
Furthermore, when $\ell\ge 3$ is odd, \eqref{Mac-ineq} implies 
\begin{equation}\label{gl}
\begin{split}
g(\ssc{x}{1},\ldots,\ssc{x}\ell) &= \big(1 - 2^{\ell-3} \big) M^{(c,\eps)}(\ssc{x}{1})\cdots M^{(c,\eps)}(\ssc{x}\ell)\\
&\ge \big(2^{\ell-3}-1\big) (3/4)^{\ell}\one(\ssc{x}{1},\dots,\ssc{x}\ell\ge \epsilon) \ge 0.
\end{split}
\end{equation}

For $0<\eta < \eps$,  define
\[
\kappa(\eta) := \frac{1}{(1-\gamma) \eta} \;\rho\( \frac{1-\gamma}{\eta}-1 \).
\]
Since $\kappa(\eta)\to 0$ as $\eta\to 0$, there is a choice of $\eta\in(0,\epsilon)$ so that
$\kappa(\eta) \le 1/(4\ssc{g}{0}).$
In this way, by Lemma \ref{lem:M-bounds}, for any $\eta\le \a \le 1$,
\be\label{deleps}
-1 \le M^{(1-\gamma,\eta)}(\a) \le -1 + \kappa(\eta) \le -1 + \frac{1}{4\ssc{g}{0}}.
\ee

Now define
\[
f(\bbeta) = \tl^{(1-\gamma,\eta)}(\bbeta) + \frac{g(\bbeta)}{\ssc{g}{0}}.
\]
Since $\eta<\eps$, $\sF^*_\eps(\gamma)\subseteq \sF^*_\eta(\gamma)$.
Thus, by Lemma \ref{lem:lambda-Fdelgam} with $m=1$, $f\in \sF^*_{\eta}(\gamma)$.
By \eqref{g1} and \eqref{deleps},
\[
f(1) = M^{(1-\gamma,\eta)}(1) + \frac{g(1)}{\ssc{g}{0}} \le -1 + \frac{1}{4\ssc{g}{0}} - \frac{1}{2\ssc{g}{0}} = -1 - \frac{1}{4\ssc{g}{0}} < -1.
\]
Suppose $k\ge 2$ and $\bbeta=(\beta_1,\ldots,\beta_k)$.  If $\beta_i<\eta$ for some $i$ then $\tilde{\lambda}^{(1-\gamma,\eta)}(\bbeta)=g(\bbeta)=0$, so we may assume that $\beta_i\ge \eta$ for all $i$. If $k$ is even,
\eqref{deleps} implies that
 \[
 f(\bbeta) \ge 0 + \frac{g(\bbeta)}{\ssc{g}{0}} \ge -1,
 \]
 while if $k$ is odd, \eqref{gl} and \eqref{deleps} imply that
 \[
 f(\bbeta) \ge -1 + \frac{g(\bbeta)}{\ssc{g}{0}} \ge  -1.
 \]
Therefore, the hypotheses of Theorem \ref{thm:fsl} are satisfied, and this 
completes the proof. \qed

\bigskip
%
%
%
%
%
{\Large \section{Asymptotic for primes}\label{sec:asymptotic}}
%
%
%
%
%
%

In this section we prove Theorems \ref{thm:asymptotic-R1}
and \ref{thm:asymptotic-R1-A1A2}, the former being straightforward and
the latter requiring lengthy case-by-case analysis.
Together, these establish Theorem \ref{thm:asymptotic}.

Recall that $\cR$ is the set of vectors, of arbitrary dimension, which
have sum of components 1, all components in $(0,1-\gamma)$ and no subsum
in $[\theta,\theta+\nu]$.
  Recall also that $M=\fl{1/(1-\gamma)}$, so that
\begin{equation}\label{Mdef}
\frac{1}{M+1} < 1-\gamma \le \frac{1}{M}, \qquad M\ge 2.
\end{equation}

\medskip

\subsection{Proof of Theorem \ref{thm:asymptotic-R1} when $\cR$ is empty}
\label{sec:sieve_nu>1-gamma}

We will show that for any $A > \varpi$, if $B$ is large enough as a function
of $A,\gamma,\theta,\nu$ and $(w_n)$ satisfies \eqref{w},
\eqref{eq:TypeI} and \eqref{eq:TypeII}, then
\begin{equation}\label{wp-goal-large-nu}
\sum_p w_p \ll \frac{x}{(\log x)^{A}}.
\end{equation}

First, assume that $\nu\ge 1-\gamma$ ($\cR$ is always empty with this condition).
We begin with an application of Lemma \ref{lem:PrimeSplit} with
$\psi_{u,v} = \one_{u=v=1}$, $k=1$, $\cU=\{1\}$, $g(1)=1$, $\sigma=1/2$, $\cT_{1,1}=[\frac12,1]$, 
$D=\varpi+A$, $M=\fl{(\log x)^D}$ and $\ell=6\cl{ 1/(1-\gamma) }$.  This gives
\[
\sum_p w_p \ll \frac{x}{(\log x)^A} + (\log x)^{\ell+3D+6}
\sup_{\substack{\beta_1,\ldots,\beta_\ell \\ \text{1-bounded}}}
|S_{\bbeta}|, 
\]
where
\[
S_{\bbeta}:= \ssum{n=d_1\cdots d_{\ell}\sim x \\
 d_i < n/x^{\gamma}\, \forall i} w_n
\beta_1(d_1)\cdots \beta_{\ell}(d_{\ell}).
\]

Now $d_i\le x^{1-\gamma} \le x^{\nu}$ for all $i$, and it follows that
for a unique $k\le \ell$ we have $d_1\cdots d_{k-1} \le (x/2)^{\theta}$ and  $d_1\cdots d_k \in ((x/2)^\theta,x^{\theta+\nu}]$.  Then
\[
S_{\bbeta} =  \sum_{k=1}^{\ell}  \ssum{n = d_1\cdots,d_{\ell} \sim x \\ 
d_i < n/x^{\gamma} \; (1\le i\le \ell) \\ 
d_1\cdots d_{k-1} \le (x/2)^{\theta} \\ 
d_1\cdots d_k \in ((x/2)^\theta,x^{\theta+\nu}]} w_{n} \beta_1(d_1)\cdots \beta_{\ell}(d_{\ell}).
\]
The condition $d_i<n/x^{\gamma}$ is equivalent to $n/d_i > \fl{x^{\gamma}}+\frac12$,
and $d_1\cdots d_{k-1}\le (x/2)^\theta$ is equivalent to $d_1\cdots d_{k-1} < \fl{(x/2)^\theta}+\frac12$.
Thus, after $\ell+1$ successive applications of Lemma \ref{lem:Integration},
we see that
\begin{equation}\label{eq:nu>1-gamma-1}
\sup_{\substack{\beta_1,\ldots,\beta_\ell \\ \text{1-bounded}}}
|S_{\bbeta}| \ll (\log x)^{\ell+1} \sup_{\substack{\beta_1,\ldots,\beta_\ell \\ \text{1-bounded}}} \bigg|
\sum_{k=1}^{\ell}  \ssum{n = d_1\cdots,d_{\ell} \sim x \\ 
d_1\cdots d_k \in ((x/2)^\theta,x^{\theta+\nu}]} w_{n} \beta_1(d_1)\cdots \beta_{\ell}(d_{\ell})\bigg|.
\end{equation}
Writing $\ssc{e}{1}=d_1\cdots d_k$, $\ssc{e}{2}=d_{k+1}\cdots d_\ell$ and
\[
E_1(\ssc{e}{1}) := \sum_{\ssc{e}{1}=d_1\cdots d_k} \beta_1(d_1)\cdots \beta_k(d_k),
\qquad E_2(\ssc{e}{2}) := \sum_{\ssc{e}{2}=d_{k+1}\cdots d_\ell} \beta_{k+1}(d_{k+1})\cdots \beta_\ell(d_\ell),
\]
we have that the inner sum in the right side of \eqref{eq:nu>1-gamma-1} is
\[
\ssum{n=\ssc{e}{1}\ssc{e}{2}\sim x \\ \ssc{e}{1}\in((x/2)^\theta,x^{\theta+\nu}]}
w_n E_1(\ssc{e}{1})E_2(\ssc{e}{2}).
\]
By Lemma \ref{lem:tauk-tau}, for $j=1$ and $j=2$, $|E_j(\ssc{e}{j})|\le \tau_k(\ssc{e}{j})\le \tau(\ssc{e}{j})^{k-1}$.  Hence, by the Type II bound \eqref{eq:TypeII},
if $B\ge \max(k-1,2\ell+3D+7+A)$ we have
\[
S_{\bbeta} \ll \frac{x}{(\log x)^{\ell+3D+6+A}}.
\]
This establishes \eqref{wp-goal-large-nu} when $\nu\ge 1-\gamma$.

Next, assume that $\nu<1-\gamma$ and $\cR$ is empty.
We apply Proposition \ref{prop:general} with
$g(\emptyvec)=1$, $g(\bx)=0$ for $\bx\ne \emptyvec$, $\lambda(1)=1$ and $\lambda(d)=0$ for $d\ne 1$,
and with $\sigma=\nu$.  Then $H(n)=1$ for all $n$
and $\lambda(n)$ is splittable with respect to $L$ for any $L$.  Then
\[
\sum_{n\text{ not prime}} w_n \ll \frac{x}{(\log x)^A}.
\]
But \eqref{eq:TypeI} implies that $|\sum_n w_n| \le x(\log x)^{-B}$, so that
if $B\ge A$ we conclude \eqref{wp-goal-large-nu} when $\nu<1-\gamma$ too. This completes the proof of Theorem  \ref{thm:asymptotic-R1} when $\cR$ is empty.

\medskip

\subsection{Proof of Theorem \ref{thm:asymptotic-R1} when $\cR$ is nonempty}\label{sec:asymptotic-R1}
Let $P=(\gamma,\theta,\nu) \in \cQ$ and $\cR = \cR(P)$.
We must show that if $\cR\ne \emptyset$ then $C^-(P)<1<C^+(P)$.
We first need a positive measure subset of $\cR$ 
with $\cyrL_{1-\gamma}(\bx)\ne 0$.  This is accomplished with the next two lemmas.


\begin{lem}\label{lem:combine}
Fix $\rho>0$.
Given any $\bx$ with $0\le \ssc{x}i < \rho$ for all $i$, 
there is a coagulation $\by$ of $\bx$ which satisfies
$y_i<\rho$ for all $i$ and $y_i+y_j \ge \rho$ for all $i\ne j$.
\end{lem}

\begin{proof}
If  $i\ne j$ and $\ssc{x}i+\ssc{x}j < \rho$,
replace the two components $\ssc{x}i,\ssc{x}j$ with the singleton $\ssc{x}i+\ssc{x}j$, and repeat this process
until there are no more pairs $(i,j)$ with $\ssc{x}i+\ssc{x}j<\rho$.
\end{proof}

\begin{lem}\label{lem:R1-nbhd-Q}
Assume that $\cR$ is nonempty.
Define $\cT_k = \{ \bx\in \RR^k : |\bx|=1 \}$ for each $k$.
There is a $k\ge 3$, $\bz\in \cR\cap \cT_k$ and $\eta>0$ so that
the set $\cD =\{ \bu\in \cT_k : |u_i-z_i| \le \eta\; (1\le i\le k) \}$
lies in $\cR$ and every element $\bu \in \cD$ 
satisfies the following:
\begin{enumerate}
\item $\eta < u_1<\cdots<u_k < 1-\gamma-\eta$;
\item for all $i\ne j$, $u_i+u_j > 1-\gamma+\eta$;
\item all subsums of $\bu$ avoid $[\theta-\eta,\theta+\nu+\eta]$;
\item $\cyrL_{1-\gamma}(\bu)=(-1)^{k-1}(k-1)!$.
\end{enumerate}
\end{lem}

\begin{proof}
For some $n\ge 3$, $S=\cR \cap \cT_n$ is nonempty.  As $S$ is open relative
to $\cT_n$, there is some point $\by \in S$ avoiding all of the hyperplanes
 $\sum_{i\in I} y_i=1-\gamma$ for $I\subseteq [n]$ with $0<|I|<n$, 
and also avoiding the hyperplanes $\sum_{i\in I} y_i = \sum_{j\in J} y_j$ for distinct $I,J\subseteq [n]$, since each such hyperplane intersects $\cT_n$ in a set of dimension $n-2$ (or has no intersection, e.g. if $I=\emptyset$ and $J=[n]$).  This means
that  $\by$ has distinct subsums, none of which are equal to $1-\gamma$. 
By Lemma \ref{lem:combine}, some coagulation $\bz\in \cT_k$ of $\by$ is in $\cR$
and satisfies $z_i+z_j \ge 1-\gamma$ for all $i\ne j$.  This means
that $z_i+z_j > 1-\gamma$ for all $i\ne j$, and also that $z_i \ne z_j$ for all $i\ne j$.
By reordering the components, we may suppose that $z_1<\cdots<z_k$. 
Therefore, for some $\eta$ the set  
$\cD=\{ \bu\in \cT_k : |u_i-z_i| \le \eta\; (1\le i\le k) \}$
will satisfy $\cD \subseteq \cR$
and for all $\bu \in \cD$, $\eta < u_1<\cdots<u_k< 1-\gamma-\eta$
and $u_i+u_j > 1-\gamma+\eta$ for all $i\ne j$.  It follows from Lemma \ref{lem:tc}
(with $k=0$ in that lemma)
that $\cyrL_{1-\gamma}(\bu) = (-1)^{k-1}(k-1)!$ for all $\bu \in \cD$.
Furthermore, since all subsums of $\bz$ avoid $\two$,
 if $\eta$ is small enough then all subsums of any $\bu\in \cD$ avoid
$[\theta-\eta,\theta+\nu+\eta]$ as well.
\end{proof}

Let $\cD \subseteq \cR$ be the set guaranteed by Lemma \ref{lem:R1-nbhd-Q},
and $\eta,k$ the associated parameters, and where
where $\cyrL_{1-\gamma}(\bu) = (-1)^{k-1}(k-1)!$ for all $\bu\in \cD$.

For $\bu\in \cD$ define $f(\bu)=\delta$ for a sufficiently small real number $\delta$ (which can be positive or negative),
and let $f(\bu)=0$ for all other $k-$dimensional vectors $\bu$
with $u_1 \le \cdots \le u_k$.  Extend $f$ to a function symmetric in all variables,
and then extend the definition of $f$ to vectors of smaller dimension using
\eqref{fsl}.  Then it is clear that $f$ is supported on the subset of $\cC(\cR)$
consisting of vectors with all components $>\eta$.  Therefore, $f\in \sF_\eta$
(recall Definition \ref{defn:FetaP} for the definition
of $\sF_\eta$).

Moreover, on the right side of \eqref{fsl}, $u_{j,b}>\eta$ 
for all $j,b$ and hence $\ell$ is bounded, the functions $\cyrL_{1-\gamma}(\bu_j)$
are bounded and also $\prod_{j,b} u_{j,b}^{-1}$ is bounded.  
Hence, for $\delta$ small enough (in terms of $\eta,k,\theta,\nu,\gamma$), we have $f(\bu) \ge -1$ for all $\bu$.
Furthermore, by \eqref{fsl}, we have
\[
f(1) = f_{0,1}(1) = \delta (-1)^{k-1}(k-1)! \mint{\cD} \frac{d\bu}{u_1\cdots u_k}.
\]
Now we apply Theorem \ref{thm:constructions} (b).
Taking $\delta$ so that $\delta(-1)^{k-1}<0$ we get $C^-(P)<1$
and taking $\delta$ so that $\delta (-1)^{k-1} >0$ we get  $C^+(P)>1$. This completes the proof of Theorem \ref{thm:asymptotic-R1}.

\bigskip

\subsection{Proof of Theorem \ref{thm:asymptotic-R1-A1A2}}\label{sec:asymptotic-R1-A1A2}
Recall the definitions of $\cQ$ and $\cR$, given in \eqref{Q2} and Definition \ref{defn:R1}, respectively.
Consider the subset $\cQ_1$ of $\cQ$, defined by
\begin{equation}\label{eq:Q1}
\cQ_1 := \big\{ (\gamma,\theta,\nu) : \tfrac12 \le \gamma < 1-\theta-\nu\text{ or } 1-\theta \le \gamma < 1,\, 0\le \theta < \tfrac12, \,
0\le \nu < \tfrac12-\theta\text{ or }\nu=1-2\theta \big\}.
\end{equation}
This captures the fact that by Proposition \ref{prop:first-reduction},
having $\nu=\frac12-\theta$ is essentially equivalent to having $\nu=1-2\theta$
and having $\gamma=1-\theta-\nu$ is essentially equivalent to
having $\gamma=1-\theta$.  With respect to the analysis of $\cR$,
these reductions are exact as we shall now see.

Let $P=(\gamma,\theta,\nu) \in \cQ$ and $\cR = \cR(P)$.
If $P\in \cQ \setminus \cQ_1$, then  one of the following holds:
\begin{itemize}
\item[(i)]  $1-\gamma=\theta+\nu$,
\item[(ii)]$\theta+\nu=\frac12$.
\end{itemize}
 In case (i), let $P'=(\gamma',\theta,\nu)$ with $\gamma'=1-\theta$,
so that $P'\in \cQ_1$. Then \eqref{eq:A2} holds for both $P$ and $P'$ (since $1-\gamma,1-\gamma'\in\{\theta,\theta+\nu\}$), \eqref{eq:A1} for $P$
holds  if and only if \eqref{eq:A1} holds for $P'$ (since $(n-1)/n\in [\theta,\theta+\nu]$ for $\lceil 1/(1-\gamma)\rceil< n\le \lceil 1/(1-\gamma')\rceil$), and $\cR(P')=\cR(P)$.
That is, the claim in Theorem \ref{thm:asymptotic-R1-A1A2} holds for $P$ if and only
if it holds for $P'$.  Similarly, in case (ii) let $P'=(\gamma,\theta,1-2\theta)$,
so that $P'\in \cQ_1$, $\cR(P)=\cR(P')$ and each of the conditions 
\eqref{eq:A1} and \eqref{eq:A2} holds for $P$ if and only if it holds for $P'$.
Again,  the claim in Theorem \ref{thm:asymptotic-R1-A1A2} holds for $P$ if and only
if it holds for $P'$.  It therefore suffices to prove  the claim in Theorem \ref{thm:asymptotic-R1-A1A2} when $P\in \cQ_1$, which we henceforth assume.

Our first task is to prove a weaker version of Theorem \ref{thm:asymptotic-R1-A1A2},
where we suppose that, in addition to \eqref{eq:A1} and \eqref{eq:A2}, we have the following.

\medskip

\noindent
\textbf{Hypothesis (A3).}  For every vector $\bx$ of non-negative real 
numbers summing to 1 and also satisfying
\begin{enumerate}
\item[(i)]  $\ssc{x}i\in (\nu,1-\gamma)$ for all $i$;
\item[(ii)] for all $i\ne j$ with have $\ssc{x}i+\ssc{x}j \ge 1-\gamma$; 
\item[(iii)] for all $i\ne j$, either $\ssc{x}i=\ssc{x}j$ or $|\ssc{x}i-\ssc{x}j|>\nu$; and
\item[(iv)] there are $i\ne j$ with $\ssc{x}i\ne \ssc{x}j$,
\end{enumerate}
some proper subsum of $\bx$ lies in \two.

\medskip

\begin{lem}\label{lem:BDEF}
Let $P\in \cQ_1$.
Then $\cR(P)$ is empty if and only if all of \eqref{eq:A1}, \eqref{eq:A2} and \emph{(A3)} hold.
\end{lem}

\begin{proof}
Assume $\cR = \cR(P)$ is empty.
For all integers $n\ge M+1$,  $(\frac{1}{n},\cdots,\frac{1}{n})\not\in \cR$, which implies \eqref{eq:A1}.  If (A3) fails, then there is a vector $\bx$ of non-negative real components summing to 1 satisfying conditions (i)--(iv) in the definition
of Hypothesis (A3) and with no proper subsum in $[\theta,\theta+\nu]$, and evidently $\bx\in \cR$.  Thus, (A3) holds.
Now $1-\gamma > \frac{1}{M+1}$, thus for sufficiently small $\eps>0$,
\[
\bx_\eps := (1-M(1-\gamma-\eps), 1-\gamma-\eps,\ldots,1-\gamma-\eps)
\]
 has all components in $(0,1-\gamma)$ and is not in $\cR$.
Then some proper subsum of $\bx_\eps$ lies in \two.  Equivalently, there is
a positive integer $h\le M$ such that $h(1-\gamma-\eps)\in \two \cup [1-\theta-\nu,1-\theta]$.
As this holds for every $\eps>0$, it is also true for $\eps=0$ and thus \eqref{eq:A2} holds.

Now suppose that $\cR$ is nonempty, and let $\bx$ be an element of $\cR$.
By Lemma \ref{lem:combine}, there is a coagulation of $\bx$ satisfying (ii)
in the definition of Hypothesis (A3) and with all components in $(0,1-\gamma)$,
thus this coagulation is in $\cR$ as well.
Now suppose $\bx\in \cR$ satisfies (ii).
Suppose further that $x_{i_1},\ldots,x_{i_k}$ 
 all lie in an interval of length $\le \nu$ and are not all equal.
For each $\ell\le k$, let $T_\ell$ be the set of all nonzero subsums of $\ell$ elements of $(x_{i_1},\ldots,x_{i_k})$.
If $s$ is any subsum of the remaining components of $\bx$,
then the numbers
$\{ t+s : t\in T_\ell \}$ must be either all  $<\theta$ or all $>\theta+\nu$.
To see this, recall that the proper subsums of $\bx$ avoid $\two$ and observe that
the consecutive terms in the sequence $s+t$  for $t\in T_\ell$ differ by at most $\nu$, and so must all be $<\theta$ or $>\theta+\nu$.
If we then create a new vector $\bx'$ by replacing 
each component $x_{i_j}$ with the average $(x_{i_1}+\cdots+x_{i_k})/k$,
it is clear that $\bx'$ has the same property (for all $\ell\le k$).
Thus, $\bx' \in \cR$, while preserving (ii). 
We repeat this until we have a vector
satisfying (iii).  At each step the number of distinct components $x_i$
decreases, thus the process will terminate after a finite number of steps.
Thus, there is a $\bx\in \cR$ satisfying (ii) and (iii).

Suppose now that $\bx=(\ssc{x}{1},\cdots,\ssc{x}{k})\in \cR$ satisfies parts (ii) and (iii).  In particular, all subsums of $\bx$ avoid $\two$.
If all of the $x_i$ are equal, then they equal $1/k$ where $1/k < 1-\gamma \le 1/M$, thus $k\ge M+1$ and \eqref{eq:A1} fails.
Now assume that at least two of the $x_i$ are different, so part (iv) in Hypothesis (A3) holds.  If $x_i > \nu$ for all $i$ then part (i) in the definition of (A3) also
holds and hence (A3) fails.
Otherwise, if $x_i\le \nu$ 
for some $i$, by reordering we may suppose that $\ssc{x}{1}=\min x_i \le \nu$.  For $j > 1$, (ii) implies
\[
x_j \ge 1-\gamma - \ssc{x}{1} \ge 1-\gamma - \nu.
\]
Hence, all of the numbers $x_j$ for $j>1$ lie in $[1-\gamma-\nu,1-\gamma)$, an interval of length $\nu$.  By (iii), $\ssc{x}{2}=\cdots=\ssc{x}{k}$ and hence $\ssc{x}{1}<\ssc{x}{2}$ as well.
If $\ssc{x}{2} \le \nu$ then $|\ssc{x}{1}-\ssc{x}{2}|<\nu$, 
contradicting (iii).  Hence,
\[
0<\ssc{x}{1} \le \nu < \ssc{x}{2}=\cdots = \ssc{x}{k} < 1-\gamma.
\]
In particular, we must have $\theta > 0$ (else $\ssc{x}{1}\in\two$).
This implies, by \eqref{Mdef}, that
\[
1 = \ssc{x}{1}+\cdots + \ssc{x}{k} < k(1-\gamma) \le k/M,
\]
and hence $k \ge M+1$.  Since $\ssc{x}{1}\le \nu$ and all subsums of $\bx$ avoid $\two$ (using $\theta>0$),
for any integer $h\in [0,k-1]$, the points $h\ssc{x}{2}$ and $h\ssc{x}{2}+\ssc{x}{1}$ lie on the
same side of $\two$.  That is, either $h\ssc{x}{2} + \ssc{x}{1} < \theta$ or $h\ssc{x}{2} > \theta+\nu$.
This same property is preserved if we deform $\bx$ by increasing $\ssc{x}{2},\ldots,\ssc{x}{k}$
by $\delta>0$ and decreasing $\ssc{x}{1}$ by $(k-1)\delta$, which keep the sum of all components equal to 1.  So long as
  $\ssc{x}{2}+\delta < 1-\gamma$ and $(k-1)\delta \le \ssc{x}{1}$, the components remain in $[0,1-\gamma]$.
If $k\ge M+2$ we may choose $\delta = \ssc{x}{1}/(k-1)$ and our new vector is
$(0,\frac{1}{k-1},\ldots,\frac{1}{k-1})$.  As $\theta>0$ and this vector has no
subsum in $\two$, \eqref{eq:A1} fails.
If $k=M+1$ we choose $\delta = 1-\gamma-\ssc{x}{2}$, so that
$(1-M(1-\gamma),1-\gamma,\ldots,1-\gamma)$ has no subsum in \two.  Thus, \eqref{eq:A2} fails.
\end{proof}

In light of Lemma \ref{lem:BDEF}, Theorem \ref{thm:asymptotic-R1-A1A2}
 will follow from the next result, which has a long proof
especially in the case $1-\gamma \le \theta$.

\begin{prop}\label{prop:A123}
Fix $(\gamma,\theta,\nu)\in \cQ_1$.
If \eqref{eq:A1} and \eqref{eq:A2} hold, then \emph{(A3)} holds.
\end{prop}

The proof occupies the next two subsections.
A few of our results are stated for the slightly larger set $\cQ$ as we
will need them for the proof of Theorem \ref{thm:continuity}.

\begin{defn}[The parameters $a$ and $k$]\label{def:ak}
Assuming \eqref{eq:A1} holds, we define two further parameters $a$ and $k$.  
Let $a$ be the smallest positive 
integer so that $\frac{a}{M+1} \in \two$, and let $k\ge 1$ be the unique integer
such that $\frac{a}{M+k+1}<\theta \le\frac{a}{M+k}$.  
\end{defn}
Assuming \eqref{eq:A1}, we see $\frac{a+1}{M+k+1} \in \two.$
It follows immediately that
\be\label{nu-lower}
\begin{split}
\nu &\ge \max \Big( \frac{a}{M+1}-\frac{a}{M+k}, \frac{a+1}{M+k+1}-\frac{a}{M+k}\Big)
\\ &= \max\Big( \frac{(k-1)a}{(M+1)(M+k)}, \frac{M+k-a}{(M+k+1)(M+k)} \Big).
\end{split}
\ee

\subsection{Proof of Proposition \ref{prop:A123} when $1-\gamma > \theta+\nu$.}
We begin with a preliminary observation.

\begin{lem}\label{lem:smallgamma}
Suppose that $(\gamma,\theta,\nu)\in \cQ$ and
 $1-\gamma \ge  \theta+\nu$.  Then \eqref{eq:A1} is equivalent to the interval
\two containing  $[\frac{1}{2M+1},\frac{1}{M+1}]$, $[\frac{1}{2M},\frac{2}{2M+1}]$
or $[\frac{1}{2M-1},\frac1{M}]$ (the last case only if $1-\gamma=\theta+\nu=\frac{1}{M}$).  Furthermore, \eqref{eq:A2} is equivalent to 
the assertion that $1-\gamma=\theta+\nu$ or $1-M(1-\gamma)\in \two$.
\end{lem}

\begin{proof}
By \eqref{Mdef},
\[
\theta + \nu \le 1-\gamma \le \frac1{M}.
\]
If $\theta+\nu= \frac{1}{M}$ then $1-\gamma=\frac{1}{M}$,
and we see that \eqref{eq:A1} holds if and only if $\frac{1}{2M-1}\in \two$
since $\frac{2}{2M-1} > \frac{1}{M} > \theta+\nu$.
If $\theta+\nu \in [\frac{2}{2M+1}, \frac{1}{M})$,  
then \eqref{eq:A1} holds if and only if $\frac{1}{2M} \in \two$,
since $\frac{1}{M+1} < \frac{2}{2M+1}$.
Finally, if $\theta+\nu < \frac{2}{2M+1}$, then \eqref{eq:A1} holds if and only if
$\two$ contains $\frac{1}{2M+1}$ and $\frac{1}{M+1}$.

 Hypothesis \eqref{eq:A2} asserts that there is a positive integer $h$ such that
 (a) $h(1-\gamma)\in \two$ or (b) $1-h(1-\gamma)\in \two$. 
If $1-\gamma=\theta+\nu$ then (a) holds with $h=1$. 
If $1-\gamma > \theta+\nu$, (a) is impossible and (b) is only possible 
 for $h=M$ since $1-(M+1)(1-\gamma) <0$ and $1 - (M-1)(1-\gamma)  \ge 1/M \ge 1-\gamma > \theta+\nu$.
 \end{proof}

\begin{lem}\label{DFE:small-gamma}
Assume that $(\gamma,\theta,\nu)\in \cQ$ with 
$1-\gamma \ge \theta+\nu$.  If \eqref{eq:A1} and \eqref{eq:A2} hold, then (A3) holds.
In particular, Proposition \ref{prop:A123} holds in the case $1-\gamma>\theta+\nu$.
\end{lem}

\begin{proof}
Assume \eqref{eq:A1} and \eqref{eq:A2}, and suppose that (A3) fails.
Then there is a vector $\bx=(\ssc{x}{1},\ldots,\ssc{x}{k})$ with $|\bx|=1$ and with no proper subsum in $\two$ satisfying (i), (ii), (iii) and (iv) from the definition of (A3).
In particular, for each $i$ we have $x_i \in (\nu,\theta) \cup (\theta+\nu,1-\gamma)$
and furthermore for all $i\ne j$, $x_i+x_j \ge 1-\gamma$ and either 
$x_i=x_j$ or $|x_i-x_j|>\nu$. (If $\theta<\nu$ this simply means all components lie in $(\theta+\nu,1-\gamma)$.)

In the case $1-\gamma=\theta+\nu=\frac{1}{M}$, all of the $x_i$ are in $(\nu,\theta)$. Also, Lemma \ref{lem:smallgamma}
implies that $\theta \le \frac{1}{2M-1}$ and thus, recalling that $M\ge 2$,
we have $\nu \ge \frac{1}{M}-\frac{1}{2M-1}
\ge \theta/2$.  By (iii), all of the  $x_i$ are equal, violating condition (iv),
a contradiction.

Assume now that $\theta+\nu < \frac{1}{M}$.
By Lemma \ref{lem:smallgamma},
$\two$ contains $\frac{1}{2M}$ and $\frac{1}{M+1}$, hence
\be\label{smallgamma-two}
\theta \le \frac{1}{2M} \le \frac{1}{M+1} \le \theta+\nu \le 1-\gamma \le \frac{1}{M}.
\ee
 If $k \le M$ then $\sum x_i < M(1-\gamma) \le 1$,
a contradiction.
If $k=M+1$ then suppose that there are $\ell$ variables $x_i$ lying in $(\nu,\theta)$.
If $\ell=0$, then, utilizing \eqref{smallgamma-two}, all components $x_i$ 
lie in $(\theta+\nu,1-\gamma)$ and hence $|\bx| > 1$. If $\ell\ge 2$ then 
$|\bx| < \frac{\ell}{2M} + \frac{M+1-\ell}{M} < 1$.  Hence $\ell=1$
and without loss of generality $\nu < \ssc{x}{1} < \theta$.  By Lemma \ref{lem:smallgamma}, 
\eqref{eq:A2} implies that $1-M(1-\gamma) \ge \theta > \ssc{x}{1}$
and thus
\[
|\bx| < \theta + M(1-\gamma) \le 1-M(1-\gamma) + M(1-\gamma) = 1,
\]
a contradiction.  Therefore, $k \ge M+2$.

Using \eqref{smallgamma-two} again, we see that $(M+1)(\theta+\nu)>1$, so there
are at most $M$ components $x_i$ in $(\theta+\nu,1-\gamma)$ and, consequently,
there are at least two components $x_i$ in $(\nu,\theta)$.
 Furthermore, by Lemma \ref{lem:smallgamma}, $\nu > \frac{1}{4M} \ge \theta/2$,
hence by (iii), all of the $x_i$ which are in $(\nu,\theta)$ are equal, and furthermore the common value is
$\ge \frac{1-\gamma}{2}$ by (ii).  By relabeling the components, we may suppose that $\ssc{x}{1}=\cdots=x_b \in [\frac{1-\gamma}{2},\theta)$, where $b\ge 2$.

By Lemma \ref{lem:smallgamma}, $\two$ contains either $[\frac{1}{2M},\frac{2}{2M+1}]$ or $[\frac{1}{2M+1},\frac{1}{M+1}]$ (recall that we have already handled the case $\theta+\nu=\frac{1}{M}$).  If
 $\two$ contains $[\frac{1}{2M},\frac{2}{2M+1}]$,
then $\ssc{x}{1} \ge \frac{1-\gamma}{2} > \frac{1}{2M+1}$ and thus all 
components $x_i$ lie in
\[
\Big(\frac{1}{2M+1},\frac{1}{2M}\Big) \cup \Big(\frac{2}{2M+1},\frac{2}{2M}\Big).
\]
Hence the sum of the $x_i$, namely 1, lies in $(\frac{h}{2M+1},\frac{h}{2M})$ for some integer
$h$, which is clearly impossible.

In the second case, suppose that $\two$ contains $[\frac{1}{2M+1},\frac{1}{M+1}]$.
Here we have $\ssc{x}{1} \ge \frac{1-\gamma}{2} > \frac{1}{2M+2}.$
Let $c=k-b$ be the number of components $x_i$ lying in $(\theta+\nu,1-\gamma)$.  Then we have
\[
1  = |\bx| > \frac{b}{2M+2} + \frac{c}{M+1} = \frac{b+2c}{2M+2},
\]
whence $b+2c \le 2M+1$.  Recall that \eqref{eq:A2} implies
$1-(1-\gamma)M \ge \theta$ (using Lemma \ref{lem:smallgamma} again), whence 
\[
1-\gamma \le \frac{1-\theta}{M}.
\]
Therefore,
\begin{align*}
1 &= |\bx| < b\theta + c(1-\gamma) \le b\theta + c\pfrac{1-\theta}{M}\\
&\le (2M+1-2c)\theta + c\pfrac{1-\theta}{M} = \frac{c}{M}+\theta \big(2M+1-2c-c/M\big).
\end{align*}
Since $b\ge 2$, $2c\le 2M-1$ and thus $2M+1-2c-c/M > 0$.  Therefore,
\[
1 < \frac{c}{M} + \frac{2M+1-2c-c/M}{2M+1} =1,
\]
a contradiction.  This completes the proof.
\end{proof}

\bigskip

\subsection{Proof of Proposition \ref{prop:A123} in the case $1-\gamma \le \theta$.}

\begin{lem}\label{lem:A3-ratio-less-one-half}
If \emph{(A3)} fails then $\nu < \frac{1-\gamma}{2} \le  \frac{1}{2M}$.
\end{lem}

\begin{proof}
By assumption, there is a vector $\bx$ with non-negative components satisfying 
$|\bx|=1$ and (i),(ii),(iii),(iv) of Hypothesis (A3), but with
no subsum in $\two$. In particular, by (iv)
there are at least two distinct values $x_i$.  By (i) and (iii), $1-\gamma>2\nu$.
In particular, $\nu < \frac{1-\gamma}{2} \le \frac{1}{2M}$. 
\end{proof}

\begin{lem}\label{lem:nu=1-2theta}
Assume that $(\gamma,\theta,\nu)\in \cQ_1$ and $\theta+\nu=1-\theta$.
Then \eqref{eq:A1} implies \emph{(A3)}.
\end{lem}

\begin{proof}
Let $n$ be the smallest odd integer
larger than $M$, so that $n\in \{M+1,M+2\}$.  Assuming \eqref{eq:A1}, and exploiting the symmetry of $\two$, we see that $\two$ contains both $\frac{(n-1)/2}{n}$ and $\frac{(n+1)/2}{n}$.  Thus, 
\[
\nu \ge \frac{1}{n} > \frac{1}{2M},
\]
which implies (A3) by Lemma \ref{lem:A3-ratio-less-one-half}.
\end{proof}

\medskip

It remains to prove Proposition \ref{prop:A123} in the case where $1-\gamma\le \theta$
and $\theta+\nu<\frac12$.  We note that the statements ``$P\in \cQ_1$ and
$1-\gamma\le \theta<\theta+\nu<\frac12$'' and ``$P\in \cQ$ and
$1-\gamma\le \theta<\theta+\nu<\frac12$'' are equivalent.  Although our main result,
Proposition \ref{prop:ratio} is stated for $P\in \cQ_1$, it follows that it holds
for $P\in \cQ$ as well.

Recall the definition of $a$ and $k$ from Definition \ref{def:ak}.
Since $\frac{1}{M+1} < 1-\gamma \le \theta \le \frac{a}{M+1} \le \theta+\nu < \frac12$ we have
\be\label{aMk}
2 \le a \le M/2, \qquad M\ge 4, \qquad k\ge 1.
\ee

\begin{prop}\label{prop:ratio}
Suppose that $(\gamma,\theta,\nu)\in \cQ_1$ and that
$1-\gamma \le \theta < \theta+\nu < \frac12$.  
Assume \eqref{eq:A1} and \eqref{eq:A2}.  Then $\nu > \frac{1-\gamma}{2}$ except
in the following cases:
\begin{enumerate}
\item[(a)] $1-\gamma=\frac4{33}$, $[\theta,\theta+\nu] = [\frac3{11},\frac13]$, so that
$\nu = \frac{1-\gamma}{2}$;
\item[(b)] $1-\gamma=\frac2{15}$, $\two = [\frac13, \frac25]$,  so that
$\nu = \frac{1-\gamma}{2}$;
\item[(c)] $1-\gamma=\frac3{22}$, $[\theta,\theta+\nu] = [\frac2{11},\frac14]$, so that
$\nu = \frac{1-\gamma}{2}$;
\item[(d)] $1-\gamma=\frac7{45}$, $\two = [\frac29, \frac3{10}]$,  so that
$\nu = \frac{1-\gamma}{2}$;
\item[(e)] $\frac16 \le 2\nu \le 1-\gamma \le \frac{4}{21}$, $\two$ contains
 $[\frac14, \frac13]$, and $4(1-\gamma)\le 1-\theta$; 
\item[(f)] $\frac{5}{28} \le 2\nu \le 1-\gamma \le \frac{4}{21}$, $\two$ contains $[\frac27,\frac38]$ and $2(1-\gamma)\le \theta+\nu$;
\item[(g)] $\frac4{21} \le 2\nu \le 1-\gamma \le \frac{8}{35}$, $\two$ contains
$[\frac13,\frac37]$ and $3(1-\gamma)\le 1-\theta$. 
\item[(h)] $\frac{8}{35} \le 2\nu \le (1-\gamma) \le \frac{6}{25}$,
$\two$ contains $[\frac27,\frac25]$ and $3(1-\gamma) \le 1-\theta$.
\end{enumerate}
Furthermore, in all cases (a)--(h) we have $\nu \ge \frac5{12}(1-\gamma)$,
and that cases (e)--(h) occur only when $a=2$ and $4\le M\le 6$.
\end{prop}

\begin{proof}
Throughout the proof, we assume that $\nu \le \frac{1-\gamma}{2}$.
In particular, $\nu \le \frac{1}{2M}$.
By \eqref{nu-lower},
\begin{align*}
\nu &\ge \frac12 \( \frac{(k-1)a}{(M+1)(M+k)} + \frac{M+k-a}{(M+k)(M+k+1)} \)
\\ &= \frac{1}{2(M+k+1)} \(1 - \frac{a}{M+k} + \frac{(k-1)a(M+k+1)}{(M+1)(M+k)} \).
\end{align*}
The right side is $> \frac{1}{2M}$ if and only if
\[
-a + \frac{(k-1)a (M+k+1)}{M+1} > \frac{(k+1)(M+k)}{M},
\]
equivalently
\be\label{nu**}
a \( k-2 + \frac{k(k-1)}{M+1}\) > k+1+\frac{k(k+1)}M.
\ee
Hence, if \eqref{nu**} holds then we reach a contradiction.

\textbf{The case $k\ge 5$.}
By \eqref{aMk}, $M\ge 4$ and $a\ge 2$, thus 
\begin{align*}
\frac{k(k+1)}{M}-\frac{k(k-1)a}{M+1} &= \frac{k(k+1)}{M}\(1- a\cdot \frac{k-1}{k+1}\cdot \frac{M}{M+1} \) \\
&\le \frac{k(k+1)}{M}\( 1 - 2 \cdot \frac{4}{6} \cdot \frac{4}{5} \) <0
\end{align*}
and $(k-2)a \ge 2k-4 \ge k+1$.  Thus, \eqref{nu**} holds.

\medskip

\textbf{The case $k=4$.}
The inequality \eqref{nu**} becomes
\[
a \(2+ \frac{12}{M+1}\)  > 5 + \frac{20}{M}.
\]
When $a\ge 3$ the left side above is at least $6+ \frac{36}{M+1} > 5 + \frac{20}{M}$
since $M\ge 4$ and thus \eqref{nu**} holds.
If $a=2$, \eqref{nu-lower} implies that
\[
\nu \ge \frac{6}{(M+1)(M+4)} > \frac{1}{2M} \qquad (4\le M\le 6)
\]
and 
\[
\nu \ge \frac{M+2}{(M+4)(M+5)} > \frac{1}{2M} \qquad (M\ge 8).
\]
Finally, when $a=2$, $M=7$, \eqref{eq:A1} implies that
$\two$ contains $[\frac2{11},\frac14]$
and hence $\nu \ge \frac{3}{44}$. Then
\[
\frac{3}{22} \le 2\nu \le 1-\gamma \le \frac17, 
\qquad \two \subseteq \Big[ \frac14-\frac1{14}, \frac{2}{11}+\frac{1}{14}  \Big]=\Big[\frac{5}{28},\frac{39}{154}  \Big].
\]
We thus have $1-\gamma < \theta$, $2(1-\gamma)>\theta+\nu$,
$5(1-\gamma) < 1-\theta-\nu$ and $7(1-\gamma) > 1-\theta$.
Hence, \eqref{eq:A2} implies that $6(1-\gamma) \in [1-\theta-\nu, 1-\theta]$.
Since $\nu \ge \frac14-\theta$, we get
\[
12 \nu \le 6(1-\gamma) \le 1-\theta \le \frac34 + \nu,
\]
which is only possible if every inequality is an equality, i.e., $\nu=\frac3{44}$,
$1-\gamma = \frac{3}{22}$, $\theta = \frac{2}{11}$.  This is item (c).

\medskip

\textbf{The case $k=3$.}
When $a\ge 4$ the left side of \eqref{nu**}
is at least $a(1+\frac{6}{M+1}) > 4 + 12/M$, so \eqref{nu**} holds.
When $a=3$, \eqref{nu-lower} gives
\begin{align*}
\nu &\ge \frac{6}{(M+1)(M+3)} > \frac{1}{2M} \qquad (4\le M\le 7),\\
\nu &\ge \frac{M}{(M+3)(M+4)} > \frac{1}{2M} \qquad (M\ge 9).
\end{align*}
When $a=3$ and $M=8$, \eqref{eq:A1} implies that
$\two \supseteq [\frac3{11},\frac13]$ and
$\nu \ge \frac{2}{33}$.  Then
\[
\frac{4}{33} \le 2\nu \le 1-\gamma \le \frac18, 
\qquad \two \subseteq \Big[ \frac13-\frac1{16}, \frac{3}{11}+\frac{1}{16} \Big]=\Big[\frac{13}{48},\frac{59}{176}\Big].
\]
We thus have $2(1-\gamma) < \theta$, $3(1-\gamma)>\theta+\nu$,
$5(1-\gamma) < 1-\theta-\nu$ and $7(1-\gamma) > 1-\theta$.
Hence, \eqref{eq:A2} implies that $6(1-\gamma) \in [1-\theta-\nu, 1-\theta]$.
Since $\nu \ge \frac13-\theta$, we get
\[
12 \nu \le 6(1-\gamma) \le 1-\theta \le \frac23 + \nu,
\]
which is only possible if every inequality is an equality, i.e., 
 $\nu=\frac2{33}$, $\theta=\frac{3}{11}$, $1-\gamma=\frac{4}{33}$.  This is item (a).

When $a=2$ and $M\ge 7$ then \eqref{nu-lower} implies
\[
\nu \ge \frac{M+1}{(M+3)(M+4)} > \frac{1}{2M}.
\]

If $M=6$ then $\two$ contains $[\frac29,\frac3{10}]$ and hence $\nu\ge \frac{7}{90}$.
Thus, $\frac{7}{45} \le 2\nu \le 1-\gamma\le \frac16$.  As $\nu\le \frac1{12}$,
we have $\two \subseteq [0.216,0.306]$.  It follows that $1-\gamma < \theta$,
$2(1-\gamma) > \theta+\nu$, $4(1-\gamma)\le \frac23 < 1-\theta-\nu$,
 and $6(1-\gamma) > 1-\theta$.  By \eqref{eq:A2},
we must have $5(1-\gamma) \in [1-\theta-\nu, 1-\theta]$.  Thus,
\[
10\nu \le 5(1-\gamma) \le 1-\theta \le \frac7{10}+\nu,
\]
whence $\nu \le \frac{7}{90}$.  This implies that $\nu=\frac{7}{90}$,
$1-\gamma = 2\nu$ and $\two=[\frac29,\frac3{10}]$.  This is item (d).

If $M=5$ then $\two$ contains $[\frac14,\frac13]$, $\frac16 \le 2\nu \le 1-\gamma \le \frac15$ and we also have $1-\gamma>\frac16$.  
Consequently, $\two \subseteq [0.233,0.35]$.
We have $1-\gamma < \theta$.  Also $2(1-\gamma)> \theta+\nu$, since otherwise
we have $4\nu \le \theta+\nu \le \frac14 + \nu$, then $\nu\le \frac{1}{12}$,
then  $\theta+\nu=\frac13\ge 2(1-\gamma)$, a contradiction.  Now $\frac12 < 3(1-\gamma) \le \frac35 < 1-\theta-\nu$, $5(1-\gamma)>1-\theta$ and $4(1-\gamma) > \frac23 \ge 1-\theta-\nu$,
thus \eqref{eq:A2} is equivalent to  $4(1-\gamma) \le 1-\theta$.  This implies that
$8\nu \le 4(1-\gamma) \le 1-\theta \le \frac23+\nu$, so $\nu \le \frac{2}{21}$ and
consequently
\[
1-\gamma \le \frac{1-\theta}{4}=\frac{1-(\theta+\nu)+\nu}{4} \le \frac{2/3+2/21}{4}=\frac{4}{21}.
\]
 This gives most of item (e), excluding only the case $1-\gamma=\frac16$.

When $M=4$, $\two$ contains $[\frac27,\frac25]$ and thus
$\frac8{35} \le 2\nu \le 1-\gamma \le \frac14$.  Since $\two \subseteq [0.275,0.4108]$, we have $1-\gamma < \theta$, $\theta+\nu< 2(1-\gamma) \le \frac12$ and $4(1-\gamma)>1-\theta$.  By \eqref{eq:A2}, we must have $3(1-\gamma)\le 1-\theta$.
This implies that $6\nu \le 1-\theta \le \frac35+\nu$, so that $\nu \le \frac{3}{25}$
and $1-\gamma \le \frac{1-\theta}{3} \le  \frac{3/5+3/25}{3}=\frac{6}{25}$. 
 This is item (h).

\medskip

\textbf{The case $k=2$.}
If $a\le \frac{M-2}{2}$ then \eqref{aMk} implies that $M\ge 6$ and hence \eqref{nu-lower} gives
\[
\nu \ge  \frac{M/2+3}{(M+2)(M+3)} = \frac{1}{2M} \cdot \frac{1+6/M}{(1+5/M+6/M^2)} \ge \frac{1}{2M}.
\]
Hence, $\nu > \frac{1-\gamma}{2}$ except in the case $M=6$, $a=2$, $1-\gamma=\frac16$,
and $\two=[\frac14,\frac13]$, which is the final part of item (e).

Thus we may assume that $a\in\{\frac{M-1}{2},\frac{M}{2}\}$ according to the parity of $M$. First suppose $M$ is odd, so $M\ge 5$ and $a=\frac{M-1}{2}$. Then
$\two$ contains $\frac{a}{M+k}=\frac{(M-1)/2}{M+2}$ and $\frac{a+1}{M+k+1}=\frac{(M+1)/2}{M+3}$.  When $M\ge 9$,
\[
\frac{(M+1)/2}{M+5} < \frac{(M-1)/2}{M+2}, \qquad \frac{(M+3)/2}{M+5} > \frac{(M+1)/2}{M+3},
\]
and thus, by \eqref{eq:A1}, $\two$ also contains $\frac{(M+1)/2}{M+5}$ or $\frac{(M+3)/2}{M+5}$.  Hence,
\begin{align*}
\nu &\ge \min \bigg( \frac{(M+1)/2}{M+3} - \frac{(M+1)/2}{M+5}, \frac{(M+3)/2}{M+5} - \frac{(M-1)/2}{M+2} \bigg)\\
&= \min \bigg( \frac{M+1}{(M+3)(M+5)}, \frac{M+11}{2(M+2)(M+5)} \bigg)\\
&> \frac{1}{2M},
\end{align*}
a contradiction.

When $M=7$ and $a=3$, $\two$ contains $[\frac13,\frac25]$ and $\nu \ge \frac{1}{15}$.
Then $\frac{2}{15} \le 2\nu \le 1-\gamma \le \frac17$.
Consequently, $\theta \ge \frac25-\nu \ge 0.328$ and $\theta+\nu\le \frac13+\nu\le 0.405$.
With these restrictions, $2(1-\gamma)<\theta$, $\theta+\nu < 4(1-\gamma) < 1-\theta-\nu$ and $6(1-\gamma) > 1-\theta$.  By \eqref{eq:A2}, we have either $3(1-\gamma) \le \theta+\nu$ 
or $5(1-\gamma) \le 1-\theta$.  In the former case,
\[
6\nu \le 3(1-\gamma) \le \theta+\nu \le \frac13+\nu
\]
and in the latter case,
\[
10\nu \le 5(1-\gamma) \le 1-\theta \le \frac35+\nu.
\]
In either case, we have $\nu \le \frac{1}{15}$, and thus $\nu=\frac{1}{15}$,
 $\two=[\frac13,\frac25]$ and $1-\gamma=\frac{2}{15}$.  This is item (b).

Consider the case $M=5$, $a=2$.  Then $\two$ contains $[\frac27,\frac38]$ and
$\frac5{28} \le 2\nu\le 1-\gamma\le \frac15$.
If $1-\gamma \le \frac{3}{16}$ then we have $2(1-\gamma) \in \two$ (in particular, \eqref{eq:A2} always holds) and this is part of item (f).  Now suppose that $\frac3{16} < 1-\gamma \le \frac15$.
Since $\nu \le \frac{1}{10}$ we have $\two \subseteq [0.275,\frac{27}{70}]$.
Thus, $1-\gamma < \theta$, $\frac12 < 3(1-\gamma) < 1-\theta-\nu$ and $4(1-\gamma) \ge \frac34 > 1-\theta$.  Also, $2(1-\gamma) > \frac38$, hence by \eqref{eq:A2}, we must have
$2(1-\gamma) \le \theta+\nu$.  Also,
$4\nu \le 2(1-\gamma) \le \theta+\nu\le \frac27+\nu$ implies $\nu \le \frac{2}{21}$
and $2(1-\gamma) \le \frac27+\frac2{21} = \frac{8}{21}$.
 This is the other part of item (f).

Now assume that $M$ is even, so $M\ge 4$ and $a=\frac{M}{2}$.
Here, $\two$ contains $\frac{a}{M+2}=\frac12-\frac{1}{M+2}$ and
$\frac{a+1}{M+3}=\frac12 -\frac1{2M+6}$.  By \eqref{eq:A1}, $\two$
contains either $\frac{a+1}{M+5} = \frac12-\frac{3}{2M+10}$ or 
$\frac{a+2}{M+5}=\frac12-\frac1{2M+10}$ since 
$\frac{3}{2M+10} > \frac{1}{M+2}$ and $\frac{1}{2M+10} < \frac{1}{2M+6}$.
Hence
\begin{align*}
\nu &\ge \min \Big( \frac{3}{2M+10}-\frac{1}{2M+6}, \frac{1}{M+2} - \frac{1}{2M+10} \Big) \\
&= \min \Big( \frac{M+2}{(M+5)(M+3)}, \frac{M+8}{2(M+2)(M+5)} \Big).
\end{align*}
When $M\ge 12$ this implies $\nu > \frac{1}{2M}$.
When $M=10$ this gives $\nu \ge \frac{1}{20}$, and so $\nu> \frac{1-\gamma}{2}$
unless $\nu=\frac{1}{20}$, $1-\gamma=\frac{1}{10}$ and $\two=[\frac5{12},\frac{7}{15}]$.  But then $h(1-\gamma)\not\in \two \cup [1-\theta-\nu,1-\theta]$ for all $h\in \NN$, so that \eqref{eq:A2} fails.
  When $M=8$, $\two$ contains $[\frac{5}{13},\frac5{11}]$ or $[\frac25,\frac{6}{13}]$.  In the former case, $\nu > \frac{1}{16}=\frac{1}{2M}$ and in the latter, $\nu \ge \frac{4}{65}$ and $\two$ contains $[\frac25,\frac6{13}]$.  Since $\nu \le \frac{1-\gamma}{2}$ then $\frac{8}{65} \le 2\nu \le 1-\gamma \le \frac18$.  But then $h(1-\gamma)\not\in \two \cup [1-\theta-\nu,1-\theta]$ for all $h\in \NN$, so that \eqref{eq:A2} again fails.
 When $M=6$, $\two$ contains either $[\frac4{11},\frac49]$ or $[\frac38,\frac{5}{11}]$,
 and thus $\frac7{44} \le 2\nu \le 1-\gamma \le \frac16$.  Again, \eqref{eq:A2} fails.

 This leaves the case $M=4$.  Here, $\two$ contains $[\frac13,\frac37]$,
  $\frac4{21} \le 2\nu \le 1-\gamma$
 and $\frac15 < 1-\gamma \le \frac14$.  If $1-\gamma \le \frac29$,
 then $3(1-\gamma) \le 1-\theta$ and so we are in case (g). 
 If $\frac29 < 1-\gamma \le \frac14$, then \eqref{eq:A2} requires
 either $2(1-\gamma) \le \theta+\nu$ or $3(1-\gamma) \le 1-\theta$.
 In the first case, $4\nu \le \theta+\nu$, so $\nu \le \theta/3 \le \frac19$
 and then $2(1-\gamma) \le \frac43 \theta  \le \frac49$, a contradiction.
 Hence, $3(1-\gamma) \le 1-\theta$,
 which implies that $6\nu\le 3(1-\gamma)\le 1-(\theta+\nu)+\nu\le \frac47+\nu$,
 and hence $\nu\le \frac{4}{35}$ and $1-\gamma\le \frac{4/7+\nu}{3} \le \frac{8}{35}$.
 Thus we are also in case (g).
 
 \medskip
 
\textbf{The case $k=1$.}
By \eqref{nu-lower}, 
\[
\nu \ge \frac{M+1-a}{(M+1)(M+2)}.
\]
If $a\le \frac{M}{2}-1$ then the numerator is at least $\frac{M+4}{2}$ and we get
\[
\nu \ge \frac{1+4/M}{2M(1+1/M)(1+2/M)} > \frac{1}{2M},
\]
as desired.
If $a=\frac{M}{2}$ then $\two$ contains $\frac{a+1}{M+2}=\frac12$, contradicting
that we are in the case $\theta+\nu < 1/2$.

Finally, suppose that $a=\frac{M-1}{2}$, so that $M\ge 5$ and $M$ is odd.
Then $\two$ contains $[\frac{(M-1)/2}{M+1},\frac{(M+1)/2}{M+2}]$.
Since $\frac{(M+1)/2}{M+4} \le \frac{(M-1)/2}{M+1}$ (with equality  when $M=5$)
and $\frac{(M+3)/2}{M+4} > \frac{(M+1)/2}{M+2}$, by \eqref{eq:A1} one of the numbers
$\frac{(M+1)/2}{M+4}, \frac{(M+3)/2}{M+4}$ also lies in $\two$.
Therefore,
\begin{align*}
\nu &\ge \min\bigg( \frac{(M+1)/2}{M+2} - \frac{(M+1)/2}{M+4}, \frac{(M+3)/2}{M+4}-
\frac{(M-1)/2}{M+1} \bigg)\\
&= \min \bigg( \frac{M+1}{(M+2)(M+4)}, \frac{M+7}{2(M+1)(M+4)} \bigg).
\end{align*}
When $M\ge 7$, this shows that $\nu > \frac{1}{2M}$.
When $M=5$, $\two$ contains $[\frac13,\frac37]$
and we have
$\frac4{21} \le 2\nu \le 1-\gamma \le \frac15$.  Since $2(1-\gamma)\in \two$,
\eqref{eq:A2} holds for all such choices of parameters. 
Furthermore, $3(1-\gamma) \le \frac35 \le 1-\theta$.
 Thus we are also in case (g).

The claim that cases (e)--(h) occur only when $a=2$ and $4\le M\le 6$ 
follows from the above case-by-case analysis. We also have in each case (e)--(h) the following bounds:
\begin{itemize}
\item[(e)] here $\frac{\nu}{1-\gamma} \ge \frac{1/12}{4/21}=\frac{7}{16}>\frac5{12}$;
\item[(f)] here $\frac{\nu}{1-\gamma} \ge \frac{5/56}{4/21}>\frac5{12}$;
\item[(g)] here $\frac{\nu}{1-\gamma} \ge \frac{2/21}{8/35} = \frac5{12}$; 
\item[(h)] here $\frac{\nu}{1-\gamma} \ge \frac{4/35}{6/25} > \frac5{12}$.
\end{itemize}
This completes the proof of Proposition \ref{prop:ratio}.
\end{proof}

Since we will need it later in the proof of Theorem \ref{thm:continuity}, we prove a slightly stronger form of the
statement ``\eqref{eq:A1} and \eqref{eq:A2} implies (A3)''.
As remarked earlier, although the next Proposition is stated for $P\in \cQ_1$
it also holds for $P\in \cQ$.

\begin{prop}\label{prop:A123-epsilon-tweak}
Fix $P_0=(\gamma,\theta,\nu)\in \cQ_1$ with $1-\gamma \le \theta < \theta+\nu < \frac12$.
\begin{itemize}
\item[($\alpha$)] If \eqref{eq:A1} and \eqref{eq:A2} hold for $P_0$ then \emph{(A3)} also holds for $P_0$.

\item[($\beta$)] Suppose that \eqref{eq:A1} and \eqref{eq:B} hold for $P_0$.
For $\eps>0$, let $P_\eps = (\gamma-\eps,\theta+\eps,\nu-2\eps).$
Then for some $\epszero > 0$ and all
$0\le \eps \le \epszero$, \emph{(A3)} is true for $P_\eps$,
except when $P_0=(\frac56,\frac14,\frac1{12})$.
\end{itemize}
\end{prop}

Recall that hypothesis \eqref{eq:B} holding for $P_0$ implies that there is an $\epszero>0$ so that \eqref{eq:A2} holds for $0\le \eps\le \epszero$.

\begin{proof}
Assume that \eqref{eq:A1} and \eqref{eq:A2} hold for $P_0$ 
(recall that (B) implies \eqref{eq:A2}).
If $\nu > \frac{1-\gamma}{2}$, then for small enough $\epszero>0$,
for all $0\le \eps\le \epszero$ we have
$\nu -2\eps > \frac{1-\gamma+\eps}{2}$ and parts ($\alpha$) and ($\beta$) follow from Lemma 
\ref{lem:A3-ratio-less-one-half}.
Part ($\alpha$) also follows when $\nu = \frac{1-\gamma}{2}$, again using
Lemma \ref{lem:A3-ratio-less-one-half}.
By Proposition \ref{prop:ratio}, it thus suffices to prove ($\alpha$)
in cases (e)--(h) and ($\beta$) in cases (a)--(h).  Note that in all
of these cases,
\[
\frac5{12}(1-\gamma) \le \nu \le \frac{1-\gamma}{2}.
\]
Let $\epszero$ he sufficiently small, $0\le \eps \le \epszero$,
and suppose that $\bu$ is a vector satisfying the four conditions (i)--(iv) 
 of Hypothesis (A3) for $P_\eps$.
Since $\nu > \frac{1-\gamma}{3}$,
if $\eps_0$ is small enough then $\bu$ has exactly
two distinct components, $x$ and $y$ in $(\nu-2\eps,1-\gamma+\eps)$
with $y-x>\nu-2\eps$. Suppose $\bu$ has $b$ components equal to $x$ and $c$ components equal to $y$,
so that $bx+cy=1$, $b\ge 1$ and $c\ge 1$.
 Then
\be\label{xy}
\nu-2\eps < x < y < 1-\gamma+\eps, \qquad y-x>\nu-2\eps, \qquad 
\text{ if } b\ge 2 \text{ then } x \ge \frac{1-\gamma+\eps}{2}.
\ee
To prove ($\alpha$), we will show that if $\eps=0$ and $\bu$ exists,
then $\bu$ has a proper subsum in $\two$.
To prove ($\beta$), we will show that if $\eps_0$ is small enough, $0\le \eps\le \epszero$,
 $\bu$ exists and (B) holds for $P_0$
 then $\bu$ has a proper subsum in $[\theta+\eps,\theta+\nu-\eps]$.

In cases (a)--(d) of Proposition \ref{prop:ratio}, write
 $1-\gamma = \frac{d}{e}$, where $d,e$ are integers
with  $(d,e)=1$.  Since $1-\gamma=2\nu$, \eqref{xy} implies that $x= \frac{d}{2e}+O(\eps)$ and
$y = \frac{d}{e}+O(\eps)$, and thus $bd+2cd=2e$ if $\eps$  is small enough.  
In particular, $d|2e$, which
does not occur in cases (a),(c) and (d), hence $\bu$ does not exist.
  In case (b), (B) fails for $P_0$.
Thus, part ($\beta$) holds in cases (a)--(d).

Next, suppose we are in case (e).  
By \eqref{xy}, $x+y< \frac8{21}-\frac1{12}+4\eps<0.3$ for small enough $\eps$.
Thus, if $x+y \ge \frac14+\eps$ then $x+y\in [\theta+\eps,\theta+\nu-\eps]$,  as required.  Now assume that $x+y < \frac14+\eps$.  By \eqref{xy}, 
$x+y>3(\nu-2\eps)$, so this is impossible 
if $\eps=0$, thus proving ($\alpha$).  Also, $x+y<\frac14+\eps$ is impossible if
$\nu>\frac1{12}$ and $\eps$ is small enough, hence $\nu=\frac1{12}$ and 
$\two=[\frac14,\frac13]$.  
If $1-\gamma>\frac16$, \eqref{xy} implies
 we have
 \[
 \frac14+\eps > x+y > 1-\gamma + \frac1{12} - \eps,
 \]
 which is false for small enough $\eps$.  Hence, $1-\gamma = \frac16$.
 In this special case $\two=[\frac14,\frac13]$, $1-\gamma=\frac16$, (B) holds
 and (A3) fails with the choice $x=\frac1{12}+\frac{\eps}{2},y=\frac16-\frac{\eps}{5},b=2,c=5$.
This proves ($\beta$).

In case (f), \eqref{xy} implies, for small enough $\eps$, that
\[
0.0892 < x < 0.1012, \qquad 0.1785 < y < 0.1905.
\]
Since $bx+cy=1$, either $c\ge 3$ or $b\ge 7$.
If $c\ge 3$ and $y\le 0.1874$ then $2y\in [\frac27+\eps,\frac38-\eps]$.  Hence, $c\ge 3$
implies $0.1874 < y < 0.1905$.  However, for all $c\in\{3,4,5\}$, the range of $1-cy$
does not contain an integer multiple of $x$.  Therefore, 
$c\le 2$ and $b\ge 7$.  If $x\le 0.0937$ then $4x \in [\frac27+\eps,\frac38-\eps]$,
and if $x\ge 0.0953$ then $3x \in [\frac27+\eps,\frac38-\eps]$, hence
$0.0937<x<0.0953$.  In both cases $c\in \{1,2\}$, the range of $1-cy$
does not contain an integer multiple of $x$.
Thus, $\bu$ does not exist and hence ($\alpha$) and ($\beta$) hold in case (f).

Suppose we are in case (h).  By \eqref{xy} we have
\[
0.1142 < x < 0.1258, \qquad 0.2285 < y < 0.2401.
\]
In all possible cases $1\le c\le 4$, we see that $1-cy$ cannot equal any
integer multiple of $x$.  So $\bu$ does not exist,
and hence ($\alpha$) and ($\beta$) hold in case (h).

Finally, consider case (g).  By \eqref{xy},
\[
0.0953 <x < 0.1334, \quad 0.1904 <y < 0.2286
\]
and hence
$x+y < 0.362$, that is, well below $\frac37$. Thus,
 if $x+y \ge \theta+\eps$ then we have a contradiction.  Therefore, 
\[
x+y < \theta+\eps \le \frac13+\eps.
\]
We cannot have $c\ge 5$ since $5y \ge 0.952$.  Hence we have one
of the following:
\begin{itemize}
\item[(i)] $b=c=3$, $x+y=\frac13$, $\theta=\frac13$;
\item[(ii)] $b\ge 4$; or
\item[(iii)] $c=4$.
\end{itemize}

In case (i), since $3x\notin [\theta+\eps,\theta+\nu-\eps]$, $x<\frac19+\frac{\eps}{3}$ and $y> \frac29-\frac{\eps}{3}$.
Since $y< 1-\gamma + \eps$ and $3(1-\gamma)\le 1-\theta$ (the latter is part of 
condition (g)), we have
$3(1-\gamma)=1-\theta$ and consequently $1-\gamma=\frac29$,
$y=\frac29+O(\eps)$ and $x=\frac19+O(\eps)$.
As $\nu \le \frac{1-\gamma}{2}=\frac19$, $\theta+\nu \le \frac49$.
Thus, (B) fails.
When $\eps=0$, no such $x$ exists, so $\bu$ doesn't exist.
This is sufficient for both parts ($\alpha$) and ($\beta$).

For case (ii),  if $x\ge \frac19 + \frac{\eps}{3}$ then 
$3x \in  [\theta+\eps,\theta+\nu-\eps]$ and if $x \le \frac{3}{28}-\frac{\eps}{4}$
then $4x \in  [\theta+\eps,\theta+\nu-\eps]$.
Thus, $x\in (\frac{3}{28}-\frac{\eps}{4},\frac19+\frac{\eps}{3})$.
Since $y<\frac13+\eps-x$ and $y>x+\nu-2\eps$ we have
$0.2023<y<0.2262$.  Then
\[
1 = bx+cy \in (0.101(b+2c),0.114(b+2c))
\]
which implies that $b+2c=9$.
Thus, $b=5,c=2$ or $b=7,c=1$ and in both cases, $y > \frac29 - 3\eps$
and hence $x>0.1095$.
When $\eps=0$ it follows that $x<\frac19$ and $y>\frac29$, so
$b/c > 1$ gives 
\[
x+y = x + \frac{1-bx}{c} > \frac19 \bigg( 1 - \frac{b}{c} \bigg) + \frac{1}{c} = \frac13,
\]
which is impossible.  Thus, $\bu$ doesn't exist.
Now assume (B) holds and $0<\eps \le \epszero$.  Since $y> \frac29-3\eps$,
$1-\gamma \ge \frac29$.  As $1-\gamma < \theta < \theta+\nu < 2(1-\gamma)$,
(B) implies that $3(1-\gamma) < 1-\theta$.  So, for some $\delta>0$,
$3(1-\gamma) = 1-\theta-\delta$, so that $1-\theta \ge \frac23+\delta$.
We then have $x+y=x+\frac{1-bx}{c}=x(3-\frac{9}{c})+\frac1{c} < \theta+\eps$ and hence
\begin{align*}
x > \frac{1/c-\theta-\eps}{9/c-3} \ge \frac{1/c-1/3+\delta-\eps}{9/c-3} = \frac19 +
\frac{\delta-\eps}{9/c-3}.
\end{align*}
This contradicts $x<\frac19+\frac{\eps}{3}$ for $\eps< \delta/10$, say.  
Thus, $\bu$ doesn't exist in case (ii).

For case (iii), since $4y > \frac{16}{21}-O(\eps)$ and $x\ge \frac2{21}-\eps$,
we have $b=1$ or $b=2$.  If $b=2$ then $y=\frac{1-2x}{4} \in [0.183,0.203]$
and we get $2y\in [\theta+\eps,\theta+\nu-\eps]$.
Now suppose that $b=1$.  The relation $1-3y=x+y < \theta+\eps \le \frac13+\eps$
implies that $y > \frac29 - \frac{\eps}{3}$, and hence $1-\gamma \ge \frac29$,
$x< \frac19+ \frac{4\eps}{3}$ and $\nu \le \frac19$.
Since $1-\gamma \le \frac{8}{35} < \theta$ and $2(1-\gamma) \ge \frac49$
and $3(1-\gamma) \ge \frac23$, \eqref{eq:A2} implies that 
$h(1-\gamma)\in [\theta,\theta+\nu]\cup[1-\theta-\nu,1-\theta]$ for 
 $h=2$ or $h=3$.
If $h=2$ then $1-\gamma=\frac29$ and $\nu=\frac19$.
If $\eps=0$ then $y>\frac29$ is not possible, 
and we also note that (B) fails.
Thus, $h=3$.  If $\eps=0$ then
\[
1-3y = x+y < \theta \le 1-3(1-\gamma),
\]
and so $y > 1-\gamma$, a contradiction.
If $\eps>0$, (B) implies that
 $3(1-\gamma) = 1-\theta-\delta$, where $\delta > 0$.  Then
\[
1-3y = x+y < \theta+\eps \le 1-3(1-\gamma)+\eps-\delta,
\]
whence $y\ge 1-\gamma + \frac{\delta-\eps}{3}$, again a contradiction
if $\eps < \delta/10$, say.  This suffices for ($\alpha$) and ($\beta$).
\end{proof}

\begin{proof}[Proof of Proposition \ref{prop:A123} in the case $1-\gamma\le 1-\theta$]
When $\theta+\nu=1-\theta$, the desired result follows from Lemma \ref{lem:nu=1-2theta}.
If $1-\gamma \le \theta < \theta+\nu < \frac12$,
the result follows from Proposition \ref{prop:A123-epsilon-tweak}, part ($\alpha$).
\end{proof}

\bigskip

%
%
%
{\Large \section{Continuities and discontinuities}\label{sec:continuity}}
%
%
%

In this section we prove Theorems \ref{thm:continuity}
and \ref{thm:continuity-1/2}.  Recall that for $P=(\gamma,\theta,\nu)$ fixed, 
\[
P_\eps := (\gamma-\eps, \theta+\eps, \nu-2\eps).
\]

\subsection{Discontinuities}\label{sec:discontinuities}

In this subsection we prove Theorem \ref{thm:continuity}
in the case where (B) fails and $\theta>0$.  Our goal is to prove that
\begin{equation}\label{eq:dicontinuity-1}
\sup_{\eps>0} C^-(P_\eps) < 1 < \inf_{\eps>0} C^+(P_\eps).
\end{equation}

We begin with some lemmas needed in the proof.
The first lemma is closely related to the problem of counting the number of integers 
$\le x$ with exactly $k$ prime factors.

\begin{lem}\label{lem:mk}
For $k\ge 1$, $\eps>0$ and $y\in \RR$ define
\[
m_k(y,\eps) := \mint{\eps \le \xi_1 \le \cdots \le \xi_k \\ \xi_1+\cdots+\xi_k=y}
\frac{d\bxi}{\xi_1\cdots \xi_k}.
\]
(a) For all $k\ge 1$ and $y \ge \eps >0$ we have
\[
m_k(y,\eps) \le \frac{(\log \frac{y}{\eps})^{k-1}}{(k-1)! y}.
\]
(b) For any fixed $A\ge 1$ and uniformly for $y\ge 100A^2 \eps>0$ and $1\le k\le A \log(y/\eps)$, we have
\[
m_k(y,\eps) \gg_A \frac{(\log \frac{y}{\eps})^{k-1}}{(k-1)! y}.
\]
\end{lem}

\begin{proof}
We begin with
\begin{align}
m_k(y,\eps) &= \frac{1}{k!} \mint{\xi_1,\ldots,\xi_k \ge \eps \\ \xi_1+\cdots+\xi_k=y}
\frac{d\bxi}{\xi_1\cdots \xi_k} 
= \frac{1}{yk!} \mint{\xi_1,\ldots,\xi_k \ge \eps \\ \xi_1+\cdots+\xi_k=y}
\frac{\xi_1+\cdots+\xi_k}{\xi_1\cdots \xi_k}d\bxi \notag \\
&= \frac{1}{y(k-1)!}  \mint{\xi_1,\ldots,\xi_k \ge \eps \\ \xi_1+\cdots+\xi_k=y}
\frac{d\bxi}{\xi_1\cdots \xi_{k-1}}. \label{mky1}
\end{align}
The multiple integral on the right side of \eqref{mky1} is at most
$(\int_\eps^y d\xi/\xi)^{k-1} = (\log \frac{y}{\eps})^{k-1}$
and this proves (a).

When $k=1$ and $y\ge \eps$ we have $m_1(y,\eps) = 1/y$.  
Now suppose that $k\ge 2$ and let $B=10A$, so that $y \ge B^2 \eps$.  
On the right side of \eqref{mky1}, we have $\xi_1+\cdots+\xi_{k-1}\le y-\eps$.
Thus, we obtain a lower bound for $m_k(y,\eps)$ by integrating over $(\xi_1,\ldots,\xi_{k-1})$ such that $\xi_1+\cdots+\xi_{k-1}\le y-\eps$ and $\eps \le \xi_i \le y/B$ for $1\le i\le k-1$.  This implies that
\begin{align*}
m_k(y,\eps)
&\ge \frac{1}{y(k-1)!}  \mint{\eps \le \xi_1,\ldots,\xi_{k-1} \le y/B} \frac{1}{\xi_1\cdots \xi_{k-1}} \bigg( 1- \frac{\xi_1+\cdots+\xi_{k-1}}{y-\eps}\bigg)d\xi_1 \cdots d\xi_{k-1} \\
&= \frac{I_1 - I_2}{y(k-1)!},
\end{align*}
say.  Here 
\[
I_1 =  \mint{\eps \le \xi_1,\ldots,\xi_{k-1} \le y/B} \frac{d\xi_1\cdots d\xi_{k-1}}{\xi_1\cdots \xi_{k-1}} 
= \big( \log \tfrac{y}{B\eps} \big)^{k-1}
\]
and 
\begin{align*}
I_2 &= \frac{k-1}{y-\eps} \mint{\eps \le \xi_1,\ldots,\xi_{k-1} \le y/B} \frac{d\xi_1\cdots d\xi_{k-1}}{\xi_1\cdots \xi_{k-2}} \\
&= (k-1) \Big( \frac{y/B-\eps}{y-\eps} \Big) \Big( \log \frac{y}{B\eps} \Big)^{k-2}\\
&\le \Big(\frac{k-1}{B}\Big)  \Big( \log \frac{y}{B\eps} \Big)^{k-2}.
\end{align*}
It follows that
\[
m_k(y,\eps) \ge \frac{\big(\log \frac{y}{B\eps}\big)^{k-1}}{y(k-1)!}
\bigg[ 1 - \frac{k-1}{B \log (\frac{y}{B\eps})} \bigg].
\]
By the given range of $y$ and $k$,
\begin{align*}
k &\le A\log(y/\eps) = 2A \log \Big( \tfrac{y}{\eps} \sqrt{\tfrac{\eps}{y}}\Big)
\le 2A \log \big( \tfrac{y}{B\eps} \big) = \tfrac{B}{5}  \log \big( \tfrac{y}{B\eps} \big)
\end{align*}
and we get that
\begin{equation}\label{mky2}
m_k(y,\eps) \ge \frac45\,\cdot \, \frac{\big(\log \frac{y}{B\eps}\big)^{k-1}}{y(k-1)!}.
\end{equation}
since $y/\eps \ge B^2$ and $k\le A \log(y/\eps)$, we have
\begin{align*}
\frac{\big(\log \frac{y}{\eps} \big)^{k-1}}{\big( \log \frac{y}{B\eps} \big)^{k-1}}
&= \bigg(1 + \frac{\log B}{\log\frac{y}{\eps} - \log B} \bigg)^{k-1}\\
&\le \exp \bigg\{ (k-1) \frac{2\log B}{\log \frac{y}{\eps}} \bigg\} \le e^{2A\log B}=B^{2A}.
\end{align*}
Combined with \eqref{mky2}, this proves (b).
\end{proof}

Next, we make precise the notion that $\cR(P_\eps)$ has a large `mass'
 which is independent of $\eps$.
 
 \begin{lem}\label{lem:discontinuity-bigmass}
 Suppose that $P=(\gamma,\theta,\nu)\in \cA^*$, $M=\fl{1/(1-\gamma)}$, (B) fails,
 $\theta>0$, $\eps>0$  and 
 \[
\cV_\eps = \Bigg\{ (u_1,\ldots,u_M^{\phantom{a}},\xi_1,\ldots,\xi_k) : \;\;
\begin{matrix} |\bu|+|\bxi|=1,\;  1-\gamma < u_i < 1-\gamma + \eps\; (1\le i\le M),\\\xi_i>\eps\; (1\le i\le k),\; 1\le k\le 2M\log\tfrac{1}{\eps}\end{matrix} \Bigg\}.
\]
Then $\cV_\eps \subseteq \cR(P_\eps)$.
 \end{lem}

\begin{proof}
Since \eqref{eq:A2} holds, for some $h\in \NN$, $h(1-\gamma)\in [\theta,\theta+\nu] \cup [1-\theta-\nu,1-\theta]$.  But \eqref{eq:B} fails, thus some integer multiple of $(1-\gamma)$ lies in $\{\theta+\nu,1-\theta\}$ and no integer multiple of $1-\gamma$ lies in 
$[\theta,\theta+\nu) \cup [1-\theta-\nu,1-\theta).$
In particular, we have $1-\gamma < \frac{1}{M}$, since if $1-\gamma=\frac{1}{M}$
then $k(1-\gamma) \in \{\theta+\nu,1-\theta\}$ implies that
$(M-k)(1-\gamma) \in \{ \theta, 1-\theta-\nu \}$.

There is a positive integer $h$ so that either (i) $h(1-\gamma)=1-\theta$
or (ii) $(M-h+1)(1-\gamma) = \theta+\nu$.  In case (i),
\[
(M-h+1)(1-\gamma) = (M+1)(1-\gamma) - (1-\theta) > \theta,
\]
and hence  $(M-h+1)(1-\gamma) \ge \theta+\nu$.
In case (ii),
\[
h(1-\gamma) = (M+1-(M+1-h))(1-\gamma) = (M+1)(1-\gamma) -(\theta+\nu) > 1-\theta-\nu,
\]
and hence $h(1-\gamma) \ge 1-\theta$.  In either case, we have both of the inequalities
\[
h(1-\gamma) \ge 1-\theta, \qquad (M-h+1)(1-\gamma) \ge \theta+\nu.
\]
Thus, if $(\bu,\bxi)\in \cV_\eps$ and $A\subseteq [M]$ with $|A| \ge M-h+1$, then
\[
|\ssc{\bu}{A}| > (M-h+1)(1-\gamma) \ge \theta+\nu.
\]
If  $A\subseteq [M]$ with $|A| \le M-h$ then
\[
|\ssc{\bu}{A}| + |\bxi| = 1- |\ssc{\bu}{[M]\setminus A} |
< 1 - h(1-\gamma) \le \theta.
\]
This proves that $(\bu,\bxi)\in \cR(P_\eps)$ for all $\eps \ge 0$.
\end{proof}

We now conclude the proof of \eqref{eq:dicontinuity-1}.
Fix $P=(\gamma,\theta,\nu)$ in $\cA^*$, such that (B) holds, and let $\epszero$
be a sufficiently small positive number and $0\le \eps \le \epszero$.
Let $K \in (0,1]$, depending on $P$ but not on $\eps$,
and for all $k\in\NN$, all $\bu\in\RR^M$, $\bxi\in \RR^k$ with $(\bu,\bxi)\in \cV_\eps$ (where $\cV_{\eps}$ is defined in Lemma \ref{lem:discontinuity-bigmass}), set
\begin{equation}\label{f-def}
f(\bu,\bxi) = K(-1)^{k+M}.
\end{equation}
Note that, for each fixed $k$,
 $f$ is symmetric in $u_1,\ldots,u_M$ and in $\xi_1,\ldots,\xi_k$.
For each $k$, extend the definition of $f$ to a function symmetric in all $k+M$ variables. 
(This is well-defined since $(M+1)(1-\gamma)>1$ so $\xi_i<1-\gamma$ for all $i$.)
 Set $f(\bx)=0$ for vectors $\bx$ whose components are all $<1-\gamma+\eps$,
but no permutation of the components lies in $\cV_\eps$.
 Define $f$ for vectors that have at least one component $\ge 1-\gamma+\eps$ using \eqref{fsl}, with $\gamma$ replaced
 by $\gamma+\eps$ and with $\eta=\eps$.
 The function $f$ is now supported on $\cC(\cV_\eps)$, which is a subset of
 $\cC(\cR(P_\eps))$ by Lemma \ref{lem:discontinuity-bigmass},
  and also any element has all components are $>\eps$.
  In the notation of Theorem \ref{thm:constructions}, $f\in \sF_\eps$.

We first estimate $f(1)$.
Let $w=1-M(1-\gamma)$ so that $w< 1-\gamma$.
By Lemma \ref{lem:tc}, for every $(\bu,\bxi)\in \cV_\eps \cap \RR^{M+k}$,
\[
\cyrL_{1-\gamma+\eps}(\bu,\bxi) = (-1)^{k+1+M} (M-1)! M^k.
\]
Then, by \eqref{fsl},
\begin{align*}
-f(1) &= K(M-1)! \sum_{1\le k\le 2M\log\frac{1}{\eps}} M^k \mint{(\bu,\bxi)\in \cV_\eps\cap \RR^{M+k} \\ 1-\gamma<u_1<\cdots < \ssc{u}{M} < 1-\gamma+\eps \\
\eps < \xi_1 < \ldots < \xi_k} \frac{d (\bu,\bxi)}{u_1\cdots u_M \xi_1\cdots \xi_k}\\
&= K(M-1)! \sum_{1\le k\le 2M\log\frac{1}{\eps}} M^k \mint{1-\gamma<u_1 < \cdots < u_M^{\phantom{a}}<1-\gamma+\eps} \frac{m_k(1-|\bu|,\eps)}{u_1\cdots u_M} d\bu ,
\end{align*}
where $m_k(\bx,\eps)$ is defined in Lemma \ref{lem:mk}.
For $\eps>0$, sufficiently small in terms of $M$ and $w$,
if $m_k(1-|\bu|,\eps)\ne 0$ then $1-|\bu| \ge w-M\eps \ge w/2$,
$k\le 2M \log\frac{1}{\eps} \le 4M \log( \frac{w/2}{\eps} )$
and $w/2 \ge 100(4M)^2 \eps$.
Thus, by Lemma \ref{lem:mk} (b),
\[
m_k(1-|\bu|,\eps) \ssc{\gg}{M} \frac{(\log (\frac{w}{2\eps}))^{k-1}}{t(k-1)!}.
\]
Also, $u_1 \cdots \ssc{u}{M} > (1-\gamma)^M$ and the measure of the set of possible vectors $(u_1,\cdots,u_M^{\phantom{a}})$ equals $\eps^M/M!$.   Hence, for suffiently small $\eps$,
\begin{align*}
-f(1) &\gg_M^{\phantom{a}} K \eps^M   \sum_{1\le k\le 2 M\log(\frac{w}{2\eps})}
 \frac{(M\log (\frac{w}{2\eps}))^{k-1}}{(k-1)!} \\
&\gg_M  K \eps^M \cdot \tfrac12\big(\tfrac{w}{2\eps}\big)^M
\\ &\gg_{w,M^{\phantom{a}}} K.
\end{align*}
It is crucial that
the right side is independent of $\eps$.

Next, we verify that $f(\bx)\ge -1$ for any $\bx$.
Once we have accomplished this, Theorem \ref{thm:continuity} will
follow from Theorem \ref{thm:constructions} (b) with $\eta=\eps$.

Suppose $\by \in \cC(\cV_\eps)$
and suppose $\by$ is a coagulation of $(\bx,\bxi)\in \cV_\eps$.
Since $1-\gamma < x_i < 1-\gamma+\eps$ for $1\le i\le M$,
 $|\bxi| < w$.  Therefore, all components of $\by$ which
are $<1-\gamma+\eps$ lie in $(\eps,w) \cup (1-\gamma,1-\gamma+\eps)$ 
and the larger components lie in $J_1 \cup \cdots \cup J_M$, where
\begin{align*}
J_1 &= [1-\gamma+\eps,1-\gamma+w), \\ 
J_j &= (j(1-\gamma),j(1-\gamma)+w) \;\;\;\; (2\le j\le M-1),\\
J_M &= (M(1-\gamma),1].
\end{align*}
These are disjoint since $w<1-\gamma$.
Moreover, a component of $\by$ that is in $J_i$, for $i\ge 1$, is the sum of exactly
$i$ of the variables $\ssc{x}{1},\ldots,x_M$ plus a subset of the variables $\xi_i$.

Now write $\by = (\xi_1,\ldots,\xi_\ell,\beta_1,\ldots,\beta_r,\alpha_1,\ldots,\alpha_s)$,
where $\xi_i \in (\eps,w)$ for all $i$, $\beta_i\in (1-\gamma,1-\gamma+\eps)$
for each $i$, $s\ge 1$ and $\alpha_i \in J_{j_i}$ for each $i$, where $j_i\in \{1,\ldots,M\}$
for each $i$ and
\begin{equation}\label{rjsM}
r + j_1 + \cdots + j_s = M.
\end{equation}
Suppose that in \eqref{fsl}, 
for $1\le i\le s$, $\alpha_i$ fragments into $j_i$ variables 
from $(1-\gamma,1-\gamma+\eps)$, call them $x_{i,1},\ldots,x_{i,j_i}$, and $k_i$ variables in $(\eps,w)$, call them $\xi_{i,1},\ldots,\xi_{i,k_i}$ (here $k_i$
is not fixed).
By Lemma \ref{lem:tc},
\[
\cyrL_{1-\gamma+\eps}(x_{i,1},\ldots,x_{i,j_i},\xi_{i,1},\ldots,\xi_{i,k_i}) = 
(-1)^{1+j_i+k_i} (j_i-1)! j_i^{k_i}.
\]
Using \eqref{rjsM}, the product of these equals
\[
(-1)^{s+j_1+\cdots+j_s + k_1+\cdots k_s} \prod_{i=1}^s (j_i-1)! j_i^{k_i} = 
(-1)^{s+M-r+k_1 + \cdots + k_s}  \prod_{i=1}^s (j_i-1)! j_i^{k_i}.
\]
Let $\bz$ be the vector consisting of $\xi_1,\ldots,\xi_\ell,\beta_1,\ldots\beta_r$
and all of the variables $x_{i,j}$ and $\xi_{i,j}$.  There are 
$\ell+k_1+\cdots+k_s$ total variables in $(\eps,w)$, and thus, by \eqref{f-def},
\[
f(\bz) = K (-1)^{M+\ell + k_1+ \cdots + k_s}.
\]
By \eqref{fsl},
\begin{align*}
f(\by) &=
\alpha_1 \cdots \alpha_s \sum_{k_1,\ldots,k_s\ge 0} \;\;\;  \mint{\alpha_i=\sum_{j=1}^{j_i} x_{i,j} + \sum_{j=1}^{k_i} \xi_{i,j} \\ (1\le i\le s) \\ 1-\gamma < x_{i,1} 
< \cdots < x_{i,j_i} < 1-\gamma+\eps\;\; \forall i  \\
\eps < \xi_{i,1} < \cdots < \xi_{i,k_i} \;\; \forall i}
\frac{K(-1)^{s+\ell-r} \prod_{i=1}^s (j_i-1)! j_i^{k_i} }{\prod_{i=1}^s (\prod_{j=1}^{j_i} x_{i,j} \prod_{j=1}^{k_i} \xi_{i,j})}\\
&= (-1)^{\ell+s+r} K \alpha_1 \cdots \alpha_s \Big( \prod_{i=1}^s (j_i-1)! \Big) \, W_1 \cdots W_s,
\end{align*}
where, for $1\le i\le s$,
\begin{align*}
W_i &= \mint{1-\gamma < x_{i,1} < \cdots < x_{i,j_i} < 1-\gamma+\eps \\
x_{i,1}+\cdots+x_{i,j_i} = \alpha_i} \frac{d\bx_i}{x_{i,1}\cdots x_{i,j_i}}   \\
&\qquad + 
 \mint{1-\gamma < x_{i,1} < \cdots < x_{i,j_i} < 1-\gamma+\eps \\
x_{i,1}+\cdots+x_{i,j_i} < \alpha_i-\eps}  \frac{1}{x_{i,1}\cdots x_{i,j_i}}  
\sum_{k_i\ge 1} j_i^{k_i} m_{k_i}(\alpha_i-(x_{i,1}+\cdots+x_{i,j_i}),\eps)\, d\bx_i.
\end{align*}
The first term above corresponds to $k_i=0$, and we note that this is zero if $j_i=1$.

We have $x_{i,1} \cdots x_{i,j_i} > (1-\gamma)^{j_i}$,
and the measure of the set of vectors $(x_{i,1},\ldots,x_{i,j_i})$ is at most
 $\eps^{j_i-1}/j_i!$ when
$k_i=0$ and is otherwise is at most $\eps^{j_i}/j_i!$.
Using Lemma \ref{lem:mk} (a), for any $\eps \le y\le 1$ and $j\ge 1$ we have
\[
\sum_{k\ge 1} j^k m_k(y;\eps) \le \sum_{k=1}^\infty \frac{(\log \frac{y}{\eps})^{k-1} j^k}{y(k-1)!} = \frac{j}{y} \pfrac{y}{\eps}^j \le \frac{j}{\eps^j}.
\]
Therefore,
\begin{align*}
j_i! W_i \le \frac{\eps^{j_i-1}}{(1-\gamma)^{j_i}} + \frac{\eps^{j_i}}{(1-\gamma)^{j_i}} \cdot 
\frac{j_i}{\eps^{j_i}} = \frac{\eps^{j_i-1}+j_i}{(1-\gamma)^{j_i}} \le \frac{2j_i}{(1-\gamma)^{j_i}}.
\end{align*}
In conclusion, we obtain from \eqref{rjsM},
\[
|f(\by)| \le \frac{K\cdot  2^s}{(1-\gamma)^{j_1+\cdots+j_s}} =
\frac{K\cdot 2^s}{(1-\gamma)^{M-r}} \le K\cdot 2^M (M+1)^M.
\]
Taking $K = 2^{-M} (M+1)^{-M}$ gives $|f(\by)|\le 1$ for all $\by$, which
suffices.

%
%
\subsection{Continuities}\label{sec:continuities}
%
%

\subsubsection{Initial steps for the continuity case}
Define the invariants for each triple $(\gamma,\theta,\nu)$:
\begin{align*}
M(\gamma) &:= \fl{\frac{1}{1-\gamma}}, \text{ so that } \frac{1}{M(\gamma)+1} < 1-\gamma \le \frac{1}{M(\gamma)},\\
\sM(\gamma,\theta,\nu) &:= \{ m \ge M(\gamma)+1 : [\theta,\theta+\nu] \text{ contains no rational }a/m \text{ with } a\ge 1 \}.
\end{align*}

In this way, condition \eqref{eq:A1} in Theorem \ref{thm:asymptotic} is equivalent to
$\sM(\gamma,\theta,\nu) = \emptyset$.

In the case where $\theta=0$ or \eqref{eq:B} holds,
it will suffice to prove Theorem \ref{thm:continuity}
when $\theta+\nu \ne 1-\theta$. Indeed, if $(\gamma,\theta,1-2\theta)\in \cA^*$ then certainly $(\gamma,\theta,1/2-\theta)\in \cA^*$. By monotonicity (Proposition \ref{prop:C-monotonicity}), for $\eps>0$ we have
\begin{align*}
C^-(\gamma-\eps,\theta+\eps,\tfrac12-\theta-2\eps) &\le C^-(\gamma-\eps,\theta+\eps,1-2\theta-2\eps), \\
C^+(\gamma-\eps,\theta+\eps,\tfrac12-\theta-2\eps) &\ge C^+(\gamma-\eps,\theta+\eps,1-2\theta-2\eps).
\end{align*}
Thus, if $(\gamma,\theta,1-2\theta)\in \cA^*$ we have $(\gamma,\theta,1/2-\theta)\in\cA^*$ and if \eqref{C-continuous} holds for $\nu=\frac12-\theta$ then \eqref{C-continuous}
holds for $\nu=1-2\theta$. So the claim of Theorem \ref{thm:continuity} for $(\gamma,\theta,1-2\theta)$ follows from the claim for $(\gamma,\theta,1/2-\theta)$.

Our arguments naturally break into three cases: $1-\gamma \ge \theta+\nu$, $1-\gamma\le \theta < \theta+\nu < \frac12$ and $1-\gamma\le \theta < \theta+\nu = \frac12$.

Suppose that $P=(\gamma,\theta,\nu)\in \cA^*$, $\eqref{eq:B}$ holds and hence that 
$\gamma>\frac12$.
Since $\eqref{eq:B}$ holds for $P$, we see that there is a
sufficiently small $\epszero>0$ (depending on $P$) such that $\eqref{eq:A2}$ holds for $P_\eps$ whenever $0<\eps \le \epszero$. Since $P\in \cA^*$, we have that $P_\eps\notin \cA$ for $\eps>0$, so by Theorem \ref{thm:asymptotic}, $\eqref{eq:A1}$ fails for $P_\eps$ for $\eps>0$ sufficiently small but $\eqref{eq:A1}$ holds for $P_0=P$. (This final claim is automatic from Theorem \ref{thm:asymptotic} if $P\in\cQ_1$; if $P\in\cA\setminus\cQ_1$ then $\theta+\nu=1/2$ or $\gamma=1-\theta-\nu$, so 
\eqref{eq:A1} holds for $(\gamma,\theta,\nu)$ since it holds for $(\gamma,\theta,1-2\theta)$ or $(1-\theta,\theta,\nu)$ by Theorem \ref{thm:asymptotic}).
We also note that
\[
\gamma + \nu \le 1,
\]
for otherwise \eqref{eq:A1} holds for $P_\eps$ for all sufficiently small $\eps$.

Analogous to $\cR(P)$, define $\cR_{\text{cl}}(P)$ to be the set of vectors, of
arbitrary dimension, all of whose components are in $[0,1-\gamma]$,
the sum of components is 1, and where no subsum of components lies
in $(\theta,\theta+\nu)$.  This will be a more convenient set to work with
here.
If $\cW$ is a collection of subsets of $[k]$, and
\begin{equation}\label{eq:T2kPW}
T_{2,k}(P;\cW):=\big\{\bx\in \cR_{\text{cl}}(P) \cap \RR^k: |\ssc{\bx}{W}| \le \theta \; (\forall W\in \cW), |\ssc{\bx}{W}| \ge \theta+\nu \; (\forall W \not\in \cW)\big\},
\end{equation}
then $\cR_{\text{cl}}(P)\cap\RR^k$ is the union of $T_{2,k}(P;\cW)$ over all choices of $\cW$. 
The sets $T_{2,k}(P;\cW)$ are disjoint since if $W\in \cW_1\triangle \cW_2$ then 
the sets $T_{2,k}(P;\cW_1)$ and $T_{2,k}(P;\cW_2)$ lie on opposite sides of
the region $\{ \bx \in \RR^k: \theta < |\ssc{\bx}{W}| < \theta+\nu \}$.

For brevity, write
\[
\cR^\eps := \cR(P_\eps), \quad \cR_{\text{cl}}^\eps  := \cR_{\text{cl}}(P_\eps),
\]
so that $\cR^0 = \emptyset$ since $P\in \cA$ (and using Theorem \ref{thm:asymptotic}). Clearly we have the inclusions
\begin{equation}\label{eq:R1R2}
\cR_{\text{cl}}^{\eps} , \cR^{\eps'} \subseteq \cR_{\text{cl}}^{\eps'}\quad\text{ and }\quad \cR_{\text{cl}}^\eps\cap(0,1]^k\subseteq \cR^{\eps'} \qquad (0\le \eps < \eps',\,k\ge 1).
\end{equation}

When $1-\gamma < 1/M(\gamma)$, we have $M(\gamma-\eps)=M(\gamma)$ for $\eps>0$ sufficiently small.  Also, the elements of $\sM(\gamma,\theta,\nu)$ are all $<1/\nu$.
Hence, for small enough $\eps>0$ we have
\[
\sM(P_\eps) = \big\{ m\ge M(\gamma)+1 : \text{ there is no rational } a/m \in 
(\theta,\theta+\nu) \big\}.
\]
When $1-\gamma = 1/M(\gamma)$,  we have $M(\gamma-\eps)=M(\gamma)-1$ for $\eps>0$ sufficiently small, and thus for $\eps>0$ small enough we have
\[
\sM(P_\eps) = \big\{ m\ge M(\gamma) : \text{ there is no rational } a/m \in 
(\theta,\theta+\nu) \big\}.
\]
Thus, regardless of whether $1-\gamma=1/M(\gamma)$ or not, for sufficiently small $\epszero>0$ the set $\sM(P_\eps)$ is constant for $0<\eps\le \epszero$ and in this range of $\eps$ we abbreviate
\[
\sM = \sM(P_\eps).
\]
For each $m\in \sM$, and $0<\eps\le \epszero$,
both $\cR_{\text{cl}}^0$ and $\cR^\eps$ contain the point $(1/m,\dots,1/m)$,
and these sets contain no other points  of the form $(1/n,\dots,1/n)$.  Therefore,
\begin{equation}\label{eq:sM-alt}
\sM = \Bigg\{ m \ge \cl{\frac{1}{1-\gamma}} :  \, \Big( \frac{1}{m},\ldots,\frac{1}{m} \Big) \in \cR_{\text{cl}}^0 \Bigg\}.
\end{equation}
For example, if $[\theta,\theta+\nu]=[\frac14,\frac13]$ and $\frac16 \le 1-\gamma < \frac3{16}$ then $\sM = \{ 6,8,9,12 \}$.

For the remainder of this subsection, $\epszero$ will be assumed to be sufficiently small
in terms of $P=P_0$, and $0<\eps \le \epszero$.

\begin{lem}\label{lem:sM}
Let $(\gamma,\theta,\nu)\in \cA^*$.  For every $m\in \sM$,
there is an integer $a$ with $a/m\in\{\theta,\theta+\nu\}$.
\end{lem}

\begin{proof}
Let $M=M(\gamma)$.
If $m\ge M+1$, the assertion holds since  \eqref{eq:A1} holds for $P_0$ but fails for $P_\eps$ when $0<\eps\le \epszero$.  If $m=M$, then $1-\gamma=\frac{1}{M}$.
By \eqref{eq:B}, there is an integer $h$ with $\frac{h}{m} \in [\theta,\theta+\nu) \cup
[1-\theta-\nu,1-\theta)$.  But $M=m\in \sM$ implies that for every integer $h$,
$\frac{h}{M} \not\in (\theta,\theta+\nu)$.  Hence, for all $h$,
 $\frac{h}{M} \not\in (1-\theta-\nu,1-\theta)$ also.  Therefore, there
 is an integer $h$ with $\frac{h}{M} \in \{ \theta, 1-\theta-\nu \}$.
 Finally, if $\frac{h}{M}=1-\theta-\nu$ then $\frac{M-h}{M}=\theta+\nu$.
\end{proof}

\begin{lem}\label{lem:nu-lower-for 1/2-theta}
Suppose that $P=(\gamma,\theta,\nu)\in \cQ$ with $\nu=\frac12-\theta$
and \eqref{eq:A1} holds. Then
\[
\nu \ge \begin{cases}
\frac{1}{2M+2} & \text{ if } M \text{ is even},\\
\frac{1}{2M+4} & \text{ if } M \text{ is odd}.
\end{cases}
\]
\end{lem}

\begin{proof}
Let $I=[\theta,\tfrac12]$ and $M=M(\gamma)$.  If $M$ is even, then \eqref{eq:A1} holds  if and only if $\frac{M/2}{M+1} \in I$. In this case $\nu=\tfrac12-\theta\ge \tfrac12-\frac{M/2}{M+1}=\frac{1}{2M+2}$. If instead $M$ is odd then \eqref{eq:A1} holds if and only if $\frac{(M+1)/2}{M+2}\in I$ (note that for any even $m$, $\frac{m/2}{m}\in I$, so we only need to consider odd $m\ge M+1$ in \eqref{eq:A1}). In this case $\nu=\tfrac12-\theta\ge \tfrac12-\frac{(M+1)/2}{M+2}=\frac{1}{2M+4}$.
\end{proof}

The principal tool in proving \eqref{C-continuous} of Theorem \ref{thm:continuity} is that the vectors on $\cR_{\text{cl}}^\eps$ must have special forms.

\begin{prop}\label{prop:ClaimA}
Let $(\gamma,\theta,\nu)\in \cA^*$ be such that $\theta=0$ or \eqref{eq:B} holds, and let $k\ge M(\gamma)+1$.  Suppose that $0<\eps\le \epszero$ and
 $\bx \in \cR_{\text{cl}}^\eps \cap \RR^k$ with
$S = \{ i: x_i < \nu-2\eps \}$ and $L=\{ i : x_i \ge \nu-2\eps \}$.
Then $|\ssc{\bx}{S}| \le m^2 \eps$ 
and there is an integer $m$ such that one of the following holds:
\begin{enumerate}
\item $m\in \sM$, and $|x_i - \frac{1}{m}| \le 2m^2\eps$ for all $i\in L$;
\item $2m\in \sM$, for all $i\in L$  either  $|x_i - \frac{1}{m}| \le 2m^2\eps$ or $|x_i - \frac{1}{2m}| \le 2m^2\eps$;
\item $m\in \sM$, and there exists distinct $j_1,j_2 \in L$  such that all of the following hold:
\begin{enumerate}
\item $|x_{j_1}+x_{j_2} - \frac{1}{m}| \le m^2\eps$; 
\item For all $i\notin\{j_1,j_2\}$ with $i\in L$, we have $|x_i - \frac{1}{m}| \le m^2\eps$;
\item Either $\nu-2\eps \le x_{j_1} < \frac{1}{2m}-2m^2\eps$ or $x_{j_2}>\frac{1}{2m}+2m^2\eps$.
\end{enumerate}
\end{enumerate}
\end{prop}

Case (3) is possible, for example when $[\theta,\theta+\nu]=[\frac15,\frac13]$
and $\frac13 < 1-\gamma < \frac25$, $\cR_{\text{cl}}^\eps$ contains points
of the shape $\big(\frac13+O(\eps),\frac13+O(\eps),x,y\big)$, with $x+y=\frac13+O(\eps)$ and
$\frac2{15} \le x \le \frac16$.

\begin{lem}\label{lem:ClaimA-3}
Let $P=(\gamma,\theta,\nu)\in \cA^*$, \eqref{eq:B} holds, $0 < \eps \le \epszero$. For some $m^*\in \{M(\gamma),M(\gamma)+1\}$, depending only on $P$,
any vector in $\cR_{\text{cl}}^\eps$ satisfying case (3) of 
Proposition \ref{prop:ClaimA} has $m=m^*$.
\end{lem}

\begin{proof}
Let $M=M(\gamma)$.
Assume that there is a vector $\bx\in \cR_{\text{cl}}^\eps$
satisfying case (3) in Proposition \ref{prop:ClaimA}.
We have either $x_{j_1} < \frac{1}{2m} - 2m^2 \eps$ or 
$x_{j_2} > \frac{1}{2m} + 2m^2\eps$.  Since $x_{j_1}+x_{j_2}\le \frac1{m}+m^2\eps$,
in either case we have $\nu-2\eps \le x_{j_1} \le \frac{1}{2m} - m^2\eps$, and hence
\begin{equation}\label{nu-2m}
\nu < \frac{1}{2m} \le \frac{1-\gamma}{2}.
\end{equation}

\textbf{Case I. $1-\gamma \le \theta < \theta+\nu < \frac12$}. 
By \eqref{nu-2m} and Proposition \ref{prop:ratio},
we are in one of the cases (e)--(h) in that Proposition.
In all cases, $\nu \ge \frac{1}{12}$, and thus $m\le 5$.
For cases (e) and (f), this is not possible since $1-\gamma < \frac15$.
For case (h), $\nu \ge \frac{4}{35} > \frac{1}{10}$, so $m\le 4$, but 
$1-\gamma < \frac14$, so this is also impossible.
In case (g), $1-\gamma < \frac14$, so $m=5$.  But $[\theta,\theta+\nu]$ contains
$[\frac13,\frac37]$ with the point $\frac25$ in the interior, and thus $5\not\in \sM$.
Therefore, there are no such vectors $\bx$ in this case.

\textbf{Case II.  $1-\gamma \le \theta < \theta+\nu = \frac12$}. 
Recall that $m\ge M$.  If $M$ is even, Lemma \ref{lem:nu-lower-for 1/2-theta}
and \eqref{nu-2m} imply that $\frac{1}{2M+2} \le\nu < \frac{1}{2m}$, hence $m=M$.
If $M$ is odd, then by Lemma \ref{lem:nu-lower-for 1/2-theta}
and \eqref{nu-2m}, $\frac{1}{2M+4} \le \nu < \frac{1}{2m}$,
so that $m=M$ or $m=M+1$.  If $m=M$ then $1-\gamma=\frac{1}{M}$.
As $m$ is odd, by Lemma \ref{lem:sM}, $\theta = \frac{a}{M}$ for some integer
$a$.  Thus, $\theta \le \frac{M-1}{2M}$, so $\nu\ge \frac{1}{2M}$.  But this
contradicts \eqref{nu-2m}.  Thus, $m=M+1$.

\textbf{Case III. $1-\gamma\ge \theta+\nu$.}
By Lemma \ref{lem:smallgamma}, $[\theta,\theta+\nu]$ contains
one of the intervals $[\frac{1}{2M+1},\frac{1}{M+1}]$, $[\frac{1}{2M},\frac{2}{2M+1}]$ or $[\frac{1}{2M-1},\frac{1}{M}]$.  Since $M\ge 2$, $\nu >\frac{1}{2M+4}$. 
By \eqref{nu-2m}, $m\le M+1$.
If $m=M$ then $1-\gamma=\frac{1}{M}$. By Lemma \ref{lem:sM}, this forces $\theta+\nu=1/M$. 
If $m=M+1$, then Lemma \ref{lem:sM} plus the inequality $1-\gamma \le \frac{1}{M}<
\frac{2}{M+1}$ implies that $\frac{1}{M+1} \in \{ \theta, \theta+\nu \}$.
If $\theta=\frac{1}{M+1}$, then $M=2$, $\theta+\nu=\frac12 =1-\gamma$,
which is not possible as $\gamma>\frac12$ if $P\in \cA^*$.
Thus, $\theta+\nu=\frac{1}{M+1}$. Hence $m^*$ can be taken to be $M$ or $M+1$ depending on whether $\theta+\nu=\frac{1}{M}$ or $\theta+\nu=\frac{1}{M+1}$.
\end{proof}

In the next subsection, we deduce the continuity case of Theorem \ref{thm:continuity}
from Proposition \ref{prop:ClaimA} and Lemma \ref{lem:ClaimA-3}.
The following four subsections are devoted to the proof of Proposition \ref{prop:ClaimA}.

\subsubsection{The sieve argument}


Recall the definitions \eqref{eq:VHZGdefs}, and 
for $0<\eps\le \epszero$ define $\cH^{\eps}:=\cH(P_\eps)$,
$\cZ^\eps:=\cZ(P_\eps)$, and $\cG_1^\eps:=\cG_1(P_\eps)$.
We remark that $\cZ^\eps$ and $\cG_1^\eps$ both contain the empty vector $\varnothing$.

Let $d=100m^3$.
It follows from Proposition \ref{prop:ClaimA} that the vectors in $\cH^\eps$
have two possible forms:
\begin{itemize}
\item[(i)] for some $m\in \sM$, each component $x_i$ lies in 
\[
\cI_m := \bigcup_{a=1}^{m-1} \bigg[ \frac{a}{m} - d\eps, \frac{a}{m} + d\eps \bigg];
\]
\item[(ii)]  for some $\ell \in \sM$, all components $x_i$ but two lie in $\cI_\ell$,
while the other two components have sum in $\cI_\ell \cup \{1\}$ and each lie in
\[
\cK_\ell := \bigcup_{a=0}^{\ell-1} \bigg[ \frac{a}{\ell} + \nu - d\eps, \frac{a+1}{\ell} - \nu + d\eps \bigg];
\]
\end{itemize}
Vectors of type (ii) only occur as coagulations of some vector satisfying type (3)
in Proposition \ref{prop:ClaimA}.
It is possible for a vector to be of type (ii) for some $\ell\in \sM$ and also of type (i) with $m=2\ell$ if $2\ell\in \sM$, $\nu\le \frac{1}{2\ell}+O(\eps)$
and the two components which are in $\cK_\ell$ are in $\cI_{2\ell}$.
Also, a vector of type (i) need not be associated to a unique $m$, e.g. if
(i) holds for $m=6$ and $12\in \sM$ then (i) holds also for $m=12$.

For brevity, let
\[
\cH_k^\eps = \cH^\eps \cap \RR^k.
\]

The claim that $|\ssc{\bx}{S}|\le m^2 \eps$ in Proposition \ref{prop:ClaimA} implies that
components of a vector in $\cR^\eps$ which are $\le \nu-2\eps$ have sum at most $m^2 \eps$.  Consequently, $\psi(\bx) \le m^2 \eps$ for all 
$\bx \in \cZ^\eps$
with $x_i > \nu-2\eps$ for all $i$. 
Therefore, $\cG_1^\eps$ contains all those
vectors in $\cZ^{\eps}$ with components $>\nu-\eps$ and
 sum of components $\le \gamma-(m^2+1)\eps$.
In particular, since $\gamma>\frac12$, $\cG_1^\eps$ contains
all such vectors with sum of components $\le \frac12$.

We will not make use of $\cG_2(P_\eps)$ for these constructions.
By Theorem \ref{thm: Main sieving}, it suffices to find functions $g_1^+$ and $g_1^-$,
each supported in $\cG_1^\eps$, and satisfying
\begin{enumerate}
\item[(a)] $g_1^+(\emptyset) = g_1^-(\emptyset)=1$;
\item[(b)] for all $\bx\in \cH^\eps$, $(\bone\star g_1^-)(\bx) \le 0\le (\bone\star g_1^+)(\bx)$;
\item[(c)] for each $\ssc{g}{1} \in \{ g_1^-, g_1^+ \}$,
\[
\sum_{k\ge 2} \;\int_{\cH_k^\eps} \frac{(\bone\star \ssc{g}{1})(\bx)}{\ssc{x}{1}\cdots \ssc{x}{k}} = O(\eps).
\]
\end{enumerate}

\textbf{Lower bound sieve.}  We take, for a singleton $x$,
$g_1^-(x) = - \one(x\le 1/2)$,
$g_1^-(\emptyvec)=1$ and $g_1^-(\bx)=0$ otherwise.
As every vector in $\cH^\eps$ has at least one component $\le \frac12$, (b) follows
for $g_1^-$.
For all $k\ge 2$, the $(k-1)-$dimensional volume of 
the part of $\cH_k^\eps$ of type (i) is $O(\eps^{k-1})$ (note $\cH_k^\eps=\emptyset$ if $k>\fl{1/(\nu-2\eps)}$),
and the $(k-1)-$dimensional measure of 
the part of $\cH_k^\eps$ of type (ii) is $O(\eps^{k-2})$.
Since $k \le k_0=\fl{1/(\nu-2\eps)}$, $|(\bone\star g_1^-)(\bx)| \le k_0$ 
for all $\bx\in \cH^\eps$.
When $k=2$ and $\bx=(\ssc{x}{1},\ssc{x}{2})\in \cH^\eps_2$, $(\bone\star g_1^-)(\bx)=0$ except when
$\ssc{x}{1}=\ssc{x}{2}=\frac12$, a set of measure zero.  It follows that (c) holds for $\ssc{g}{1}=g_1^-$.

\medskip

\textbf{Upper bound sieve.}  If all of the vectors in $\cH$ are of type (i),
the simple choice of $g_1^+(\emptyvec)=1$, $g_1(\bx)=0$ otherwise, clearly satisfies
(a) and (b).  The measure of $\cH_2^\eps$ is also $O(\eps)$, since 
all vectors are of type (i).  Part (c) then follows for $\ssc{g}{1}=g_1^+$.

When there are vectors of type (ii) in $\cH_2$, which are not of type (i),
the measure of $\cH_2^\eps$ is potentially $\gg 1$
and we must be more careful in choosing $g_1^+$.   Our goal is to find $g_1^+$
 so that $(\bone\star g_1^+)(\bx)=0$ on almost all of $\cH_2^\eps$, while ensuring
 that (b) also holds.
 By Lemma \ref{lem:ClaimA-3}, there is a only a single number $\ell$ 
 so that there are vectors $\bx\in \cH^\eps$ of type (ii).
 Thus, for  all vectors in $\cH^\eps$, all of its components are in $\cK_{\ell} \cup (\cup_{m\in \sM}\cI_m)$.
 Define
 \[
 \cK^* = \cK_\ell \setminus \bigcup_{m\in \sM} \cI_m.
 \]
 Let $g_1^+(\emptyvec)=1$ and
 \begin{align*}
 g_1^+(x) &= -\one(x\in \cK^*, x<1/2) + \one(x\in \cI_\ell, x < 1/2), \\
 g_1^+(\ssc{x}{1},\ssc{x}{2}) &= \one(\ssc{x}{1}\in \cK^*, \ssc{x}{2}\in \cK^*, \ssc{x}{1}+\ssc{x}{2}\le 1/2),
 \end{align*}
 with $g_1^+(\bx)=0$ if $\bx$ has three or more components.
 
 We now verify (b) for $\bx=(\ssc{x}{1},\ldots,\ssc{x}{k})\in \cH_k^\eps$.  
 If at most one component is in $\cK^*$ and $<1/2$
  then clearly $(\bone\star g_1^+)(\bx) \ge 0$.
 Otherwise, by condition (ii), there are exactly two components in $\cK^*$,
 both $<1/2$ (call them $y,z$) and the remaining components are in $\cI_\ell$.  In particular, $k\ge 3$.
 If $y+z\le 1/2$ then we have
 $(\bone\star g_1^+)(\bx) \ge 1 - 2 + g_1^+(y,z) \ge 0$,
 and if $y+z>1/2$ then there is a component in $\cI_\ell$ which is $<1/2$
 we likewise get $(\bone\star g_1^+)(\bx) \ge 0$.  This verifies (b).
 
 Now we verify (c).
 It is clear that  $(\bone\star g_1^+)(\bx)$ is uniformly bounded over all $k$, thus
 \[
 \sum_{k\ge 3} \int_{\cH_k^*}  \frac{(\bone\star g_1)(\bx)}{\ssc{x}{1}\cdots \ssc{x}{k}} = O(\eps).
 \]
Furthermore, when $\bx=(\ssc{x}{1},\ssc{x}{2})\in \cH_2^{\eps}$ with  neither component
 in $\{1/2\} \cup \cup_{m\in \sM} \cI_m$, then both components are in $\cK^*$
 and hence $(\bone\star g_1^+)(\bx)=0$.  The set of other pairs $(\ssc{x}{1},\ssc{x}{2})$, with some
 component in $\{1/2\} \cup \cup_{m\in \sM} \cI_m$, has measure $O(\eps)$
 and we deduce that
 \[
 \int_{\cH_2^*}  \frac{(\bone\star g_1)(\ssc{x}{1},\ssc{x}{2})}{\ssc{x}{1} \ssc{x}{2}} = O(\eps)
 \]
 as well, and this establishes (c).

%
%
\subsubsection{Proof of Proposition \ref{prop:ClaimA} when $\theta=0$}\label{sec:theta=0}
%
%

Suppose that $P_0=(\gamma,0,\nu) \in \cA^*$ and let $M=M(\gamma)$ so that $M\ge 2$.
If $\nu > 1-\gamma$ then both \eqref{eq:A1} and \eqref{eq:A2} hold for
$P_\eps$ when $0\le \eps < \frac{\nu-(1-\gamma)}{3}$, since for such $\eps$
we have $\nu-2\eps > 1-\gamma+\eps$.  Thus, $\nu \le 1-\gamma$.
Also, $\nu \ge \frac{1}{M+1}$ since \eqref{eq:A1} holds for $P_0$.
Suppose that $\bx \in \cR_{\text{cl}}^\eps$. Since $\bx$ has no subsums in $(\eps,\nu-\eps)$, we see that after rearranging the coordinates $\bx$ is of the form $(u_1,\ldots,u_k,\xi_1,\ldots,\xi_r)$ with $\nu - \eps \le u_i \le 1-\gamma+\eps$ and $0\le \xi_i \le \eps$ for all $i$. For $\eps<\nu/3$, $|\bxi| \le \eps$ as well,
for otherwise $\bxi$ has a subsum in $(\eps,2\eps]$.
Since $u_i \ge \nu-\eps \ge \frac{1}{M+1}-\eps$, $u_i\le 1-\gamma+\eps\le \frac1{M}+\eps$ for all $i$
and $|\bu|=1-|\bxi| \in[ 1-\eps,1]$, it follows that $k\in \{M,M+1\}$. 
If $k=M$ then $\frac{1}{M}-M\eps \le u_i \le \frac{1}{M}+\eps$ for all $i$,
and if $k=M+1$ then $u_i\in \big[\frac{1}{M+1}-\eps,\frac{1}{M+1}+M\eps\big]$ for all $i$.  In particular, $k\in \sM$ and $|u_i-1/k|\le k\eps$ for all $i$, so case $(1)$ of Proposition \ref{prop:ClaimA} holds.

%
%
\subsubsection{First reduction when $\theta>0$}
%
%

In this section, we deduce Proposition \ref{prop:ClaimA}  from the following.
Recall the definition \eqref{eq:T2kPW} of the polytopes $T_{2,k}(P_\eps;\cW)$.

\begin{prop}\label{prop:ClaimB}
Let $(\gamma,\theta,\nu)\in \cA^*$ with $\theta>0$ be such that \eqref{eq:B} holds, and let $k\ge M(\gamma)+1$.
Then for any $\bx\in \cR_{\text{cl}}^\eps\cap \mathbb{R}^k$, there is some $m\in \sM$ and $\by\in\cR_{\text{cl}}^\eps\cap \mathbb{R}^k$ in the same polytope $T_{2,k}(P_\eps;\cW)$ as $\bx$, such that $y_i\in \{0,1/(2m),1/m\}$ for all $i\in[k]$. Furthermore, if $y_i=\frac{1}{2m}$ for some $i$ then $2m \in \sM$ as well.
\end{prop}

\begin{lem}\label{lem:BtoA}
Let $(\gamma,\theta,\nu)\in \cA^*$ with $\theta > 0$.
If Proposition \ref{prop:ClaimB} holds for a given value of $k$
then Proposition \ref{prop:ClaimA} is true for the same $k$.
\end{lem}

\begin{proof}
Assume that Proposition \ref{prop:ClaimB} holds for $k$. Let $\bx$ lie in  a polytope 
$T=T_{2,k}(P_\eps;\cW)$, and let $\by=(y_1,\ldots,y_k)$ be a vector in the same polytope $T$ that is guaranteed by Proposition \ref{prop:ClaimB}, corresponding to $m\in \sM$. After possibly permuting the coordinates, we may assume that $y_i=0$ for $1\le i\le j$, $y_i = \frac{1}{2m}$ for $j+1\le i\le j+2n$, and $y_i = \frac{1}{m}$ for $j+2n+1 \le i\le k$ for some $j,n$ with $j+m+n=k$ and $2m\in \sM$ if $n>0$.

By Lemma \ref{lem:sM} and the fact that $\theta>0$, there is some integer $a$ with $1\le a\le m-1$ and either $\frac{a}{m}=\theta$ or $\frac{a}{m}=1-\theta-\nu$. Since $\bx,\by\in T$,  for any $J\subseteq [k]$ we have
 $|\ssc{\bx}{J}|\le \frac{a}{m}+\eps$ 
if and only if $|\ssc{\by}{J}|\le \frac{a}{m}+\eps$.  Let
\[
z_0=\ssc{x}{1}+\cdots+x_j, \qquad
z_i = x_{j+2i-1} + x_{j+2i} \;\; (1\le i\le n), \qquad z_i=x_{j+n+i} \;\; (n+1\le i\le m).
\]
Since the subsums of $\bx$ avoid the interval $(\frac{a}{m}+\eps,\frac{a}{m} + \nu-\eps)$, if $x_\ell < \nu-2\eps$ then
for any $J\subseteq [k] \setminus \{\ell\}$,  $|\ssc{\bx}{J}|\le \frac{a}{m}+\eps$ 
if and only if $x_\ell + |\ssc{\bx}{J}|\le \frac{a}{m}+\eps$.
 Then $\by$
 has the same property, namely  $|\ssc{\by}{J}|\le \frac{a}{m}+\eps$ 
if and only if $y_\ell + |\ssc{\by}{J}|\le \frac{a}{m}+\eps$.
Since $1\le a\le m-1$, there is a set $J$, not containing $\ell$, with $|\ssc{\by}{J}| = \frac{a}{m}$, and this implies that $y_\ell=0$ and $\ell\le j$.
Thus, in particular, 
\begin{equation}\label{z0}
z_0\ge\sum_{i:x_i < \nu-2\eps}x_i. 
\end{equation}

For any sets $I \subseteq [m]$ with $|I|=a$  we have
\[
y_1+\cdots+y_j+
\sum_{i\in I, i\le n} \big( y_{j+2i-1}+y_{j+2i} \big) + \sum_{i\in I, i>n} y_{j+n+i} = \frac{a}{m},
\]
and consequently
\begin{equation}
z_0 + |\ssc{\bz}{I}| \le \frac{a}{m}+\eps.
\label{eq:TypeIICondition}
\end{equation}
Fix $1\le \ell\le m$. Since $0<a<m$, we can sum \eqref{eq:TypeIICondition} over all such choices $I$ not containing $\ell$.  This gives
\[
\binom{m-2}{a-1}\sum_{i\notin \{0,\ell\}}z_i\le \binom{m-1}{a}\Bigl(\frac{a}{m}+\eps-z_0\Bigr).
\]
Recalling that $|\bz|=1$ we see that
\[
z_\ell\ge \frac{1}{m}-\frac{m-1}{a}(\eps-z_0)-z_0 \ge \frac1{m}-m\eps \qquad (1\le \ell\le m).
\]
This implies that $z_0 \le m^2 \eps$.
By \eqref{z0}, 
$\sum_{i:x_i < \nu-2\eps}x_i\le z_0 \le m^2\eps$, giving the first claim.
Since $z_1+\cdots + z_m \le 1$ it follows that $z_\ell \le \frac{1}{m}+m^2\eps$ for $1\le \ell\le m$, and so $|z_\ell-\frac{1}{m}|\le m^2\eps$ for $1\le \ell\le m$.

The same argument works after
permuting the variables $x_{j+1},\ldots,x_{j+2n}$,
that is, grouping them arbitrarily into pairs to form $z_1,\ldots,z_n$.
 Thus we have
\begin{align*}
|x_{j+i} + x_{j+h} - 1/m| &\le m^2\eps \;\; (1\le i<h\le 2n), \\
\qquad |x_i-1/m| &\le m^2\eps \;\; (j+2n+1 \le i\le k), \\
\qquad x_{1}+\cdots + x_j &\le m^2 \eps.
\end{align*}
If $n\ge 2$ then for distinct $i_1,i_2,i_3\in \{1,\dots,2n\}$
\[
2\bigg|x_{j+i_1}-\frac{1}{2m}\bigg|\le \bigg|x_{j+i_1}+x_{j+i_2}-\frac{1}{m}\bigg|+\bigg|x_{j+i_1}+x_{j+i_3}-\frac{1}{m}\bigg|+\bigg|x_{j+i_2}+x_{j+i_3}-\frac{1}{m}\bigg|\le 3m^2\eps.
\]
Thus if either $n=0$ or $n\ge 2$, we have that $|x_{j+i}-\frac{1}{2m}|\le 2m^2 \eps$
for all $1\le i\le 2n$, and we conclude that $\bx$ satisfies condition (1)
or condition (2) in Proposition \ref{prop:ClaimA}. If instead $n=1$, we have that $\bx$ satisfies condition (2) 
in Proposition \ref{prop:ClaimA} if $|x_i-\frac{1}{2m}|\le 2m^2 \eps$ 
for $i\in \{j+1,j+2\}$ and condition (3) otherwise. In particular, we see that Proposition \ref{prop:ClaimA} holds for $k$.
\end{proof}

%
%
\subsubsection{Proof of Proposition \ref{prop:ClaimB}}
%
%

\begin{lem}\label{lmm:Sizes}
Let $(\gamma,\theta,\nu)\in \cA^*$ with $\theta>0$ be such that \eqref{eq:B} holds,
and $0<\eps \le \epszero$.  Suppose that numbers $x,y$ satisfy 
\[
\nu-2\eps \le x<y\le 1-\gamma+\eps, \qquad y-x\ge \nu+2\eps.
\]
If $\bx \in \cR_{\text{cl}}^\eps$ has all components equal to $x$ or $y$, with at least one of each, then one of the following holds:
\begin{enumerate}
\item[(i)] for some integer $m\ge M$, 
$x=\frac{1}{2m}+O(\eps)$ and $y=\frac{1}{m}+O(\eps)$, or
\item[(ii)] there are two components of $\bx$ with sum $< 1-\gamma+\eps$.
\end{enumerate}
\end{lem}

The proof  will be given in the next section.

\begin{lem}\label{lem:GammaNu}
Given $(\gamma,\theta,\nu) \in \cA$ with $\theta>0$, we have
$1-\gamma \le 3\nu$, with equality only for the triples
\[
\bigg( \frac35, \frac15, \frac2{15} \bigg), \qquad \bigg( \frac7{10}, \frac25, \frac1{10} \bigg).
\]
In particular, if \eqref{eq:B} holds for $(\gamma,\theta,\nu)\in\cA$ with $\theta>0$, then $1-\gamma<3\nu$.
\end{lem}
\begin{proof}
Here we write $M=M(\gamma)$.

\medskip

\textbf{Case 1. $1-\gamma \ge \theta+\nu$.}
By \eqref{eq:A1} and Lemma \ref{lem:smallgamma}, $[\theta,\theta+\nu]$ contains $[\frac{1}{2M+1},\frac{1}{M+1}]$,
$[\frac{1}{2M}, \frac{2}{2M+1}]$ or $[\frac{1}{2M-1},\frac{1}{M}]$.  If $M\ge 4$ then all intervals have length $> \frac{1}{3M} \ge \frac{1-\gamma}{3}$.    

  If $M=3$, then $\frac14 < 1-\gamma\le \frac13$ and $[\theta,\theta+\nu]$
  contains $[\frac17,\frac14]$, $[\frac16,\frac27]$ or $[\frac15,\frac13]$.
If  $[\theta,\theta+\nu]$ contains $[\frac16,\frac27]$ or $[\frac15,\frac13]$
then $\nu > \frac19 \ge \frac{1-\gamma}{3}$.
If   $[\theta,\theta+\nu]$ contains $[\frac17,\frac14]$ then
 $\nu \ge \frac14-\frac17 = \frac3{28}$, so if
  $1-\gamma < \frac{9}{28}$ then $\nu > \frac{1-\gamma}{3}$.  
  If $\frac{9}{28} \le 1-\gamma \le \frac13$,
  \eqref{eq:A2} implies that either $\theta+\nu\ge \frac{9}{28}$,
  $1-\theta-\nu \le \frac23$ or $1-\theta \ge \frac{27}{28}$, and in all cases
  $\nu \ge 0.17 > \frac{1-\gamma}{3}$.
  
  If $M=2$, then  $\frac13 < 1-\gamma < \frac12$ so $\theta+\nu<1/2$ and $[\theta,\theta+\nu]$ contains either $[\frac15,\frac13]$
  or $[\frac14,\frac25]$ (recall that elements of $\cA$ have $\gamma>\frac12$).  
  Thus $\nu \ge \frac{2}{15}$ and  $\theta\le \frac{1}{4}$, so we may assume that $1-\gamma\ge \frac{2}{5}$ (since otherwise $1-\gamma<3\nu$). We now invoke \eqref{eq:A2}.  As $3(1-\gamma)>1$, we have $h\in \{1,2\}$.  If $h=1$ then
  $\theta+\nu \ge 1-\gamma$, so $\nu \ge 1-\gamma-\frac{1}{4}> \frac{1-\gamma}{3}$ since $1-\gamma\ge 2/5$.  If $h=2$ then $\theta\le 1-2(1-\gamma)$ so $\nu\ge \frac{1}{3}-\theta\ge 2(1-\gamma)-\frac{2}{3}\ge \frac{1-\gamma}{3}$, with equality if and only if $(\gamma,\theta,\nu)=(\frac{3}{5},\frac{1}{5},\frac{2}{15})$.
\medskip

\textbf{Case 2. $1-\gamma \le \theta, \theta+\nu=\frac12$.}
Let $n$ be the smallest odd integer larger than $M$.  By \eqref{eq:A1}, $[\theta,\frac12]$
contains $\frac{(n-1)/2}{n}$ and thus $\nu \ge \frac{1}{2n}$.  
If $M\ge 5$ then 
\[
\nu \ge \frac{1}{2(M+2)} \ge \frac{5}{14M} > \frac{1}{3M}.
\]
If $M=4$ then $\nu \ge \frac{1}{10} > \frac{1}{3M}$.  When $M=3$ we have
$\nu \ge \frac{1}{10}$. By \eqref{eq:A2}, for some $h\in \NN$, 
$h(1-\gamma)\in[\theta,1-\theta]$.
  If $h=1$ then $\theta=1-\gamma \le \frac13$, so $\nu \ge \frac16 \ge (1-\gamma)/2$.  If  $h\ge 2$, then $1-\theta\ge 2(1-\gamma)$, so $\nu\ge 2(1-\gamma)-\frac12\ge \frac{1-\gamma}{3}$ with equality if and only if $(\gamma,\theta,\nu)=(\frac7{10},\frac35,\frac1{10})$.
$M=2$ is impossible, since then \eqref{eq:A1} implies $\theta \le \frac13$ but $1-\gamma > \frac13$.

\medskip

\textbf{Case 3. $1-\gamma \le \theta, \theta+\nu=1-\theta$.}
Here, $\nu=1-2\theta$ is double that in Case 2.

\medskip

\textbf{Case 4.  $1-\gamma \le \theta, \theta+\nu < \frac12$.}
Use Proposition \ref{prop:ratio}.  This case may also be proven
more quickly by directly applying \eqref{nu-lower} and \eqref{aMk}.
\end{proof}

\begin{lem}\label{lem:small-components}
Let $P=(\gamma,\theta,\nu)\in \cQ_0$ and $T = T_{2,k}(P;\cW)$ for some $\cW$.
If $\bx \in T$, with 
\begin{equation}
\label{small-large}
x_i < \nu \;\; (i\in I), \quad  x_i \ge \nu\;\; (i\in [k]\setminus I).
\end{equation}
Then, for all $W\in \cW$ with $W\subseteq [k]\setminus I$, $W \cup I \in \cW$.
\end{lem}

\begin{proof}
If $W\subseteq [k]\setminus I$, $W\in \cW$ and $W \cup I \not\in \cW$, then 
$|\ssc{\bx}{W}| \le \theta$ and
$|\ssc{\bx}{W\cup I}| \ge \theta+\nu$.  Hence, for some
$I'\subseteq I$, the subsum $|\ssc{\bx}{W \cup I'}|$ lies in $(\theta,\theta+\nu)$,
a contradiction.
\end{proof}

\begin{lem}\label{lem:move-by-less-than-nu}
Let $T=T_{2,k}(P_\eps;\cW)$, $\bx\in T$. Let $J$ be a nonempty subset of $[k]$  such that the numbers $x_j$ for $j\in J$ lie in an open interval of length $\nu-2\eps$. Then $\bx'\in T$, where 
\[
x_i'=\begin{cases}
x_i,\qquad &i\notin J,\\
\tfrac{1}{|J|} |\ssc{\bx}{J}|, &i\in J.
\end{cases}
\]
\end{lem}
\begin{proof}
Let $J=\{j_1,\ldots,j_r\}$ with $x_{j_1} \le \cdots \le x_{j_r}$ and let
$\a = (x_{j_1}+\cdots+x_{j_r})/r$.
If $j,j' \in J$ then $|x_{j} - x_{j'}| < \nu-2\eps$,
 thus if $W$ contains $j$ but not $j'$, then $W\in \cW$ if and only if $W \cup \{j'\} \setminus \{j\} \in \cW$. 
 Repeating this argument, we see that for any subsets $I,I' \subseteq J$
 of the same size, and for $V \subseteq [k] \setminus J$, 
 $V \cup I \in \cW$ if and only if $V \cup I' \in \cW$.
Now let $W \subseteq [k]$, with $L = W \cap J$ of size $\ell$.
If $L$ is empty, then $\ssc{\bx}{W} = \bx'_W$.  If $L$ is nonempty
and $W\in \cW$ (that is, $|\ssc{\bx}{W}| \le \theta+\eps$) then $(W\setminus L) \cup \{ x_{j_{r-\ell+1}},\ldots,x_{j_r} \} \in \cW$
as well.  Since $x_{j_{r-\ell+1}}+\cdots+x_{j_r} \ge \alpha \ell$,
we have $|\bx'_W| \le \theta+\eps$ as well.
Likewise, if $W\not \in \cW$, so that $|\ssc{\bx}{W}| \ge \theta+\nu-\eps$,
then $(W\setminus L) \cup \{ x_{j_1},\ldots,x_{j_\ell} \} \in \cW$,
$x_{j_1}+\cdots+x_{j_\ell} \le \ell \a$ and hence $|\bx'_W| \ge \theta+\nu-\eps$
as well.  Therefore, $\bx' \in T$.
\end{proof}

\subsubsection{Proof of Proposition \ref{prop:ClaimB}}
We prove Proposition \ref{prop:ClaimB} by induction on $k$.
Any $\bx\in \mathbb{R}^k\cap R_2^\eps$ has $x_i \le 1-\gamma+\eps$ for all $i$ and $|\bx|=1$, so certainly the claim is trivially true for $k < 1/(1-\gamma)$
and for $\epszero$ sufficiently small. For the purposes of induction, we now assume that Proposition \ref{prop:ClaimB} holds whenever $\bx\in \mathbb{R}^k\cap \cR_{\text{cl}}^\eps$ for $k<k_0$, and we wish to establish it for $\bx\in \mathbb{R}^{k_0}\cap R_2^\eps$. 

Let $\bx \in T=T_{2,k_0}(P_\eps;\cW)$.
By repeatedly applying Lemma \ref{lem:move-by-less-than-nu},
we produce a vector $\by \in T$ that satisfies
$|y_i-y_j| \ge \nu-2\eps$ whenever $y_i\ne y_j$
(each invocation of Lemma \ref{lem:move-by-less-than-nu} produces a 
vector with fewer distinct components, so the number of iterations is finite).
By Lemma \ref{lem:GammaNu}, we have $1-\gamma < 3\nu$. For small enough $\eps$ we have $1-\gamma+ \eps < 3(\nu-2\eps)$, and so $\by$ has at most 3 distinct coordinate values. If all nonzero coordinates have the same value then Proposition \ref{prop:ClaimB} holds,
 since then \eqref{eq:sM-alt} implies that all nonzero components must be of the form $\frac{1}{m}$ for some $m\in \sM$. Therefore we may assume that
there are at least two distinct nonzero coordinates in $\by$. 
Since $\by\in \cR_{\text{cl}}^\eps$, all subsums of $\by$ lie in $J_1\cup J_2 \cup J_3$
where
\[
J_1 =[0,\theta+\eps], \quad J_2 = [\theta+\nu-\eps,1-\theta-\nu+\eps], \quad
J_3=[1-\theta-\eps,1].
\]

First we consider the case when there are three distinct nonzero coordinates $v_1,v_2,v_3$ in $\by$, where
 $0< v_3 < v_2 < v_1 \le 1-\gamma+\eps$, occurring 
  $n_1,n_2$ and $n_3$ times, respectively.  
Then $v_1\in (2(\nu-2\eps),1-\gamma+\eps]$,
$v_2\in [\nu-2\eps,1-\gamma-\nu+3\eps)$ and $v_3\in[0,\nu-2\eps)$.
 Clearly
 \begin{equation}\label{n1n2n3}
 n_1+n_2+n_3 = k_0, \qquad n_1 v_1 + n_2 v_2 + n_3 v_3=1.
 \end{equation}
 In light of Lemma \ref{lem:small-components}, and since $v_3 < \nu-2\eps$,
 for any $0\le h_1 \le n_1$ and $0\le h_2 \le n_2$, $h_1 v_1 + h_2 v_2$
 and $h_1 v_1 + h_2 v_2 + n_3 v_3$ lie in the same interval $J_i$. 
 Thus, if
 \begin{equation}\label{adjust}
 n_3 v_3 \le n_1(1-\gamma+\eps-v_1) + n_2 (1-\gamma+\eps-v_2),
 \end{equation}
 we may replace $v_3$ by 0, and $v_2,v_3$ by larger values, each $\le 1-\gamma+\eps$,
 while retaining \eqref{n1n2n3} and staying in the same polytope $T$.
 If \eqref{adjust} is false, we may replace $v_1$ and $v_2$ by $1-\gamma+\eps$
 and $v_3$ by 
 \[
 v_3 - \frac{n_1(1-\gamma+\eps-v_1) + n_2 (1-\gamma+\eps-v_2)}{n_3},
 \]
  which is positive, again retaining \eqref{n1n2n3} and staying in the same polytope.
  In either case, the new vector has at most two distinct, non-zero coordinates.
   The new vector can be adjusted further using Lemma \ref{lem:move-by-less-than-nu}
 to make the nonequal components $\nu-2\eps$ separated, without increasing the 
 number of distinct nonzero coordinates. As shown earlier, we are done if all nonzero coordinates in the final vector are equal.
Therefore, it suffices to prove the Proposition
for vectors $\bx$ that have exactly two nonzero components, $v_1$ and $v_2$,
occurring $n_1$ and $n_2$ times, respectively, and with
 $0<v_2 < v_1\le 1-\gamma+\eps$ and $v_1 - v_2 \ge \nu-2\eps$.
 
If $v_2 < \nu-2\eps$, we may similarly decrease $v_2$ and increase $v_1$
until either $v_2=0$ or $v_1=1-\gamma+\eps$.  In the former case,
the new vector has only one distinct nonzero component and we are done.
In the latter case, Lemma \ref{lem:R1-properties} implies
that $v_2 n_2 \le 1-\gamma-\nu+3\eps$ (that lemma is written for $\cR^\eps$,
but applies to $\cR_{\text{cl}}^\eps$ as well by \eqref{eq:R1R2}).
This implies that $1 = n_1 v_1 + n_2 v_2 < (n_1+1)(1-\gamma) \le \frac{n_1+1}{M}$,
thus $n_1 \ge M$.  We also have that $1 \ge n_1 (1-\gamma+\eps) > \frac{n_1}{M+1}$,
so $n_1<M+1$.   Therefore, $n_1=M$.
As the new vector has all subsums avoiding $[\theta+\eps,\theta+\nu-\eps]$, 
this vector violates condition \eqref{eq:A2}, which holds for $P_\eps$ for sufficiently small $\eps>0$, since \eqref{eq:B} holds for $P$. 

We may therefore assume that $v_2 \ge \nu-2\eps$.
Using Lemma \ref{lem:move-by-less-than-nu} again, we may assume that 
$v_1 \ge v_2 + (\nu-2\eps)$ as well.  Thus, we have
\[
\nu-2\eps \le v_2 \le v_1 - (\nu-2\eps), \;\;\; v_1 \le 1-\gamma+\eps.
\]
Let $\bx'$ denote the corresponding vector, and recall that $\bx,\bx'$
lie in the same polytope $T$.
Apply Lemma \ref{lmm:Sizes} to the subvector $\by\in R_2^\eps$ with 
$n_1+n_2$ components, $n_1$ components equal to $v_1$ and $n_2$
components equal to $v_2$. 
The vector $\by$ evidently satisfies the conditions in the hypothesis of
Lemma \ref{lmm:Sizes}.  Since $v_1+v_2 \ge 3(\nu-2\eps) > 1-\gamma+\eps$,
if conclusion (ii) does not hold than we must have 
$v_2 < (1-\gamma+\eps)/2$ and $n_2 \ge 2$.  Thus, we either have

\begin{enumerate}
\item[(i)] $v_2 < (1-\gamma+\eps)/2$ and $n_2 \ge 2$, or
\item[(ii)] $v_2 = \frac{1}{2M}+O(\eps)$ and $v_1=\frac{1}{M}+O(\eps)$.
\end{enumerate}

We claim that
\begin{equation}\label{v1v2-claim}
\exists \, m\in \NN\, : \; v_2=\frac{1}{2m}+O(\eps), \;\; 
v_1 = \frac{1}{m}+O(\eps).
\end{equation}

In case (ii), the claim \eqref{v1v2-claim} follows with $m=M$.
Now suppose that (i) holds. We now consider the point $\bx''$ where two copies of $v_2$ in $\bx'$ are replaced by one copy of $2v_2$. 
That is, $\bx''$ has $n_1$ copies of $v_1$,
one copy of $2v_2$ and $n_2-2$ copies of $v_2$.
 This is clearly still in $\cR_{\text{cl}}^\eps$, but now $\bx''$ has $k_0-1$ components. By the induction hypothesis, Proposition \ref{prop:ClaimB}
holds for $\bx''$.  By Lemma \ref{lem:BtoA},
Proposition \ref{prop:ClaimA} holds for $\bx''$.
  As $v_1-v_2 \ge \nu-2\eps$,
if condition (1) in Proposition \ref{prop:ClaimA} holds then $n_2=2$ and $v_1=1/m+O(\eps)$ and $2v_2=1/m+O(\eps)$, giving the claim
\eqref{v1v2-claim}. 
If condition (2) holds, we note that $4v_2$ is much larger than $v_1$,
hence we must have $n_2-2>0$, $v_2 = \frac{1}{2m}+O(\eps)$ and
hence $2v_2$ and $v_1$ are each $\frac{1}{m}+O(\eps)$.  This gives the claim
\eqref{v1v2-claim}. 
If condition (3) holds, then, since $x_{j+1},x_{j+2}$ and $x_{j+3}$
are distinct, $\bx''$ has three distinct components, and hence $n_2-2>0$
and those three components are $v_2,2v_2$ and $v_1$.  Furthermore, $v_2$
is the smallest of the three.
Thus, $v_2 + \min(2v_2,v_1) = \max(2v_2,v_1)+O(\eps)$ and hence either
$3v_2=v_1+O(\eps)$ or $v_2+v_1=2v_2+O(\eps)$.  The former
is impossible for small $\eps>0$ by $1-\gamma<3\nu$, and the latter implies
$v_1=v_2+O(\eps)$, which is false. Thus condition (3) cannot hold, and so in all cases we have proven the claim \eqref{v1v2-claim}.

Let $m$ be the constant guaranteed by \eqref{v1v2-claim}. For every $1\le a\le 2m-1$, 
$\bx'$ has a subsum of the form $\frac{a}{2m}+O(\eps)$.
If $2m\not\in \sM$ then for some $a\in \NN$, $\frac{a}{2m}\in (\theta,\nu)$
and for small enough $\eps$ this is a contradiction since $\bx$ has no subsum in $(\theta,\nu)$. Thus $2m\in \sM$. Let $\bz$ be the vector formed by replacing each component $v_2$ of $\bx'$ with $\frac{1}{2m}$
and replacing each component $v_1$ of $\bx'$ with $\frac{1}{m}$. Since $2m\in \sM$, 
there is no rational of the form $\frac{a}{2m}$ in $(\theta,\theta+\nu)$ and hence
for $\eps$ small enough, we have $\bz\in \cR_{\text{cl}}^\eps$. Since $\bz=\bx'+O(\eps)$ and the  polytopes are disconnected from one another, for $\eps$ small enough this vector lies in the same polytope as $\bx'$. This concludes the proof of Proposition \ref{prop:ClaimB} when $k=k_0$.

\subsubsection{Proof of Lemma \ref{lmm:Sizes}}

When $1-\gamma \le \theta < \theta+\nu < \frac12$
 and $P=P_0\ne (\frac56,\frac14,\frac1{12})$, the lemma follows
from Proposition \ref{prop:A123-epsilon-tweak}, part ($\beta$),
which states that (A3) holds for $P_\eps$ and $0<\eps\le \epszero$.
Recalling the definition of Hypothesis (A3), we see that 
there are no vectors satisfying the hypotheses 
of Proposition \ref{prop:A123-epsilon-tweak} which fail conclusion (ii).

When $P=(\frac56,\frac14,\frac1{12})$, $M=6$ and if $x,y$ satisfy the hypotheses
of Proposition \ref{prop:A123-epsilon-tweak}, then we have $y-x \ge \nu-2\eps=\frac{1}{12}-2\eps$ and hence $x=\frac{1}{12}+O(\eps)$ and $y=\frac16+O(\eps)$, as required.

\begin{proof}[Proof of Lemma \ref{lmm:Sizes} for the case $\nu=\frac12-\theta$
and $1-\gamma\le \theta$]
We will prove more, that conclusion (i) in Lemma \ref{lmm:Sizes} must always hold.
If $\nu > \frac{1-\gamma}{2}$, then for small enough $\eps$, the lemma holds
vacuously. 
Lemma \ref{lem:nu-lower-for 1/2-theta} furnishes a lower bound on $\nu$,
 Thus, we may assume that
\begin{equation}
\label{nu-upper}
\frac{1}{2M} \ge \frac{1-\gamma}{2} \ge \nu \ge
\begin{cases}
\frac{1}{2M+4} & M\text{ odd } \\
\frac{1}{2M+2} & M\text{ even. } 
\end{cases}
\end{equation}
In particular, we cannot have $M=2$ since then $\nu \ge \frac16$ and hence $\theta \le \frac13<1-\gamma$.
Thus, $M\ge 3$.
Now suppose $\bx$ is a vector satisfying the hypotheses of Lemma \ref{lmm:Sizes}.
By Lemma \ref{lem:GammaNu}, for small enough $\eps>0$ we have
\begin{equation}\label{gap}
1-\gamma < 3\nu - 100\eps,
\end{equation}
say.  Suppose $\bx$ takes exactly two distinct values, $x$ and $y$, where
\begin{equation}
\label{xy-2}
\nu - 2\eps \le x \le y- (\nu-2\eps), \qquad x + \nu-2\eps \le y \le 1-\gamma+\eps.
\end{equation}
Suppose that $\bx$ contains $b$ components equal to $x$ and $c$ components equal to
$y$, so that
\[
bx+cy =1 ,\quad b\ge 1, \quad c\ge 1.
\]
We will establish the lemma by successively proving a number of claims:
\begin{enumerate}
\item[(a)] $\bx$ has a subsum in $[\frac12-\eps,\frac12+\eps]$.
\item[(b)] $x \le \frac29$ and $|2x-y| < \nu - 3\eps$.
\item[(c)] $|2x-y| \le  10\eps$.
\end{enumerate}

\textbf{Proof of claim (a).}
Suppose the claim is false, so all subsums of $\bx$ are $ \le \theta+\eps$
or $\ge 1-\theta-\eps$.  By \eqref{gap}, the excluded interval has length
 $2\nu - 4\eps > 1-\gamma-\nu + 4\eps > x$.
Thus, if $d$ is the largest integer $\le c$ such that $dy \le \theta+\eps$,
then  $dy + bx \le \theta+\eps$ and $(d+1)y \ge 1-\theta-\eps$ (in particular, $d<c$).
Hence, $y-bx > 1-\gamma-\nu+4\eps$, and since $x\ge \nu-2\eps$ we  get 
$y>1-\gamma+2\eps$, a contradiction.

\medskip

\textbf{Proof of claim (b).}
Since $M\ge 3$,
by \eqref{xy-2} and \eqref{gap},
\[
x\le 1-\gamma-\nu + 3\eps < \frac{2(1-\gamma)}{3} \le \frac{2}{3M} \le \frac29.
\]
 Next, using \eqref{xy-2} and \eqref{gap} again,
\[
2x-y \le y-2\nu + 4\eps \le 1-\gamma-2\nu+5\eps \le \nu - 3\eps
\]
and 
\[
y-2x \le 1-\gamma+\eps - 2x \le 1-\gamma-2\nu+5\eps \le \nu-3\eps.
\]
This proves the second part.

\medskip

\textbf{Proof of claim (c).}
This has a longer proof.
By Claim (a), there are non-negative integers $b',b'',c',c''$ with
$b'+b''=b,c'+c''=c$ and with $b'x+c'y$ and $b'' x + c'' y$ both in 
$\cK := [\frac12-\eps,\frac12+\eps]$.

\textbf{Case 1. $\max(c',c'')=c$.}  Without loss of generality $c'=c$ and $c''=0$.  By the first part of Claim (b) and $b'' x\in \cK$, we have $b'' \ge 3$.  Hence, by the second part of Claim (b),
$(b'+2)x+(c-1)y$ is also in $\cK$, and it follows that
$|2x-y|=|(b'+2)x+(c-1)y-(b'x+cy)| \le 2\eps$.

\textbf{Case 2. $c'<c,c''<c,\max(b',b'')\ge 2$.}  Without loss of generality $b'=2$.
Again by the second part of Claim (b), $(b'-2)x+(c'+1)y \in \cK$, and it 
follows that $|2x-y| \le 2\eps$.

\textbf{Case 3.  $c'<c,c''<c,b=1$.}  Without loss of generality $b'=1,b''=0$.  Then both
$x+c'y$ and $c'' y$ are in $\cK$, thus $|x+(c'-c'')y|  \le 2\eps$.
If $c'=c''$ then $|x|\le 2\eps$, contradicting \eqref{xy-2}, and if $c'\ne c''$
then $2\eps \ge y-x$, again contradicting \eqref{xy-2} and \eqref{nu-lower}.

\textbf{Case 4.  $b=2,b'=b''=1$, $c'=c''$.}  Here $x+ (c/2)y=\frac12$.
By Claim (b), $2x < \frac49 < \frac12-\eps$, thus $2x \le \theta+\eps$
and $cy \ge 1-\theta-\eps$. 
Now (B) implies that $h(1-\gamma)\in [\theta,1-\theta)$ for some positive integer $h$.
For small enough $\eps$ we have
\[
hy \le h(1-\gamma+\eps) < 1-\theta-\eps \le cy,
\]
so that $h\le c-1$.  It follows that $hy$ is a subsum of $\bx$, and thus
$hy \in \cK$ or $hy\le \theta+\eps$.   Also, $(c/2)y=\frac12-x\le \theta+\eps$,
thus if $hy\in \cK$ then $h \ge \frac{c}{2}+1$.  This implies that $(c+2)y \le 1+2\eps$.  By \eqref{xy-2}
we then have $1=cy+2x=(c+2)y - 2(y-x) \le 1-2\nu +6\eps < 1$, a contradiction.
Therefore, $hy \le \theta+\eps$ and $y\le \frac{\theta+\eps}{h}$.

If $M$ is even, let $h=\frac{M}{2}$.  This works since 
$(M/2)(1-\gamma)\in (\frac{M/2}{M+1},\frac12] \subseteq I$
using \eqref{nu-upper}.  By \eqref{nu-upper} again,
\[
y \le \frac{2}{M}(\theta+\eps) = \frac{1-2\nu+2\eps}{M} \le \frac1{M+1} + \frac{2\eps}{M} \le 2\nu + \frac{2\eps}{M} \le 2\nu + \eps.
\]
Using \eqref{xy-2} again, we have $2\nu - 4\eps \le y \le 2\nu + \eps$.
This implies that
 $\nu-2\eps \le x \le y-(\nu-2\eps) \le \nu+3\eps$, 
and, consequently, $|2x-y| \le 10\eps$.

Now suppose that $M$ is odd.
If  $1-\gamma=2\nu$, \eqref{xy-2}
likewise implies that $2\nu-4\eps \le y\le 2\nu+\eps$ and again $|2x-y| \le 10\eps$.
Now assume $1-\gamma>2\nu$.  We claim that we may take $h=\frac{M+1}{2}$.
Now $\frac{1}{M+1}<1-\gamma\le \frac{1}{M}$ implies that
$\frac{M+1}{2}(1-\gamma) > \frac12 > \frac{M-1}{2}(1-\gamma)$,\
hence (B) must hold with $h=\frac{M-1}{2}$ or $\frac{M+1}{2}$.
 Furthermore,
$\frac{M+1}{2}(1-\gamma) - \frac12 \le \frac12 - \frac{M-1}{2}(1-\gamma)$,
with equality in the last expression if and only if $1-\gamma=\frac{1}{M}$.
Thus, if $1-\gamma<\frac{1}{M}$ then $\frac{M+1}2(1-\gamma)$
is closer to $\frac12$ than $\frac{M-1}{2}(1-\gamma)$ is, and so
$h=\frac{M+1}{2}$ works for (B).
If $1-\gamma=\frac{1}{M}$ then $h=\frac{M+1}{2}$ works unless
$\theta = \frac{M-1}{2}(1-\gamma) = \frac12 - \frac1{2M}$,
which implies that $\nu=\frac{1}{2M}=\frac{1-\gamma}{2}$.

  By Lemma \ref{lem:nu-lower-for 1/2-theta},
\[
y \le \frac{2}{M+1}(\theta+\eps) = \frac{1-2\nu+2\eps}{M+1} \le \frac{1}{M+2} + 
\frac{2\eps}{M+1} \le 2\nu+\eps,
\]
which, once again, implies that $|2x-y|\le 10\eps$.
This concludes the proof of Claim (c).

\medskip

Now we conclude the argument, using
Claim (c).  We have $1=bx+cy=(b+2c)x+O(\eps)$, which implies that
\[
x = \frac{1}{b+2c}+O(\eps), \qquad y = \frac{2}{b+2c}+O(\eps).
\]
The fact that $y\le 1-\gamma+\eps \le \frac{1}{M}+\eps$ implies that 
$b+2c \ge 2M$, with equality only possible if $1-\gamma=\frac{1}{M}$.
If $b+2c=2\ell+1$ is odd, then there is a subsum of $\bx$ equal to
$\frac{\ell}{2\ell+1} + O(\eps) = \frac12 - \frac{1}{4\ell+2}+O(\eps)$.
This must be $\le \frac12-\nu+\eps$, and thus, by \eqref{nu-upper},
we have $4\ell+2\le 2M+4$, so that $b+2c\le M+2$.  As $M\ge 3$, this
contradicts $b+2c\ge 2M$ just established.
Therefore, $b+2c$ is even, hence $b$ is even and
(i) holds with $m=b/2+c$.
\end{proof}

\begin{proof}[Proof of Lemma \ref{lmm:Sizes} when $1-\gamma \ge \theta+\nu$]
We may assume that $\nu \le \frac{1-\gamma}{2}$, for otherwise Lemma \ref{lmm:Sizes}
follows vacuously if $\epszero$ is sufficiently small (there are no vectors satisfying the conditions).

By Lemma \ref{lem:smallgamma}, we have that $[\theta,\theta+\nu]$
contains one of the intervals $[\frac{1}{2M+1},\frac{1}{M+1}]$, 
 $[\frac{1}{2M},\frac{2}{2M+1}]$ or $[\frac{1}{2M-1},\frac{1}{M}]$.
In particular,
\[
\nu > \frac{1}{2M+4}.
\]

By \eqref{eq:B}, for some positive integer $h$ we have $h(1-\gamma)\in [1-\theta-\nu,1-\theta)$.  But $(M+1)(1-\gamma)>1$ and $(M-1)(1-\gamma)\le 1 - \frac{1}{M} \le 1-\theta-\nu$,
with equality if and only if $\theta+\nu=1-\gamma=\frac{1}{M}$.
Also, $M(1-\gamma) \ge M(\theta+\nu) \ge 1-(\theta+\nu)$.
Hence, \eqref{eq:B} is equivalent to the statement that
 either $1-\gamma=\theta+\nu=\frac{1}{M}$ or $M(1-\gamma)<1-\theta$.

Now suppose that $\bx$ is a vector satisfying the conditions of Lemma \ref{lmm:Sizes},
with two distinct components $x<y$ with $y-x\ge \nu-2\eps$.
  Also assume that conclusion (ii)
fails, that is, every pair of components of $\bx$ has sum $\ge 1-\gamma+\eps$.
As the subsums of $\bx$ avoid $(\theta+\eps,\theta+\nu-\eps)$,
$x$ and $y$ lie in $[\nu-2\eps,\theta+\eps] \cup [\theta+\nu-\eps,1-\gamma+\eps]$.

\textbf{Case 1: $1-\gamma=\frac{1}{M}=\theta+\nu$.}
In this case $M\ge 3$, $\theta\le \frac{1}{2M-1}$ and thus
$\nu \ge \frac{M-1}{M(2M-1)} > \frac23 \theta$. Hence, for $\epszero$ small enough,
$x$ and $y$ cannot both lie in $[\nu-2\eps,\theta+\eps]$. 
As $[\theta+\nu-\eps,1-\gamma+\eps]$ has length $2\eps$, $x$ and $y$ cannot both lie
in this interval. This implies that
\[
y\in\Big[\frac{1}{M}-\eps,\frac{1}{M}+\eps\Big],\quad x\in \Big[\frac{1}{M}-\frac{1}{2M-1}-2\eps,\frac{1}{2M-1}+\eps\Big].
\]
Let $\bx$ have $b$ copies of $x$ and $c$ copies of $y$. Since $b\ge 1$ and $1=bx+cy$, we must have $c\le M-1$, which in turn implies that $b\ge 2$. By our assumption that (ii) fails, this means that $x\ge \frac{1-\gamma+\eps}{2}\ge \frac{1}{2M}$. We wish to show that $x=\frac{1}{2M}+O(\eps)$. If $x>\frac{1}{2M}+M\eps$ then 
\[
1=bx+cy> \frac{b}{2M} + \frac{c}{M} + (2Mb-c)\eps > \frac{2c+b}{2M}
\]
so $b+2c\le 2M-1$. But then
\[
1=bx+cy\le \frac{c}{M}+\frac{b}{2M-1}+2M\eps\le 1-\frac{c}{M(2M-1)}+2M\eps
\]
which is a contradiction if $\eps$ is small enough. Thus we must have that $\frac{1}{2M} \le x < \frac{1}{2M}+M\eps$, as desired.  Hence, conclusion (i)
in Lemma \ref{lmm:Sizes} holds with $m=M$.

\textbf{Case 2: $1-\gamma < \frac{1-\theta}{M}$.}
Here $M\ge 2$, and for $\eps$ small enough,
\[
1-\gamma \le \frac{1-\theta}{M} - 3M\eps.
\]
If $\theta> \frac{1}{2M+1}$ then  $[\theta,\theta+\nu]$ must contain $[\frac{1}{2M},\frac{2}{2M+1}]$ or $[\frac{1}{2M-1},\frac{1}{M}]$, thus we have $\theta+\nu\ge \frac{2}{2M+1}$. This implies that
\[
\frac{2}{2M+1} \le \theta+\nu \le 1-\gamma < \frac{1}{M} \bigg( 1 - \frac{1}{2M+1} \bigg) = \frac{2}{2M+1},
\]
a contradiction.  Therefore 
$\theta\le \frac{1}{2M+1}$ and $\theta+\nu \ge \frac{1}{M+1}$.
In particular, $\nu \ge \frac{M}{(2M+1)(M+1)} \ge \frac23 \theta$,
and again this implies that $x,y$ cannot both be in $[\nu-2\eps,\theta+\eps]$
and hence that $y\ge \theta+\nu-\eps\ge \frac{1}{M+1}-\eps$.
Also, $[\theta+\nu-\eps,1-\gamma+\eps]$ has length at most
\[
\frac{1-\theta}{M} - \frac1{M+1} \le \frac{1+\nu-\frac{1}{M+1}}{M}-\frac{1}{M+1}=\frac{\nu}{M},
\]
which implies that $x,y$ cannot both be in $[\theta+\nu-\eps,1-\gamma+\eps]$
and hence $x\le \theta+\eps$.

Let $\bx$ have $b$ copies of $x$ and $c$ copies of $y$.  
If $c\ge M+1$ then $bx+cy \ge \nu-2\eps + (M+1)(\frac{1}{M+1}-\eps)>1$, a 
contradiction.  Therefore, $c\le M$ and 
\begin{align*}
1=bx+cy &\le b (\theta+\eps) + c \Big( \frac{1-\theta}{M}-3M\eps \Big) \\ 
&< \Big(b- \frac{c}{M}\Big)\theta + \frac{c}{M} \\ &\le 
\Big(b- \frac{c}{M}\Big)\frac{1}{2M+1} + \frac{c}{M} =\frac{b+2c}{2M+1},
\end{align*}
thus $b+2c\ge 2M+2$.  Since $c\le M$, we have $b\ge 2$
and, by assumption,
this implies that $x \ge \frac{1-\gamma+\eps}{2}>\frac{1}{2M+2}$ for small enough $\eps$ (since $1-\gamma>\frac{1}{M+1}$). We also have $y\ge \theta+\nu\ge 1/(M+1)$, hence
\[
1=bx+cy > \frac{b+2c}{2M+2} \ge 1,
\]
a contradiction.  This completes the proof of Lemma \ref{lmm:Sizes} in the case
$1-\gamma \ge \theta+\nu$.
\end{proof}


%

\bigskip

\noindent
\textbf{Acknowledgements.}
This project started in 2017 while the second author was a Member of the Institute for Advanced Study and visited the University of Illinois.  Special thanks go to Enrico Bombieri
for helpful discussions and his subtle encouragement for the two authors to work together.
Part of this work was accomplished when both authors visited the University of Montreal 
in 2018 and the Institute for Advance Study in 2022, and while the first author was
a Visiting Fellow of Magdalen College, Oxford, in 2019.  The authors thank Andrew
Granville, Ben Green and  Peter Sarnak for these invitations.
The first author also enjoyed the hospitality of the Mathematical Institute of Oxford University
and the Institute of Mathematics and Informatics of the Bulgarian Academy of Sciences.
The authors thank Denka Kutzarova for assistance with the proof of the
geometric lemmas \ref{lem:covering} and \ref{lem:R1-nbhd-Q}.

The first author was supported by National Science Foundation grants
DMS-1501982, DMS-1802139, and DMS-2301264, and the second author was
was supported  by the European Research Council (ERC) under the European Union’s
 Horizon 2020 research and innovation programme (grant agreement No 851318).



\begin{thebibliography}{99}

\bibitem{BakerHarmanPintz} R. C. Baker, Glyn Harman and Janos Pintz, \emph{The difference between consecutive primes. {II}}.
Proceedings of the London Mathematical Society (3), {\bf 83} (2001), no. 3, 532--562.

\bibitem{BI17} R. C. Baker and A. J. Irving,
{\it Bounded intervals containing many primes.}
Mathematische Zeitschrift, {\bf 286} (2017), no. 3-4, 821--841.

\bibitem{Bombieri}  Enrico Bombieri, {\it The asymptotic sieve}, Mem. Acad. Naz. 
dei XL, {\bf 1/2} (1976), 243--269.

\bibitem{BFI} Enrico Bombieri, John B. Friedlander, and Henryk Iwaniec, 
{\it Primes in arithmetic progressions to large moduli.},
Acta Mathematica, {\bf 156}, (1986), no. 3-4, 203--251.

\bibitem{Brezis} Haim Br\'ezis,
{\it  Functional analysis, Sobolev spaces and partial differential equations.}
Springer-Verlag,  2011.

\bibitem{Brunner} Hermann Brunner, \emph{Volterra integral equations},
Cambridge University Press, 2017.

\bibitem{DFI} William Duke, John B. Friedlander and Henryk Iwaniec,
\emph{Equidistribution of roots of a quadratic congruence to prime moduli.}
Annals of Mathematics (2) {\bf 141} (1995), no.2, 423--441.

\bibitem{Ford} Kevin Ford, {\it On Bombieri's Asymptotic Sieve},
Transactions of the American Mathematical Society {\bf 357} (2004), 1663--1674.

\bibitem{FI98} John Friedlander and Henryk Iwaniec,
\emph{The polynomial $X^2+Y^4$ captures its primes.}
Annals of Mathematics (2) {\bf 148} (1998), no. 3, 945--1040. 

\bibitem{Opera} John Friedlander and Henryk Iwaniec,
\emph{Opera de Cribro}, American Mathematical Socirty, 2009.

\bibitem{GK} S. W. Graham and G. Kolesnik,
\emph{van der Corput's Method of Exponential Sums},
London Mathematical Society Lecture Notes Series {\bf 126},
Cambridge University Press, 1991.

\bibitem{Harman} Glyn Harman, {\it Prime detecting sieves},
London Mathematical Society, 2006.

\bibitem{H-B} D. R. Heath-Brown,
 \emph{Prime numbers in short intervals and a generalized Vaughan identity},
Canadian Journal of  Mathematics {\bf 34} (1982), no. 6, 1365--1377.

\bibitem{HB01} D. R. Heath-Brown,
\emph{Primes represented by $x^3+2y^3$},
Acta Mathematica {\bf 186} (2001), no.1, 1--84.

\bibitem{Ivic} Aleksandar Ivi\'c,
\emph{The Riemann zeta-function.  Theory and Applications},
John Wiley \& Sons, New York, 1985. 

\bibitem{IwaniecKowalski}
Henryk Iwaniec and Emmanuel Kowalski, \emph{Analytic number theory},
American Mathematical Society, 2004.

\bibitem{Jia} Chaohua Jia, {\it on the distribution of $\a p$ modulo 1 (II)},
Science in China A {\bf 43}, no. 7\, (2000), 703--721.

\bibitem{Lassak} Marek Lassak,
\emph{Covering the boundary of a convex set by tiles}, 
Proceedings of the American Mathematical Society {\bf 104}, no. 1 (1988), 269--272.

\bibitem{Linnik}  Yuri V. Linnik,
\emph{The dispersion method in binary additive problems.}
Translated by S. Schuur,
American Mathematical Society, Providence, RI, 1963. x+186pp.
\newblock English translation of \emph{The dispersion method in binary additive problems}, in Russian,
Izdat. Leningrad. Univ., Leningrad, 1961. 208 pp.

\bibitem{MDP} James Maynard,
\emph{Primes with restricted digits}. Inventiones Mathematicae {\bf 217} (2019), no. 1, 127--218. 

\bibitem{NormForms} James Maynard,  \emph{Primes represented by incomplete norm forms}.
Forum Mathematics Pi, {\bf 8} (2020), e3, 128 pages.

\bibitem{MaynardI} James Maynard,  \emph{Primes in arithmetic progressions to large moduli I: fixed residue classes}.  Memoirs of the American Mathematical Society, \textit{to appear}.

\bibitem{Merikoski} Jori Merikoski,
\emph{The polynomials $X^2+(Y^2+1)^2$  and $X^2+(Y^3+Z^3)^2$ also capture their primes}, 
Proceedings of the London Mathematical Society (3) {\bf 127} (2023), no. 4, 1057--1133.
 
\bibitem{Pitt} Nijel J. E. Pitt,
\emph{On an analogue of {T}itchmarsh's divisor problem for holomorphic cusp forms}.
Journal of the American Mathematical Society {\bf 26} (2013), no. 3, 735--776.
 
\bibitem{Polymath} D. H. J. Polymath,
\emph{New equidistribution estimates of {Z}hang type}. 
Algebra and Number Theory {\bf 8} (2014), no. 9, 2067--2199.
 
\bibitem{Rockafellar} R. Tyrrell Rockafellar,
\emph{Convex Analysis}, Princeton Univ. Press, 1970.
 
\bibitem{SU} Peter Sarnak and Adri\'an Ubis,
\emph{The horocycle flow at prime times.}
Journal de Math\'ematiques Pures et Appliqu\'ees. Neuvi\`eme S\'erie
 (9) {\bf 103} (2015), no.2, 575--618.
 
\bibitem{Selberg} Atle Selberg, \emph{The general sieve-method and its place in prime number theory.} Proceedings of the International
Congress of Mathematicians, Cambridge, Mass., 1950, vol. 1, pp. 286--292. 
American Mathematical Society, Providence, Rhode Island, 1952.

\bibitem{Tenenbaum-book}
G\'erald Tenenbaum,
\emph{Introduction to Analytic and Probabilistic Number Theory}, 3rd edition,
American Mathematical Society, 2015

\bibitem{TZ}
Jesse Thorner and Asif Zaman, \emph{Refinements to the prime number theorems
for arithmetic progressions}, Mathematische Zeitschrift {\bf 306} (2024),
article 54.
 
\bibitem{Zhang}
Yitang Zhang, \emph{Bounded gaps between primes},
Annnals of Mathematics (2), {\bf 179} (2014), no. 3, 1121--1174.
 
\end{thebibliography}
\end{document}